\documentclass[12pt]{article}%
\usepackage[colorlinks,bookmarks=true]{hyperref}
\usepackage{amsmath}
\usepackage{graphicx}
\usepackage{amsfonts}
\usepackage{amssymb}%
\providecommand{\U}[1]{\protect\rule{.1in}{.1in}}
\newtheorem{theorem}{Theorem} [section]

\newtheorem{corollary}{Corollary}[section]

\newtheorem{definition}{Definition} [section]

\newtheorem{lemma}{Lemma}[section]

\newtheorem{proposition}{Proposition} [section]
\newtheorem{remark}{Remark} [section]

\newenvironment{proof}[1][Proof]{\textbf{#1.} }{\ \rule{1em}{1em}}
\numberwithin{equation}{section}

\newcommand{\ep}{\epsilon}

\newcommand{\BFR}{\mathbf{R}}

\newcommand{\BFX}{\mathbf{X}}
\newcommand{\BFA}{\mathbf{A}}

\DeclareMathOperator{\graph}{graph}
\begin{document}

\title{Instability, index theorem, and exponential trichotomy for Linear Hamiltonian PDEs}
\author{Zhiwu Lin and Chongchun Zeng\\School of Mathematics\\Georgia Institute of Technology\\Atlanta, GA 30332, USA}
\date{}
\maketitle

\begin{abstract}
Consider a general linear Hamiltonian system $\partial_{t}u=JLu$ in a Hilbert
space $X$. We assume that$\ L: X \to X^{*}$ induces a bounded and symmetric
bi-linear form $\left\langle L\cdot,\cdot\right\rangle $ on $X$, which has
only finitely many negative dimensions $n^{-}(L)$. There is no restriction on
the anti-self-dual operator $J: X^{*} \supset D(J) \to X$. We first obtain a
structural decomposition of $X$ into the direct sum of several closed
subspaces so that $L$ is blockwise diagonalized and $JL$ is of upper
triangular form, where the blocks are easier to handle. Based on this
structure, we first prove the linear exponential trichotomy of $e^{tJL}$. In
particular, $e^{tJL}$ has at most algebraic growth in the finite
co-dimensional center subspace. Next we prove an instability index theorem to
relate $n^{-}\left(  L\right)  $ and the dimensions of generalized eigenspaces
of eigenvalues of$\ JL$, some of which may be embedded in the continuous
spectrum. This generalizes and refines previous results, where mostly $J$ was
assumed to have a bounded inverse. More explicit information for the indexes
with pure imaginary eigenvalues are obtained as well. Moreover, when
Hamiltonian perturbations are considered, we give a sharp condition for the
structural instability regarding the generation of unstable spectrum from the
imaginary axis. Finally, we discuss Hamiltonian PDEs including dispersive long
wave models (BBM, KDV and good Boussinesq equations), 2D Euler equation for
ideal fluids, and 2D nonlinear Schr\"{o}dinger equations with nonzero
conditions at infinity, where our general theory applies to yield stability or
instability of some coherent states.

\end{abstract}
\tableofcontents

\section{Introduction}

In this paper, we consider a general linear Hamiltonian system
\begin{equation}
\partial_{t}u=JLu,\ u\in X \label{E:Hamiltonian}%
\end{equation}
in a real Hilbert space $X$. We assume that the operator $J:X^{\ast}\supset
D(J)\rightarrow X$ satisfies $J^{\ast}=-J$ and $L:X\rightarrow X^{\ast}$ is
bounded and satisfies $L^{\ast}=L$. This abstract equation is motivated by the
linearization of a large class of Hamiltonian PDEs at equilibria or relative
equilibria. Our first goal is to understand the structural and spectral
properties of (\ref{E:Hamiltonian}), its linear stability/instability, and the
persistence of these properties under small perturbations in a general
setting. Secondly, the general results on (\ref{E:Hamiltonian}) will be
applied to study the linearization at some coherent states of nonlinear
Hamiltonian PDEs such as the 2-dim incompressible Euler equation, generalized
Bullough-Dodd equation, Gross-Pitaevskii type equation, and some long wave
models like KdV, BBM, and the good Boussinesq equations.

Our main assumption is that the quadratic form $\langle L\cdot,\cdot\rangle$
admits a decomposition $X=X_{-}\oplus\ker L\oplus X_{+},$ such that
\[
\dim X_{-}=n^{-}\left(  L\right)  <\infty,\ \langle L\cdot,\cdot
\rangle|_{X_{-}}<0,\ \text{and }\langle L\cdot,\cdot\rangle|_{X_{+}}\geq
\delta>0.
\]
An additional regularity assumption is required when $\dim\ker L=\infty$ (see
\textbf{(H3)} in Section \ref{SS:setup}). We note that there is no additional
restriction on the symplectic operator $J$, which can be unbounded,
noninvertible, or even with infinite dimensional kernel. \newline

\textit{* Background: stability/instability and local dynamics near an
equilibrium.} As our motivation for studying the linear system
(\ref{E:Hamiltonian}) is to understand the stability/instability of and the
local dynamics near coherent states (steady states, traveling waves, standing
waves etc.) of a nonlinear PDE, we first give a brief discussion of several
standard notions of stability/instability and local dynamics. In a simple case
of an ODE system
\[
x_{t}=f(x),\quad x\in\mathbf{R}^{n},
\]
the local dynamics near an equilibrium $x_{0}$, without loss of generality
assuming $x_{0}=0$, is very much related to the dynamics of its linearized
equation
\[
x_{t}=Ax,\quad A_{n\times n}=Df(0).
\]
On the one hand, if $A$ has an unstable eigenvalue $\lambda$
($\operatorname{Re}\lambda>0$), then the above linearized equation has an
exponential growing solution and is therefore linearly unstable. Here, linear
stability means $e^{tA}$ is uniformly bounded for all $t\geq0$. While it is
clearly linearly stable if Re$\lambda<0$ for all $\lambda\in\sigma(A)$, there
might be linear solutions with polynomial growth if Re$\lambda\leq0$ for all
$\lambda\in\sigma(A)$, which is often referred to as the spectrally stable
case. Nonlinear instability immediately follows from spectral instability for
ODEs. However, it is a much more subtle issue what properties in addition to
the spectral (or even linear) stability would ensure nonlinear stability. On
the other hand, assume $\sigma_{1}\subset\sigma(A)$ and Re$\lambda<\alpha$ (or
Re$\lambda>\alpha$) for all $\lambda\in\sigma_{1}$. Let $E_{1}$ be the
eigen-space of $\sigma_{1}$ which is invariant under $e^{tA}$, then we have
the spectral mapping property
\[
\text{(SM) there exists }C>0\text{ s.t. }|e^{tA}x|\leq Ce^{\alpha
t}|x|,\,\forall x\in E_{1},\,t\geq0\text{ (or }t\leq0).
\]
Suppose $\alpha_{+}>\alpha_{-}$ and $\sigma(A)=\sigma_{+}\cup\sigma_{-}$ with
Re$\lambda>\alpha_{+}$ for all $\lambda\in\sigma_{+}$ and Re$\lambda
<\alpha_{-}$ for all $\lambda\in\sigma_{-}$. Let $E_{\pm}$ be the eigen-spaces
of $\sigma_{\pm}$, then the above spectral mapping property (SM) and
$\alpha_{+}>\alpha_{-}$ imply an \textit{exponential dichotomy} of $e^{tA}$:
in the decomposition $\mathbf{R}^{n}=E_{+}\oplus E_{-}$ which is invariant
under $e^{tA}$, the relative minimal exponential expanding rate of
$e^{tA}|_{E_{+}}$ is greater than the maximal rate of $e^{tA}|_{E_{-}}$. For
the nonlinear ODE system, the classical invariant manifold theory, based on
the cornerstone of the exponential dichotomy, implies the existence of locally
invariant (pseudo-)stable and unstable manifolds near $0$. They often provide
more detailed dynamic structures than the mere stability/instability and also
help to organize the local dynamics.

It often happens that $f(x)$ and thus $A$ depend on a small parameter
$\epsilon$, so one naturally desires to understand the dynamics of the
perturbed systems for $0<|\epsilon|<<1$ based on that of $\epsilon=0$. A
system is said to be structurally stable if its dynamics does not change
qualitatively under any sufficiently small perturbation. For ODEs, it is well
known that the local dynamics is structurally stable if $A$ is hyperbolic,
namely $\sigma(A)\cap i\mathbf{R}=\emptyset$.

The above ODE results may serve as \textit{guidelines} in the study of local
dynamics of PDEs near equilibria and relative equilibria while one has to keep
in mind the following issues (among others):\newline$\bullet$ Sometimes it is
highly non-trivial to analyze the spectra of linearized PDEs, particularly
when the linear operator is not self-adjoint and has continuous spectrum.
\newline$\bullet$ On the eigen-space $E_{1}$ of a spectral subset $\sigma_{1}%
$, the above spectral mapping type property (SM) may not hold for solutions of
the linearized PDEs, due to the existence of continuous spectrum of the
linearized operator (see e.g. \cite{renardy-sm}). \newline$\bullet$ Regularity
issues in spatial variables can cause serious complications in proving
nonlinear properties (stability/instability, local invariant manifolds,
\textit{etc.}) based on linear ones (spectral stability/instability,
exponential dichotomy, \textit{etc.}). The existing systematic results are
mainly for semilinear PDEs.

\textit{* Background: regarding Hamiltonian systems.} On a Hilbert space $X$,
a Hamiltonian system takes the form
\begin{equation}
u_{t}=J\nabla H\left(  u\right)  ,\label{E:NLHPDE}%
\end{equation}
where the symplectic operator $J:X^{\ast}\rightarrow X$ satisfies $J^{\ast
}=-J$ and $H:X\rightarrow\mathbf{R}$ is the Hamiltonian energy functional.
In a more general setting, $J=J(u)$ may depend on $u$ or \eqref{E:NLHPDE} may
be posed on a symplectic manifold $M$ where $J(u):T^{\ast}M\rightarrow TM$. In
the classical setting, the symplectic structure $\omega\in T^{\ast}M\otimes
T^{\ast}M$ is a 2-form given by
\[
\omega(u)(U_{1},U_{2})=\langle J(u)^{-1}U_{1},U_{2}\rangle,\;U_{1,2}\in
T_{u}M,
\]
which is required to be closed, namely $d\omega=0$. It is standard that $H$
and $\omega$ are invariant under the Hamiltonian flow associated with
(\ref{E:NLHPDE}). Suppose $u_{\ast}$ is a steady state of (\ref{E:NLHPDE})
(possibly in an appropriate reference frame, see examples in Section
\ref{SS:example}), then the linearized equation at $u_{\ast}$ takes the form
of \eqref{E:Hamiltonian} with $L=\nabla^{2}H\left(  u_{\ast}\right)  $. In
some cases, even though the nonlinear equation is not written in a
straightforward Hamiltonian form, the linearization at an equilibrium
$u_{\ast}$ can still be put in the Hamiltonian form (\ref{E:Hamiltonian}), see
Section \ref{SS:Euler} for the example of 2D Euler equation. It is standard
for Hamiltonian ODEs and also proved for many Hamiltonian PDEs that the
spectrum $\sigma(JL)$ is \textit{symmetric} with respect to both real and
imaginary axes. Therefore, either \eqref{E:Hamiltonian} is spectrally unstable
or its spectrum must lie on the imaginary axis. Even though the latter falls
into the spectral stability category, it is often subtle to obtain properties
of even the linear dynamics, such as linear stability and exponential
dichotomy, based on the spectral properties, particular when there is
continuous spectrum. Existing results in the literature often take advantage
of the conservation of $H$ or $\omega$.

The structural stability is also more subtle even for linear Hamiltonian PDEs.
On the one hand, the linearized operator $JL$ associated with the
linearization of Hamiltonian PDEs arising from physics and engineering usually
has most of its spectrum lie on the imaginary axis. Therefore, the structural
stability results based on the hyperbolicity of $JL$ are hardly applicable. On
the other hand, properties of Hamiltonian systems, such as the notions of
Krein signatures and the conservation of $H$ and $\omega$, provide crucial
additional tools. The structural stability of linear Hamiltonian PDEs
addressed in this paper is mainly related to spectral properties and linear
exponential dichotomy.

For Hamiltonian PDEs, there have been some works on local nonlinear dynamics
based on properties of the linearized equations. For semilinear Hamiltonian
PDEs $u_{t}=JH^{\prime}\left(  u\right)  $ with nonlinear terms of subcritical
growth,
such as nonlinear Klein-Gordon equation,
nonlinear Schr\"{o}dinger equation, and Gross-Pitaevskii equation, local
invariant manifolds can be constructed by combining ODE techniques with
dispersive estimates (e. g. \cite{BJ89} \cite{jin-et-GP}
\cite{nakanishi-schlag-book}). Such results for traveling wave solutions of
the generalized KdV equation had also been obtained (\cite{jin-et-kdv}) with
the help of smoothing estimates. The construction of invariant manifolds for
quasilinear PDEs is more difficult, and was only done in very few cases (e. g.
\cite{lin-zeng-invariant}). However, the passing from linear to nonlinear
instability, which is a much weaker statement than the existence of invariant
manifolds, had been done for many quasilinear PDEs (e.g. \cite{grenier-2000}
\cite{guo-strauss95} \cite{lin-liao-jin-modulational} \cite{lin-euler-imrn}
\cite{lin-cpam-bgk}). Several techniques were introduced to overcome the
difficulties of loss of derivative of nonlinear terms and the growth due to
the essential spectra of the linearized operators (see above references). The
passing from spectral (or linear) stability to nonlinear stability is more
subtle, particularly when $\left\langle Lu,u\right\rangle $ is not positive
definite after the symmetry reduction. When such positivity holds, the
nonlinear stability can usually be proved by using the Lyapunov functional,
see e.g. \cite{gss-87} \cite{gss-90} for Hamiltonian PDEs. If such positivity
fails, there is currently no general approach to study the nonlinear stability
based on the linear one. \newline

Our motivation of analyzing the linearized Hamiltonian system
(\ref{E:Hamiltonian}) in such a general form is to understand the
stability/instability of and the local dynamics near a coherent state $u_{*}$
of a nonlinear Hamiltonian PDE in the form of \eqref{E:NLHPDE} with $L
=\nabla^{2} H(u_{*})$. We first make some comments on the hypotheses.

On $L$, the assumption $n^{-}\left(  L\right)  <\infty$ is equivalent to that
$H\left(  u\right)  $ has a finite Morse index at the critical point $u_{\ast
}$. This assumption is automatically satisfied if $u_{\ast}$ is constructed by
minimizing $H\left(  u\right)  $ subject to finitely many constraints. In
applications to continuum mechanics (fluids, plasmas etc.), the PDEs are often
of a noncononical Hamiltonian form $u_{t}=J\left(  u\right)  \nabla H\left(
u\right)  $, with a symplectic operator $J\left(  u\right)  \ $depending on
the solution $u$. In many cases, the linearization at an equilibrium $u_{\ast
}$ can still be written in the Hamiltonian form (\ref{E:Hamiltonian}) and the
assumption $n^{-}\left(  L\right)  <\infty$ is satisfied (see Section
\ref{SS:Euler} for the example of 2D Euler equation). The uniform positivity
of $L$ on $X_{+}$ could be relaxed to positivity by defining a new phase space
(see Section \ref{S:degenerate}).

In the existing literature on systems in the form of (\ref{E:Hamiltonian}),
$J^{-1}:X\rightarrow X^{\ast}$ is mostly assumed to be a bounded operator,
which is not only for technical convenience but also natural in the sense that
the symplectic 2-form $\omega$ is defined in terms of $J^{-1}$. However, it
happens that $J$ does not have a bounded inverse for many important
Hamiltonian PDEs such as the KdV, BBM, the good Boussinesq equations, 2D Euler
equation, \textit{etc.}, see Section \ref{SS:example}.

The goal of this paper regarding the general Hamiltonian PDE
(\ref{E:Hamiltonian}) is to study its spectral structures, linear dynamics, as
well as certain structural stability properties under the assumption
$n^{-}(L)<\infty$, but without any assumption on $J$ in addition to $J^{\ast
}=-J$. Our main general results include the symmetry of the spectrum
$\sigma(JL)$, an index theorem relating certain spectral properties of $JL$ to
$n^{-}\left(  L\right)  $ which is useful for linear stability analysis, the
linear exponential trichotomy of $e^{tJL}$, and the persistence of these
properties for slightly perturbed Hamiltonian systems. These results are
mostly achieved based on a structural decomposition of (\ref{E:Hamiltonian}).
In Section \ref{SS:example}, several Hamiltonian PDEs are studied using these
general results.

In the below, we briefly describe our main results and some key ideas in the
proof. More details of the main theorems can be found in Section
\ref{S:MainResults} and proofs in later sections.

\textbf{Structural decomposition.} Most of the general theorems in this paper
are based on careful decompositions of the phase space into closed subspaces
through which $L$ and $JL$ take rather simple block forms. One of the most
fundamental decomposition is given in Theorem \ref{T:decomposition}. In this
decomposition,
\[
JL\longleftrightarrow%
\begin{pmatrix}
0 & A_{01} & A_{02} & A_{03} & A_{04} & 0 & 0\\
0 & A_{1} & A_{12} & A_{13} & A_{14} & 0 & 0\\
0 & 0 & A_{2} & 0 & A_{24} & 0 & 0\\
0 & 0 & 0 & A_{3} & A_{34} & 0 & 0\\
0 & 0 & 0 & 0 & A_{4} & 0 & 0\\
0 & 0 & 0 & 0 & 0 & A_{5} & 0\\
0 & 0 & 0 & 0 & 0 & 0 & A_{6}%
\end{pmatrix}
,
\]%
\[
L\longleftrightarrow%
\begin{pmatrix}
0 & 0 & 0 & 0 & 0 & 0 & 0\\
0 & 0 & 0 & 0 & B_{14} & 0 & 0\\
0 & 0 & L_{X_{2}} & 0 & 0 & 0 & 0\\
0 & 0 & 0 & L_{X_{3}} & 0 & 0 & 0\\
0 & B_{14}^{\ast} & 0 & 0 & 0 & 0 & 0\\
0 & 0 & 0 & 0 & 0 & 0 & B_{56}\\
0 & 0 & 0 & 0 & 0 & B_{56}^{\ast} & 0
\end{pmatrix}
,
\]
where $L$ takes an almost diagonal block form with $L_{X_{3}}\geq\delta$ for
some $\delta>0$ and $JL$ takes a blockwise upper triangular form. Moreover,
all the blocks of $JL$ are bounded operators except for $A_{3}$ which is
anti-self-adjoint with respect to the equivalent inner product $\langle
L_{X_{3}}\cdot,\cdot\rangle$ on $X_{3}$. In particular, all other diagonal
blocks are matrices and therefore have only eigenvalues of finite
multiplicity. The upper triangular form of $JL$ simplifies the spectral
analysis on $JL$ tremendously and plays a fundamental role in the proof of the
exponential trichotomy of $e^{tJL}$, the index formula, and the structural
stability/instability of \eqref{E:Hamiltonian}.

We briefly sketch some ideas in the construction of the decomposition here
under the assumption $\ker L=\left\{  0\right\}  $, from which the
decomposition in the general case follows. First, we observe that $JL$ is
anti-self-adjoint in the indefinite inner product $\left\langle L\cdot
,\cdot\right\rangle $. Thus, by a Pontryagin type invariant subspace Theorem
for symplectic operators in an indefinite inner product space, there exists an
invariant (under $JL$) subspace $W\subset X,$ satisfying that $L|_{W}\leq0$
and $\dim W=n^{-}\left(  L\right)  $.

It would be highly \textit{desirable} to extend $W$ to a finite dimensional
invariant subspace $\tilde{W}$ such that $L|_{\tilde{W}}\ $is non-degenerate.
This would yield the \textit{invariant} decomposition $X=\tilde{W}\oplus
\tilde{W}^{\perp L}$, where $\tilde{W}^{\perp L}$ is the orthogonal complement
of $\tilde{W}$ with respect to $\left\langle L\cdot,\cdot\right\rangle $ and
$L|_{\tilde{W}^{\perp L}}>0$. Since $JL|_{\tilde{W}^{\perp L}}$ is
anti-self-adjoint in the equivalent inner product $\left\langle L\cdot
,\cdot\right\rangle $ and $\tilde{W}$ is finite dimensional, this immediately
gives the decomposition we want.

However, such an invariant decomposition $X=\tilde{W}\oplus\tilde{W}^{\perp
L}$ is in general impossible since it would imply that $L$ is non-degenerate
on the subspace of generalized eigenvectors of any purely imaginary eigenvalue
of $JL$ (Lemma \ref{lemma-non-degenerate-finite-d}), while the counterexample
in Section \ref{SS:non-deg} shows that $L$ can be degenerate on such subspaces
of embedded eigenvalues in the continuous spectra. Our proof is by a careful
decomposition of the invariant spaces $W,\ W^{\perp L}$ and their complements.

\textbf{Exponential trichotomy.} Our second result is the exponential
trichotomy of $e^{tJL}$ in $X$ and more regular spaces (Theorem
\ref{theorem-dichotomy}). More precisely, we decompose $X=E^{u}\oplus
E^{c}\oplus E^{s}$, such that: $E^{u,c,s}$ are invariant under $e^{tJL}$,
\[
\dim E^{u}=\dim E^{s}\leq n^{-}(L),\ E^{c}=\left(  E^{u}\oplus E^{s}\right)
^{\perp L},
\]
and $A_{5}=e^{tJL}|_{E^{u}}$ $\left(  A_{6} =e^{tJL}|_{E^{s}}\right)  \ $has
exponential decay when $t<0$ $\left(  t>0\right)  \ $and $e^{tJL}|_{E^{c}}$
has possible polynomial growth for all $t$ with the optimal algebraic rate
explicitly given. Roughly speaking, the unstable (stable) spaces $E^{u}\left(
E^{s}\right)  $ are subspaces of generalized eigenvectors of the unstable
(stable) eigenvalues of $JL$ and the center space $E^{c}$ corresponds to the
spectra in the imaginary axis.

Such exponential trichotomy is an important step to prove nonlinear
instability, and furthermore to construct local invariant (stable, unstable,
center) manifolds which are crucial for a complete understanding of the local
dynamics, see, for example, \cite{BJ89, CLL91, CL88}. Such exponential
trichotomy or dichotomy might be tricky to get due to the spectral mapping
issue, that is, generally $\sigma\left(  e^{tJL}\right)  \subsetneq
e^{t\sigma\left(  JL\right)  }$. So even if the spectra of $JL$ is understood,
it is still a subtle issue to prove the estimates for $e^{tJL}$. In the
literature, the exponential dichotomy is usually obtained either by resolvent
estimates (e.g. \cite{latushkin-et-2000}) or compact perturbations of simpler
semigroups (\cite{vidav70} \cite{shizuta83}). The proofs were often technical
(particularly for resolvent estimates) and only worked for specific classes of
problems. Our result gives the exponential trichotomy for general Hamiltonian
PDEs (\ref{E:Hamiltonian}) with $n^{-}\left(  L\right)  <\infty$. Moreover,
the growth rates (particularly on the center space) obtained are sharp. In
particular, our sharp polynomial growth rate estimate on the center space
implies a stronger result than the usual spectral mapping statement. Our proof
of the exponential trichotomy which is very different from traditional
methods, is based on the upper triangular form of $JL$ in the decomposition
given in Theorem \ref{T:decomposition}. It can be seen that the Hamiltonian
structure of (\ref{E:Hamiltonian}) plays an important role in the proof.

\textbf{Index theorems.} Our third result is an index formula to relate the
counting of dimensions of some eigenspaces of $JL$ to $n^{-}\left(  L\right)
$. Denote the sum of algebraic multiplicities of all positive eigenvalues of
$JL$ by $k_{r}$ and the sum of algebraic multiplicities of eigenvalues of $JL$
in the first quadrant by $k_{c}$. Let $k_{i}^{\leq0}$ be the total number of
\textit{nonpositive} dimensions $n^{\leq0}(L|_{E_{i\mu}})$ of the quadratic
form $\left\langle L\cdot,\cdot\right\rangle $ restricted to the subspaces
$E_{i\mu}$ of generalized eigenvectors of all purely imaginary eigenvalues
$i\mu\in\sigma(JL)\cap i\mathbf{R}$ of $JL$ with positive imaginary parts, and
$k_{0}^{\leq0}$ be the number of nonpositive dimensions of $\left\langle
L\cdot,\cdot\right\rangle $ restricted to the generalized kernel of $JL$
modulo $\ker L$. We note that, when all purely imaginary eigenvalues are
semi-simple and $\left\langle L\cdot,\cdot\right\rangle $ restricted to these
kernels is non-degenerate, $k_{i}^{\leq0}$ is equal to $k_{i}^{-}$ which
represents the number of purely imaginary eigenvalues (with positive imaginary
parts) of negative Krein signature. The situation is more complicated if the
eigenvalue is not semi-simple or even embedded into the continuous spectra. In
the general case, we have
\begin{equation}
k_{r}+2k_{c}+2k_{i}^{\leq0}+k_{0}^{\leq0}=n^{-}\left(  L\right)  \text{.}
\label{formula-index}%
\end{equation}
Two immediate corollaries of (\ref{formula-index}) are: $n^{-}\left(
L\right)  =k_{0}^{\leq0}$ implies spectral stability and the oddness of
$n^{-}\left(  L\right)  -k_{0}^{\leq0}$ implies linear instability. Since by
(\ref{formula-index}) all the negative directions of $\left\langle
L\cdot,\cdot\right\rangle $ are associated to eigenvalues of $JL$,
conceptually the continuous spectrum of $JL\ $is only associated to positive
directions of $\left\langle L\cdot,\cdot\right\rangle $.

There have been lots of work on similar index formulae under various settings
in the literature. In the finite dimensional case where $L$ and $JL$ are
matrices, such index formula readily follows from arguments in a paper of
Mackay \cite{mackay86}, although was not written explicitly there. In the past
decade, there have been lots of work trying to extend it to the infinite
dimensional case. In most of these papers, $J$ is assumed to have a bounded
inverse (\cite{chu-pilinovsky} \cite{cug-Pelinovsky05} \cite{kapitula-et-04}
\cite{kollar-miller}), or $J|_{\left(  \ker J\right)  ^{\perp}}\ $\ has a
bounded inverse, as in the cases of periodic waves of dispersive PDEs
(\cite{bronski-et-quadratic-pencils} \cite{bronski-et-index-KDV}
\cite{kappitula-haragus} \cite{deconinck-kapitula}). Recently, in
\cite{kapitula-stefanov-kdv} \cite{pelinovsky-KDV}, the index formulae were
studied for KDV type equations in the whole line for which $J=\partial_{x}$
does not have bounded inverse. Our result (\ref{formula-index}) gives a
generalization of these results since we allow $J$ to be an arbitrary
anti-self-dual operator. In particular, $J|_{\left(  \ker J\right)  ^{\perp}}$
does not need to have a bounded inverse. This is important for applications to
continuum mechanics (e.g. fluids and plasmas) where $J$ usually has an
infinite dimensional kernel with $0$ in the essential spectrum of $J$ in some
appropriate sense (see Section \ref{SS:Euler} for the example of 2D Euler equation).

We should also point out some differences of (\ref{formula-index}) with
previous index formulae even in the case with bounded $J^{-1}$. In previous
works on index formula, it is assumed that $\left\langle L\cdot,\cdot
\right\rangle $ is non-degenerate on $(JL)^{-1}\left(  \ker L\right)  /\ker
L$. Under this assumption, the generalized kernel of $JL$ only have Jordan
blocks of length 2 and $k_{0}^{\leq0}=n^{-}\left(  L|_{(JL)^{-1}\left(  \ker
L\right)  /\ker L}\right)  $ (see Propositions \ref{prop-counting-k-0-1} and
\ref{prop-counting-k-0-2}). In (\ref{formula-index}), we do not impose such
non-degeneracy assumption on $L|_{(JL)^{-1}\left(  \ker L\right)  /\ker L}$
and thus the possible structures may be much richer. In the counting of
(\ref{formula-index}), we use $k_{i}^{\leq0},k_{0}^{\leq0}$, which are the
total dimensions of \textit{non-positive} directions of $L$ restricted on the
subspaces $E_{i\mu}$ of generalized eigenvectors of purely imaginary
eigenvalues $i\mu$ or zero eigenvalue (modulo $\ker L$). Since $\langle
L\cdot,\cdot\rangle$ might be degenerate on such subspace $E_{i\mu}\ $of an
embedded eigenvalue (see example in Section \ref{SS:non-deg}), they can not be
replaced by $k_{i}^{-},k_{0}^{-}$ (i.e. the dimensions of negative directions
of $L$) as used in the index formula of some papers (e.g.
\cite{kapitula-et-04}). However, in Proposition \ref{P:non-deg}, we show that
if a purely imaginary spectral point $i\mu$ is isolated, then $L$ is
non-degenerate on its generalized eigenspace $E_{i\mu}$ which consists of
generalized eigenvectors only. In this case, we also get an explicit formula
(\ref{counting-pure-imaginary}) for $n^{-}\left(  L|_{E_{i\mu}}\right)  $ by
its Jordan canonical form, which is independent of the choice of the basis
realizing the canonical form. This formula suggests that even for embedded
eigenvalues which might be of infinite multiplicity, the number and length of
nontrivial Jordan chains are bounded in terms of $n^{-}\left(  L\right)  $.

Moreover, even\ for the case where $\left\langle L\cdot,\cdot\right\rangle $
is degenerate on $E_{i\mu}$, we give a block decomposition of $JL$ and $L$ on
$E_{i\mu}$ (Proposition \ref{P:basis}). In this decomposition, $L$ is
blockwise diagonal and $JL$ takes an upper triangular form with three diagonal
blocks corresponding to the degenerate part of $L$, the simple eigenspaces and
the Jordan blocks of $i\mu$ of $JL$. Furthermore,we construct a special basis
for each Jordan block such that the corresponding $L$ is in an anti-diagonal
form (\ref{L-anti-diagonal}). The above decomposition of $E_{i\mu}$ yields
formula (\ref{counting-pure-imaginary}) for the case where $\left\langle
L\cdot,\cdot\right\rangle |_{E_{i\mu}}$ is non-degenerate and also plays an
important role on the constructive proof of Pontryagin type invariant subspace
Theorem \ref{T:Pontryagin} and the proof of structural instability Theorem
\ref{T:USImSpec}. To our knowledge, the formula (\ref{counting-pure-imaginary}%
) and the decomposition in Proposition \ref{P:basis} are new even for the
finite dimensional case.

We also note that for an eigenvalue $\lambda$ with $\operatorname{Re}%
\lambda\neq0$, $\left\langle L\cdot,\cdot\right\rangle |_{E_{\lambda}}=0$ and
by Corollary \ref{C:symmetry} $\left\langle L\cdot,\cdot\right\rangle
|_{E_{\lambda}\oplus E_{-\bar{\lambda}}}$ is non-degenerate with
\begin{equation}
n^{-}\left(  L|_{E_{\lambda}\oplus E_{-\bar{\lambda}}}\right)  =\dim
E_{\lambda}. \label{formula-unstable-L-index}%
\end{equation}
Therefore, we get the matrix form
\[
\left\langle L\cdot,\cdot\right\rangle |_{E_{\lambda}\oplus E_{-\bar{\lambda}%
}}\longleftrightarrow\left(
\begin{array}
[c]{cc}%
0 & A\\
A^{\ast} & 0
\end{array}
\right)  ,
\]
where $A$ is a nonsingular $n\times n$ matrix with $n=\dim E_{\lambda}$.

Now we discuss some ideas in our proof of index formula and the decomposition
in Proposition \ref{P:basis} after we briefly review previous approaches for
the index formulae. Like in the literature (\cite{kapitula-et-04}
\cite{cug-Pelinovsky05}), the index formula was usually proved by reducing the
eigenvalue problem $JLu=\lambda u$ to a generalized eigenvalue problem
$\left(  R-zS\right)  v=0\ $(so called linear operator pencil), where
$z=-\lambda^{2}\ $\ and $R,S$ are self-adjoint operators with $\ker S=\left\{
0\right\}  $. To get such reduction it is required that $J$ has a bounded
inverse and $L$ is non-degenerate on $(JL)^{-1}\left(  \ker L\right)  /\ker
L$. Notice that the operator $S^{-1}R$ is self-adjoint in the indefinite inner
product $\left\langle S\cdot,\cdot\right\rangle $. So by the Pontryagin
invariant subspace theorem (\cite{Gr90} \cite{chu-pilinovsky}
\cite{Krein-fixed-point} \cite{pontryagin}) for self-adjoint operators, there
is an $n^{-}\left(  S\right)  $-dimensional invariant (under $S^{-1}R$)
subspace $W$ such that $\left\langle S\cdot,\cdot\right\rangle |_{W}\leq0$,
where
\[
n^{-}\left(  S\right)  =n^{-}\left(  L\right)  -n\left(  L|_{(JL)^{-1}\left(
\ker L\right)  /\ker L}\right)  .
\]
Going back to the original problem $JLu=\lambda u$, an index formula can be
obtained by counting the negative dimensions of $L$ on the eigenspaces for
real, complex and pure imaginary eigenvalues. However, it should be pointed
out that the counting in some papers used the formula
(\ref{formula-unstable-L-index}), for which the required non-degeneracy of
$L|_{E_{\lambda}\oplus E_{-\bar{\lambda}}}$ seemed to be assumed but not proved.

In \cite{kappitula-haragus} and later also in \cite{bronski-et-index-KDV}
\cite{bronski-et-quadratic-pencils} \cite{deconinck-kapitula}, the index
formula was proved without reference to the Pontryagin invariant subspace
theorem. In these papers, some conditions on $J$ and $L$ were imposed to
ensure that the generalized eigenvectors of $JL$ form a complete basis of $X$.
Then the index formula follows by the arguments as in the finite dimensional
case (\cite{mackay86}). Such requirement of a complete basis is very strong
and mostly true only in some cases where the eigenvalues of $JL$ are all discrete.

Our proof of the index formula (\ref{formula-index}) is based on the
decomposition in Theorem \ref{T:decomposition}, where we used the Pontryagin
invariant subspace theorem for the anti-self-adjoint operator $JL$ in the
indefinite inner product $\left\langle L\cdot,\cdot\right\rangle $. The proof
of the detailed decompositions of $JL$ and $L\ $on $E_{i\mu}$ given in
Proposition \ref{P:basis}, particularly the construction of the special basis
realizing the Jordan canonical form, is carried out in two steps. First, in
the finite dimensional case, we construct a special basis of the eigenspace
$E_{i\mu}\ $of $JL\ $to skew-diagonalize $L$ on the Jordan blocks by using an
induction argument on the length of Jordan chains. Second, for the infinite
dimensional case, we decompose $E_{i\mu}\ $into subspaces corresponding to
degenerate eigenspaces, simple non-degenerate eigenspaces and Jordan blocks.
Since the Jordan block part is finite dimensional, the special basis is
constructed as in the finite dimensional case.

\textbf{Hamiltonian perturbations.} Our fourth main result is about the
persistence of exponential trichotomy and a sharp condition for the structural
stability of linear Hamiltonian systems under small Hamiltonian perturbations.
Consider a perturbed Hamiltonian system $u_{t}=J_{\#}L_{\#}u$ where
$J_{\#},L_{\#}$ are small perturbations of $J,L$ in the sense of (\ref{E:ep}).
This happens when the symplectic structure or the Hamiltonian of the system
depends on some parameters.

First, we show that the exponential trichotomy of $e^{tJL}$ persists under
small perturbations. More precisely, we show in Theorem \ref{T:PET} that there
exists a decomposition $X=E_{\#}^{u}\oplus E_{\#}^{s}\oplus E_{\#}^{c}$,
satisfying that: $E_{\#}^{u,s,c}$ are invariant under $e^{tJ_{\#}L_{\#}}$ and
are obtained as small perturbations of $E^{u,s,c}$ in the sense that
$E_{\#}^{u,s,c}=\graph(S_{\#}^{u,s,c})$ where
\[
S_{\#}^{u}:E^{u}\rightarrow E^{s}\oplus E^{c},\quad S_{\#}^{s}:E^{s}%
\rightarrow E^{u}\oplus E^{c},\quad S_{\#}^{c}:E^{c}\rightarrow E^{s}\oplus
E^{u}, \quad|S_{\#}^{u,s,c}|\leq C\epsilon,
\]
and $\epsilon$ is roughly the size of perturbations $L_{\#} - L$ and $J_{\#}
-J$ (see \eqref{E:ep}). Moreover, $e^{tJ_{\#}L_{\#}}$ has exponential decay on
$E_{\#}^{u}$ and $E_{\#}^{s}$ in negative and positive times respectively with
at most $O\left(  \epsilon\right)  $ loss of decay rates compared with
$e^{tJL}|_{E^{u,s}}$; on $E_{\#}^{c}$, $\ e^{tJ_{\#}L_{\#}}\ $has at most
small exponential growth at the rate $O\left(  \epsilon\right)  $. We note
that $J_{\#}L_{\#}|_{E_{\#}^{c}}$ might contain eigenvalues with small real
parts which are perturbed from the spectra of $JL$ in the imaginary axis and
thus the small exponential growth on $e^{tJ_{\#}L_{\#}}$ is the best one can
get. In the perturbed decomposition $E_{\#}^{u,s,c}$, we obtain the uniform
control of the growth rate and the bounds in semigroup estimates for
$e^{tJ_{\#}L_{\#}}$ on $E_{\#}^{u,s,c}$. Such uniform estimates of the
exponential trichotomy (or dichotomy) are important for many applications of
nonlinear perturbation problems, such as the modulational instability of
dispersive models (see Lemma \ref{lemma-modulational-localized-semigroup}).

We briefly discuss some ideas in the proof of Theorem \ref{T:PET}. The spaces
$E_{\#}^{u,s}$ are constructed as the ranges of the projection operators
$\tilde{P}_{\#}^{u,s}\ $by the Riesz projections associated with the operator
$J_{\#}L_{\#}$ in a contour enclosing $\sigma\left(  JL|_{E^{u,s}}\right)  $
and $E_{\#}^{c}=\left(  E_{\#}^{u,s}\right)  ^{\perp L_{\#}}$. The smallness
assumption (\ref{E:ep}) is used in the resolvent estimates to show that
$E_{\#}^{u,s,c}$ are indeed $O\left(  \epsilon\right)  $ perturbations of
$E^{u,s,c}$. It is actually not so straightforward to prove the small
exponential growth of $e^{tJ_{\#}L_{\#}}$ on $E_{\#}^{c}$ since the
perturbation term $J(L_{\#}-L)$ may be unbounded. We again use the
decomposition Theorem \ref{T:decomposition}, where in the decomposition for
$JL$,
only one block is infinite dimensional, with good structure, and others blocks
are all bounded.

In Theorems \ref{T:SImSpec} and \ref{T:USImSpec}, we prove that a pure
imaginary eigenvalue $i\mu\neq0\ $of $JL$ is structurally stable, in the sense
that the spectra of $J_{\#}L_{\#}$ near $i\mu$ stay in the imaginary axis,
\textit{if and only if} either $L|_{E_{i\mu}}>0$ or $i\mu$ is isolated and
$L|_{E_{i\mu}}<0$. In particular, when $\left\langle L\cdot,\cdot\right\rangle
$ is indefinite on $E_{i\mu}$ or $i\mu$ is an embedded eigenvalue and
$\left\langle Lu,u\right\rangle \leq0$ for some $0\neq u\in$ $E_{i\mu}$, there
exist perturbed operators $JL_{\#}$ with unstable eigenvalues near $i\mu$ and
$\left\vert L_{\#}-L\right\vert $ being arbitrarily small. The structural
stability of finite dimensional Hamiltonian systems had been well studied in
the literature (see \cite{ekeland90} \cite{mackay86} and references therein).
It was known that (see e.g. \cite{mackay86}) a purely imaginary eigenvalue
$i\mu\neq0$ is structurally stable if and only if $L$ is definite on $E_{i\mu
}$. As a consequence, for a family of Hamiltonian systems, the equilibrium can
lose spectral stability only by the collision of purely imaginary eigenvalues
of opposite Krein signatures (i.e. sign of $\left\langle L\cdot,\cdot
\right\rangle $)$\ $. For Hamiltonian PDEs, the situation is more subtle due
to the possible embedded eigenvalues in the continuous spectrum. In
\cite{Gr90}, the linearized equation at excited states of a nonlinear
Schr\"{o}dinger equation was studied and the structural instability was shown
for an embedded simple eigenvalue with negative signature. A similar result
was also obtained in \cite{cug-Pelinovsky05} for semi-simple embedded
eigenvalues. The assumptions in Theorems \ref{T:SImSpec} and \ref{T:USImSpec}
are much more general and they give a sharp condition for the structural
stability of nonzero pure imaginary eigenvalues of general Hamiltonian
operator $JL$. In particular, in Theorem \ref{T:USImSpec}, structural
instability is proved even for the case when the embedded eigenvalue is
degenerate, which was not included in \cite{Gr90} or \cite{cug-Pelinovsky05}
for linearized Schr\"{o}dinger equations.

In the below, we discuss some ideas in the proof of Theorems \ref{T:SImSpec}
and \ref{T:USImSpec}. In the finite dimensional case, the structural stability
of an eigenvalue $i\mu$ of $JL\ $with a definite energy quadratic form
$L|_{E_{i\mu}}$ can be readily seen from an argument based on Lyapunov
functions. The above intuition can be used to show structural stability in
Theorem \ref{T:SImSpec} for isolated eigenvalues with definite energy
quadratic forms. The proof is more subtle for embedded eigenvalues with
positive energy quadratic forms. We argue via contradiction by showing that if
there is a sequence of unstable eigenvalues perturbed from $i\mu$, then this
leads to a non-positive direction of $L|_{E_{i\mu}}$. In this proof, the
decomposition Theorem \ref{T:decomposition} again plays an important role. The
proof of structural instability Theorem \ref{T:USImSpec} is divided into
several cases. When $L|_{E_{i\mu}}$ is non-degenerate and indefinite, it can
be reduced to the finite dimensional case for which we can construct a
perturbed matrix to have unstable eigenvalues. In particular, in the case when
$E_{i\mu}$ contains a Jordan chain on which $L$ is non-degenerate, we use the
special basis in Proposition \ref{P:basis} to construct a perturbed matrix
with unstable eigenvalues.

The proof is more subtle for an embedded eigenvalue $i\mu\ $with non-positive
and possibly degenerate $\left\langle L\cdot,\cdot\right\rangle |_{E_{i\mu}}$.
First, we construct a perturbed Hamiltonian system $J\tilde{L}_{\#}$ near $JL$
such that $i\mu$ is an isolated eigenvalue of $J\tilde{L}_{\#}$ and there is a
positive direction of $\tilde{L}_{\#}|_{E_{i\mu}\left(  J\tilde{L}%
_{\#}\right)  }$. In this construction, we use the decomposition Theorem
\ref{T:decomposition} once again along with spectral integrals. Then by
Proposition \ref{P:non-deg}, $\tilde{L}_{\#}|_{E_{i\mu}\left(  J\tilde{L}%
_{\#}\right)  }$ is non-degenerate and is indefinite by our construction. Thus
it is reduced to the previously studied cases. In a rough sense, the
structural instability is induced by the resonance between the embedded
eigenvalue (with $\left\langle L\cdot,\cdot\right\rangle $ non-positive in the
directions of some generalized eigenvectors) and the pure continuous spectra
whose spectral space has only positive directions due to the index formula
(\ref{formula-index}).\newline

In some applications (see e.g. Subsection \ref{SS:2dGGP}), it is not easy to
get the \textit{uniform} positivity for $L|_{X^{+}}$ (i.e. assumption
(\textbf{H2.b})) in an obvious space $X$ and only the positivity $L|_{X^{+}}$
is available. In Theorem \ref{T:degenerate}, we show that under some
additional assumptions ((\textbf{B1})-(\textbf{B5}) in Section
\ref{SS:degenerate}), one can construct a new phase space $Y$ such that $X$ is
densely embedded into $Y$; the extension $L_{Y}$ of $L$ satisfies the uniform
positivity in $\left\Vert {\cdot}\right\Vert _{Y}$; $J_{Y}:D(J)\cap Y^{\ast
}\rightarrow Y$ is the restriction of $J$, and $\left(  J_{Y},L_{Y},Y\right)
$ satisfy the main assumptions (\textbf{H1-3}). Then we can apply the theorems
to $\left(  J_{Y},L_{Y},Y\right)  $.\newline

\textbf{Hamiltonian PDE models.} In Section \ref{SS:example} (see also
Subsection \ref{SS:Applications} for a summary), we study the stabilities and
related issues of various concrete Hamiltonian PDEs based on our above general
theory, including: stability of solitary and periodic traveling waves of long
wave models of BBM, KDV, and good Boussinesq types; the eigenvalue problem of
the form $Lu=\lambda u^{\prime}$ arising from the stability of solitary waves
of generalized Bullough--Dodd equation; modulational instability of periodic
traveling waves; stability of steady flows of 2D Euler equations; traveling
waves of 2D nonlinear Schr\"{o}dinger equations with nonzero condition at infinity.

This paper is organized as follows. In Section 2, we give the precise set-up
and list the main general results more precisely with some comments, where the
readers are directed to the corresponding subsequent sections for detailed
proofs. \textit{For some readers, who would like to see the general results
but do not desire to get into the technical details of the proofs, it is
possibly sufficient to read Subsections 2.1--2.6 only.} The stability analysis
of various Hamiltonian PDEs are outlined in Subsection 2.7. The proofs of the
main general results are given in Sections 3 to 10. Section 3 studies some
basic properties of linear Hamiltonian systems. Section 4 is about the finite
dimensional Hamiltonian systems. In particular, the special basis in
Proposition \ref{P:basis} is constructed. Section 5 is about the Pontryagin
type invariant subspace Theorem for anti-self-adjoint operators in an
indefinite inner product space. Two proofs are given. One is by the fixed
point argument as found in the literature (\cite{chu-pilinovsky}
\cite{ky-Fan63} \cite{Krein-fixed-point}), which provides the existence of an
invariant Pontryagin subspace abstractly. The second one in separable Hilbert
spaces is via Galerkin approximation which also yields an \textit{explicit}
construction of a maximally non-positive invariant subspace. Section 6 is to
prove decomposition Theorem \ref{T:decomposition} which plays a crucial role
in the proof of most of the main results. Section 7 contains the proof of the
exponential trichotomy of $e^{tJL}$. In Section 8, the index theorem is
proved. Besides, the structures of the generalized eigenspaces are studied and
more explicit formula for the indexes $k_{i}^{\leq0},k_{0}^{\leq0}$, etc. are
proved. The non-degeneracy of $L|_{E_{i\mu}}$ for any isolated spectral point
$i\mu$ is also proved there. In Section 9, we prove the persistence of the
exponential trichotomy and the structural stability/instability Theorems. In
Section 10, we prove that the uniform positivity assumption (\textbf{H2.b})
can be relaxed under some assumptions. We study the stability and related
issues of various Hamiltonian PDEs in Section 11. In the Appendix, we prove
some functional analysis facts used throughout the paper, including some basic
decompositions of the phase space, the well-posedness of the linear
Hamiltonian system, and the standard complexification procedure.

\section{Main results}

\label{S:MainResults}

In this section, we give details of the main results described in the
introduction. The detailed proofs are left for later sections.\newline

\noindent\textbf{A remark on notations:} Throughout the paper, given a densely
defined linear operator $T$ from a Banach space $X$ to a Banach space $Y$ we
will always use $T^{\ast}$ to denote its dual operator from a subspace of
$Y^{\ast}$ to $X^{\ast}$. It would \textit{never} mean the adjoint operator
even if $X=Y$ is a Hilbert space. Given a Hilbert space $X$ and a linear
operator $L:X \to X^{*}$, since $L^{*}: (X^{*})^{*} =X \to X^{*}$, it is
legitimate to compare whether $L=L^{*}$.

\subsection{Set-up}

\label{SS:setup}

Consider a linear Hamiltonian system
\begin{equation}
\partial_{t}u=JLu,\quad u\in X \label{eqn-hamiltonian}%
\end{equation}
where $X$ is a real Hilbert space. Let $\left(  \cdot,\cdot\right)  $ denote
the inner product on $X$ and $\left\langle \cdot,\cdot\right\rangle $ the dual
bracket between $X^{\ast}$ and $X$. We make the following assumptions:

\begin{enumerate}
\item[(\textbf{H1})] $J:X^{\ast} \supset D(J) \rightarrow X$ is
anti-self-dual, in the sense $J^{*} = -J$.

\item[(\textbf{H2})] The operator $L:X\rightarrow X^{\ast}$ is bounded and
symmetric (i.e. $L^{*}=L$) such that $\left\langle Lu,v\right\rangle $ is a
bounded symmetric bilinear form on $X$. Moreover, there exists a decomposition
of $X$ into the direct sum of three closed subspaces
\[
X=X_{-}\oplus\ker L\oplus X_{+}, \quad n^{-}(L) \triangleq\dim X_{-} < \infty
\]
satisfying

\begin{enumerate}
\item[(\textbf{H2.a})] $\left\langle Lu, u\right\rangle <0$ for all $u \in
X_{-}\backslash\{0\}$;

\item[(\textbf{H2.b})] there exists $\delta>0$ such that
\[
\left\langle Lu,u\right\rangle \geq\delta\left\Vert u\right\Vert ^{2}\ ,\text{
for any }u\in X_{+}.
\]

\end{enumerate}

\item[(\textbf{H3})] The above $X_{\pm}$ satisfy
\[
\ker i_{X_{+}\oplus X_{-}}^{*}= \{f\in X^{*} \mid\langle f, u\rangle=0, \,
\forall u\in X_{-} \oplus X_{+}\} \subset D(J)
\]
where $i_{X_{+}\oplus X_{-}}^{*}: X^{*} \to(X_{+}\oplus X_{-})^{*}$ is the
dual operator of the embedding $i_{X_{+}\oplus X_{-}}$.
\end{enumerate}

\begin{remark}
\label{R:orthogonal-d-1} If in addition we assume
\begin{equation}
\label{E:orthogonal-d}\ker i_{(\ker L)^{\perp}}^{*} = \{ f \in X^{*}
\mid\langle f, u\rangle=0, \forall u \in(\ker L)^{\perp}\} \subset D(J),
\end{equation}
where
\begin{equation}
\label{E:orthogonal-c}(\ker L)^{\perp}= \{ u \in X \mid(u, v) =0, \;\forall v
\in\ker L\},
\end{equation}
it is possible to choose $X_{\pm}\subset(\ker L)^{\perp}$. See Lemma
\ref{L:decom1} and Remark \ref{R:orthogonal-d-2}.
\end{remark}

Regarding the operator $L$, what often matters more is its associated
symmetric quadratic form $\langle Lu,v\rangle$, $u,v\in X$, (or the Hermitian
symmetric form after the complexification). We say a bounded symmetric
quadratic form $B(u,v)$ is \textit{non-degenerate} if
\begin{equation}
\inf_{v\neq0}\ \sup_{u\neq0}\ \frac{|B(u,v)|}{\Vert u\Vert\Vert v\Vert}>0,
\label{E:non-degeneracy-def}%
\end{equation}
or equivalently, $v\rightarrow f=B(\cdot,v)\in X^{\ast}$ defines an
isomorphism from $X$ to $X^{\ast}$ (or a complex conjugate (sometimes called
anti-linear) isomorphism -- satisfying $av\rightarrow\bar{a}f$ for any
$a\in\mathbf{C}$ -- after the complexification). Under assumptions
(\textbf{H1-3}), $\langle Lu,v\rangle$ is non-degenerate if and only if $\ker
L=\{0\}$ (see Lemma \ref{L:non-degeneracy}).

\begin{remark}
\label{R:assumptions} It is worth pointing out that $n^{-}(L)= \dim X_{-}$ is
actually the maximal dimension of subspaces where $\langle L\cdot,
\cdot\rangle<0$, see Lemma \ref{L:Morse-Index}. Thus $n^{-}(L)$ is the Morse
index of $L$.

By Riesz Representation Theorem, there exists a unique bounded symmetric
linear operator $\mathbb{L}:X\rightarrow X$ such that $(\mathbb{L}u,v)=\langle
Lu,v\rangle$. Let $\Pi_{\lambda}$, $\lambda\in\mathbf{R}$, denote the
orthogonal spectral projection operator
from $X$ to the closed subspace corresponding to the spectral subset
$\sigma(\mathbb{L})\cap(-\infty,\lambda]$. From the standard spectral theory
of self-adjoint operators, assumption \textbf{(H2)} is equivalent to that
there exists $\delta^{\prime}>0$ such that \newline i.) $\sigma(\mathbb{L}%
)\cap\lbrack-\delta^{\prime},\delta^{\prime}]\subset\{0\}$, which is
equivalent to the closeness of $R(L)$, and \newline ii.) $\dim(\Pi
_{-\delta^{\prime}}X)<\infty$. \newline The subspaces
\[
X_{-}=\Pi_{-\frac{\delta^{\prime}}{2}}X\qquad X_{+}=(I-\Pi_{\frac
{\delta^{\prime}}{2}})X,
\]
along with $\ker L$ lead to a decomposition of $X$ orthogonal with respect to
both $(\cdot,\cdot)$ and $\langle L\cdot,\cdot\rangle$, satisfying
(\textbf{H2}).
\end{remark}

\begin{remark}
\label{R:H3} We would like to point out that (\textbf{H3}) is automatically
satisfied if $\dim\ker L<\infty$. In fact in this case,
\[
\dim\ker i_{X_{+}\oplus X_{-}}^{*} =\dim\{f\in X^{*} \mid\langle f,
u\rangle=0, \, \forall u\in X_{-} \oplus X_{+}\} = \dim\ker L <\infty.
\]
Let $\{f_{1}, \ldots, f_{k}\}$ be a basis of $\ker i_{X_{+}\oplus X_{-}}^{*}$.
As $D(J)$ is dense in $X^{*}$, one may take $g_{j} \in D(J)$ sufficiently
close to $f_{j}$, $j=1, \ldots, k$. Let
\[
X_{1}= \{ u \in X\mid\langle g_{j}, u\rangle=0, \ \forall j=1, \ldots, k\}.
\]
Since $X_{1}$ is close to $X_{+}\oplus X_{-}$, it is easy to show that there
exist closed subspaces $X_{1\pm} \subset X_{1}$ satisfying (\textbf{H2}) and
$X_{1} = X_{1+} \oplus X_{1-}$.

In fact, if we had treated $L$ and $J$ as operators from $X$ to $X$ through
the Riesz Representation Theorem and $X_{\pm}$ happen to be given as in Remark
\ref{R:assumptions} then (\textbf{H3}) would take the form $\ker L\subset
D(J)$.

Assumption (\textbf{H3}) does ensure that $JL$ is densely defined, see Lemma
\ref{L:decomJL}.
\end{remark}

\begin{remark}
\label{R:closedness} Assumption (\textbf{H2.b}) requires that the quadratic
form $\langle Lu,u\rangle$ has a uniform positive lower bound on $X_{+}$. This
corresponds to that $0$ is an isolated eigenvalue of $\mathbb{L}$ defined in
Remark \ref{R:assumptions}, which also implies that $R(L)$ is closed and
$R(L)=\{\gamma\in X^{\ast}\mid\langle\gamma,u\rangle=0,\,\forall u\in\ker L\}$.

For some PDE systems, (\textbf{H2.b}) may not hold or be hard to verify, see,
e.g. Subsection \ref{SS:2dGGP}. In Subsection \ref{SS:degenerate}, we consider
a framework where assumption (\textbf{H2.b}) for the uniform positivity of
$L|_{X_{+}}\ $is weakened to the positivity of $L|_{X_{+}}$, if some
additional and more detailed structures are present. In that situation, we
construct a new phase space $Y\supset X\ $and extend the operators $L$ and $J$
to $Y\ $accordingly so that (\textbf{H1-3}) are satisfied.
\end{remark}

\subsection{Structural decomposition}

Our first main result is to construct a decomposition of the phase space $X$
which helps understanding both structures of $JL$ and $L$ simultaneously.

\begin{theorem}
\label{T:decomposition} Assume (\textbf{H1-H3}). There exist closed subspaces
$X_{j}$, $j=1, \ldots, 6$, and $X_{0}= \ker L$ such that

\begin{enumerate}
\item $X = \oplus_{j=0}^{6} X_{j}$, $X_{j} \subset\cap_{k=1}^{\infty
}D\big((JL)^{k}\big)$, $j\ne3$, and
\[
\dim X_{1} = \dim X_{4}, \; \dim X_{5} = \dim X_{6}, \; \dim X_{1} + \dim
X_{2} + \dim X_{5} = n^{-}(L);
\]

\item $JL$ and $L$ take the following forms in this decomposition
\begin{equation}
JL\longleftrightarrow%
\begin{pmatrix}
0 & A_{01} & A_{02} & A_{03} & A_{04} & 0 & 0\\
0 & A_{1} & A_{12} & A_{13} & A_{14} & 0 & 0\\
0 & 0 & A_{2} & 0 & A_{24} & 0 & 0\\
0 & 0 & 0 & A_{3} & A_{34} & 0 & 0\\
0 & 0 & 0 & 0 & A_{4} & 0 & 0\\
0 & 0 & 0 & 0 & 0 & A_{5} & 0\\
0 & 0 & 0 & 0 & 0 & 0 & A_{6}%
\end{pmatrix}
, \label{block-JL}%
\end{equation}%
\begin{equation}
L\longleftrightarrow%
\begin{pmatrix}
0 & 0 & 0 & 0 & 0 & 0 & 0\\
0 & 0 & 0 & 0 & B_{14} & 0 & 0\\
0 & 0 & L_{X_{2}} & 0 & 0 & 0 & 0\\
0 & 0 & 0 & L_{X_{3}} & 0 & 0 & 0\\
0 & B_{14}^{\ast} & 0 & 0 & 0 & 0 & 0\\
0 & 0 & 0 & 0 & 0 & 0 & B_{56}\\
0 & 0 & 0 & 0 & 0 & B_{56}^{\ast} & 0
\end{pmatrix}
. \label{block-L}%
\end{equation}

\item $B_{14}: X_{4} \to X_{1}^{*}$ and $B_{56}: X_{6} \to X_{5}^{*}$ are
isomorphisms and there exists $\delta>0$ satisfying $\mp\langle L_{X_{2,3}} u,
u\rangle\ge\delta\Vert u\Vert^{2}$, for all $u \in X_{2,3}$;

\item all blocks of $JL$ are bounded operators except $A_{3}$, where $A_{03}$
and $A_{13}$ are understood as their natural extensions defined on $X_{3}$;

\item $A_{2,3}$ are anti-self-adjoint with respect to the equivalent inner
product $\mp\langle L_{X_{2,3}} \cdot, \cdot\rangle$ on $X_{2,3}$;

\item the spectra $\sigma(A_{j})\subset i\mathbf{R}$, $j=1,2,3,4$,
$\pm\operatorname{Re}\lambda>0$ for all $\lambda\in\sigma(A_{5,6})$, and
$\sigma(A_{5})=-\sigma(A_{6})$;

\item $n^{-}(L|_{X_{5}\oplus X_{6}})=\dim X_{5}$ and $n^{-}(L|_{X_{1}\oplus
X_{4}})=\dim X_{1}$.

\item $(u, v) =0$ for all $u \in X_{1} \oplus X_{2} \oplus X_{3} \oplus X_{4}$
and $v \in\ker L$.
\end{enumerate}
\end{theorem}

Through straightforward calculations, one may naturally rewrite the operator
$J$ and obtain additional relations among those blocks $A_{jk}$ using
$J^{\ast}=-J$.

\begin{corollary}
\label{C:decomposition} Let $P_{j}$, $j=0, \ldots, 6$ be the projections
associated to the decomposition in Theorem \ref{T:decomposition} and $\tilde
X_{j}^{*}= P_{j}^{*} X_{j}^{*} \subset X^{*}$. In the decomposition $X^{*} =
\Sigma_{j=0}^{6} \tilde X_{j}^{*}$, $J$ has the block form
\[
J \longleftrightarrow%
\begin{pmatrix}
J_{00} & J_{01} & J_{02} & J_{03} & J_{04} & 0 & 0\\
J_{10} & J_{11} & J_{12} & J_{13} & J_{14} & 0 & 0\\
J_{20} & J_{21} & J_{22} & 0 & 0 & 0 & 0\\
J_{30} & J_{31} & 0 & J_{33} & 0 & 0 & 0\\
J_{40} & J_{41} & 0 & 0 & 0 & 0 & 0\\
0 & 0 & 0 & 0 & 0 & 0 & J_{56}\\
0 & 0 & 0 & 0 & 0 & J_{65} & 0
\end{pmatrix}
.
\]
where the blocks, except $J_{00}$, are given by
\begin{align*}
&  - J_{10}^{*} = J_{01} = A_{04}B_{14}^{-1}, \; \; -J_{20}^{*} = J_{02} =
A_{02} L_{2}^{-1}\\
&  - J_{30}^{*} = J_{03} = A_{03} L_{X_{3}}^{-1}, \; \; -J_{40}^{*} = J_{04} =
A_{01} (B_{14}^{*})^{-1}\\
&  J_{11} = A_{14} B_{14}^{-1}, \; \; J_{12} = A_{12} L_{X_{2}}^{-1}, \; \;
J_{13} = A_{13} L_{X_{3}}^{-1}, \; \; J_{14} = A_{1} (B_{14}^{*})^{-1}\\
&  J_{21} = A_{24} B_{14}^{-1}, \; \; J_{22} = A_{2} L_{X_{2}}^{-1}, \; \;
J_{31} = A_{34}B_{14}^{-1}, \; \; J_{33} = A_{3} L_{X_{3}}^{-1}\\
&  J_{41} = A_{4} B_{14}^{-1}, \; \; J_{56} = A_{5} (B_{56}^{*})^{-1}, \; \;
J_{65} = A_{6} B_{56}^{-1}.
\end{align*}
Due to $J^{*} + J=0$, we also have $L_{X_{j}} A_{j} + A_{j}^{*} L_{X_{j}}$=0,
$j=2,3$, and
\begin{align*}
&  B_{14}^{*} A_{14} + A_{14}^{*} B_{14}=0, \; \; L_{X_{2}} A_{24} +
A_{12}^{*} B_{14}=0, \; \; L_{X_{3}} A_{34} + A_{13}^{*} B_{14} =0\\
&  B_{14} A_{4} + A_{1}^{*} B_{14}=0, \; \; \; \; B_{56} A_{6} + A_{5}^{*}
B_{56}=0.
\end{align*}

\end{corollary}

\begin{remark}
\label{R:decomposition} From the corollary, we have the following observations.

(i) $A_{4}$ and $- A_{1}^{*}$ are similar through $B_{14}$ and thus have the
same spectrum, contained in $i\mathbf{R}$ and symmetric about the real axis.
This in turn implies that $\sigma(A_{1}) =\sigma(A_{4})$.

(ii) $A_{24}$ and $A_{34}$ can be determined by other blocks
\[
A_{24}=-L_{X_{2}}^{-1}A_{12}^{\ast}B_{14},\;\;A_{34}=-L_{X_{3}}^{-1}%
A_{13}^{\ast}B_{14}.
\]
Consequently,
\[
J_{21}=-L_{X_{2}}^{-1}A_{12}^{\ast},\;\;J_{31}=-L_{X_{3}}^{-1}A_{13}^{\ast}.
\]

\end{remark}

The proof of Theorem \ref{T:decomposition} is given in Section
\ref{S:decomposition}, largely based on the Pontryagin invariant subspace
theorem \ref{T:Pontryagin}. Theorem \ref{T:decomposition} decomposes the
closed operator $JL$ into an upper triangular block form, all of which are
bounded except for one block anti-self-adjoint with respective to an
equivalent norm. This decomposition plays a fundamental role in proving the
linear evolution estimates, the index theorem, the spectral analysis, and the
perturbation analysis.

\subsection{Exponential Trichotomy}

One of our main results is the exponential trichotomy of the semigroup
$e^{tJL}$ on $X$ and more regular spaces, to be proved in Section \ref{S:ET}.
Such linear estimates are important for studying nonlinear dynamics,
particularly, the construction of invariant manifolds for nonlinear
Hamiltonian PDEs.

\begin{theorem}
\label{theorem-dichotomy} Under assumptions (\textbf{H1})-(\textbf{H3}), $JL$
generates a $C^{0}$ group $e^{tJL}$ of bounded linear operators on $X$ and
there exists a decomposition%
\[
X=E^{u}\oplus E^{c}\oplus E^{s},\quad\dim E^{u}=\dim E^{s}\leq n^{-}(L)
\]
satisfying: \newline i) $E^{c}$ and $E^{u},E^{s}\subset D(JL)$ are invariant
under $e^{tJL}$; Here, $E^{u}=X_{5},\ E^{s}=X_{6}$ are the unstable and stable
spaces defined in Theorem \ref{T:decomposition}, and the center space $E^{c}$
is defined by
\[
E^{c}=\{u\in X\mid\langle Lu,v\rangle=0,\ \forall v\in E^{s}\oplus
E^{u}\} = \oplus_{j=0}^4 X_j;\newline 
\]
\newline ii) $\langle L\cdot,\cdot\rangle$ completely vanishes on $E^{u,s}$,
but is non-degenerate on $E^{u}\oplus E^{s}$;  \newline iii) let $\lambda
_{u}=\min\{$Re$\lambda\mid\lambda\in\sigma(JL),\ \text{Re}\lambda>0\}$, there
exist $M>0$ and an integer $k_{0}\geq0$, such that
\begin{equation}%
\begin{split}
&  \left\vert e^{tJL}|_{E^{s}}\right\vert \leq M(1+t^{\dim E^{s}%
-1})e^{-\lambda_{u}t},\quad\forall\;t\geq0;\ \\
&  |e^{tJL}|_{E^{u}}|\leq M(1+|t|^{\dim E^{u}-1})e^{\lambda_{u}t},\quad
\forall\;t\leq0,
\end{split}
\label{estimate-stable-unstable}%
\end{equation}%
\begin{equation}
\ |e^{tJL}|_{E^{c}}|\leq M(1+\left\vert t\right\vert ^{k_{0}}),\quad
\forall\;t\in\mathbf{R},\label{estimate-center}%
\end{equation}
and
\[
k_{0}\leq1+2\big(n^{-}(L)-\dim E^{u}\big);
\]
Moreover, for $k\geq1$, define the space $X^{k}\subset X$ to be
\[
X^{k}=D\big((JL)^{k}\big)=\left\{  u\in X\ |\ \left(  JL\right)  ^{n}u\in
X,\ n=1,\cdots,k.\right\}
\]
and
\begin{equation}
\left\Vert u\right\Vert _{X^{k}}=\left\Vert u\right\Vert +\left\Vert
JLu\right\Vert +\cdots+\Vert(JL)^{k}u\Vert.\label{norm-X-k}%
\end{equation}
Assume $E^{u,s}\subset X^{k}$, then the exponential trichotomy for $X^{k}$
holds true: $X^{k}$ is decomposed as a direct sum%
\[
X^{k}=E^{u}\oplus E_{k}^{c}\oplus E^{s},\ E_{k}^{c}=E^{c}\cap X^{k}%
\]
and the estimates (\ref{estimate-stable-unstable}) and (\ref{estimate-center})
still hold in the norm $X^{k}$.
\end{theorem}

An immediate corollary of the theorem is that there are only finitely many
eigenvalues of $JL$ outside the imaginary axis in the complex plane.

\begin{remark}
The above growth estimates is optimal as one may easily construct finite
dimensional examples which achieve upper bounds in the estimates.
\end{remark}

\begin{remark}
Naturally, the above invariant decomposition and exponential trichotomy are
based on the spectral decomposition of $JL$. The unstable/stable subspaces
$E^{u,s}$ are the eigenspaces of the stable/unstable spectrum, which have
finite total dimensions. Therefore, it is easy to obtain the exponential decay
estimates of $e^{tJL}|_{E^{u,s}}$. While $E^{c}$ is the eigenspace of the
spectrum residing on the imaginary axis, the growth estimate of $e^{tJL}%
|_{E^{c}}$ is far from obvious as the spectral mapping is often a complicated
issue especially when continuous spectra is involved. Normally some
sub-exponential growth estimates, like in the form of
\[
\forall\epsilon>0,\ \exists\ C>0\Longrightarrow|\ e^{tJL}|_{E^{c}}\ |\leq
Ce^{\epsilon|t|},\ \forall t\in\mathbf{R},
\]
are already sufficient for some nonlinear local analysis. Our above polynomial
growth estimate on $e^{tJL}|_{E^{c}}$ with uniform bound on the degree of the
polynomial based on $\dim X_{-}$ is a much stronger statement.
\end{remark}

\begin{remark}
Often the invariant subspaces $E^{u,s,c}$ are defined via spectral
decompositions where the $L$-orthogonality between $E^{s}\oplus E^{u}$ and
$E^{c}$ is not immediately clear. In fact, this is a special case of more
general $L$-orthogonality property. See Lemma \ref{L:L-orth-eS} and Corollary
\ref{C:L-orth-eS}.
\end{remark}

\subsection{Index Theorems and spectral properties}

\label{section-index theorem}

Roughly our next main result is on the relationship between the number of
negative directions of $L$ (the Morse index) and the dimensions of various
eigenspaces of $JL$, which may have some implications on $\dim E^{u,s}$ and
thus the stability/instability of the group $e^{tJL}$.

We first introduce some notations. Given any subspace $S\subset X$, denote
$n^{-}\left(  L|_{S}\right)  $ and $n^{\leq0}(L|_{S})$ as the maximal negative
and non-positive dimensions of $\left\langle Lu,u\right\rangle \ $restricted
to $S$, respectively. Clearly, $n^{-}(L|_{s})\leq n^{-}(L)<\infty$.

In order to state and prove our results on the index theorems, we will work
with the standard \textit{complexified} spaces, operators, and quadratic
forms, see Appendix (Section \ref{S:appendix}) for details.

For any eigenvalue $\lambda$ of $JL$ let $E_{\lambda}$ be the generalized
eigenspace, that is,
\[
E_{\lambda}= \{ u\in X\ |\ ( JL-\lambda I) ^{k} u=0,\ \text{for some integer
}k\geq1\}.
\]

\begin{remark}
As $JL$ generates a $C^{0}$ semigroup (Proposition \ref{P:well-posedness}),
$(JL-\lambda)^{k}$ is a densely defined closed operator (see
\cite{hille-phillips}) and thus $E_{\lambda}$ is indeed a closed subspace. It
will turn out that $E_{\lambda}=\ker(JL-\lambda I)^{2n^{-}(L)+1}$ for any
eigenvalue $\lambda$. See Theorem \ref{theorem-counting} for $\lambda\notin
i\mathbf{R}$ and Proposition \ref{P:finite-dim-E} for more details.
\end{remark}

Let $k_{r}$ be the sum of algebraic multiplicities of positive eigenvalues of
$JL$ and $k_{c}$ be the sum of algebraic multiplicities of eigenvalues of $JL$
in the first quadrant (i.e. both real and imaginary parts are positive).
Namely,
\begin{equation}
k_{r}=\sum_{\lambda>0}\dim E_{\lambda},\quad k_{c}=\sum_{Re\lambda
,\ Im\lambda>0}\dim E_{\lambda}. \label{E:kr-kc}%
\end{equation}
For any purely imaginary eigenvalue $i\mu$ $\left(  0\neq\mu\in\mathbf{R}%
^{+}\right)  $ of $JL,$ let
\begin{equation}
k^{\leq0}\left(  i\mu\right)  =n^{\leq0}\left(  L|_{E_{i\mu}}\right)  ,\quad
k_{i}^{\leq0}=\sum_{0\neq\mu\in\mathbf{R}^{+}}k^{\leq0}\left(  i\mu\right)  .
\label{E:ki}%
\end{equation}
The index counting on $E_{0}$ is slightly more subtle due to the possible
presence of nontrivial $\ker L\subset E_{0}$. Observe that, for any subspace
$S\subset X$, $L$ induces a quadratic form $\langle L\cdot,\cdot\rangle$ on
the quotient space $S\slash(\ker L\cap S)$. As $\ker L\subset E_{0}$, define
\begin{equation}
k_{0}^{\leq0}=n^{\leq0}\left(  \langle L\cdot,\cdot\rangle|_{E_{0}\slash\ker
L}\right)  . \label{defn-k-0}%
\end{equation}
Equivalently, let $\tilde{E}_{0}\subset E_{0}$ be any subspace satisfying
$E_{0}=\ker L\oplus\tilde{E}_{0}$. Define
\[
k_{0}^{\leq0}=n^{\leq0}\left(  L|_{\tilde{E}_{0}}\right)  .
\]
It is easy to see that $k_{0}^{\leq0}$ is independent of the choice of
$\tilde{E}_{0}$. We have the following index formula which is proved in
Subsection \ref{SS:Index-counting}.

\begin{theorem}
\label{theorem-counting} Assume (\textbf{H1})-(\textbf{H3}), we have

(i) If $\lambda\in\sigma(JL)$, then $\pm\lambda, \pm\bar\lambda\in\sigma(JL)$.

(ii) If $\lambda$ is an eigenvalue of $JL$, then $\pm\lambda, \pm\bar\lambda$
are all eigenvalues of $JL$. Moreover, for any integer $k>0$,
\[
\dim\ker(JL \pm\lambda)^{k} = \dim\ker(JL\pm\bar\lambda)^{k}.
\]

(iii) The indices satisfy
\begin{equation}
k_{r}+2k_{c}+2k_{i}^{\le0}+k_{0}^{\le0}=n^{-}\left(  L\right)  \text{.}
\label{counting-formula}%
\end{equation}

\end{theorem}

Combining Theorem \ref{theorem-dichotomy} and \ref{theorem-counting}, we have
the following corollary.

\begin{corollary}
\label{cor-instability-index} (i) If $k_{0}^{\leq0}=n^{-}\left(  L\right)  $,
then (\ref{eqn-hamiltonian}) is spectrally stable. That is, there exists no
exponentially unstable solution of (\ref{eqn-hamiltonian}).

(ii) If $n^{-}\left(  L\right)  -k_{0}^{\le0}$ is odd, then there exists a
positive eigenvalue of (\ref{eqn-hamiltonian}), that is, $k_{r}>0$. In
particular, if $n^{-}\left(  L\right)  -k_{0}^{\le0}=1$, then $k_{r}=1$ and
$k_{c}=k_{i}^{\le0}=0$, that is, (\ref{eqn-hamiltonian}) has exactly one pair
of stable and unstable simple eigenvalues.
\end{corollary}

\begin{remark}
\label{R:subtlety} The formula (\ref{counting-formula}) might seem more
intuitive if those above $k^{\leq0}$ had been replaced by $k^{-}$. In fact
such an index formula with $k^{-}$ instead of $k^{\leq0}$ is true only if the
quadratic form $\langle Lu,v\rangle$ is non-degenerate on all $E_{i\mu}$,
$\mu\in\mathbf{R}^{+}$ and $\tilde{E}_{0}$, which would imply $n^{-}%
(L|_{E_{i\mu}})=n^{\leq0}(L|_{E_{i\mu}})$. However, the degeneracy is indeed
possible and the correct choice has to be $k^{\leq0}$. Such an example is
given in Subsection \ref{SS:non-deg}.
\end{remark}

Even though we can not claim $\dim E_{i\mu}<\infty$ for an eigenvalue $i\mu\in
i\mathbf{R}$ which might be embedded in the continuous spectrum, in fact
$E_{i\mu}$ is spanned by eigenvectors along with finitely many generalized
eigenvectors, except for $\mu=0$. More precisely, we prove the following two
propositions in Lemma \ref{L:e-space-1} and Subsection \ref{SS:e-space}.

\begin{proposition}
\label{P:finite-dim-E} Assume (\textbf{H1})-(\textbf{H3}). For any $i\mu
\in\sigma(JL) \cap i \mathbf{R}\backslash\{0\}$, it holds
\[
E_{i\mu} = \ker(JL- i\mu)^{2k^{\le0} (i\mu)+1}, \quad\dim\big((JL- i\mu) E_{i
\mu}\big) \le2 k^{\le0} (i\mu).
\]
Moreover,
\[
E_{0} = \ker(JL)^{2k_{0}^{\le0} +2}, \quad\dim\big((JL)^{2} E_{0}%
\big) \le2k_{0}^{\le0}.
\]

\end{proposition}

The above proposition does not hold if $(JL)^{2} E_{0}$ is replaced by $JL
E_{0}$ as in the case of $\mu\ne0$. See an example in Remark \ref{R:(JL)^2} in
Subsection \ref{SS:e-space}.

For $\mu\in\mathbf{R}$, Theorem \ref{theorem-counting} and Proposition
\ref{P:finite-dim-E} mean that, in addition to eigenvectors, $JL|_{E_{i\mu}}$
has only finitely many nontrivial Jordan blocks with the total dimensions
bounded in term of $n^{-}(L)$. The number and the lengths of nontrivial Jordan
chains of $JL|_{E_{i\mu}}$ are independent of the choice of the basis
realizing the Jordan canonical form. Intuitively if a basis consisting of
generalized eigenvectors simultaneously diagonalizes the quadratic form
$\langle Lu,u\rangle$ and realizes the Jordan canonical form of $JL$, it would
greatly help us to understand the structure of \eqref{eqn-hamiltonian}.
However, usually this is not possible. Instead, we find a `good' basis for the
Jordan canonical form of $JL$ which also `almost' diagonalizes the quadratic
form $L$. To our best knowledge, we are not aware of such a result even in
finite dimensions.

\begin{proposition}
\label{P:basis} Assume (\textbf{H1})-(\textbf{H3}). For $i\mu\in\sigma(JL)\cap
i\mathbf{R}\backslash\{0\}$, there exists a decomposition of $E_{i\mu}$ into
closed subspaces $E_{i\mu}=E^{D}\oplus E^{1}\oplus E^{G}$ such that $L$ and
$JL$ take the block forms
\[
\langle L\cdot,\cdot\rangle\longleftrightarrow%
\begin{pmatrix}
0 & 0 & 0\\
0 & L_{1} & 0\\
0 & 0 & L_{G}%
\end{pmatrix}
,\quad JL\longleftrightarrow%
\begin{pmatrix}
A_{D} & A_{D1} & A_{DG}\\
0 & i\mu & 0\\
0 & 0 & A_{G}%
\end{pmatrix}
.
\]
For $\mu=0$, there exists a decomposition $E_{0}=\ker L\oplus E^{D}\oplus
E^{1}\oplus E^{G}$ such that $L$ and $JL$ take the block form
\[
\langle L\cdot,\cdot\rangle\longleftrightarrow%
\begin{pmatrix}
0 & 0 & 0 & 0\\
0 & 0 & 0 & 0\\
0 & 0 & L_{1} & 0\\
0 & 0 & 0 & L_{G}%
\end{pmatrix}
,\quad JL\longleftrightarrow%
\begin{pmatrix}
0 & A_{0D} & A_{01} & A_{0G}\\
0 & A_{D} & A_{D1} & A_{DG}\\
0 & 0 & 0 & 0\\
0 & 0 & 0 & A_{G}%
\end{pmatrix}
.
\]
In both cases, all blocks are bounded operators, $L_{1}$ and $L_{G}$ are
non-degenerate, $\sigma(A_{G})=\sigma(A_{D})=\{i\mu\}$, and
\[
\dim E^{G}\leq3\big(k^{\leq0}(i\mu)-\dim E^{D}-n^{-}(L|_{E^{1}})\big),\quad
\dim E^{1}\leq\infty.
\]
Moreover, $\ker(A_{G}-i\mu)\subset(A_{G}-i\mu)E^{G}$, namely, the Jordan
canonical form of $JL$ on $E^{G}$ has non-trivial blocks only. Let
$1<k_{1}<\cdots<k_{j_{0}}$ be the dimensions of Jordan blocks of $A_{G}$ in
$E^{G}$. Suppose there are $l_{j}$ Jordan blocks of size $k_{j}\times k_{j}$.
For each $j=1,\ldots,j_{0}$, there exist linearly independent vectors
\begin{equation}
\{u_{p,q}^{(j)}\mid p=1,\ldots,l_{j},\ q=1,\ldots,k_{j}\}\subset E^{G}
\label{chain-k-j}%
\end{equation}
such that

\begin{enumerate}
\item $\forall\ 1\leq p\leq l_{j}$,
\[
\left\{  u_{p,q}^{(j)}=(JL-i\mu)^{q-1}u_{p,1}^{(j)},\ q=1,\ldots
,k_{j}\right\}
\]
form a Jordan chain of length $k_{j}$. More explicitly ,
\[
\text{on span }\{u_{1,1}^{(j)},\ldots,u_{1,k_{j}}^{(j)},\ldots,u_{l_{j}%
,1}^{(j)},\ldots,u_{l_{j},k_{j}}^{(j)}\}:
\]%
\[
A_{G}\longleftrightarrow%
\begin{pmatrix}
i\mu & 0 & \cdots & 0 & 0 & \cdots & 0 & 0 & \cdots & 0 & 0\\
1 & i\mu & \cdots & 0 & 0 & \ \cdots\  & 0 & 0 & \cdots & 0 & 0\\
&  & \cdots &  &  &  &  &  &  &  & \\
0 & 0 & \cdots & 1 & i\mu & \ \cdots\  & 0 & 0 & \cdots & 0 & 0\\
&  &  &  & \cdots &  &  &  &  &  & \\
0 & 0 & \cdots & 0 & 0 & \ \cdots\  & i\mu & 0 & \cdots & 0 & 0\\
0 & 0 & \cdots & 0 & 0 & \ \cdots\  & 1 & i\mu & \cdots & 0 & 0\\
&  &  &  &  &  &  & \cdots &  &  & \\
0 & 0 & \cdots & 0 & 0 & \ \cdots\  & 0 & 0 & \cdots & 1 & i\mu
\end{pmatrix}
\]
The above count for all Jordan blocks of $A_{G}$ of size $k_{j}$.

\item $\langle Lu_{p, q}^{(j)}, u_{p^{\prime}, q^{\prime}}^{(j^{\prime}%
)}\rangle=0$ if $p\ne p^{\prime}$ or $j \ne j^{\prime}$.

\item $\forall\ 1\leq p\leq l_{j}$, the $k_{j}\times k_{j}$ representation
matrix of $L$ on a chain (\ref{chain-k-j}) is
\begin{equation}
\big(\langle Lu_{p,q}^{(j)},u_{p,r}^{(j)}\rangle\big)_{q,r}=%
\begin{pmatrix}
0 & 0 & \cdots & 0 & a_{p,1}^{(j)}\\
0 & 0 & \cdots & a_{p,2}^{(j)} & 0\\
\cdots &  &  &  & \\
a_{p,k_{j}}^{(j)} & 0 & \cdots & 0 & 0
\end{pmatrix}
, \label{L-anti-diagonal}%
\end{equation}
where the entries satisfy
\[
a_{p,q^{\prime}}^{(j)}=(-1)^{q^{\prime}-q}a_{p,q}^{(j)}\neq0,\quad\quad
a_{p,k_{j}+1-q}^{(j)}=\overline{a_{p,q}^{(j)}}%
\]
and thus the above matrix is non-degenerate.

\item If $k_{j}$ is odd, then $a_{p,\frac{1}{2}(k_{j}+1)}^{(j)}=\pm1$ and the
$k_{j}$-th Krein signature of $i\mu$ defined by
\[
n_{k_{j}}^{-}(i\mu)=\sum_{p=1}^{l_{j}}\min\{0,\ a_{p,\frac{1}{2}(k_{j}%
+1)}^{(j)}\}
\]
is independent of the choices of such bases $\{u_{p,q}^{(j)}\}$.
\end{enumerate}
\end{proposition}

\begin{remark}
\label{R:JordanChain1}Since $\langle Lu,u\rangle$ is symmetric (Hermitian
after the complexification), we can normalize the above $a_{p,q}^{(j)}$ such
that $a_{p,q}^{(j)}=\pm1$ if $k_{j}$ is odd and $a_{p,q}^{(j)}=\pm i$ if
$k_{j}$ is even. In particular, when $\mu=0$, since the generalized eigenspace
is spanned by real functions in $X$, it follows that the Jordan chains in
$E^{G}\subset E_{0}$ are all of odd length. 
\end{remark}

In the splitting of $E_{i\mu}$, we note that only $E^{1}$ may be infinite
dimensional, where $L$ is positive except in finitely many directions. If
$\langle L\cdot,\cdot\rangle$ is non-degenerate on $E_{i\mu}$, the subspace
$E^{D}$ may be eliminated and many of our results can be improved. However,
this degeneracy indeed is possible. See such an example in Subsection
\ref{SS:non-deg}. On the positive side, in that subsection, we also prove the
following proposition on the non-degeneracy of $L|_{E_{i\mu}}$ for isolated
eigenvalues $i\mu$. In particular, the isolation assumption for $i\mu\in
\sigma(JL)\cap i\mathbf{R}$ usually holds if the problem comes from PDEs
defined on bounded or periodic domains.

\begin{proposition}
\label{P:non-deg} If $i\mu\in\sigma(JL) \cap i\mathbf{R}$ is isolated in
$\sigma(JL)$, then

(i) $i\mu$ is an eigenvalue, i.e. $E_{i\mu} \ne\{0\}$, and $\langle L\cdot,
\cdot\rangle$ is non-degenerate on $E_{i\mu}\slash (\ker L \cap E_{i\mu})$.

(ii) there exists a closed subspace $E_{\#}\subset X$ invariant under $JL$
such that $X=E_{i\mu}\oplus E_{\#}$ and $\langle Lu,v\rangle=0$ for all $u\in
E_{i\mu}$ and $v\in E_{\#}$.

(iii) $\sigma\big( (JL)|_{E_{\#}}\big) =\sigma(JL) \backslash\{i\mu\}$.
\end{proposition}

In the case of an isolated spectral point $i\mu$, one may define the invariant
eigenspaces and its complement eigenspace via contour integral in operator
calculus. Usually it is not guaranteed that such $i\mu$ is an eigenvalue and
its eigenspace coincides with $E_{i\mu}$. This proposition implies that, under
assumptions (\textbf{H1-3}), this is exactly the case and $\langle
L\cdot,\cdot\rangle$ is non-degenerate on $E_{i\mu}$. As a corollary, we prove

\begin{proposition}
\label{P:direct decomposition} In addition to (\textbf{H1-3}), we assume

\begin{enumerate}
\item[(\textbf{H4})] $\langle L\cdot, \cdot\rangle$ is non-degenerate on
$E_{\lambda}$ for any non-isolated $\lambda\in\sigma\left(  JL\right)  \cap
i\mathbf{R}\backslash\{0\}$ and also on $E_{0}/\ker L$ if $0\in\sigma(JL)$ is
not isolated,
\end{enumerate}

then there exist closed subspaces $N$ and $M$, which are $L$-orthogonal, such
that $N\oplus\ker L$ and $M\oplus\ker L$ are invariant under $JL$, $X=N\oplus
M\oplus\ker L$, $\dim N<\infty$, and $L\geq\delta$ on $M$ for some $\delta>0$.
\end{proposition}

In particular, if eigenvalues of $JL$ are isolated, then by Proposition
\ref{P:non-deg}, (\textbf{H4}) is automatically satisfied and Proposition
\ref{P:direct decomposition} holds. If we further assume $\ker L=\{0\}$, then
$X=N\oplus M$ and both $N$ and $M$ are invariant under $JL$. Proposition
\ref{P:direct decomposition} can be used to construct invariant decompositions
for $L-$self-adjoint operators. The next proposition gives a generalization of
Theorem A.1 in \cite{pego-kollar-et04}, which was proved for a compact
$L$-self-adjoint operator $A$ with $\ker A=\left\{  0\right\}  $. Such
decomposition was used to study the damping of internal waves in a stably
stratified fluid (\cite{pego-kollar-et04}).

\begin{proposition}
\label{P:pego} Let $X$ be a complex Hilbert space along with a Hermitian
symmetric quadratic form $B(u, v) = \langle Lv, u \rangle$ defined by an
(anti-linear) operator $L:X\rightarrow X^{\ast}$ satisfying (\textbf{H2}) with
$\ker L=\left\{  0\right\}  $. Let $A:X\rightarrow X$ be a $L-$self-adjoint
complex linear operator (i.e. $\langle LAu, v \rangle= \langle Lu, Av \rangle
$) such that nonzero eigenvalues of $A$ are isolated. If $L|_{\ker A}$ is
non-degenerate, then there exists a decomposition $X=N\oplus M$ such that $N$
and $M$ are $L$-orthogonal and invariant under $A$, $\dim N<\infty$ and
$L|_{M}$ is uniformly positive.
\end{proposition}

We will extend the notion of the \textit{Krein signature} to eigenvalues
$i\mu$ for which $\langle L\cdot,\cdot\rangle$ on $E_{i\mu}$ is
non-degenerate, and give more detailed descriptions of $k_{i}^{-}$ and
$k_{0}^{-}$. As commented above, the non-degeneracy assumption means $E^{D}$
is eliminated in $E_{i\mu}$. For such $\mu$, define
\[
E_{i\mu,0}=\{v\in\ker(JL-i\mu)\mid\langle Lv,u_{p,q}^{(j)}\rangle
=0,\ \forall1\leq j\leq j_{0},\ 1\leq p\leq l_{j},\ 1\leq q\leq k_{j}\}
\]
which is the complementary subspace of $R(JL-i\mu)\cap\ker(JL-i\mu)$ inside
$\ker(JL-i\mu)$. It corresponds to the diagonalized part of $JL|_{E_{i\mu}}$.

\begin{definition}
For $\mu\geq0$ such that $\langle L\cdot,\cdot\rangle$ is non-degenerate on
$E_{i\mu}$, define the first Krein signature
\[
n_{1}^{-}(i\mu)=n^{-}(L|_{E_{i\mu,0}})
\]
and $k_{j}$-th Krein signatures as $n_{k_{j}}^{-}(i\mu)$ given in Proposition
\ref{P:basis}, for odd $k_{j}=2m-1\geq1$.
\end{definition}

\begin{remark}
\label{R:index1} The Krein signature $n_{k_{j}}^{-}(i\mu)$, for odd
$k_{j}=2m-1\geq1$, does not have to be defined as in Proposition \ref{P:basis}
using the above special bases. In fact, for any $j$, let $\left\{
v_{p,q}^{(j)}\right\}  $ be an arbitrary complete set of Jordan chains of
length $k_{j}$. Define the $l_{j}\times l_{j}\ $matrix $\tilde{M}_{j}=\left(
\left\langle Lv_{p_{1},m}^{(j)},v_{p_{2},m}^{(j)}\right\rangle \right)  ,$
$1\leq p_{1},p_{2}\leq l_{j}$. Then $n_{k_{j}}^{-}\left(  i\mu\right)
=n^{-}\left(  \tilde{M}_{j}\right)  $, the negative index (Morse index) of
$\tilde{M}_{j}$.
\end{remark}

\begin{remark}
The signatures $n_{k_{j}}^{-} (\mu)$ may also be defined in an intrinsic way
independent of bases. See Definition \ref{D:signature1} and equation \eqref{E:Q_K}.
\end{remark}

According to Proposition \ref{P:basis}, the 2-dim subspace span$\{u_{p,q}%
^{(j)},u_{p,k_{j}+1-q}^{(j)}\}$ and 1-dim subspace span$\{u_{p,\frac{1}%
{2}(k_{j}+1)}^{(j)}\}$ for odd $k_{j}$ are $L$-orthogonal to each other. With
respect to the basis $\{u_{p,q}^{(j)},u_{p,k_{j}+1-p}^{(j)}\}$ there, $L$
takes the form of the Hermitian symmetric matrix $%
\begin{pmatrix}
0 & a\\
\bar{a} & 0
\end{pmatrix}
$ with $a\neq0$, whose Morse index is clearly 1. Therefore, we obtain the
following formula for $k_{i}^{-}$.

\begin{proposition}
\label{prop-counting-pure-imaginary} In addition to (\textbf{H1}%
)-(\textbf{H3}), assume $i\mu\in\sigma(JL) \cap i\mathbf{R}$ satisfies that
$\langle L\cdot, \cdot\rangle$ is non-degenerate on $E_{i\mu}$. Then we have
\begin{equation}
k^{\le0} (i\mu) = k^{-}\left(  i\mu\right)  =\sum_{k_{j}\ \text{even}}%
\frac{l_{j}k_{j}}{2}+\sum_{k_{j}\ \text{odd}}\left[  \frac{l_{j}\left(
k_{j}-1\right)  }{2}+n_{k_{j}}^{-}\left(  i \mu\right)  \right]  .
\label{counting-pure-imaginary}%
\end{equation}

\end{proposition}

As Hamiltonian systems often possess additional symmetries which generate
nontrivial $\ker L$, $k_{0}^{\leq0}$ deserves some more discussion if $\ker
L\neq\{0\}$. The following propositions are proved in Subsection \ref{SS:E_0},
based on a decomposition of the subspace $E_{0}$. Recall that for any subspace
$S\subset X$, $L$ also induces a quadratic form $\langle L\cdot,\cdot\rangle$
on the quotient space $S\slash(S\cap\ker L)$.

\begin{proposition}
\label{prop-counting-k-0-1} Assume (\textbf{H1})-(\textbf{H3}), then
$(JL)^{-1}(\ker L)$ is a closed subspace. Furthermore, let
\[
n_{0}=n^{\leq0}(\langle L\cdot,\cdot\rangle|_{(JL)^{-1}(\ker L)\slash\ker
L}).
\]
Then

(i) $k_{0}^{\le0}\geq n_{0}$.

(ii) If $\left\langle L\cdot,\cdot\right\rangle $ is non-degenerate on
$(JL)^{-1} (\ker L) \slash \ker L$, then
\[
k_{0}^{\le0}=n_{0} = n^{-} (\langle L\cdot, \cdot\rangle|_{(JL)^{-1} (\ker L)
\slash \ker L}).
\]

\end{proposition}

\begin{remark}
\label{R:counting-k-0-1} Practically, in order to compute $n_{0}$ in the above
proposition, let $S \subset(JL)^{-1} (\ker L)$ be a closed subspace such that
\begin{equation}
\label{E:subspaceS}(JL)^{-1} (\ker L) = \ker L \oplus S,
\end{equation}
then $n_{0} = n^{\le0} (L|_{S})$. Often $S$ can be taken as $(\ker L)^{\perp
}\cap(JL)^{-1} (\ker L)$.
\end{remark}

It is worth comparing the above results with some classical results (e.g.
\cite{gss-87, gss-90}). Consider a nonlinear Hamiltonian equation
\begin{equation}
\partial_{t}u=JDH(u) \label{eqn-nonlinear-Hamiltonian}%
\end{equation}
which has an additional conserved quantity $P(u)$ (often the momentum, mass
\textit{etc.}) due to some symmetry. Assume that for $c$ in a neighborhood of
$c_{0}$, there exists $u_{c}$ such that $DH(u_{c})-cDP(u_{c})=0$, which gives
a relative equilibrium of (\ref{eqn-nonlinear-Hamiltonian}) such as traveling
waves, standing waves, etc. The linearized equation of
(\ref{eqn-nonlinear-Hamiltonian}) in some reference frame at $u_{c_{0}}$ takes
the form of \eqref{eqn-hamiltonian} with $L=D^{2}H(u_{c_{0}})-c_{0}%
D^{2}P(u_{c_{0}})$. It can be verified that $JDP(u_{c_{0}})\in\ker L$ and
$L\partial_{c}u_{c}|_{c=c_{0}}=DP(u_{c_{0}})$. In the case where $\ker
L=span\{JDP(u_{c_{0}})\}$ and $J$ is one to one (not necessarily with bounded
$J^{-1}$ as assumed in \cite{gss-87, gss-90}), we have
\[
(JL)^{-1}(\ker L)=span\{JDP(u_{c_{0}}),\partial_{c}u_{c}|_{c=c_{0}}\}
\]
when $\frac{d}{dc}P(u_{c})|_{c=c_{0}}\neq0\ $and
\[
n_{0}=\left\{
\begin{array}
[c]{cc}%
0 & \text{if \ }\frac{d}{dc}P(u_{c})|_{c=c_{0}}<0\\
1 & \text{if \ }\frac{d}{dc}P(u_{c})|_{c=c_{0}}>0
\end{array}
\right.  .
\]
If we further assume $n^{-}(L)=1$, then the combination of Proposition
\ref{prop-counting-k-0-1} and Theorem \ref{theorem-counting} implies the
result in \cite{gss-87} that equation \eqref{eqn-hamiltonian} is stable if
$\frac{d}{dc}P(u_{c})|_{c=c_{0}}\leq0$ and unstable if $\frac{d}{dc}%
P(u_{c})|_{c=c_{0}}>0$.

In the following special cases, $k_{0}^{\leq0}$ as well as $n_{0}$ can be
better estimated, which is often useful in applications.

\begin{lemma}
\label{L:counting-k-0-2} Assume (\textbf{H1})-(\textbf{H3}). we have

(i) $\langle Lu, v\rangle=0, \quad\forall u\in\ker(JL), \, v \in\overline{R(
J)}$.

(ii) $\left\langle Lu, u\right\rangle $ is non-degenerate on $\ker(JL)
\slash \ker L$ if and only if it is non-degenerate on $\overline{R(
J)}\slash \big(\ker L \cap\overline{R( J)}\big)$.
\end{lemma}

While the statement of the lemma and the following proposition in the language
of quotient spaces make them independent of choices of subspaces transversal
to $\ker L$, practically it might be easier to work with subspaces. The
following is an equivalent restatement of Lemma \ref{L:counting-k-0-2} using
subspaces. Actually the proof in Subsection \ref{SS:E_0} will be carried out
by using subspaces.

\begin{corollary}
\label{C:counting-k-0-2-a} Let $S_{1},S^{\#}\subset X$ be closed subspaces
such that
\begin{equation}
\ker(JL)=\ker L\oplus S_{1},\qquad\overline{R(J)}=\big(\overline{R(J)}\cap\ker
L\big)\oplus S^{\#}. \label{E:subspaceTS}%
\end{equation}
We have that $\langle L\cdot,\cdot\rangle$ is non-degenerate on $S_{1}$ if and
only if it is non-degenerate on $S^{\#}$.
\end{corollary}

Under this non-degeneracy, we have

\begin{proposition}
\label{prop-counting-k-0-2} Assume (\textbf{H1})-(\textbf{H3}), and that
$\left\langle Lu,u\right\rangle $ is non-degenerate on $\ker(JL)\slash\ker L$
which is equivalent to $\ker(JL)\cap\overline{R(J)}\subset\ker L$, then

(i) $X=\ker(JL)+\overline{R(J)}$ and
\[
n^{-}(L)=n^{-}\big(L|_{\ker(JL)\slash\ker L}\big)+n^{-}\big(L|_{\overline
{R(J)}\slash\big(\ker L\cap\overline{R(J)}\big)}\big).
\]

(ii) Let
\[
\tilde S = \overline{R( J)} \cap(JL)^{-1} (\ker L).
\]
Then
\[
k_{0}^{\le0} \ge n^{-} (L|_{\ker(JL) \slash \ker L}) + n^{\le0} (L|_{\tilde
S\slash (\ker L \cap\tilde S)}).
\]

(iii) If, in addition, $\langle Lu, u\rangle$ is non-degenerate on $\tilde
S\slash (\ker L \cap\tilde S)$, then
\begin{equation}
\label{index-formula-D}%
\begin{split}
&  k_{0}^{\le0} = n^{-} (L|_{\ker(JL) \slash \ker L}) + n^{-} (L|_{\tilde
S\slash (\ker L \cap\tilde S)})\\
&  k_{r}+2k_{c}+2k_{i}^{\le0}=n^{-}\big( L|_{\overline{R( J)}\slash \big(\ker
L \cap\overline{R( J)}\big)} \big)-n^{-}\left(  L|_{\tilde S\slash (\ker L
\cap\tilde S)}\right)  .
\end{split}
\end{equation}

\end{proposition}

We notice that the last equality is only a consequence of the previous two
equalities on $n^{-}$ and $k_{0}^{\leq0}$ and the index Theorem
\ref{theorem-counting}.

In terms of subspaces, equivalently we have

\begin{corollary}
\label{C:counting-k-0-2-b} Let $S_{1},S^{\#}\subset X$ be closed subspaces
assumed in Corollary \ref{C:counting-k-0-2-a} and $S_{2}\in X$ be a closed
subspace such that
\begin{equation}
\overline{R(J)}\cap(JL)^{-1}(\ker L)=S_{2}\oplus\big(\overline{R(J)}\cap\ker
L\big). \label{E:subspaceS2}%
\end{equation}
Assume the non-degeneracy of $\left\langle Lu,u\right\rangle $ on $S_{1}$.
Under this condition, we have
\[
X=\ker L\oplus S_{1}\oplus S^{\#},
\]
and this decomposition is orthogonal with respect to the quadratic form
$\langle L\cdot,\cdot\rangle$. Moreover, we have
\[
n^{-}(L)=n^{-}(L_{S_{1}})+n^{-}(L|_{S^{\#}})\;\text{ and }\;k_{0}^{\leq0}\geq
n^{-}(L|_{S_{1}})+n^{\leq0}(L|_{S_{2}}).
\]
The additional non-degeneracy assumption of $\langle Lu,u\rangle$ on
$\tilde{S}\slash(\ker L\cap\tilde{S})$ is equivalent to its non-degeneracy on
$S_{2}$ and it implies
\[%
\begin{split}
&  k_{0}^{-}=n^{-}\left(  L|_{S_{1}}\right)  +n^{-}\left(  L|_{S_{2}}\right)
\\
&  k_{r}+2k_{c}+2k_{i}^{-}=n^{-}\big(L|_{S^{\#}}\big)-n^{-}\left(  L|_{S_{2}%
}\right)  .
\end{split}
\]

\end{corollary}

Very often subspaces $S_{1}, S^{\#}, S_{2}$ can be taken as various
intersections with $(\ker L)^{\perp}$.

\subsection{Structural stability/instability}

\label{SS:SS}

Our next main result is on the spectral properties of the Hamiltonian operator
$JL$ under small bounded perturbations. Consider the perturbed linear
Hamiltonian system
\begin{equation}
u_{t}=J_{\#}L_{\#}u,\qquad J_{\#}=J+J_{1},\quad L_{\#}=L+L_{1},\quad u\in X.
\label{E:PHam}%
\end{equation}
We assume the perturbations satisfy

\begin{enumerate}
\item[(\textbf{A1})] $J$ and $L$ satisfies (\textbf{H1-2}) and the
perturbations $J_{1}:X^{\ast}\rightarrow X$ and $L_{1}:X\rightarrow X^{\ast}$
are bounded operators with $J_{1}^{\ast}=-J_{1}$ and $L_{1}^{\ast}=L_{1}$.

\item[(\textbf{A2})] $\dim\ker L <\infty$;

\item[(\textbf{A3})] $D(JL) \subset D(J L_{1})$.
\end{enumerate}

We note that (\textbf{A2}) implies (\textbf{H3}) for $JL\ $by Remark
\ref{R:H3}. From the Closed Graph Theorem, $JL_{1}$ is a bounded operator on
the Hilbert space $D(JL)$ equipped with the graph norm
\begin{equation}
||u||_{G}^{2}\triangleq||u||^{2}+||JLu||^{2},\;u\in D(JL);\;|JL_{1}%
|_{G}\triangleq\sup_{||u||_{G}=1}||JL_{1}u||.\label{E:graph-norm}%
\end{equation}
We first point out that assumptions (\textbf{A1-3}) imply (\textbf{H1-3}) for
$J_{\#}L_{\#}$ when the perturbations are sufficiently small as assumed in
Theorem \ref{T:PET} below. See Lemma \ref{L:P-H3}. As indicated in assumption
(\textbf{A1}) we consider bounded perturbations to both the symplectic
structure $J$ and the energy quadratic form $L$, while the Hamiltonian
structure is preserved. Assumption (\textbf{A2}) ensures $n^{-}(L_{\#}%
)<\infty$ so that the perturbed problem is still in our framework. Assumption
(\textbf{A3}) is a regularity assumption which implies that $J_{\#}L_{\#}$ is
not more unbounded compared to $JL$. Therefore, the resolvent $(\lambda
-J_{\#}L_{\#})^{-1}$ is a small perturbation of $(\lambda-JL)^{-1}$ as proved
in Lemma \ref{L:resolvent}.

Let $E^{u,s,c}$ be the unstable/stable/center subspaces of $J L$, as well as
the constants $\lambda_{u}>0$, as given in Theorem \ref{theorem-dichotomy}.
The next theorem and the following proposition will be proved in Subsection
\ref{SS:PET}.

\begin{theorem}
\label{T:PET} Assume {(A1-3)}. There exist $C, \epsilon_{0}>0$ depending only
on $J$ and $L$ such that, if
\begin{equation}\label{E:ep}
|J_{1}|+ |L_{1}|+ |JL_{1}|_{G} \le \ep \le\epsilon_{0},
\end{equation}
then

\begin{enumerate}
\item[(a)] There exist bounded operators
\[
S_{\#}^{u}: E^{u} \to E^{s} \oplus E^{c}, \quad S_{\#}^{s}: E^{s} \to E^{u}
\oplus E^{c}, \quad S_{\#}^{c}: E^{c} \to E^{s} \oplus E^{u},
\]
such that
\[
|S_{\#}^{u,s,c}| \le C\epsilon, \quad e^{tJ_{\#} L_{\#}} E_{\#}^{u,s,c} =
E_{\#}^{u,s,c}, \quad\text{ where } E_{\#}^{u,s,c} = \graph (S_{\#}^{u,s,c}),
\]
for all $t \in\mathbf{R}$. Moreover,
\begin{equation}%
\begin{split}
&  \left\vert e^{tJ_{\#} L_{\#}}|_{E_{\#}^{s}}\right\vert \leq C (1+ t^{\dim
E^{s} -1}) e^{-(\lambda_{u}- C\epsilon) t},\quad\forall\;t\geq0;\ \\
&  |e^{tJ_{\#} L_{\#}}|_{E_{\#}^{u}}|\leq C (1+ |t|^{\dim E^{u} -1})
e^{(\lambda_{u} - C \epsilon)t},\quad\forall\;t\leq
0,\label{E:P-stable-unstable}%
\end{split}
\end{equation}
\begin{equation}
\ |e^{tJ_{\#} L_{\#}}|_{E_{\#}^{c}}|\leq C\epsilon^{\frac 1{2(1+ n^-(L) - \dim E^u)}-1} e^{C\epsilon^{\frac 1{2(1+ n^-(L) - \dim E^u)}} |t|},\; \forall
 t\in\mathbf{R}. \label{E:P-center}%
\end{equation}
\item[(b)] $\langle L_{\#} \cdot, \cdot\rangle$ vanishes on $E_{\#}^{u,s}$,
but is non-degenerate on $E_{\#}^{s} \oplus E_{\#}^{u}$, and
\[
E_{\#}^{c} = \{u\mid\langle L_{\#} u, v \rangle=0, \, \forall v \in E_{\#}^{u}
\oplus E_{\#}^{s}\}.
\]
\item[(c)] If $\langle L \cdot, \cdot\rangle\ge\delta>0$ on $E^{c}$, then
there exists $C^{\prime}>0$ depending on $\delta$, $J$, and $L$ such that
$|e^{tJ_{\#} L_{\#}}|_{E_{\#}^{c}}|\le C^{\prime}$ for any $t \in\mathbf{R}$.
\end{enumerate}
\end{theorem}

Due to assumption (\textbf{A3}), the resolvent $(\lambda-J_{\#}L_{\#})^{-1}$
is only a small perturbation of $(\lambda-JL)^{-1}$ as proved in Lemma
\ref{L:resolvent}. Therefore, the existence of the invariant subspaces
$E_{\#}^{u,s,c}$ as a small perturbation to $E^{u,s,c}$ follows immediately.
Statements (b) and (c) basically result from the Hamiltonian structure and the
estimates of $e^{tJ_{\#}L_{\#}}$ on $E_{\#}^{u,s}$ are basically due to their
finite dimensionality. If $J_{\#}L_{\#}-JL$ had been a bounded operator,
estimate \eqref{E:P-center} would follow easily from the standard spectral
theory as well. However, since $J:X^{\ast}\supset D(J)\rightarrow X$ is only
assumed to satisfy $J^{\ast}=-J^{\ast}$, the term $JL_{1}$ may not be bounded
and thus \eqref{E:P-center} does not follow from the standard spectral theory.
Our proof heavily relies on the decomposition given by Theorem
\ref{T:decomposition}. In fact, the usual resolvent estimate often neglects
the Hamiltonian structure of the problem which actually plays an essential
role here. Otherwise a counterexample without the Hamiltonian structure is
$J=J_{\#}=i$ and $L_{\#}=\partial_{xx}+\epsilon\partial_{x}$ with
$X=H^{1}(S^{1},\mathbf{C})$, for which the equation $u_{t}=J_{\#}L_{\#}u$ is
not even well-posed in $X\ $for $\epsilon\neq0$.

Another consequence of Lemma \ref{L:resolvent} of the resolvent estimate and
Lemma \ref{L:L-orth-eS} is the following structural stability type result.

\begin{proposition}
\label{P:PSubS} Suppose closed subsets $\sigma_{1,2} \subset\sigma(JL)$ satisfy

\begin{enumerate}
\item $\sigma(JL) = \sigma_{1} \cup\sigma_{2}$, $\sigma_{1} \cap\sigma_{2}
=\emptyset$, and $\sigma_{2}$ is compact.

\item For any $\lambda\in\sigma_{1}$ and $0\ne u \in E_{\lambda}$, it holds
$\langle Lu, u \rangle>0$.
\end{enumerate}

Then there exist $\alpha, \epsilon_{0}>0$ depending only on $J$ and $L$ such
that \eqref{E:ep} implies
\[
\{ \lambda\in\sigma(J_{\#}L_{\#}) \mid d(\lambda, \sigma_{2}) > \alpha\}
\subset i\mathbf{R}.
\]

\end{proposition}

From Proposition \ref{P:decomposition2}, any $\lambda\in\sigma(JL)\backslash
i\mathbf{R}$ is an eigenvalue, i.e. $E_{\lambda}\neq\{0\}$, and $\langle
L\cdot,\cdot\rangle$ vanishes on $E_{\lambda}$. Therefore, it must hold that
$\sigma_{1}\subset i\mathbf{R}$. Even though the second assumption on
$\sigma_{1}$ seems weaker than that $\langle L\cdot,\cdot\rangle$ is uniformly
positive on its eigenspaces, it along with Theorem \ref{theorem-counting}
actually implies the latter. This proposition means that, under small
perturbations, unstable eigenvalues can not bifurcate from such $\sigma_{1}$.

In the next we consider the deformation of purely imaginary spectral points of
$JL$ under perturbations as they are closely related to generation of linear
instability. The next two theorems are proved in Subsection \ref{SS:ImSpec}.
Firstly we prove that if $i\mu\in\sigma(JL)$ and $\langle L\cdot,\cdot\rangle$
has certain definite sign on $E_{i\mu}$, then $\sigma(J_{\#}L_{\#})$ would not
have nearby unstable eigenvalues.

\begin{theorem}
\label{T:SImSpec} Assume {(A1-3)}, $i\mu\in\sigma(JL) \cap i\mathbf{R}$, and
either a.) there exists $\delta>0$ such that $\langle L u, u\rangle\ge
\delta\vert| u\vert|^{2}$ for all $u \in E_{i\mu}$ or b.) $i\mu$ is isolated
in $\sigma(J L)$and $\langle Lu, u\rangle\le-\delta\Vert u \Vert^{2}$ for all
$u \in E_{i\mu}$, then there exist $\alpha, \epsilon_{0}>0$ depending on $J$,
$L$, $\mu$, and $\delta$ such that, if \eqref{E:ep} holds, then
\[
\{ \lambda\in\sigma(J_{\#} L_{\#}) \mid| \lambda-i\mu| \le\alpha\} \subset i
\mathbf{R}.
\]

\end{theorem}

\begin{remark}
On the one hand, note that in the above theorem, we do not require $i\mu$
being an isolated eigenvalue or even an eigenvalue of $JL$. If $i\mu$ is not
an eigenvalue, $E_{i\mu}=\{0\}$ and the sign definiteness assumption is
automatically satisfied. On the other hand, if $i\mu$ is an isolated spectral
point, then Proposition \ref{P:non-deg} implies that $E_{i\mu}$ is nontrivial
and is precisely the eigenspace of $i\mu$. Moreover, from Lemma
\ref{lemma-orthogonal-condition-chain} and the sign definiteness of $L$ on
$E_{i\mu}$, we have $E_{i\mu}=\ker(JL-i\mu)$.
\end{remark}

On the one hand, the above theorem indicates that under Hamiltonian
perturbations, hyperbolic (i.e. stable and unstable) eigenvalues can not
bifurcate from either a.) any $i\mu\in\sigma(JL)$, whether isolated or not,
for which $\langle L\cdot,\cdot\rangle$ is positive on $E_{i\mu}$, or b.) any
isolated eigenvalue $i\mu$ where $\langle L\cdot,\cdot\rangle$ has a definite
sign on $E_{i\mu}$. Theorem \ref{T:SImSpec}, as well as Theorem \ref{T:PET}
can be viewed as robustness or structural stability type results.

On the other hand, as given in the next theorem, the structural stability
conditions in Theorem \ref{T:SImSpec} are also necessary for an eigenvalue
$i\mu\neq0$. As in many applications parameters mostly appear in the energy
operator $L$ instead of the symplectic operator $J$, we will study
perturbations only to $L$ for possible bifurcations of unstable eigenvalues
near $i\mu$.

\begin{theorem}
\label{T:USImSpec} Assume that $(J,L)$ satisfies (\textbf{H1-3}) and $0\ne
i\mu\in\sigma(JL) \cap i\mathbf{R}$ satisfies

\begin{enumerate}
\item $\langle L\cdot, \cdot\rangle$ is neither positive nor negative definite
on $E_{i\mu}$ or

\item $i\mu$ is non-isolated in $\sigma(JL)$ and there exists $u\in E_{i\mu}$
with $\langle Lu,u\rangle\leq0$,
\end{enumerate}

then for any $\epsilon>0$, there exist a symmetric bounded linear operator
$L_{1}:X\rightarrow X^{\ast}$ such that: $|L_{1}|<\epsilon$ and there exists
$\lambda\in\sigma\big(J(L+L_{1})\big)$ with $\operatorname{Re}\lambda>0$ and
$|\lambda-i\mu|<C\epsilon$, for some constant $C$ depending only on $\mu,J,L$.
\end{theorem}

It is easy to see that conditions in Theorem \ref{T:USImSpec} are exactly
complementary to those in Theorem \ref{T:SImSpec} for $i\mu\neq0$ and thus
they give necessary and sufficient conditions on whether unstable eigenvalues
can bifurcate from $0\neq i\mu\in\sigma(JL)\cap i\mathbf{R}$ under Hamiltonian perturbations.

\begin{remark}
\label{R:Grillakis90} In \cite{Gr90}, Grillakis proved that an embedded purely
imaginary eigenvalue with negative energy of the linearized operator at
excited states of a semilinear nonlinear Schr\"{o}dinger equation is
`structurally unstable' under small perturbations and unstable eigenvalues can
be generated. The linearized operator is of the form $JL$, where
\[
J=\left(
\begin{array}
[c]{cc}%
0 & 1\\
-1 & 0
\end{array}
\right)  ,\ L=\left(
\begin{array}
[c]{cc}%
-\Delta+V_{1}\left(  x\right)  & 0\\
0 & -\Delta+V_{2}\left(  x\right)
\end{array}
\right)  .
\]
Here, $V_{1}\left(  x\right)  ,V_{2}\left(  x\right)  \rightarrow\omega>0$
exponentially when $\left\vert x\right\vert \rightarrow\infty$. Under some
assumptions, Theorem 2.4 in \cite{Gr90} implies that that if $i\mu\neq0$ is an
embedded eigenvalue of $JL$ with $\left\langle Lu,u\right\rangle <0$ for some
eigenfunction $u$, then an unstable eigenvalue may bifurcate from $i\mu$ under
Hamiltonian perturbations. Similar result was also obtained in
\cite{cug-Pelinovsky05}. This is a special case of the above theorem.
Actually, we can relax the structural instability condition to be that
$\left\langle L\cdot,\cdot\right\rangle $ is not positive definite on
$E_{i\mu}$, including cases of degeneracy of $L|_{E_{i\mu}}$ or with Jordan chains.

However, it should be pointed out that it is not clear that the above
structural instability may be realized by the linearized equation of the
nonlinear Schr\"{o}dinger equation at a perturbed excited state. It would be
interesting to see if one can prove the structural instability in the sense
that there is linear instability for nearby excited states.
\end{remark}

\begin{remark}
\label{R:0-e-v} The case $\mu=0$ is not included in Theorem \ref{T:USImSpec}
since this may be related to some additional degeneracy of $L$ or $J$. See for
example Cases 3b and 3d in Subsection \ref{SS:ImSpec}. The analysis of
possible bifurcations of unstable eigenvalues from $\mu=0$ could be carried
out in a similar fashion based on the Propositions \ref{P:basis},
\ref{P:non-deg}, Lemma \ref{L:isolation}, \textit{etc.}, but more carefully.
We feel that it might be easier to work on this case directly in concrete
applications and thus do not include it in the above theorem.
\end{remark}

\subsection{A theorem where $L$ does not have a positive lower bound on
$X_{+}$}

\label{SS:degenerate}

Among our global assumptions (\textbf{H1-3}), (\textbf{H2}) requires that the
phase space $X$ is decomposed into the direct sum of three subspaces
$X=X_{-}\oplus\ker L\oplus X_{+}$, such that the quadratic form $\langle
L\cdot,\cdot\rangle$ is uniformly positive/negative on $X_{\pm}$. This
assumption plays a crucial role in the analysis throughout the paper. However,
in some Hamiltonian PDEs $L$, which usually appears as the Hessian of the
energy functional at a steady state, may not have a positive lower bound on
$X_{+}$. One such simple example is $X=H^{1}(\mathbf{R}^{n})$ and
$L=-\Delta+a(x)$ where $\lim_{|x|\rightarrow\infty}a(x)=0$. Even if $a>0$
which implies $L>0$, but for any $\delta>0$, there exists $u\in H^{1}$ such
that $\langle Lu,u\rangle<\delta\Vert u\Vert_{H^{1}}^{2}$. A potential
resolution to this issue in this specific example is to take a different phase
space such as $\dot{H}^{1}$ instead of $H^{1}$. In Section \ref{S:degenerate},
we show that this observation may be applied in a rather general setting. As a
non-trivial example of this case, the stability of traveling waves of a
nonlinear Schr\"{o}dinger equation in 2-dim with non-vanishing condition at
$|x|=\infty$ is considered in Subsection \ref{SS:2dGGP}.

In this subsection, let $X$ be a real Hilbert space with the inner product
$(\cdot, \cdot)$ and we assume

\begin{enumerate}
\item[(\textbf{B1})] $Q_{0}, Q_{1}: X \to X^{*}$ are bounded positive
symmetric linear operators such that
\[
\langle(Q_{0} + Q_{1}) u, v\rangle= (u,v), \; Q_{0,1}^{*} = Q_{0,1}, \;
\langle Q_{0,1}u, u \rangle>0, \ \forall\ 0\ne u, v \in X.
\]

\item[(\textbf{B2})] $\mathbb{J}:X\rightarrow X$ is a bounded linear operator
satisfying
\[
\mathbb{J}^{-1}=-\mathbb{J},\quad\langle Q_{0}\mathbb{J}u,\mathbb{J}%
u\rangle=\langle Q_{0}u,u\rangle,\quad\forall u\in X.
\]
Let $J=\mathbb{J}Q_{0}^{-1}:X^{\ast}\supset Q_{0}(X)\rightarrow X$.

\item[(\textbf{B3})] $L:X \to X^{*}$ is a bounded symmetric linear operator
such that $L_{1} = L- Q_{1}$ satisfies
\[
|\langle L_{1} u, v\rangle|^{2} \le c_{0} (\langle Q_{0} u, u\rangle\langle
Q_{0} v, v\rangle+ \langle Q_{0} u, u\rangle\langle Q_{1} v, v\rangle+ \langle
Q_{1} u, u\rangle\langle Q_{0} v, v\rangle).
\]

\item[(\textbf{B4})] There exist closed subspaces $X_{\pm}\subset X$ such
that
\begin{align}
&  X=X_{-}\oplus\ker L\oplus X_{+}, \quad n^{-}(L) \triangleq\dim X_{-} <
\infty,\label{E:B4-decom}\\
&  \pm\langle Lu_{\pm}, u_{\pm}\rangle>0, \; \langle Lu_{+}, u_{-} \rangle=0,
\; \forall\ 0\ne u_{\pm}\in X_{\pm}. \label{E:B4-sign}%
\end{align}

\item[(\textbf{B5})] Subspaces $X_{\pm}$ satisfy
\[
\ker i_{X_{+}}^{*}= \{f\in X^{*} \mid\langle f, u\rangle=0, \, \forall u\in
X_{+}\} \subset Q_{0}(X) = D(J)
\]
where $i_{X_{+}}^{*}: X^{*} \to X_{+}^{*}$ is the dual operator of the
embedding $i_{X_{+}}$.
\end{enumerate}

Obviously the assumption in (\textbf{B1}) that $Q_{1}+Q_{0}$ is the Riesz
representation of the inner product can be weakened to that it is the Riesz
representation of an equivalent inner product. It is also easy to verify that
$J$ is closed and anti-symmetric, namely, $J\subset-J^{\ast}$. Roughly the
$L$-orthogonal decomposition of $X$ can be constructed a.) by taking $\ker
L\oplus X_{+}$ as the $L$-orthogonal complement of a carefully chosen $X_{-}$
and then $X_{+}$ as any complimentary subspace of $\ker L$ there; or b.) from
a spectral decomposition of the linear operator on $X$ corresponding to the
quadratic form $\langle L\cdot,\cdot\rangle$ through certain inner product. In
a typical application as in Subsection \ref{SS:2dGGP}, $Q_{1}$ is often a
uniformly positive elliptic operator of order $2s$, $L_{1}$ is a perturbation
containing lower order derivatives with variable coefficients, and $Q_{0}$
corresponds to the $L^{2}$ duality. It is convenient to start with $X=H^{s}$
initially. The assumption $n^{-}(L)<\infty$ may come from the construction of
the steady state via some variational approach. The lack of a positive lower
bound of $L$ restricted to $X_{+}\subset H^{s}$ is often due to the missing
control of the $L^{2}$ norm by $\left\langle L\cdot,\cdot\right\rangle $. This
also forces us to make the slightly stronger assumption (\textbf{B5}) than
(\textbf{H3}). In Section \ref{S:degenerate} we prove

\begin{theorem}
\label{T:degenerate} There exists a Hilbert space $Y$ such that \newline(a)
$X$ is densely embedded into $Y$; \newline(b) $L$ can be extended to a bounded
symmetric linear operator $L_{Y}: Y \to Y^{*}$;\newline(c) $(Y, L_{Y}, J_{Y})$
satisfy (\textbf{H1-3}), where $J_{Y}: D(J) \cap Y^{*} \to Y$ is the
restriction of $J$.
\end{theorem}

It is natural to define $Y$ through the completion of $X$ under a norm based
on $L$. To prove this theorem, the key is to show (\textbf{H1}) and
(\textbf{H3}) are satisfied.

\subsection{Some Applications to PDEs}

\label{SS:Applications}

We briefly discuss the applications of the general theory to several PDE
models in Section \ref{SS:example}. First, we consider the stability of
traveling waves of dispersive wave models of KDV, BBM and good Boussinesq
types. These PDE models arise as approximation long wave models for water
waves etc. We treat general dispersion symbols including nonlocal ones.

For solitary waves, the linearized equations are written in a Hamiltonian form
where the symplectic operators $J$ turn out to be non-invertible unbounded
operators. The index formula and the exponential trichotomy estimates are
obtained from Theorems \ref{theorem-dichotomy} and \ref{theorem-counting}.

For periodic waves, the linearized equations for perturbations of the same
period are again written in the Hamiltonian form with $J$ having nontrivial
kernels. This brings changes to the index counting formula and stability
criteria. In recent years, similar index formula had been studied in various
cases. Our results give a unified treatment for general dispersion symbols.
For both solitary waves and periodic waves, the linear stability conditions
are also shown to imply nonlinear orbital stability. For the unstable cases,
the exponential dichotomy can be used to show nonlinear instability and even
to further construct local invariant (stable, unstable and center) manifolds
near the traveling wave orbit in the energy space. Moreover, when a.) the
negative dimension of the linearized energy functional is equal to the
unstable dimension of the linearized equation and b.) the kernel of the
linearized energy functional is generated exactly by the symmetry group of the
system, the orbital stability and local uniqueness on the center manifold
could be obtained. These invariant manifolds also give a complete description
of dynamics near the orbit of unstable profiles. For more details, we refer to
recent papers (\cite{jin-et-kdv} \cite{jin-et-GP}) on the construction of
invariant manifolds near unstable traveling waves of supercritical KDV
equation and 3D Gross-Pitavaeskii equation.

We then consider the linearized problems arisen from the modulational
(Benjamin-Feir, side-band) instability of period waves. Besides obtaining an
index formula for each Floquet-Block problem, we also carry out some
perturbation analysis to justify that unstable modes in the long wave limit
can only arise from zero eigenvalue of the co-periodic problem. Subsequently
we obtain the semigroup estimates for both multi-periodic and localized
perturbations, which played an important role on the recent proof
(\cite{lin-liao-jin-modulational}) of nonlinear modulational instability of
various dispersive models.

As another application, we consider the eigenvalue problem of the form
$Lu=\lambda u^{\prime}$, which arises in the stability of traveling waves of
generalized Bullough--Dodd equation (\ref{eqn-BD}). Let $J=\partial_{x}^{-1}$,
then it is equivalent to the Hamiltonian form $JLu=\lambda u$. Thus general
theorems can be applied to get instability index formula and the stability
criterion which generalize the results in \cite{stefnov-L-prime16} by relaxing
some restrictions. In particular it implies the linear instability of any
traveling wave of generalized Bullough--Dodd equation (\ref{eqn-BD}), removing
the convexity assumption in \cite{stefnov-L-prime16}.

Next, we consider stability/instability of steady flows of 2D Euler equation
in a bounded domain. For a large class of steady flows, the linearized Euler
equation can be written in a Hamiltonian form satisfying \textbf{(H1)-(H3)}.
Here, the symplectic operator $J$ has an infinite dimensional kernel. The
index formula is obtained in terms of a reduced operator related to the
projection to $\ker L$. By using the perturbation theory in Section
\ref{SS:SS}, the structural instability in the case of the presence of
embedded eigenvalues is shown. The Hamiltonian structures are also useful in
studying the enhanced damping and inviscid damping problems.

Lastly, we study the stability of traveling waves of 2D nonlinear
Schr\"{o}dinger equations with nonzero condition at infinity. When written in
the Hamiltonian form $JL$, the quadratic form $\left\langle L\cdot
,\cdot\right\rangle $ does not have uniform lower bound on the positive
subspace $X_{+}$. The strategy used for the 3D case (\cite{lin-wang-zeng})
does not work in 2D. We use the theory in Section \ref{SS:degenerate} to
construct a new and larger phase space to recover the uniform positivity of
$\left\langle L\cdot,\cdot\right\rangle $ on the positive space. Then the
theory in Section \ref{section-index theorem} is used to prove the stability
criterion in terms of the sign of $dP/dc$, where $P\left(  c\right)  $ is the
momentum of a traveling wave of speed $c$. As a somewhat unusual application
of the index formula, we prove the positivity of the momentum $P$ for
traveling waves (in both 2D and 3D) with general nonlinear terms.

\section{Basic properties of Linear Hamiltonian systems}

\label{S:Preliminary}

In this section, we present a few basic qualitative properties of the linear
equation \eqref{eqn-hamiltonian}, including the conservation of energy, some
elementary spectral properties, \textit{etc.} As our problem is set up in a
functional analysis theoretical framework, in some cases we have to follow the
painful rigor at an orthodox level. To make it less tedious, we only keep
those basic results directly related to the dynamics of
\eqref{eqn-hamiltonian} in this section, while some more elementary properties
of \eqref{eqn-hamiltonian}, including its well-posedness (Proposition
\ref{P:well-posedness}), are left in Section \ref{S:appendix}, the Appendix.

Like any Hamiltonian flow, we have the conservation of energy and the
symplectic structure of the flow defined by \eqref{eqn-hamiltonian}.

\begin{lemma}
[\cite{mackay86}]\label{lemma-L-form-preserve} For any solutions $u(t),v(t)$
of (\ref{eqn-hamiltonian}), then we have

\begin{enumerate}
\item $\frac{d}{dt}\left\langle Lu( t) ,v( t) \right\rangle =0$;

\item $R(J)$ is invariant under $e^{tJL}$; and

\item if $J$ is one-to-one ($J^{-1}$ not necessarily bounded) and $u(0) \in
R(J)$, then $\frac{d}{dt}\left\langle J^{-1} u( t) ,v( t) \right\rangle =0$.
\end{enumerate}
\end{lemma}

\begin{proof}
Property (1) is clearly true if $u(0), v(0) \in D(JL)$ and then the general
case follows immediately from a density argument. To prove (2), we first
notice, for $x \in D(JL)$,
\[
e^{tJL} x - x = J \int_{0}^{t} L e^{t^{\prime}JL} x dt^{\prime}.
\]
Since $J$ is closed and $D(JL)$ is dense, a density argument implies that
$\int_{0}^{t} L e^{t^{\prime}JL} x dt^{\prime}\in D(J)$ for all $x$ and the
above equality holds for all $x$ and thus $R(J)$ is invariant under $e^{tJL}$.
For (3), first consider $v(0) \in D(JL)$ and the above equality yields
\begin{align*}
\frac{d}{dt}\left\langle J^{-1} u( t) ,v( t) \right\rangle = \langle L u(t),
v(t)\rangle+ \langle J^{-1} u(t), JL v(t) \rangle=0
\end{align*}
where we used the assumption that $J$ is anti-self-adjoint. Again the general
case of (3) follows from the density of $D(JL)$.
\end{proof}

An immediate consequence of the conservation of the quadratic form $\langle
L\cdot, \cdot\rangle$ is on invariant subspaces.

\begin{lemma}
\label{L:InvariantSubS} Suppose a subspace $X_{1} \subset X$ is invariant
under $e^{tJL}$, i.e. $e^{tJL} X_{1} \subset X_{1}$ for all $t\in\mathbf{R}$,
then $e^{tJL} X_{2} \subset X_{2}$ for all $t \in\mathbf{R}$ where the closed
subspace $X_{2} = \{u\in X\mid\langle Lu, v\rangle=0, \; \forall v \in
X_{1}\}$.
\end{lemma}

\begin{proof}
For any $u\in X_{2}$, $v\in X_{1}$, and $t \in\mathbf{R}$, Lemma
\ref{lemma-L-form-preserve} and the invariance of $X_{1}$ imply
\[
\langle L e^{tJL} u, v\rangle= \langle L u, e^{-tJL} v\rangle= 0
\]
which yields the conclusion.
\end{proof}

While in a substantial part of the paper, we shall work with the real Hilbert
space $X$ and real operators $J,L$, \textit{etc.}, for considerations where
complex eigenvalues are involved, we have to work with their standard
\textbf{complexification}. See the Appendix (Section \ref{S:appendix}) for details.

Let $\lambda$ be an eigenvalue of $JL$ (i.e. $\lambda\in\sigma\left(
JL\right)  $) and
\[
E_{\lambda}= \{ u\in X\ |\ ( JL-\lambda I) ^{k} u=0,\ \text{for some integer
}k\geq1\}.
\]
Then by Lemma \ref{lemma-L-form-preserve}, we have

\begin{lemma}
[Lemma 2 in \cite{mackay86} or Lemma 2.7 in \cite{kappitula-haragus}%
]\label{lemma-orthogonal-eigenspace} If $v_{1}\in E_{\lambda_{1}},v_{2}\in
E_{\lambda_{2}}$ and $\lambda_{1}+\bar{\lambda}_{2}\neq0$, then $\left\langle
Lv_{1},v_{2}\right\rangle =0$.
\end{lemma}

The following lemma will be repeatedly used to analyze the structure of
$E_{i\mu}$, $\mu\in\mathbf{R}$.

\begin{lemma}
\label{lemma-orthogonal-condition-chain} For any $i\mu\in\sigma\left(
JL\right)  $ $\left(  \mu\in\mathbf{R}\right)  $, $u_{1} \in(JL-i\mu)^{l} X$,
and $u_{2} \in\ker(JL - i\mu)^{l}$, then $\left\langle Lu_{1}, u_{2}%
\right\rangle =0$.
\end{lemma}

\begin{proof}
First, we observe that for any $u,v\in X$,
\begin{equation}
\label{E:anti-S}\left\langle L\left(  JL-i\mu\right)  u,v\right\rangle
=-\left\langle Lu,\left(  JL-i\mu\right)  v\right\rangle .
\end{equation}
Let $v\in X$ such that $(JL-i\mu)^{l} v = u_{1}$, then we have
\[
\langle Lu_{1},u_{2}\rangle= \langle L (JL-i\mu)^{l} v,u_{2}\rangle= (-1)^{l}
\langle L v, (JL-i\mu)^{l} u_{2}\rangle=0.
\]

\end{proof}

The following lemma is a direct consequence of Lemma
\ref{lemma-orthogonal-condition-chain} and \eqref{E:anti-S}.

\begin{lemma}
\label{L:e-space-1} For any $i \mu\in\sigma(JL) \cap i\mathbf{R}$, it holds
\[
E_{i\mu} = \ker(JL- i\mu)^{2k^{\le0} (i\mu)+1}, \; \mu\ne0, \; \text{ and } \;
E_{0} = \ker(JL)^{2k_{0}^{\le0} +2}.
\]

\end{lemma}

\begin{remark}
\label{R:e-space-1} As $JL-i\mu$ is a generator of a strongly $C^{0}$
semigroup, $(JL-i\mu)^{m}$ is closed for any $m$ and thus $E_{i\mu}$ is a
closed subspace and $JL|_{E_{i\mu}}$ is a bounded operator with $\sigma
(JL|_{E_{i\mu}}) = \{i\mu\}$.
\end{remark}

\begin{proof}
We first consider $\mu\neq0$ and argue by contradiction. Suppose $u\in
E_{i\mu}$ such that
\[
(JL-i\mu)^{K}u=0,\quad(JL-i\mu)^{K-1}u\neq0,\quad K\geq2k^{\leq0}(i\mu)+2.
\]
For any $K-1\geq j_{1},j_{2}\geq K-k^{\leq0}(i\mu)-1$, we obtain from Lemma
\ref{lemma-orthogonal-condition-chain}
\[
\langle L(JL-i\mu)^{j_{1}}u,(JL-i\mu)^{j_{2}}u\rangle=0.
\]
Therefore, the quadratic form $\langle L\cdot,\cdot\rangle$ vanishes on
$span\{(JL-i\mu)^{K-1}u,(JL-i\mu)^{K-2},\ldots,(JL-i\mu)^{K-k^{\leq0}(i\mu
)-1}u\}$, whose dimension is $k^{\leq0}(i\mu)+1$. This contradicts the
definition of $k^{\leq0}(i\mu)$.

To finish the proof, we consider $\mu=0$. Again we argue by contradiction.
Suppose $u\in E_{0}$ is such that
\[
(JL)^{K}u=0,\quad(JL)^{K-1}u\neq0,\quad K\geq2k_{0}^{\leq0}+3.
\]

\textit{Case 1. $(JL)^{K-1}u\notin\ker L$.} In this case, clearly
\[
span\{(JL)^{j}u\mid0\leq j\leq K-1\}\cap\ker L=\{0\}.
\]
Let $\tilde{E}_{0}\subset E_{0}$ be a subspace such that $E_{0}=\tilde{E}%
_{0}\oplus\ker L$ and $(JL)^{j}u\in\tilde{E}_{0}$ for any $0\leq j\leq K-1$.
Much as in the above, $\langle L\cdot,\cdot\rangle$ vanishes on
\[
Z\triangleq span\{(JL)^{K-1}u,(JL)^{K-2}u,\ldots,(JL)^{K-k_{0}^{\leq0}%
-1}u\}\subset\tilde{E}_{0}.
\]
Since $\dim Z=k_{0}^{\leq0}+1$, this is a contradiction to the definition of
$k_{0}^{\leq0}$.

\textit{Cases 2. $(JL)^{K-1}u\in\ker L\backslash\{0\}$.} Clearly,
\[
span\{(JL)^{j}u\mid0\leq j\leq K-2\}\cap\ker L=\{0\}.
\]
Let $\tilde{E}_{0}\subset E_{0}$ be a subspace such that $E_{0}=\tilde{E}%
_{0}\oplus\ker L$ and $(JL)^{j}u\in\tilde{E}_{0}$ for any $0\leq j\leq K-2$.
Let
\[
Z\triangleq span\{(JL)^{K-2}u,(JL)^{K-3}u,\ldots,(JL)^{K-k_{0}^{\leq0}%
-2}u\}\subset\tilde{E}_{0}.
\]
According to Lemma \ref{lemma-orthogonal-condition-chain}, for $K-k_{0}%
^{\leq0}-2\leq j_{1},j_{2}\leq K-2$ and $j_{1}+j_{2}\geq K$, we have $\langle
L(JL)^{j_{1}}u,(JL)^{j_{2}}u\rangle=0$. If $K-k_{0}^{\leq0}-2\leq j_{1}%
,j_{2}\leq K-2$ and $j_{1}+j_{2}<K$, it must hold $j_{1}=j_{2}=K-k_{0}^{\leq
0}-2$ and $K=2k_{0}^{\leq0}+3$. Using \eqref{E:anti-S} we obtain
\[
\langle L(JL)^{K-k_{0}^{\leq0}-2}u,(JL)^{K-k_{0}^{\leq0}-2}u\rangle
=(-1)^{K-k_{0}^{\leq0}-2}\langle L(JL)^{K-1}u,u\rangle=0
\]
where in the last equality we used $(JL)^{K-1}u\in\ker L$. Therefore, $\langle
L\cdot,\cdot\rangle$ vanishes on $Z$. Since $\dim Z=k_{0}^{\leq0}+1$, this is
again a contradiction to the definition of $k_{0}^{\leq0}$. The proof of the
lemma is complete.
\end{proof}

To end the section of basic properties, we prove the following Lemma on the
symmetry of $\sigma\left(  JL\right)  $ about both axes.

\begin{lemma}
\label{L:symmetry} Assume (\textbf{H1})-(\textbf{H3}), except for
$n^{-}\left(  L\right)  <\infty$. Suppose $\lambda\in\sigma\left(  JL\right)
$, then we have

i) $\pm\lambda,\pm\bar{\lambda}\in\sigma\left(  JL\right)  $.

ii) Suppose $\lambda$ is an eigenvalue of $JL$ and assume in addition $\ker
L=\{0\}$ or $\lambda\neq0$, then $\bar{\lambda}$ is also an eigenvalue of $JL$
and $-\lambda,-\bar{\lambda}$ are eigenvalues of $(JL)^{\ast}=-LJ$. Moreover,
for any $k>0$,
\begin{equation}
\ker(JL-\bar{\lambda})^{k}=\{\bar{u}\mid u\in\ker(JL-\lambda)^{k}%
\}\label{E:bar-lambda}%
\end{equation}
and
\begin{equation}
L:\ker(JL-\lambda)^{k}\rightarrow\ker\big((JL)^{\ast}+\bar{\lambda}%
\big)^{k}=\ker(LJ-\bar{\lambda})^{k}\label{E:lambda-sym}%
\end{equation}
is an anti-linear isomorphism.

iii) Suppose $\lambda$ is an isolated eigenvalue of $JL$ with finite algebraic
multiplicity, then $-\lambda,\pm\bar{\lambda}$ are also eigenvalues of $JL$
with the same algebraic and geometric multiplicities.
\end{lemma}

Here the operators $J$ and $L$ are understood as their complexification, and
thus are anti-linear mappings satisfying \eqref{E:anti-linear}.

\begin{proof}
As i) is trivial if $\lambda=0$, so we assume $\lambda\neq0$ or $\ker L=\{0\}$.

Due to \eqref{E:real} which states that $JL$ is real, \eqref{E:bar-lambda} and
$\bar{\lambda}\in\sigma(JL)$ follow immediately. We are left to prove
$-\lambda,-\bar{\lambda}\in\sigma(JL)$ and \eqref{E:lambda-sym}.

The anti-linearity property \eqref{E:anti-linear} implies
\begin{equation}
L(JL-\lambda)u=(LJ-\bar{\lambda})Lu=-{\large (}(JL)^{\ast}+\bar{\lambda
}{\large )}Lu,\;\forall u\in D(JL).\label{E:EValue1}%
\end{equation}
Therefore, we have that, for any integer $k>0$,
\begin{equation}
\big((JL)^{\ast}+\bar{\lambda}\big)^{k}Lu=(-1)^{k}L(JL-\lambda)^{k}u,\;\forall
u\in D\big((JL)^{k}\big).\label{E:EValue2}%
\end{equation}

It follows from \eqref{E:EValue2} that $L\big(\ker(JL-\lambda)^{k}%
\big)\subset\ker\big((JL)^{\ast}+\bar{\lambda}\big)^{k}$. Under the assumption
$\lambda\neq0$ or $\ker L=\{0\}$, it holds $\ker L\cap E_{\lambda}=\{0\}$ and
thus $L$ is one-to-one on $E_{\lambda}$. Therefore, if $\lambda$ is an
eigenvalue of $JL$, then $E_{\lambda}$ is nontrivial which implies
$L\big(\ker(JL-\lambda)^{k}\big)$, as well as $\ker\big((JL)^{\ast}%
+\bar{\lambda}\big)^{k}$, are nontrivial. We obtain that $-\bar{\lambda}$, as
well as $-\lambda$, is an eigenvalue of $(JL)^{\ast}$. Consequently
$-\lambda,-\bar{\lambda}\in\sigma(JL)$.

To finish the proof of \eqref{E:lambda-sym}, we only need to show
\[
L\big(\ker(JL-\lambda)^{k}\big)\supset\ker\big((JL)^{\ast}+\bar{\lambda
}\big)^{k}.
\]
This is obvious from \eqref{E:EValue2} if $\ker L=\{0\}$. In the case of
$\lambda\neq0$, it is clear $\ker\big((JL)^{\ast}+\bar{\lambda}\big)^{k}%
\subset R(L)$. Therefore, for any $v\in\ker\big((JL)^{\ast}+\bar{\lambda
}\big)^{k}$, there exists $u_{1}\in X$ such that $v=Lu_{1}$. Equation
\eqref{E:EValue2} again implies
\[
w=(JL-\lambda)^{k}u_{1}\in\ker L.
\]
As $\lambda\neq0$, let $u=u_{1}-(-\lambda)^{-k}w$, then since $\left(
JL-\lambda\right)  w=\left(  -\lambda\right)  w$,$\ $we have $v=Lu$ and
$u\in\ker(JL-\lambda)^{k}$ due to $(JL-\lambda)^{k}w=\left(  -\lambda\right)
^{k}w$. Therefore, $\ker\big((JL)^{\ast}+\bar{\lambda}\big)^{k}\subset
L\ker(JL-\lambda)^{k}$ and thus $L$ is an one to one correspondence (actually
an anti-linear isomorphism) from $\ker(JL-\lambda)^{k}$ to $\ker
\big((JL)^{\ast}+\bar{\lambda}\big)^{k}$.

Finally, suppose $\lambda\neq0$ and $R(JL-\lambda)\neq X$, we will show
$R\big((JL)^{\ast}+\bar{\lambda}\big)\neq X^{\ast}$ which implies
$-\bar{\lambda},-\lambda\in\sigma\big((JL)^{\ast}\big)=\overline{\sigma\left(
JL\right)  }$ and thus completes the proof of (i). Assume one the contrary
$R\big((JL)^{\ast}+\bar{\lambda}\big)=X^{\ast}$. Let $\gamma\in D(J)\backslash
R(L)$, according to Remark \ref{R:closedness}, there exists $u\in\ker L$ such
that $\langle\gamma,u\rangle\neq0$. One can compute $\langle\big((JL)^{\ast
}+\bar{\lambda}\big)\gamma,u\rangle=\bar{\lambda}\langle\gamma,u\rangle\neq0$
and thus $\big((JL)^{\ast}+\bar{\lambda}\big)\gamma\notin R(L)$. Therefore, if
$R\big((JL)^{\ast}+\bar{\lambda}\big)=X^{\ast}$, it must hold $\big((JL)^{\ast
}+\bar{\lambda}\big)\big(R(L)\big)=R(L)$, which is the range of the right side
of \eqref{E:EValue1}. However, since $\lambda\neq0$, we have $(JL-\lambda
)(\ker L)=\ker L$. Along with $R(JL-\lambda)\neq X$, it implies
$R(L)\not \subset R\big(L(JL-\lambda)\big)$, which is the range of the left
side of \eqref{E:EValue1}. We obtain a contradiction and thus
$R\big((JL)^{\ast}+\bar{\lambda}\big)\neq X^{\ast}$.

If $\lambda\in\sigma\left(  JL\right)  $ is isolated and of finite
multiplicity, then the same is true for $\bar{\lambda}$. By i) and ii),
$-\lambda,-\bar{\lambda}\in\sigma\big((JL)^{\ast}\big)$ are also isolated and
of the same multiplicities, this implies that $-\lambda,-\bar{\lambda}$
$\in\sigma\left(  JL\right)  $ have the same (geometric and algebraic)
multiplicities (see \cite{kato-book} P. 184).
\end{proof}

\begin{remark}
\label{R:inf-dim-nega-2} As in the proof of Lemma \ref{L:decomJ} and Corollary
\ref{L:decomJ}, the assumption $n^{-}(L)<\infty$ is not required in the above
proof. So Lemma \ref{L:symmetry} holds even when $n^{-}(L)=\infty$. On the
other hand, this lemma gives the symmetry of $\sigma(JL)$, but not for general
eigenvalues, except for purely imaginary eigenvalues or isolated eigenvalues
of finite multiplicity. If $\lambda\in\sigma\left(  JL\right)  \ $is a nonzero
eigenvalue which is non-isolated or of infinite multiplicity, then above lemma
implies that $-\lambda,-\bar{\lambda}$ are eigenvalues of $\sigma
\big((JL)^{\ast}\big)$. In general, we can not exclude the possibility that
$-\lambda,-\bar{\lambda}$ are not eigenvalues of $JL$. However, when
$n^{-}(L)<\infty$, any $\lambda\in\sigma\left(  JL\right)  $ with
$\operatorname{Re}\lambda\neq0$ must be isolated and of finite multiplicity,
and the symmetry of eigenvalues and the dimensions of their eigenspaces are
given in Corollary \ref{C:symmetry}.
\end{remark}

\section{Finite dimensional Hamiltonian systems}

\label{S:finiteD}

In this section, we consider the case where the energy space $X$ of
(\ref{eqn-hamiltonian}) is $X=\mathbf{R}^{n}$ which is complexified to
$\mathbf{C}^{n}$. The assumptions (\textbf{H.1-3}) become that $J$ is a real
anti-symmetric $n\times n\ $matrix and $L$ is a real symmetric $n\times
n\ $matrix. The counting formula (\ref{counting-formula}) essentially follows
from \cite{mackay86}, except for the formula (\ref{counting-pure-imaginary}).
We do not need to assume that $J$ is invertible as assumed in \cite{mackay86}.

For $\lambda\in\sigma\left(  JL\right)  ,\ $define
\[
I_{\lambda}=E_{\lambda}\oplus E_{-\bar{\lambda}} \text{ if } \lambda\notin
i\mathbf{R}, \text{ and } I_{\lambda}=E_{\lambda} \text{ if } \lambda\in i
\mathbf{R}.
\]
We have $\mathbf{C}^{n}=I_{\lambda_{1}}\oplus\cdots\oplus I_{\lambda_{l}}$,
where $\lambda_{j}\in\sigma\left(  JL\right)  $ are all distinct eigenvalues
of $JL$ with Re$\lambda_{j} \ge0$. By Lemma \ref{lemma-orthogonal-eigenspace},
we have
\begin{equation}
n^{-}\left(  L\right)  =\sum_{j}n^{-}\left(  L|_{I_{\lambda_{j}}}\right)  .
\label{identity-negative-index}%
\end{equation}
Based on Lemma \ref{L:symmetry} of the symmetry of $\sigma(JL)$, to prove
Theorem \ref{theorem-counting} in the finite dimensional case, it suffices to
compute $n^{-}\left(  L|_{I_{\lambda}}\right)  $ for any $0\ne\lambda\in
\sigma\left(  JL\right)  \backslash i\mathbf{R}$.

\begin{lemma}
[\cite{mackay86}]\label{lemma-non-degenerate-finite-d} Let $\lambda\in
\sigma\left(  JL\right)  $. Assume $\ker L =\{0\}$ or $\lambda\ne0$, then the
restriction $\left\langle L\cdot,\cdot\right\rangle |_{I_{\lambda}}$ is non-degenerate.
\end{lemma}

\begin{proof}
Suppose $\left\langle L\cdot,\cdot\right\rangle |_{I_{\lambda}}$ is
degenerate. Then there exists $0 \ne u\in I_{\lambda}$ such that $\left\langle
Lu,v\right\rangle =0$ for any $v\in I_{\lambda}$. Since $\mathbf{C}^{n}$ is
the direct sum of all different $I_{\lambda^{\prime}}$, $\lambda^{\prime}%
\in\sigma(JL)$, this implies that $\left\langle Lu,v\right\rangle =0$ for any
$v\in\mathbf{C}^{n}$ by Lemma \ref{lemma-orthogonal-eigenspace}. So $Lu=0$ and
thus $0\ne u\in I_{\lambda}\cap\ker L$. It implies that $\lambda=0$ and $\ker
L \ne\{0\}$, a contradiction to our assumptions.
\end{proof}

\begin{lemma}
[\cite{mackay86} or \cite{kappitula-haragus}]%
\label{lemma-negative-unstable-space} If $\operatorname{Re}\lambda>0$ and let
$m_{\lambda}$ to be the algebraic multiplicity of $\lambda$. Then
$n^{-}\left(  L|_{I_{\lambda}}\right)  =m_{\lambda}$.
\end{lemma}

\begin{proof}
From Lemma \ref{lemma-orthogonal-eigenspace}, the quadratic form $\langle
L\cdot, \cdot\rangle$ on $I_{\lambda}= E_{\lambda}\oplus E_{-\bar\lambda}$ can
be represented in the block form $%
\begin{pmatrix}
0 & A\\
A^{*} & 0
\end{pmatrix}
$. Lemma \ref{lemma-non-degenerate-finite-d} implies the non-degeneracy of $A$
and thus the lemma follows.
\end{proof}

The counting formula (\ref{counting-formula}) in the finite dimensional case
follows from these lemmas and (\ref{identity-negative-index}).

In the rest of this subsection, we carefully analyze $k^{-}(i\mu) =
n^{-}\left(  L|_{E_{i\mu}}\right)  $, $\mu\in\mathbf{R}$, and obtain
Proposition \ref{P:basis} in finite dimensions. Based on Lemma
\ref{lemma-orthogonal-condition-chain} and equation \eqref{E:anti-S}, we first prove

\begin{lemma}
\label{L:signature1} Suppose $i\mu\in\sigma\left(  JL\right)  $ $\left(
\mu\in\mathbf{R}\right)  $ and $K>0$ is an integer, then

\begin{enumerate}
\item for $u, v \in\ker(JL-i\mu)^{K}$,
\[
Q_{K} (u,v) \triangleq i^{K-1} \langle L (JL-i\mu)^{K-1} u, v\rangle
\]
defines a Hermitian form on $\ker(JL-i\mu)^{K}$; and

\item assume $\ker L =\{0\}$ or $\mu\ne0$, then
\[
Y_{K} \triangleq\big(\ker(JL-i\mu)^{K} \cap R(JL-i\mu) \big) + \ker
(JL-i\mu)^{K-1} = \ker Q_{K}.
\]

\end{enumerate}
\end{lemma}

\begin{proof}
That $Q_{K}$ is a Hermitian form on $\ker(JL-i\mu)^{K}$ is an immediate
consequence of equation \eqref{E:anti-S}. Lemma
\ref{lemma-orthogonal-condition-chain} also implies $Y_{K}\subset\ker Q_{K}$.
We will show $Y_{K}=\ker Q_{K}$ under the additional assumption $\ker L=\{0\}$
or $\mu\neq0$. Suppose $u\in\ker(JL-i\mu)^{K}$ is such that
\[
Q_{K}(u,v)=i^{K-1}\langle L(JL-i\mu)^{K-1}u,v\rangle=0,\quad\forall\ v\in
\ker(JL-i\mu)^{K}.
\]
By duality, it implies that
\[
L(JL-i\mu)^{K-1}u\in\big((JL-i\mu)^{K})\big)^{\ast}(\mathbf{C}^{n}%
)=(LJ-i\mu)^{K}(\mathbf{C}^{n}).
\]
Therefore, there exists $w\in\mathbf{C}^{n}$ such that
\[
L(JL-i\mu)^{K-1}u=(LJ-i\mu)^{K}w.
\]
Since $\mu\neq0$ or $L$ is surjective, the above equation implies $w\in R(L)$
and thus there exists $\tilde{w}\in\mathbf{C}^{n}$ such that $w=L\tilde{w}$.
Consequently,
\[
(LJ-i\mu)^{K-1}L\big(u-(JL-i\mu)\tilde{w}\big)=L(JL-i\mu)^{K-1}u-(LJ-i\mu
)^{K}w=0
\]
which along with Lemma \ref{L:symmetry} implies
\[
L\big(u-(JL-i\mu)\tilde{w}\big)\in\ker(LJ-i\mu)^{K-1}=L\ker(JL-i\mu)^{K-1}.
\]
Therefore, there exists $v\in\ker(JL-i\mu)^{K-1}$ such that
\[
y=u-(JL-i\mu)\tilde{w}-v\in\ker L.
\]
If $\mu\neq0$, let $w_{1}=\tilde{w}+\frac{1}{i\mu}y$. If $\ker L=\{0\}$, we
have $y=0$ and let $w_{1}=\tilde{w}$. In both cases, we have
\[
u=v+(JL-i\mu)w_{1},\quad v\in\ker(JL-i\mu)^{K-1}\subset\ker(JL-i\mu)^{K}.
\]
Therefore, $(JL-i\mu)w_{1}\in R(JL-i\mu)\cap\ker(JL-i\mu)^{K}$ and then $u\in
Y_{K}$. The proof is complete.
\end{proof}

\begin{corollary}
\label{C:signature1} Assume $\ker L =\{0\}$ or $\mu\ne0$, then $Q_{K}$ induces
a non-degenerate Hermitian form on the quotient space $\ker(JL-i\mu)^{K} /
Y_{K}$.
\end{corollary}

If $K$ is odd, for any $u, v \in\ker(JL-i\mu)^{K}$, clearly
\begin{equation}
\label{E:Q_K}Q_{K} (u,v) = \langle L (JL-i\mu)^{\frac{K-1}2} u, (JL-i\mu
)^{\frac{K-1}2} v\rangle.
\end{equation}

\begin{definition}
\label{D:signature1} For odd $K$, define $n_{K}^{-}(i\mu)$ to be the negative
index of the quadratic form $Q_{K}$.
\end{definition}

The above quotient space $\ker(JL-i\mu)^{K} / Y_{K}$ is closely related to
Jordan chains. Suppose a basis of $\mathbf{C}^{n}$ realizes the Jordan
canonical form of $JL$, and there are totally $l$ Jordan blocks of size $K
\times K$ corresponding to $i\mu$. There must be $l$ Jordan chains of length
$K$ in such basis, each of which is generated by some $v \in\ker(JL-i\mu)^{K}
\slash Y_{K}$ as
\[
v, \ (JL-i\mu) v, \ \ldots, \ (JL-i\mu)^{K-1} v.
\]
From standard linear algebra, we have the following lemma.

\begin{lemma}
\label{L:JordanChain1} Vectors $v_{1,1},\ldots,v_{l,1}$ generate all $l$
Jordan chains of length $K$ in the sense that
\[
v_{j,k}=(JL-i\mu)^{k-1}v_{j,1},\quad1\leq k\leq K,\text{ }1\leq j\leq l,
\]
are in a basis of $\mathbf{C}^{n}$ realizing all $l\ $Jordan blocks of size
$K$ of $JL$ corresponding to $i\mu\in\sigma(JL)$, if and only if
\[
v_{1,1}+Y_{K},\ \ldots,\ v_{l,1}+Y_{K}%
\]
form a basis of $\ker(JL-i\mu)^{K}/Y_{K}$.
\end{lemma}

The following lemma would lead to the realization of the Jordan canonical form
of $JL$ and skew-diagonalization of $L$ simultaneously.

\begin{lemma}
\label{L:JordanChain2} Assume $\ker L =\{0\}$ or $\mu\ne0$ where $i\mu
\in\sigma(JL) \cap i\mathbf{R}$. Suppose $\dim\ker(JL-i\mu)^{K}\slash Y_{K}
=l>0$ and $Z \subset\ker(JL-i\mu)^{K}$ satisfies
\[
JL (Z) = Z \; \text{ and } \; Z\slash (Y_{K} \cap Z) = \ker(JL-i\mu
)^{K}\slash Y_{K},
\]
then there exist $v_{1}, \ldots, v_{l} \in Z$ such that
\begin{equation}
\label{E:basis1}\langle L (JL-i\mu)^{m} v_{j}, v_{k}\rangle= \pm i^{K-1}
\delta_{j,k} \delta_{m, K-1}, \quad0\le m \le K-1.
\end{equation}

\end{lemma}

\begin{proof}
Since $Q_{K}$ induces a non-degenerate Hermitian form on $\ker(JL-i\mu
)^{K}\slash Y_{K}=Z\slash(Z\cap Y_{K})$, there exist $w_{1},\ldots,w_{l}\in Z$
such that $w_{1}+Y_{K},\ldots,w_{l}+Y_{K}$ form a basis of $\ker(JL-i\mu
)^{K}\slash Y_{K}$ and diagonalize $Q_{K}$, that is,
\[
\langle L(JL-i\mu)^{K-1}w_{j},w_{k}\rangle=(-i)^{K-1}Q_{K}(w_{j},w_{k})=\pm
i^{K-1}\delta_{j,k}.
\]
Therefore, we have found $w_{1},\ldots,w_{l}$ satisfying \eqref{E:basis1} for
$m=K-1$.

Suppose $1\leq m_{0}+1\leq K-1$ and we have found $w_{1},\ldots,w_{l}\in Z$
satisfying \eqref{E:basis1} for $m\geq m_{0}+1$. Denote
\[
\alpha=K-1-m_{0}\geq1,\;Q_{K}(w_{j},w_{k})=b_{j,k}=\pm\delta_{j,k},\;\langle
L(JL-i\mu)^{m_{0}}w_{j},w_{k}\rangle=c_{j,k}.
\]
In the next step we will construct $v_{1},\ldots,v_{l}$ satisfying
\eqref{E:basis1} for $m\geq m_{0}$ in the form of
\[
v_{j}=w_{j}+\sum_{j^{\prime}=1}^{j}a_{j,j^{\prime}}(JL-i\mu)^{\alpha
}w_{j^{\prime}}\in Z.
\]
According to \eqref{E:anti-S}, $\langle L(JL-i\mu)^{m}\cdot,\cdot\rangle$ is
Hermitian or anti-Hermitian. Without loss of generality, we may consider only
$j\leq k$ in \eqref{E:basis1}. Compute using \eqref{E:anti-S}
\begin{equation}%
\begin{split}
\langle L(JL &  -i\mu)^{m}v_{j},v_{k}\rangle=(-1)^{\alpha}\sum_{k^{\prime}%
=1}^{k}\overline{a_{k,k^{\prime}}}\langle L(JL-i\mu)^{\alpha+m}w_{j}%
,w_{k^{\prime}}\rangle\\
&  +\langle L(JL-i\mu)^{m}w_{j},w_{k}\rangle+\sum_{j^{\prime}=1}%
^{j}a_{j,j^{\prime}}\langle L(JL-i\mu)^{\alpha+m}w_{j^{\prime}},w_{k}\rangle\\
&  +(-1)^{\alpha}\sum_{j^{\prime}=1}^{j}\sum_{k^{\prime}=1}^{k}a_{j,j^{\prime
}}\overline{a_{k,k^{\prime}}}\langle L(JL-i\mu)^{2\alpha+m}w_{j^{\prime}%
},w_{k^{\prime}}\rangle.
\end{split}
\label{E:basis2}%
\end{equation}
If $m+\alpha=K-1-m_{0}+m\geq K$, the induction assumption and the above
equation imply
\[
\langle L(JL-i\mu)^{m}v_{j},v_{k}\rangle=\langle L(JL-i\mu)^{m}w_{j}%
,w_{k}\rangle=\pm i^{K-1}\delta_{j,k}\delta_{m,K-1}%
\]
and thus \eqref{E:basis1} for $m\geq m_{0}+1$ holds for these $v_{1}%
,\ldots,v_{l}$ with any choices of $a_{j,j^{\prime}}$. For $m=m_{0}$, i.e.
$m+\alpha=K-1$, if $j<k$, \eqref{E:basis2} implies
\[
\langle L(JL-i\mu)^{m_{0}}v_{j},v_{k}\rangle=c_{j,k}+(-1)^{\alpha}%
(-i)^{K-1}\overline{a_{k,j}}b_{j,j}.
\]
Noticing $b_{j,j}=\pm1$ and letting
\[
a_{k,j}=(-1)^{\alpha+1}(-i)^{K-1}\overline{c_{j,k}}b_{j,j},\;j<k,
\]
then we have
\[
\langle L(JL-i\mu)^{m_{0}}v_{j},v_{k}\rangle=0,\quad j<k.
\]
If $j=k$,
\[
\langle L(JL-i\mu)^{m_{0}}v_{k},v_{k}\rangle=c_{k,k}+(-i)^{K-1}b_{k,k}%
\big(a_{k,k}+(-1)^{\alpha}\overline{a_{k,k}})\big).
\]
Let
\[
a_{k,k}=-\frac{1}{2}i^{K-1}b_{k,k}c_{k,k}=-\frac{1}{2}i^{\alpha}%
b_{k,k}i^{m_{0}}c_{k,k}.
\]
Since \eqref{E:anti-S} implies $c_{k,j}=(-1)^{m_{0}}\overline{c_{j,k}}$, we
have $i^{m_{0}}c_{k,k}\in\mathbf{R}$ and thus $i^{\alpha}a_{k,k}\in\mathbf{R}$
which makes it easy to verify
\[
\langle L(JL-i\mu)^{m_{0}}v_{k},v_{k}\rangle=0.
\]
Therefore, $v_{1},\ldots,v_{l}\in Z$ satisfy \eqref{E:basis1} for all $m\geq
m_{0}$ and the lemma follows from the induction.
\end{proof}

We are in a position to prove Proposition \ref{P:basis} in finite dimensions.
\newline

\textbf{Proof of Proposition \ref{P:basis} assuming $\dim X<\infty$ and $\ker
L=\{0\}$}: Let $E^{D}=\{0\}$, then $1<k_{1}<\cdots<k_{j_{0}}$ are the
dimensions of nontrivial Jordan blocks in $E_{i\mu}$, $\mu\in\mathbf{R}$, and
there are $l_{j}>0$ Jordan blocks of size $k_{j}$. For each $1\leq j\leq
j_{0}$, we will find linearly independent
\[
\{u_{p,q}^{(j)}\mid p=1,\ldots,l_{j},\ q=1,\ldots,k_{j}\}\subset E_{i\mu}%
\]
which form all Jordan chains of length $k_{j}$ and satisfy the desired
properties. The construction is by induction on $j$.

For $j=j_{0}$, applying Lemma \ref{L:JordanChain2} to $Z=\ker(JL-i\mu
)^{k_{j_{0}}}=E_{i\mu}$, where $\dim\ker(JL-i\mu)^{k_{j_{0}}}\slash
Y_{k_{j_{0}}}=l_{j_{0}}$ according to Lemma \ref{L:JordanChain1}, there exist
$u_{1,1}^{(j_{0})},\ldots,u_{l_{j_{0}},1}^{(j_{0})}$ such that
\begin{equation}
\langle L(JL-i\mu)^{m}u_{p_{1},1}^{(j_{0})},u_{p_{2},1}^{(j_{0})}\rangle=\pm
i^{k_{j_{0}}-1}\delta_{p_{1},p_{2}}\delta_{m,k_{j_{0}}-1},\quad0\leq m\leq
j_{0}-1. \label{E:basis3}%
\end{equation}
In particular we have $Q_{K}(u_{p_{1},1}^{(j_{0})},u_{p_{2},1}^{(j_{0})}%
)=\pm\delta_{p_{1},p_{2}}$. Lemma \ref{L:signature1} and Corollary
\ref{C:signature1} imply that $u_{1,1}^{(j_{0})}+Y_{k_{j_{0}}},\dots
,u_{l_{j_{0}},1}^{(j_{0})}+Y_{k_{j_{0}}}$ form a basis of $\ker(JL-i\mu
)^{k_{j_{0}}}\slash Y_{k_{j_{0}}}$. From Lemma \ref{L:JordanChain1}, we obtain
that
\[
u_{p,q}^{(j_{0})}=(JL-i\mu)^{q-1}u_{p,1}^{(j_{0})},\quad q=1,\cdots,k_{j_{0}%
},\;p=1,\ldots,l_{j_{0}}%
\]
form $l_{j_{0}}$ Jordan chains realizing all Jordan blocks of size $k_{j_{0}}$
of $JL$ corresponding to $i\mu\in\sigma(JL)$. Moreover, equation
\eqref{E:basis3} implies
\[
\langle Lu_{p_{1},q_{1}}^{(j_{0})},u_{p_{2},q_{2}}^{(j_{0})}\rangle=\pm
i^{k_{j_{0}}-1}\delta_{p_{1},p_{2}}\delta_{q_{1}+q_{2},k_{j_{0}}+1}.
\]

Suppose $0\leq j_{\ast}<j_{0}$ and we have constructed linearly independent
$u_{p,q}^{(j)}$ for all $j_{\ast}<j\leq j_{0}$, $1\leq p\leq l_{j}$, $1\leq
q\leq k_{j}$ satisfying
\begin{equation}
u_{p,q}^{(j)}=(JL-i\mu)^{q-1}u_{p,1}^{(j)},\quad\langle Lu_{p_{1},q_{1}}%
^{(j)},u_{p_{2},q_{2}}^{(j)}\rangle=\pm i^{k_{j}-1}\delta_{p_{1},p_{2}}%
\delta_{q_{1}+q_{2},k_{j}+1}.\label{E:basis4}%
\end{equation}
Clearly,
\[
Z_{1}=span\{u_{p,q}^{(j)}\mid j_{\ast}<j\leq j_{0},\ 1\leq p\leq l_{j},\ 1\leq
q\leq k_{j}\}\subset E_{i\mu}%
\]
is a subspace invariant under $JL$. Moreover, vectors $\{u_{p,q}^{(j)}\}$ form
a basis of $Z_{1}$ realizing the Jordan canonical form of $JL$ on $Z_{1}$
consisting of all those Jordan blocks of $JL$ corresponding to $i\mu$ of size
greater than $k_{j_{\ast}}$. According to \eqref{E:basis4}, the quadratic form
$\langle L\cdot,\cdot\rangle$ is non-degenerate on $Z_{1}$. In the next step
we will construct $u_{p,q}^{(j_{\ast})}$ for $1\leq p\leq l_{j_{\ast}}$ and
$1\leq q\leq k_{j_{\ast}}$. Let
\[
Z=\{u\in E_{i\mu}\mid\langle Lu,v\rangle=0,\ \forall v\in Z_{1}\}.
\]
Due to the non-degeneracy of $\langle L\cdot,\cdot\rangle$ on both $Z_{1}$ and
$I_{i\mu}=E_{i\mu}$ (Lemma \ref{lemma-non-degenerate-finite-d}), we have
$E_{i\mu}=Z_{1}\oplus Z$. For any $u\in Z$ and $v\in Z_{1}$, due to the
symmetry of $L$ and $J$, we have
\[
\langle LJLu,v\rangle=-\langle Lu,JLv\rangle=0,\;\text{ as }JLv\in Z_{1}%
\]
which implies $JL(Z)\subset Z$. Since the Jordan canonical form of $JL$ on
$Z_{1}$ includes all Jordan blocks of $JL$ on $E_{i\mu}$ of size greater than
$k_{j_{\ast}}$, the Jordan canonical form of $JL$ on $Z$ must be those Jordan
blocks of $JL$ on $E_{i\mu}$ of size no greater than $k_{j_{\ast}}$.
Therefore, $Z\subset\ker(JL-i\mu)^{k_{j_{\ast}}}$ and then Lemma
\ref{L:JordanChain1} implies $Z\slash(Z\cap Y_{k_{j_{\ast}}})=\ker
(JL-i\mu)^{k_{j_{\ast}}}\slash Y_{k_{j_{\ast}}}$. Lemma \ref{L:JordanChain2}
provides vectors $u_{1,1}^{(j_{\ast})},\ldots,u_{l_{j_{\ast}},1}^{(j_{\ast}%
)}\in Z$. It is easy to verify that $u_{p,q}^{(j)}$, $j_{\ast}\leq j\leq
j_{0}$, satisfy the induction assumption for $j_{\ast}\leq j\leq j_{0}$.
Therefore, by induction, we find all $u_{p,q}^{(j)}$ satisfying
\eqref{E:basis4} and realizing all Jordan blocks of $JL$ on $E_{i\mu}$ of size
greater than $1$. It is straightforward to verify all the properties in
Proposition \ref{P:basis}. In particular, Lemma \ref{L:JordanChain1} and
equation \eqref{E:Q_K} imply that the Krein signature defined in Proposition
\ref{P:basis} and Remark \ref{R:index1} coincides with the one in the above
Definition \ref{D:signature1} in terms of $Q_{K}$. Therefore, it is
independent of the choice of the basis (Jordan chains) realizing the Jordan
canonical form.

Finally, let
\[
E^{1}=\{v\in E_{i\mu}\mid\langle Lu_{p,q}^{(j)},v\rangle=0,\ \forall\ 1\leq
j\leq j_{0},\ 1\leq p\leq l_{j},\ 1\leq q\leq k_{j}\}.
\]
Much as in the invariance of $Z$ in the above, $JL(E^{1})\subset E^{1}$. Since
all the Jordan blocks are realized by
\[
\left\{  u_{p,q}^{(j)},\ 1\leq j\leq j_{0},\ 1\leq p\leq l_{j},\ 1\leq q\leq
k_{j}\right\}  ,
\]
we have $E^{1}\subset\ker(JL-i\mu)$. This completes the proof. \hfill$\square$

Based on Proposition \ref{P:basis}, we give the following result to be used later.

\begin{lemma}
\label{lemma-pontr-subspace-finite-d} Let $J,\ L$ be real $n\times n$
matrices. Assume $J$ is anti-symmetric and $L$ is symmetric and nonsingular.
Then there exists an invariant (under $JL$) subspace $W\ $of $\mathbf{C}^{n}$
such that $\dim W=n^{-}\left(  L\right)  $ and $\left\langle L\cdot
,\cdot\right\rangle |_{W}\leq0$.
\end{lemma}

\begin{proof}
For any purely imaginary eigenvalue $\lambda=i\mu\in i\mathbf{R}$, we start
with the special basis of $E_{i\mu}$ given by Proposition \ref{P:basis} (as
well as Remark \ref{R:JordanChain1}). For each Jordan chain $\left\{
u_{p,1}^{(j)},\cdots,u_{p,k_{j}}^{(j)}\right\}  \ $of even length, define the
subspace $Z_{i\mu,j,p}=span\left\{  u_{p,1}^{(j)},\cdots,u_{p,k_{j}/2}%
^{(j)}\right\}  $. For each Jordan chain $\left\{  u_{p,1}^{(j)}%
,\cdots,u_{p,k_{j}}^{(j)}\right\}  \ $of odd length $k_{j}\geq1$, define the
subspace
\[
Z_{i\mu,j,p}=%
\begin{cases}
span\{u_{p,1}^{(j)},\cdots,u_{p,(k_{j}-1)/2}^{(j)}\}\quad & \text{ if }\langle
Lu_{p,(k_{j}+1)/2}^{(j)},u_{p,(k_{j}+1)/2}^{(j)}\rangle>0,\\
span\{u_{p,1}^{(j)},\cdots,u_{p,(k_{j}+1)/2}^{(j)}\} & \text{ if }\langle
Lu_{p,(k_{j}+1)/2}^{(j)},u_{p,(k_{j}+1)/2}^{(j)}\rangle<0.
\end{cases}
\]
Proposition \ref{P:basis} implies that $\left\langle Lu,u\right\rangle \leq0$
for all $u\in Z_{i\mu,j,p}$ defined above. For any eigenvalue $\lambda$ of
$JL$ with $\operatorname{Re}\lambda>0$, recall $\left\langle Lu,u\right\rangle
=0$ for all $u\in E_{\lambda}$ by Lemma \ref{lemma-orthogonal-eigenspace}.
Define%
\[
Z_{i\mu}=\oplus_{j=0}^{k_{j}}\oplus_{p=1}^{l_{j}}Z_{i\mu,j,p}%
\]
and
\[
W=\oplus_{\operatorname{Re}\lambda>0}E_{\lambda}\oplus_{i\mu\in\sigma(JL)\cap
i\mathbf{R}}Z_{i\mu}.
\]
Then $\left\langle L\cdot,\cdot\right\rangle |_{W}\leq0$ since these subspaces
are pairwise orthogonal in $\left\langle L\cdot,\cdot\right\rangle $.
Moreover, $\dim W=n^{-}\left(  L\right)  $ due to the counting formula
(\ref{counting-formula}) and (\ref{counting-pure-imaginary}).
\end{proof}

\section{Invariant subspaces}

\label{S:Pontryagin}

In this section, we study subspaces of $X$ invariant under $JL$, including
both positive and negative results. As the first step to prove our main
results, a non-positive (with respect to $\langle L\cdot,\cdot\rangle$)
invariant subspace of the maximal possible dimension $n^{-}(L)$ is derived in
Subsection \ref{SS:Pontryagin}. The existence of such subspaces is not only
useful for the linear dynamics, but also a rather interesting and delicate
result as demonstrated in the discussions and examples in Subsection
\ref{SS:CounterExample}. Throughout this section, we work under the
non-degeneracy assumption that \eqref{E:non-degeneracy-def} holds for $L$
which is equivalent to $L:X\rightarrow X^{\ast}$ is an isomorphism.

\subsection{Maximal non-positive invariant subspaces (Pontryagin invariant
subspaces)}

\label{SS:Pontryagin}

\begin{theorem}
\label{T:Pontryagin} In additional to hypotheses (\textbf{H-3}), assume $L$
satisfies the non-degeneracy assumption \eqref{E:non-degeneracy-def}, then

\begin{enumerate}
\item $\dim W\le n^{-}(L)$ holds for any subspace $W \subset X$ satisfying
$\langle Lu, u\rangle\le0$ for any $u\in W$; and

\item there exists a subspace $W \subset D(JL)$ such that
\[
\dim W = n^{-}(L), \quad JL(W) \subset W, \text{ and } \langle Lu, u\rangle
\le0, \; \forall u \in W.
\]

\end{enumerate}
\end{theorem}

\begin{remark}
\label{R:Pontryagin} Though the theorem is stated for real Hilbert spaces, the
same proof shows that it also holds for complex Hilbert space $X$ and
Hermitian forms $L$ and $J$. Furthermore, the invariance of $W$ under $JL$
implies that $W \subset\cap_{k=1}^{\infty}D\big( (JL)^{k} \big)$.
\end{remark}

This theorem is basically equivalent to the classical Pontryagin invariant
subspace theorem which is usually stated for a self-adjoint operator $A\ $with
resect to some indefinite quadratic form $\langle L\cdot,\cdot\rangle\ $on $X$
with finitely many negative directions (i.e. $L$ satisfies \textbf{(H2)} with
$n^{-}\left(  L\right)  <\infty$ and $\ker L=\{0\}$). It states that there
exists a subspace $W\subset X$ $\ $such that $W$ is invariant under $A$,
$\left\langle L\cdot,\cdot\right\rangle |_{W}\leq0$ and $\dim W=$
$n^{-}\left(  L\right)  $ (i.e. maximal non-positive dimension). Such theorems
have been proved in the literature (e.g. see \cite{Gr90} \cite{chu-pilinovsky}
and the references therein). We believe that it will play a fundamental role
in further studies of Hamiltonian systems and deserves more attention than it
currently does. For the Hamiltonian PDE (\ref{eqn-hamiltonian}) considered in
this paper, one important observation is that the operator $JL$ is
anti-self-adjoint with respect to the inner product $\langle L\cdot
,\cdot\rangle$. Since both anti-self-adjoint and self-adjoint operators are
related to unitary operators by the Cayley transform, the Pontryagin invariant
subspace theorems can be equivalently stated for unitary, self-adjoint or
anti-self-adjoint cases. By Lemma \ref{lemma-L-form-preserve}, $e^{tJL}$ is
unitary in $\langle L\cdot,\cdot\rangle$. But to study the eigenvalues of $JL$
more directly, we still use Cayley transform to relate $JL$ to an unitary
operator and then apply the Pontryagin invariant subspace theorem. For the
sake of completeness, in the following we outline a proof of Theorem
\ref{T:Pontryagin} by the arguments given in \cite{chu-pilinovsky} for the
proof of Pontryagin invariant subspace theorem via unitary operators which is
based on compactness and fixed point theorems (see also \cite{ky-Fan63}
\cite{Krein-fixed-point}).

We also give another more constructive proof of Theorem \ref{T:Pontryagin}, by
using the Hamiltonian structure of (\ref{eqn-hamiltonian}) and Galerkin
approximation. It provides more information about the invariant subspace $W$.

\begin{proof}
The assumption (\ref{E:non-degeneracy-def}) is equivalent to $\ker L=\left\{
0\right\}  $. The first statement of the Theorem follows by the same proof of
Lemma \ref{L:Morse-Index}. Below we give two different proofs of the
construction of the invariant subspace $W$ in the second part of the Theorem.

Proof (\textbf{\#1.}) Here we sketch a proof of Theorem \ref{T:Pontryagin} by
using the arguments in \cite{chu-pilinovsky}. Let $X_{\pm}\subset X$ be given
by Lemma \ref{L:decom1}. Assumptions (\textbf{H2-3}) and
\eqref{E:non-degeneracy-def} ensure that
\[
X=X_{-}\oplus X_{+},\quad X^{\ast}=\tilde{X}_{-}^{\ast}\oplus\tilde{X}%
_{+}^{\ast},\quad\pm\langle Lu,u\rangle\geq\delta\Vert u\Vert^{2},\;\ \forall
u\in X_{\pm},
\]
where $\tilde{X}_{\pm}^{\ast}=P_{\pm}^{\ast}X_{\pm}^{\ast}$ and $P_{\pm}$ are
the associated projections. As in the proof of Lemma \ref{L:decomJL}, let
$i_{X_{\pm}}:X_{\pm}\rightarrow X$ be the embedding and
\[
L_{\pm}=\pm P_{\pm}^{\ast}i_{X_{\pm}}^{\ast}Li_{X_{\pm}}P_{\pm},\quad
(u,v)_{L}\triangleq\langle(L_{+}+L_{-})u,v\rangle.
\]
There exists $\delta>0$ such that $\langle L_{\pm}u,u\rangle\geq\delta\Vert
u\Vert^{2}$, for all $u\in X_{\pm}$ and the quadratic form $(\cdot,\cdot)_{L}$
induces an equivalent norm $|u|_{L}\triangleq\sqrt{(u,v)_{L}}$ on $X$. We
denote the Hilbert space $\big(X,\langle(L_{+}+L_{-})\cdot,\cdot\rangle\big)$
by $X_{L}$.

\textit{Step 1.} It is clear that $J(L_{+}+L_{-})$ is an anti-self-adjoint
operator on $X_{L}$ and $JL_{-}$ is a bounded linear operator of finite rank
on $X_{L}$ as $L_{-}X_{-}=P_{-}^{\ast}X_{-}^{\ast}\subset D(J)$. Writing
$JL=J(L_{+}+L_{-})-2JL_{-}$, we obtain that there exists $a>0$ such that
$\alpha\notin\sigma(JL)$ if $|\operatorname{Re}\alpha|\geq a$. Let
\[
T=(JL+a)(JL-a)^{-1},\text{ then }\langle LTu,Tv\rangle=\langle Lu,v\rangle
,\;\forall u,v\in X
\]
through straightforward calculation using $J=-J^{\ast}$ and that $L$ is
bounded and symmetric. In some sense, $JL$ is anti-self-adjoint with respect
to the quadratic form $\langle L\cdot,\cdot\rangle$ and thus $T$ is formally
the Cayley transformation.

\textit{Step 2.} Let $X_{L\pm}$ be the subspaces $X_{\pm}$ equipped with the
inner product $(\cdot,\cdot)_{L}$ which is equivalent to $\pm\langle
L_{X_{\pm}}\cdot,\cdot\rangle$ on $X_{\pm}$, where $L_{X_{\pm}}$ is defined in
\eqref{E:L_y1}. One may prove (see Lemma 3.6 in \cite{chu-pilinovsky}) that a
subspace $W\subset X_{L}$ satisfies $\dim W=n^{-}(L)$ and $\langle
Lu,u\rangle\leq0$ for all $u\in W$ if and only if $W$ is the graph of a
bounded linear operator $S:X_{L-}\rightarrow X_{L_{+}}$ with operator norm
$|S|\leq1$. Denote this set of operators, i.e. the unit ball of $L(X_{L-}%
,X_{L+})$, by $B_{1}(X_{L-},X_{L+})$. This proves the first statement.

\textit{Step 3.} For any $S\in B_{1}(X_{L-},X_{L+})$, since $T$ preserves the
quadratic form $\langle L\cdot,\cdot\rangle$, one may show that
$T\big(\graph(S)\big)$ is still the graph of some $S^{\prime}\in B_{1}%
(X_{L-},X_{L+})$. Hence we define a transformation $\mathcal{T}$ on
$B_{1}(X_{L-},X_{L+})$ as
\[
\graph\big(\mathcal{T}(S)\big)=T\big(\graph(S)\big).
\]

\textit{Step 4.} The space of bounded operators $L(X_{L-},X_{L+})$ equipped
with the weak topology is a locally convex topological vector space. Since
$X_{L-}$ is finite dimensional, the unit ball $B_{1}(X_{L-},X_{L+})$ is convex
and compact under the weak topology. Using the boundedness and the finite
dimensionality of $X_{L-}$, one may prove (see \cite{chu-pilinovsky} for
details) that $\mathcal{T}$ is continuous under the weak topology. According
to the Tychonoff fixed point theorem (sometimes referred as the
Schauder-Tychonoff fixed point theorem, see \cite{Ze86}), $\mathcal{T}$ has a
fixed point $S\in B_{1}(X_{L-},X_{L+})$. Let $W=\graph (S)$ and thus
$T(W)\subset W$. According to the definition of $T$, we have
\[
(JL-a)^{-1}=\frac{1}{2a}(T-I)
\]
which implies that $W$ is invariant under $(JL-a)^{-1}$. As $W$ is finite
dimensional and $(JL-a)^{-1}$ is bounded and injective, it is clear that
$W=(JL-a)^{-1}W\subset D(JL)$ and thus $JL(W)\subset W$.
\end{proof}


\noindent\textbf{Alternative proof (\#2) of Theorem \ref{T:Pontryagin} via
Galerkin approximation on separable }$X$. On the one hand, the above proof
given in \cite{chu-pilinovsky} is elegant and is based on fixed point theorems
involving compactness, which does not yield much detailed information of the
invariant subspace W. On the other hand, clearly Theorem \ref{T:Pontryagin} is
a generalization into Hilbert spaces of Lemma
\ref{lemma-pontr-subspace-finite-d} whose constructive proof provides more
explicit information of the invariant subspaces. In fact, assuming $X$ is
separable, in the rest of this section we give an alternative proof of Theorem
\ref{T:Pontryagin} based on Lemma \ref{lemma-pontr-subspace-finite-d}.

Denote
\begin{equation}
\label{E:bracket}\left[  \cdot,\cdot\right]  =\left\langle L\cdot
,\cdot\right\rangle \text{ on } X.
\end{equation}
Let $X_{\pm}$ be the same subspaces of $X$ chosen as in the above proof \#1
(as well as in the proof of Proposition \ref{P:well-posedness}). We will study
the eigenvalues of $JL$ by a Galerkin approximation. Choose an orthogonal
(with respect to $\left[  \cdot,\cdot\right]  $) basis $\left\{  \xi
_{k}\right\}  _{k=1}^{\infty}\ $of $X$ such that $\xi_{k}\in D\left(
JL\right)  $,
\[
X_{-}=span\left\{  \xi_{1},\cdots,\xi_{n^{-}\left(  L\right)  }\right\}
,\ X_{+}= \overline{span\left\{  \xi_{k}\right\}  _{k=n^{-}\left(  L\right)
+1}^{\infty}},
\]
and $\left[  \xi_{k},\xi_{j}\right]  =0$ if $k\neq j;\ $ $\left[  \xi_{j}%
,\xi_{j}\right]  =-1$ if $1\leq j\leq n^{-}\left(  L\right)  ;$ $\left[
\xi_{j},\xi_{j}\right]  =1$ if $j\geq n^{-}\left(  L\right)  +1$. For each
$n>$ $n^{-}\left(  L\right)  $, define $X^{\left(  n\right)  }=span\left\{
\xi_{1},\cdots,\xi_{n}\right\}  $ and denote $\pi^{n}$ be the orthogonal
projection with respect to the quadratic form $[\cdot, \cdot]$ from $X$ to
$X^{\left(  n\right)  }$.

Let $X,JL$, and $[\cdot,\cdot]$ (as a Hermitian symmetric form) also denote
their complexifications as in Section \ref{S:Preliminary}. Still $\{\xi
_{1},\xi_{2},\ldots\}$ form a basis of the complexified $X$. Define the
operator $F^{\left(  n\right)  }:X^{\left(  n\right)  }\rightarrow X^{\left(
n\right)  }$ by
\[
F^{(n)}v=\pi^{n}JLv.
\]
Notice that, for $j,k\leq n$,
\[
\left[  F^{(n)}\xi_{k},\xi_{j}\right]  =\left[  \pi^{n}JL\xi_{k},\xi
_{j}\right]  =\left\langle LJL\xi_{k},\xi_{j}\right\rangle =\left\langle
L\xi_{j},JL\xi_{k}\right\rangle \triangleq\left(  J^{\left(  n\right)
}\right)  _{jk},
\]
where the $n\times n$ matrix $\left(  J^{\left(  n\right)  }\right)  $ is real
and anti-symmetric. Let $v=\sum_{j=1}^{n}y_{j}\xi_{j}\in X^{(n)}$ and denote
$\vec{y}^{\left(  n\right)  }=\left(  y_{1},\cdots,y_{n}\right)  ^{T}$ and the
$n\times n$ matrix
\[
H^{\left(  n\right)  }=\left(  \left[  \xi_{k},\xi_{j}\right]  \right)
=diag\left[  \underset{1\text{ to }n^{-}\left(  L\right)
}{\underbrace{-1,\cdots,-1}},\underset{n^{-}\left(  L\right)  +1\text{ to
}n}{\underbrace{1,\cdots,1}}\right]  \text{.}%
\]
Then $F^{\left(  n\right)  }v=\sum_{k=1}^{n}a_{k}\xi_{k}$, where
\[
\vec{a}^{\left(  n\right)  }=\left(  a_{1},\cdots,a_{n}\right)  ^{T}%
=H^{\left(  n\right)  }J^{\left(  n\right)  }\vec{y}^{\left(  n\right)  }.
\]
So the eigenvalue problem $F^{\left(  n\right)  }\left(  v\right)  =\lambda v$
is equivalent to
\begin{equation}
H^{\left(  n\right)  }J^{\left(  n\right)  }\vec{y}^{\left(  n\right)
}=\lambda\vec{y}^{\left(  n\right)  }, \label{eigen-finite-d}%
\end{equation}
Let $\vec{z}^{\left(  n\right)  }=H^{\left(  n\right)  }\vec{y}^{\left(
n\right)  }$, then the eigenvalue problem (\ref{eigen-finite-d}) becomes
\[
J^{\left(  n\right)  }H^{\left(  n\right)  }\vec{z}^{\left(  n\right)
}=\lambda\vec{z}^{\left(  n\right)  }.
\]

For any $n\ge n^{-}\left(  L\right)  $, since $n^{-}\left(  H^{\left(
n\right)  }\right)  =$ $n^{-}\left(  L\right)  $, by Lemma
\ref{lemma-pontr-subspace-finite-d}, there exists a subspace $Z^{\left(
n\right)  }\subset\mathbf{C}^{n}$ of dimension $n^{-}\left(  L\right)  $, such
that $Z^{\left(  n\right)  }$ is invariant under $J^{\left(  n\right)
}H^{\left(  n\right)  }$ and $\left\langle H^{\left(  n\right)  }z,
z\right\rangle \le0$ for any $z \in Z^{\left(  n\right)  }$. Define
$Y^{\left(  n\right)  }=H^{\left(  n\right)  }Z^{\left(  n\right)  }$ and
\[
W^{\left(  n\right)  }=\left\{  \sum_{j=1}^{n}y_{j}\xi_{j}\ |\ \left(
y_{1},\cdots,y_{n}\right)  ^{T}\in Y^{\left(  n\right)  }\right\}  .
\]
Then $W^{\left(  n\right)  }$ is invariant under the linear mapping
$F^{\left(  n\right)  }$, $\dim\left(  W^{\left(  n\right)  }\right)
=n^{-}\left(  L\right)  $ and the quadratic functions
\[
\left\langle L\cdot,\cdot\right\rangle |_{W^{\left(  n\right)  }}=\left\langle
H^{\left(  n\right)  }\cdot,\cdot\right\rangle |_{Y^{\left(  n\right)  }
}=\left\langle H^{\left(  n\right)  }\cdot,\cdot\right\rangle |_{Z^{\left(
n\right)  }}\leq0.
\]

As in the proof (\#1) above, denote $P_{\pm}:X\rightarrow X_{\pm}$ to be the
projection operators with $\ker P_{\pm}=X_{\mp}$. Since the definitions of
$X_{+}$ and $W^{(n)}$ imply $W^{(n)}\cap X_{+}=\{0\}$, it holds that
$P_{-}\left(  W^{\left(  n\right)  }\right)  =X_{-}$. So we can choose a basis
$\{w_{1}^{\left(  n\right)  },\cdots,w_{n^{-}\left(  L\right)  }^{\left(
n\right)  }\}$ of $W^{\left(  n\right)  }$ such that $w_{j}^{(n)}=\xi
_{j}+w_{j+}^{(n)}$ with $w_{j+}^{(n)}\in X_{+}$. For each $j\leq n^{-}(L)$,
since
\[
0\geq\langle Lw_{j}^{\left(  n\right)  },w_{j}^{\left(  n\right)  }%
\rangle=\langle L\xi_{j}^{\left(  n\right)  },\xi_{j}^{\left(  n\right)
}\rangle+\langle Lw_{j+}^{\left(  n\right)  },w_{j+}^{\left(  n\right)
}\rangle\geq-1+\delta_{0}\Vert w_{j+}^{(n)}\Vert^{2},
\]
so $\Vert w_{j}^{\left(  n\right)  }\Vert\leq C$ for some constant $C$
independent of $j$ and $n$. Therefore, as $n\rightarrow\infty$, subject to a
subsequence, we have $w_{j}^{\left(  n\right)  }\rightharpoonup w_{j}^{\infty
}\in X$ weakly and $P_{-}(w_{j}^{\infty})=\xi_{j}$. The subspace $W^{\infty
}=span\left\{  w_{j}^{\infty}\right\}  _{j=1}^{n^{-}\left(  L\right)  }$ is of
dimension $n^{-}\left(  L\right)  $ since $P_{-}\left(  W^{\infty}\right)
=X_{-}$.

We now show that: i) $W^{\infty}$ is invariant under the operator $JL$ and ii)
$\left\langle Lu,u\right\rangle \leq0$ for any $u\in W^{\infty}$. To prove i),
first note that since $W^{\left(  n\right)  }$ is invariant under $F^{\left(
n\right)  }$, we have
\[
F^{(n)}w_{k}^{\left(  n\right)  }=\sum_{j=1}^{n^{-}\left(  L\right)  }%
a_{kj}^{\left(  n\right)  }\ w_{j}^{\left(  n\right)  },\quad a_{ij}^{\left(
n\right)  }\in\mathbf{C}.
\]
For any integer $l\in\mathbf{N}$ and a fixed $w\in X^{\left(  l\right)  }$,
when $n\geq l$,
\begin{align}
\sum_{j=1}^{n^{-}\left(  L\right)  }a_{kj}^{\left(  n\right)  }\ [w_{j}%
^{\left(  n\right)  },w] &  =[F^{\left(  n\right)  }w_{k}^{\left(  n\right)
},w]=\langle LJLw_{k}^{\left(  n\right)  },w\rangle
\label{identity-g-eigenvector-n}\\
&  =-\langle Lw_{k}^{\left(  n\right)  },JLw\rangle=-[w_{k}^{\left(  n\right)
},JLw].\nonumber
\end{align}
We claim that $\left\{  a_{ij}^{\left(  n\right)  }\right\}  $ is uniformly
bounded for $1\leq k,j\leq n^{-}\left(  L\right)  $ and $n>n^{-}\left(
L\right)  $. Suppose otherwise, there exists $1\leq k_{0},j_{0}\leq
n^{-}\left(  L\right)  $ and a subsequence $\left\{  n_{m}\right\}
\rightarrow\infty,$ such that, for all $j\leq n^{-}(L)$,
\[
\left\vert a_{k_{0}j_{0}}^{\left(  n_{m}\right)  }\right\vert =\max_{1\leq
j\leq n^{-}\left(  L\right)  }\left\{  \left\vert a_{k_{0}j}^{\left(
n_{m}\right)  }\right\vert \right\}  \rightarrow\infty\;\text{ and }\;\forall
j,\;c_{k_{0},j}=\lim_{m\rightarrow\infty}a_{k_{0}j}^{\left(  n_{m}\right)
}/a_{k_{0}j_{0}}^{\left(  n_{m}\right)  }\;\text{ exists }.
\]
Then from (\ref{identity-g-eigenvector-n}), we get
\[
\sum_{j=1}^{n^{-}\left(  L\right)  }\frac{a_{k_{0}j}^{\left(  n_{m}\right)  }%
}{a_{k_{0}j_{0}}^{\left(  n_{m}\right)  }}\ \left[  w_{j}^{\left(
n_{m}\right)  },w\right]  =-\frac{1}{a_{k_{0}j_{0}}^{\left(  n_{m}\right)  }%
}\left[  w_{k_{0}}^{\left(  n_{m}\right)  },JLw\right]
\]
and letting $m\rightarrow\infty$, we obtain
\begin{equation}
\sum_{j=1}^{n^{-}\left(  L\right)  }c_{k_{0},j}\ \left[  w_{j}^{\infty
},w\right]  =0\label{identity-0}%
\end{equation}
where in particular we also notice $\left\vert c_{k_{0},j_{0}}\right\vert =1$.
By a density argument, the identity (\ref{identity-0}) holds also for any
$w\in X$. Therefore, $\sum_{j=1}^{n^{-}\left(  L\right)  }c_{k_{0},j}%
\ w_{j}^{\infty}=0$ by the non-degeneracy of $\left[  \cdot,\cdot\right]  $.
This is in contradiction to the independency of $\left\{  w_{k}^{\infty
}\right\}  $. So $\{a_{kj}^{\left(  n\right)  }\}$ is uniformly bounded. Let
$n\rightarrow\infty$ in (\ref{identity-g-eigenvector-n}), subject to a
subsequence, we obtain
\[
\sum_{j=1}^{n^{-}\left(  L\right)  }a_{kj}^{\infty}\ \left[  w_{j}^{\infty
},w\right]  =-\left[  w_{j}^{\infty},JLw\right]  =\left[  JLw_{k}^{\infty
},w\right]  ,\;\text{ where }\,a_{kj}^{\infty}=\lim_{n\rightarrow\infty}%
a_{kj}^{\left(  n\right)  }.
\]
By a density argument again the above equality is also true for any $w\in X$,
which implies
\[
JL\left(  w_{k}^{\infty}\right)  =\sum_{j=1}^{n^{-}\left(  L\right)  }%
a_{kj}^{\infty}\ w_{j}^{\infty}.
\]
So $W^{\infty}$ is invariant under $JL$.

Now we prove the above claim ii), that is, $\left\langle L\cdot,\cdot
\right\rangle |_{W^{\infty}}\leq0$. For any
\[
u=\sum_{j=1}^{n^{-}\left(  L\right)  }c_{j}w_{j}^{\infty}\in W^{\infty}.
\]
denote
\[
u^{\left(  n\right)  }=\sum_{j=1}^{n^{-}\left(  L\right)  }c_{j}w_{j}^{\left(
n\right)  }\in W^{\left(  n\right)  }.
\]
Clearly, $u^{\left(  n\right)  }\rightharpoonup u$ weakly in $X$ and
$\left\langle Lu^{\left(  n\right)  },u^{\left(  n\right)  }\right\rangle
\leq0$, which converges subject to a subsequence. Since
\[
\lim_{n\rightarrow\infty}\left\langle LP_{-}u^{\left(  n\right)  }%
,P_{-}u^{\left(  n\right)  }\right\rangle =\left\langle LP_{-}u,P_{-}%
u\right\rangle ,
\]
which is due to $P_{-}w_{j}^{\infty}=\xi_{j}$ and therefore $P_{-}u^{\left(
n\right)  }\rightarrow P_{-}u$ strongly in $X$, and
\[
\lim_{n\rightarrow\infty}\left\langle LP_{+}u^{\left(  n\right)  }%
,P_{+}u^{\left(  n\right)  }\right\rangle \geq\left\langle LP_{+}%
u,P_{+}u\right\rangle .
\]
as $\langle Lx,x\rangle^{\frac{1}{2}}$ is a norm on $X_{+}$. Therefore,
\begin{align}
0 &  \geq\lim_{n\rightarrow\infty}\left\langle Lu^{\left(  n\right)
},u^{\left(  n\right)  }\right\rangle =\left\langle LP_{-}u,P_{-}%
u\right\rangle +\lim_{n\rightarrow\infty}\left\langle LP_{+}u^{\left(
n\right)  },P_{+}u^{\left(  n\right)  }\right\rangle
\label{inequality-nonpositive-L-limit}\\
&  \geq\left\langle LP_{-}u,P_{-}u\right\rangle +\left\langle LP_{+}%
u,P_{+}u\right\rangle =\left\langle Lu,u\right\rangle .\nonumber
\end{align}
This complete the proof of claim ii) and thus the proof of Theorem
\ref{T:Pontryagin} under the separable assumption on $X$. \hfill$\square$

\subsection{Further discussions on invariant subspaces and invariant
decompositions}

\label{SS:CounterExample}

\noindent\textbf{Continuous dependence of invariant subspaces on }%
$JL$\textbf{.} In perturbation problems, the operator $JL$ may depend on a
perturbation parameter $\epsilon$. One would naturally wish that a family
$W_{\epsilon}$ of non-positive invariant subspaces of dimension $n^{-}(L)$ may
be found depending on $\epsilon$ at least continuously. However, this turns
out to be impossible in general, even if $L$ is assumed to be non-degenerate.
See an example in Section \ref{SS:E_0}. \newline

\noindent\textbf{Invariant splitting, I.} In the presence of $W$ invariant
under $JL$ with $\dim W=n^{-}(L)$, it is natural to ask whether it is possible
to make it into an invariant (under $JL$) decomposition of $X$, i.e. whether
there exist such $W$ and a codim-$n^{-}(L)$ invariant subspace $W_{1}\subset
X$ such that $X=W\oplus W_{1}$. This is usually not possible as in the
following example
\[
J=%
\begin{pmatrix}
0 & -1 & 1 & 0\\
1 & 0 & 0 & 1\\
-1 & 0 & 0 & 0\\
0 & -1 & 0 & 0
\end{pmatrix}
,\;L=%
\begin{pmatrix}
0 & 0 & 0 & 1\\
0 & 0 & -1 & 0\\
0 & -1 & 0 & 0\\
1 & 0 & 0 & 0
\end{pmatrix}
,\;JL=%
\begin{pmatrix}
0 & -1 & 1 & 0\\
1 & 0 & 0 & 1\\
0 & 0 & 0 & -1\\
0 & 0 & 1 & 0
\end{pmatrix}
.
\]
Here $n^{-}(L)=2$ and the only eigenvalues are $\sigma(JL)=\{\pm i\}$. The
only possible non-positive 2-dim invariant subspace, where the eigenvalues of
the restriction of $JL$ are contained in $\sigma(JL)$, has to be the geometric
kernel of $\pm i$ and thus $W=\{x_{3}=x_{4}=0\}$. There does not exist any
2-dim invariant subspace $W_{1}$ such that $\mathbf{R}^{4}=W\oplus W_{1}$
since the restriction of $JL$ on $W_{1}$ has to have eigenvectors of $\pm i$
as well. \newline

\noindent\textbf{Invariant Splitting, II.} In light of Lemma
\ref{L:InvariantSubS},
\[
W^{\perp_{L}}=\{u\in X\mid\langle Lu,v\rangle=0,\;\forall v\in W\}
\]
is invariant under $e^{tJL}$. While one may wish $X=W\oplus W^{\perp_{L}}$,
the only obstacle is that $W^{\perp_{L}}$ may intersect $W$ nontrivially as
$L_{W}$, as defined in \eqref{E:L_y1}, may be degenerate as in the above
example. A more natural question is whether it is possible to enlarge $W$ to
some closed $\tilde{W}\supset W$ such that
\[
\dim\tilde{W}<\infty,\quad JL(\tilde{W})\subset\tilde{W},\;\text{ and
}L_{\tilde{W}}\text{ an isomorphism}.
\]
If so, Lemmas \ref{L:non-degeneracy} and \ref{L:InvariantSubS} would imply
\[
\tilde{W}^{\perp_{L}}=\{u\in X\mid\langle Lu,v\rangle=0,\;\forall v\in
\tilde{W}\}
\]
is invariant under $e^{tJL}$ and $X=\tilde{W}\oplus\tilde{W}^{\perp_{L}}$.
Moreover, $L_{\tilde{W}^{\perp_{L}}}$ is positive definite due to Theorem
\ref{T:Pontryagin} and thus $e^{tJL}$ is stable on $\tilde{W}^{\perp_{L}}$.
Consequently all the index counting and stability analysis related to
$e^{tJL}$ can be reduced to the finite dimensional $\tilde{W}$, which has been
analyzed in Section \ref{S:finiteD}. For example, we would have a counting
theorem like Theorem \ref{theorem-counting}, in particular with $k_{0}^{\leq
0}$ and $k_{i}^{\leq0}$ replaced by $k_{0}^{-}$ and $k_{i}^{-}$, respectively.

Unfortunately the above splitting is not always possible either, as can be
seen from a counterexample in Subsection \ref{SS:non-deg}. In Proposition
\ref{P:direct decomposition}, we give conditions to get such a decomposition.

\section{Structural decomposition}

\label{S:decomposition}

In this section, we prove Theorem \ref{T:decomposition} and Corollary
\ref{C:decomposition} on the decomposition of $X$. Our first step to decompose
$X$ is the following proposition based on the invariant subspace Theorem
\ref{T:Pontryagin}.

\begin{proposition}
\label{P:decomposition1} In addition to (\textbf{H1-H3}), assume $\ker
L=\{0\}$. There exist closed subspaces $Y_{j}$, $j=1,2,3,4$, such that
$X=\oplus_{j=1}^{4}Y_{j}$ and
\begin{equation}
\dim Y_{1}=\dim Y_{4}=n^{-}(L)-\dim Y_{2}<\infty,\quad Y_{1,2,4}\subset
\cap_{k=1}^{\infty}D\big((JL)^{k}\big),\label{E:dim1}%
\end{equation}
and accordingly the linear operator $JL$ and the quadratic form $\langle
L\cdot,\cdot\rangle$ take the block forms
\[
JL\longleftrightarrow%
\begin{pmatrix}
\tilde{A}_{1} & \tilde{A}_{12} & \tilde{A}_{13} & \tilde{A}_{14}\\
0 & \tilde{A}_{2} & 0 & \tilde{A}_{24}\\
0 & 0 & \tilde{A}_{3} & \tilde{A}_{34}\\
0 & 0 & 0 & \tilde{A}_{4}%
\end{pmatrix}
,\quad L\longleftrightarrow%
\begin{pmatrix}
0 & 0 & 0 & \tilde{B}\\
0 & L_{Y_{2}} & 0 & 0\\
0 & 0 & L_{Y_{3}} & 0\\
\tilde{B}^{\ast} & 0 & 0 & 0
\end{pmatrix}
.
\]
Here $\tilde{B}:Y_{4}\rightarrow Y_{1}^{\ast}$ is an isomorphism and the
quadratic forms $L_{Y_{2}}\leq-\delta_{0}$ and $L_{Y_{3}}\geq\delta_{0}$ for
some $\delta_{0}>0$. Moreover, $\tilde{A}_{3}:D(\tilde{A}_{3})=Y_{3}\cap
D(JL)\rightarrow Y_{3}$ is closed, while all other blocks are bounded
operators. The operators $\tilde{A}_{2,3}$ are anti-self-adjoint with respect
to the equivalent inner product $\mp\langle L_{Y_{2,3}}\cdot,\cdot\rangle$.
\end{proposition}

Before we give the proof the proposition, we would like to make two remarks.
Firstly we observe that $(JL)^{k}$ takes the same blockwise form as the above
one of $JL$. Secondly, the bounded operator $\tilde{A}_{13}$ should be
understood as the closure of $P_{1}JL|_{Y_{3}}$, which may not be closed or
everywhere defined itself. Here $P_{j}:X\rightarrow Y_{j}$ is the projection
to $Y_{j}$, $j=1,2,3,4$, according to the decomposition.

\begin{proof}
Theorem \ref{T:Pontryagin} states that there exists $W\subset D\big((JL)^{k}%
\big)$ such that $\dim W=n^{-}(L)$, $JL(W)\subset W$, and $\langle
Lu,u\rangle\leq0$ for all $u\in W$. Let $Y_{1}=W\cap W^{\perp_{L}}$
($\perp_{L}$ defined as in Lemma \ref{L:non-degeneracy}), $\tilde{Y}%
_{2}\subset W$, and $\tilde{Y}_{3}\subset W^{\perp_{L}}$ be closed subspaces
such that $W=Y_{1}\oplus\tilde{Y}_{2}$ and $W^{\perp_{L}}=Y_{1}\oplus\tilde
{Y}_{3}$. Recall the notation $L_{Y}$ as defined in \eqref{E:L_y1} for any
closed subspace $Y$.

\textit{Claim.} $L_{Y_{1}} =0$ and there exists $\delta_{0}>0$ such that
$L_{\tilde Y_{2}} \le-\delta_{0}$ and $L_{\tilde Y_{3}}\ge\delta_{0}$.

In fact, since the quadratic form $\langle Lu,u\rangle\leq0$, for all $u\in
W$, the variational principle yields that $u\in W$ satisfies $\langle
Lu,u\rangle=0$ if and only if $\langle Lu,v\rangle=0$ for all $v\in W$, or
equivalently $u\in W\cap W^{\perp_{L}}=Y_{1}$. Therefore, $\langle
Lu,u\rangle<0$ for any $u\in\tilde{Y}_{2}\backslash\{0\}$ which along with
$\dim\tilde{Y}_{2}<\infty$ implies $L_{\tilde{Y}_{2}}\leq-\delta_{0}$ for some
$\delta_{0}>0$.

If there exists $u\in W^{\perp_{L}}\backslash W$ satisfying $\langle
Lu,u\rangle\leq0$, the definition of $W^{\perp_{L}}$ would imply $\langle
Lu,v\rangle\leq0$ for all $v\in\tilde{W}=W\oplus\mathbf{R}u$ and $\dim
\tilde{W}=n^{-}(L)+1$. This would contradict Theorem \ref{T:Pontryagin} and
thus we obtain $\langle Lu,u\rangle>0$ for all $u\in\tilde{Y}_{3}%
\backslash\{0\}$. Consequently, Lemma \ref{L:non-degeneracy} implies that
$L_{\tilde{Y}_{3}}\geq\delta_{0}$ for some $\delta_{0}>0$ and the claim is proved.

Since $L$ is assumed to non-degenerate, it is easy to see codim-$(W+W^{\perp
_{L}})=\dim Y_{1}<\infty$. Let $\tilde{Y}_{4}$ be a subspace such that
$X=(W+W^{\perp_{L}})\oplus\tilde{Y}_{4}=Y_{1}\oplus\tilde{Y}_{2}\oplus
\tilde{Y}_{3}\oplus\tilde{Y}_{4}$ and $\tilde{Y}_{4}\subset\cap_{k=1}^{\infty
}D\big((JL)^{k}\big)$, which is possible as $\cap_{k=1}^{\infty}%
D\big((JL)^{k}\big)$ is dense and $\dim\tilde{Y}_{4}=\dim Y_{1}<\infty$. With
respect to this decomposition, $L$ takes the form
\[
L\longleftrightarrow%
\begin{pmatrix}
0 & 0 & 0 & \tilde{B}_{41}^{\ast}\\
0 & L_{\tilde{Y}_{2}} & 0 & \tilde{B}_{42}^{\ast}\\
0 & 0 & L_{\tilde{Y}_{3}} & \tilde{B}_{43}^{\ast}\\
\tilde{B}_{41} & \tilde{B}_{42} & \tilde{B}_{43} & L_{\tilde{Y}_{4}}%
\end{pmatrix}
.
\]
The non-degeneracy of $L$ implies that $\tilde{B}_{41}=i_{\tilde{Y}_{4}}%
^{\ast}Li_{Y_{1}}:Y_{1}\rightarrow\tilde{Y}_{4}^{\ast}$ is an isomorphism.
Let
\[
S_{4}=-\frac{1}{2}\tilde{B}_{41}^{-1}L_{\tilde{Y}_{4}}:\tilde{Y}%
_{4}\rightarrow Y_{1},\quad S_{j}=-\tilde{B}_{41}^{-1}\tilde{B}_{4j}:\tilde
{Y}_{j}\rightarrow Y_{1},\;j=2,3.
\]
For any $u,v\in\tilde{Y}_{4}$, we have
\[%
\begin{split}
\langle L(u+S_{4}u),v+S_{4}v\rangle= &  \langle L_{\tilde{Y}_{4}}%
u,v\rangle+\langle\big(LS_{4}+(LS_{4})^{\ast}\big)u,v\rangle\\
= &  \langle L_{\tilde{Y}_{4}}u,v\rangle+\langle\big(\tilde{B}_{41}%
S_{4}+(\tilde{B}_{41}S_{4})^{\ast}\big)u,v\rangle=0.
\end{split}
\]
Similarly, for any $u\in\tilde{Y}_{j}$, $j=2,3$, and $v\in\tilde{Y}_{4}$,
\[
\langle L(u+S_{j}u),v+S_{4}v\rangle=\langle\tilde{B}_{4j}u,v\rangle
+\langle\tilde{B}_{41}S_{j}u,v\rangle=0.
\]
Let $Y_{j}=(I+S_{j})\tilde{Y}_{j}$. Clearly, it still holds $X=\oplus
_{j=1}^{4}Y_{4}$. Moreover, $Y_{1,2,4}\subset\cap_{k=1}^{\infty}%
D\big((JL)^{k}\big)$, the dimension relationship in (\ref{E:dim1}) holds, and
in this decomposition $L$ takes the desired form as in the statement of the
proposition. Due to $W=Y_{1}\oplus Y_{2}$ and $W^{\perp_{L}}=Y_{1}\oplus
Y_{3}$, the same claim as above implies the uniform positivity of $-L_{Y_{2}}$
and $L_{Y_{3}}$. The non-degeneracy of $\tilde{B}$ follows from the
non-degeneracy assumption of $L$.

The invariance of $W$, and thus the invariance of $W^{\perp_{L}}$ due to Lemma
\ref{L:InvariantSubS}, yields the desired form of $JL$. The properties that
$\tilde{A}_{2,3}$ are anti-self-adjoint with respect to $\langle L_{Y_{2,3}%
}\cdot,\cdot\rangle$ and the boundedness of other blocks can be proved by
applying Lemma \ref{L:decomJ} repeatedly to the splitting based on
$X=(Y_{1}\oplus Y_{4})\oplus(Y_{2}\oplus Y_{3})$.
\end{proof}

The following general functional analysis lemma on invariant subspaces will be
used several times in the rest of the paper.

\begin{lemma}
\label{L:InvariantSubspace} Let $Z$ be a Banach space and $Z_{1,2} \subset Z$
be closed subspaces such that $Z= Z_{1} \oplus Z_{2}$. Suppose $A$ is a linear
operator on $X$ which, in the above splitting, takes the form $%
\begin{pmatrix}
A_{1} & A_{12}\\
0 & A_{2}%
\end{pmatrix}
$, such that

\begin{itemize}
\item $A_{1,2} : Z_{1,2} \supset D(A_{1,2}) \to Z_{1,2}$ are densely defined
closed operators, one of which and $A_{12} : Z_{2} \to Z_{1}$ is bounded and

\item $\sigma(A_{1}) \cap\sigma(A_{2}) = \emptyset$,
\end{itemize}

then there exists a bounded operator $S: Z_{2} \to Z_{1}$ such that

\begin{enumerate}
\item $SZ_{2} \subset D(A_{1})$ and

\item $A\big( \tilde Z_{2} \cap D(A) \big) \subset\tilde Z_{2}$, where $\tilde
Z_{2} = (I+S)Z_{2} =\{z_{2} + S(z_{2}) \mid z_{2} \in Z_{2}\}$.
\end{enumerate}
\end{lemma}

\begin{remark}
Clearly, the above properties also imply $D(A)\cap\tilde{Z}_{2}=(I+S)D(A_{2})$
is dense in the closed subspace $\tilde{Z}_{2}$ and $A|_{\tilde{Z}_{2}%
}:D(A)\cap\tilde{Z}_{2}\rightarrow\tilde{Z}_{2}$ is a closed operator. By
using the splitting $Z=Z_{1}\oplus\tilde{Z}_{2}$, $A$ is block diagonalized
into $diag(A_{1},A_{2})$. Moreover, if $A_{2}$ is bounded, then the closed
graph theorem implies that $A|_{\tilde{Z}_{2}}$ is also bounded.
\end{remark}

The proof of this lemma may be found in some standard functional analysis
textbook. For the sake of completeness we also give a proof here.

\begin{proof}
Let us first consider the case when $A_{2}$ is bounded. Since $\sigma(A_{2})$
is compact and $\sigma(A_{2})\cap\sigma(A_{1})=\emptyset$, there exists an
open subset $\Omega\subset\mathbf{C}$ with compact closure and smooth boundary
$\Gamma=\partial\Omega$ such that $\sigma(A_{2})\subset\Omega\subset
\overline{\Omega}\subset\mathbf{C}\backslash\sigma(A_{1})$. We have
\[
\frac{1}{2\pi i}\oint_{\Gamma}(\lambda-A_{1})^{-1}d\lambda=0,\quad\frac
{1}{2\pi i}\oint_{\Gamma}(\lambda-A_{2})^{-1}d\lambda=I.
\]
Define
\[
S=\frac{1}{2\pi i}\oint_{\Gamma}T(\lambda)d\lambda,\text{ where }%
T(\lambda)=(A_{1}-\lambda)^{-1}A_{12}(A_{2}-\lambda)^{-1}.
\]
Since $(A_{j}-\lambda)^{-1}$, $j=1,2$, is analytic from $\mathbf{C}%
\backslash\sigma(A_{j})$ to $L(Z_{j})$, it is clear that $S:Z_{2}\rightarrow
Z_{1}$ is bounded. In particular, observing $T(\lambda)z\in D(A_{1})$ for any
$z\in Z_{2}$, one may verify
\begin{equation}
T(\lambda)A_{2}z-A_{1}T(\lambda)z=(A_{1}-\lambda)^{-1}A_{12}z-A_{12}%
(A_{2}-\lambda)^{-1}z\triangleq\tilde{T}(\lambda)z,\label{E:SDecomp1}%
\end{equation}
where $\tilde{T}(\lambda)\in L(Z_{2},Z_{1})$ is also analytic in $\lambda$.

We first show that $Sz\in D(A_{1})$ for any $z\in Z_{2}$. In fact, let $S_{n}%
$, $n\in\mathbf{N}$, be the values of a sequence of Riemann sums of the
integral defining $S$, such that $S_{n}\rightarrow S$. Clearly, the discrete
Riemann sums satisfy $S_{n}z\in D(A_{1})$ and along with \eqref{E:SDecomp1} we
obtain that
\[
S_{n}A_{2}z-A_{1}S_{n}z=\tilde{T}_{n}z\rightarrow\frac{1}{2\pi i}\oint%
_{\Gamma}\tilde{T}(\lambda)zd\lambda=A_{12}z
\]
where $\tilde{T}_{n}z$ is the corresponding Riemann sum of the integral on the
right side. Therefore, we obtain from the closedness of $A_{1}$ that $Sz\in
D(A_{1})$ and
\begin{equation}
SA_{2}-A_{1}S=A_{12}.\label{E:SDecomp2}%
\end{equation}
From this equation it is straightforward to verify, for any $z\in Z_{2}$,
\begin{equation}
A(z+Sz)=A_{2}z+SA_{2}z.\label{E:SDecomp3}%
\end{equation}

In the other case where $A_{1}$ is bounded, the proof is similar. In fact, let
$\Omega\subset\mathbf{C}$ be an open subset with compact closure and smooth
boundary $\Gamma=\partial\Omega$ such that $\sigma(A_{1})\subset\Omega
\subset\overline{\Omega}\subset\mathbf{C}\backslash\sigma(A_{2})$. Define
\[
S=-\frac{1}{2\pi i}\oint_{\Gamma}T(\lambda)d\lambda\in L(Z_{2},Z_{1}).
\]
It holds trivially $Sz\in D(A_{1})=Z_{1}$ for any $z\in Z_{2}$. The same
calculation, based on \eqref{E:SDecomp1} but without the need of going through
the Riemann sum as $D(A_{1})=Z_{1}$, leads us to \eqref{E:SDecomp2} which
implies \eqref{E:SDecomp3} for any $z\in D(A_{2})$. The proof is complete.
\end{proof}

In the next step, we remove the non-degeneracy assumption on $L$ and split the
phase space $X$ into the direct sum of the hyperbolic (if any) and central
subspaces of $JL$. In particular, the non-degeneracy of the quadratic form
$\langle L \cdot, \cdot\rangle$ on the hyperbolic subspace $X_{u} \oplus
X_{s}$ is of particular importance in the decomposition of $JL$.

\begin{proposition}
\label{P:decomposition2} Assume (\textbf{H1-3}). There exist closed subspaces
$X_{u, s, c} \subset X$ such that

\begin{enumerate}
\item $X=X_{c}\oplus X_{u} \oplus X_{s}$, $X_{u, s} \subset D(JL)$, $\dim
X_{u} = \dim X_{s} \le n^{-}(L)$, and $\ker L \subset X_{c}$;

\item with respect to this decomposition, $JL$ and $L$ take the forms
\[
JL\longleftrightarrow%
\begin{pmatrix}
A_{c} & 0 & 0\\
0 & A_{u} & 0\\
0 & 0 & A_{s}%
\end{pmatrix}
,\quad L\longleftrightarrow%
\begin{pmatrix}
L_{X_{c}} & 0 & 0\\
0 & 0 & B\\
0 & B^{\ast} & 0
\end{pmatrix}
;
\]

\item $B:X_{u}\rightarrow X_{s}^{\ast}$ is an isomorphism, $A_{c}$ is densely
defined, closed, and the spectral sets satisfy $\sigma(A_{c})\subset
i\mathbf{R}$ and $\pm\operatorname{Re}\lambda>0$ for any $\lambda\in
\sigma(A_{u,s})$.
\end{enumerate}
\end{proposition}

\begin{proof}
Let $X_{0}=\ker L$ and $Y=X_{-}\oplus X_{+}$ where $X_{\pm}$ are given in
Lemma \ref{L:decom1}. Let $P:X\rightarrow Y$ be the projection associated to
$X=Y\oplus X_{0}$ and $J_{Y}=PJP^{\ast}$. Lemma \ref{L:decomJ} implies that
$(Y,L_{Y},J_{Y})$ satisfy assumptions (\textbf{H1-3}), with $L_{Y}$ being an
isomorphism. Applying Proposition \ref{P:decomposition1}, we obtain closed
subspaces $Y_{j}$, $j=1,2,3,4$, such that $X=X_{0}\oplus(\oplus_{j=1}^{4}%
Y_{j})$ and $JL$ and $L$ take the forms
\[
JL\longleftrightarrow%
\begin{pmatrix}
0 & \tilde{A}_{01} & \tilde{A}_{02} & \tilde{A}_{03} & \tilde{A}_{04}\\
0 & \tilde{A}_{1} & \tilde{A}_{12} & \tilde{A}_{13} & \tilde{A}_{14}\\
0 & 0 & \tilde{A}_{2} & 0 & \tilde{A}_{24}\\
0 & 0 & 0 & \tilde{A}_{3} & \tilde{A}_{34}\\
0 & 0 & 0 & 0 & \tilde{A}_{4}%
\end{pmatrix}
,\quad L\longleftrightarrow%
\begin{pmatrix}
0 & 0 & 0 & 0 & 0\\
0 & 0 & 0 & 0 & \tilde{B}\\
0 & 0 & L_{Y_{2}} & 0 & 0\\
0 & 0 & 0 & L_{Y_{3}} & 0\\
0 & \tilde{B}^{\ast} & 0 & 0 & 0
\end{pmatrix}
,
\]
where $\tilde{B}$ is an isomorphism, $L_{Y_{2}}\leq-\delta_{0}$, $L_{Y_{3}%
}\geq\delta_{0}$, for some $\delta_{0}>0$, and $\tilde{A}_{2,3}$ are
anti-self-adjoint with respect to the equivalent inner products $\mp\langle
L_{Y_{2,3}}\cdot,\cdot\rangle$ on $Y_{2,3}$. The upper triangular structure of
$JL$ implies $\sigma(JL)=\{0\}\bigcup\cup_{j=1}^{4}\sigma(\tilde{A}_{j})$.
Moreover, we have $\sigma(\tilde{A}_{2,3})\subset i\mathbf{R}$ due to the
anti-self-adjointness of $\tilde{A}_{2,3}$.

For $j=1,4$, as $\dim Y_{j}<\infty$, let $Y_{j}=Y_{jc}\oplus Y_{jh}$, where
$Y_{jc}$ and $Y_{jh}$, are the eigenspaces of $\tilde{A}_{j}$ corresponding to
all eigenvalues with zero and nonzero real parts, respectively. For any
$x_{1,4}\in Y_{1,4}$, the above form of $JL$ and $L$ imply
\begin{align*}
\langle\tilde{B}x_{4},\tilde{A}_{1}x_{1}\rangle+\langle\tilde{B}\tilde{A}%
_{4}x_{4},x_{1}\rangle= &  \langle L\tilde{A}_{1}x_{1},x_{4}\rangle+\langle
L\tilde{A}_{4}x_{4},x_{1}\rangle\\
= &  \langle LJLx_{1},x_{4}\rangle+\langle LJLx_{4},x_{1}\rangle=0.
\end{align*}
Much as in the proof of Lemma \ref{lemma-orthogonal-eigenspace}, due to the
difference in eigenvalues, we obtain
\begin{equation}
\langle\tilde{B}x_{4h},x_{1c}\rangle=0=\langle\tilde{B}x_{4c},x_{1h}%
\rangle,\;\forall x_{jc}\in Y_{jc},\;x_{jh}\in Y_{jh},\,j=1,4.\label{E:temp2}%
\end{equation}
Therefore, the non-degeneracy of $\tilde{B}$ implies that
\begin{equation}
\langle\tilde{B}x_{4h},x_{1h}\rangle\text{ and }\langle\tilde{B}x_{4c}%
,x_{1c}\rangle\text{ are non-degenerate quadratic forms}\label{E:tildeB}%
\end{equation}
on $Y_{1h}\times Y_{4h}$ and $Y_{1c}\times Y_{4c}$.

Applying Lemma \ref{L:InvariantSubspace} to $X_{0}\oplus Y_{1h}$ and
$JL|_{X_{0}\oplus Y_{1h}}$, we obtain a linear operator $S_{1}:Y_{1h}%
\rightarrow X_{0}$ such that $X_{1h}\triangleq(I+S_{1})Y_{1h}\subset D(JL)$
satisfies $JL(X_{1h})=X_{1h}$ and $\sigma(JL|_{X_{1h}})=\sigma(\tilde{A}%
_{1}|_{Y_{1h}})$. Clearly, we still have the decomposition
\[
X=X_{0}\oplus Y_{1c}\oplus X_{1h}\oplus Y_{2}\oplus Y_{3}\oplus Y_{4c}\oplus
Y_{4h}.
\]
Applying again Lemma \ref{L:InvariantSubspace} to
\[
Z=X_{0}\oplus Y_{1c}\oplus Y_{2}\oplus Y_{3}\oplus Y_{4h}%
\]
and the projection (with the kernel $X_{1h}\oplus Y_{4c}$) of $JL|_{Z}$ to
$Z$, we obtain a bounded linear operator
\[
S_{4}:Y_{4h}\rightarrow X_{0}\oplus Y_{1c}\oplus Y_{2}\oplus Y_{3}%
\]
such that $X_{4h}\triangleq(I+S_{4})Y_{4h}\subset D(JL)$ satisfies
$JL(X_{4h})\subset X_{1h}\oplus X_{4h}$. Let $X_{h}=X_{1h}\oplus X_{4h}$, we
have
\[
X_{h}\subset D(JL),\;JL(X_{h})=X_{h},\;\sigma(JL|_{X_{h}})=\sigma
(JL)\backslash i\mathbf{R}=\sigma(\tilde{A}_{1}|_{Y_{1h}})\cup\sigma(\tilde
{A}_{4}|_{Y_{4h}}).
\]

According to \eqref{E:temp2} and the form of $L$, it holds
\[
\langle L(I+S_{1})x_{1h},(I+S_{4})x_{4h}\rangle=\langle Lx_{1h},x_{4h}%
\rangle=\langle\tilde{B}x_{4h},x_{1h}\rangle.
\]
Therefore, we obtain the non-degeneracy of $\langle L\cdot,\cdot\rangle$ on
$X_{h}$ from \eqref{E:tildeB} and the construction of $X_{h}$. Let
$X_{h}=X_{u}\oplus X_{s}$, where $X_{u,s}$ are the eigenspaces of all
eigenvalues $\lambda\in\sigma(JL|_{X_{h}})$ with $\pm$Re$\lambda>0$. Lemma
\ref{lemma-orthogonal-eigenspace} implies $\langle Lu,v\rangle=0$ on both
$X_{u}$ and $X_{s}$ and thus $\langle L\cdot,\cdot\rangle$ is a non-degenerate
quadratic form on $X_{u}\times X_{s}$ due to the non-degeneracy of $\langle
L\cdot,\cdot\rangle$ on $X_{h}$. This also yields
\[
\dim X_{u}=\dim X_{s}=\frac{1}{2}\dim X_{h}\leq\dim Y_{1}\leq n^{-}(L).
\]

Let
\[
X_{c}=X_{h}^{\perp_{L}}=\{u\in X\mid\langle Lu,v\rangle=0,\,\forall v\in
X_{h}\}.
\]
By Lemmas \ref{L:non-degeneracy} and \ref{L:InvariantSubS}, $X=X_{h}\oplus
X_{c}$ and $X_{c}$ is invariant under $e^{tJL}$. Therefore, $JL$ is densely
defined on $X_{c}$ and $A_{c}\big(D(A_{c})\cap X_{c}\big)\subset X_{c}$ where
$A_{c}=JL|_{X_{c}}$. Moreover, from the non-degeneracy of $\langle
L\cdot,\cdot\rangle$ on $X_{h}$, it is straightforward to show that $X_{c}$
can be written as a graph of a bounded linear operator from $X_{0}\oplus
Y_{1c}\oplus Y_{2}\oplus Y_{3}\oplus Y_{4c}$ to $X_{h}$. Therefore, due to the
upper triangular structure of $JL$, the spectrum $\sigma(JL|_{X_{c}})$ is
given by the union of the spectrum of those diagonal blocks of $JL$
complementary to $Y_{1h}$ and $Y_{4h}$ and thus $\sigma(JL|_{X_{c}})\subset
i\mathbf{R}$.
\end{proof}

As a by-product, we prove the symmetry of eigenvalues of $\sigma(JL)$.

\begin{corollary}
\label{C:symmetry} Suppose $\lambda\in\sigma(JL)$.

(i) If $\lambda\in\sigma(JL)\backslash i\mathbf{R}$, then $\lambda$ is an
isolated eigenvalue of finite algebraic multiplicity. Its eigenspace consists
of generalized eigenvectors only. Moreover, let $m_{\lambda}$ to be the
algebraic multiplicity of $\lambda$, then
\begin{equation}
n^{-}\left(  L|_{E_{\lambda}\oplus E_{-\bar{\lambda}}}\right)  =\dim\left(
E_{\lambda}\right)  =m_{\lambda}.
\label{formula-negative-dimension-hyperbolic}%
\end{equation}

(ii) If $\lambda$ is an eigenvalues of $JL$, then $\pm\lambda,\pm\bar{\lambda
}$ are also eigenvalues of $JL$. Moreover, for any integer $k>0$, $\dim
\ker(JL-a)^{k}$ are the same for $a=\pm\lambda,\pm\bar{\lambda}$.
\end{corollary}

For an eigenvalue $\lambda\in i\mathbf{R}$, it may happen $\dim\ker(JL
-\lambda) =\infty$.

\begin{proof}
According to Lemma \ref{L:symmetry}, we only need to prove $\lambda\in
\sigma(JL)\backslash i\mathbf{R}$ implies that $\lambda$ is an isolated
eigenvalue of finite multiplicity and $\dim\ker(JL-\lambda)^{k}=\dim
\ker(JL+\bar{\lambda})^{k}$.

In fact, if $\lambda\in\sigma(JL)\backslash i\mathbf{R}$, then Proposition
\ref{P:decomposition2} implies that $\lambda\in\sigma(A_{u})\cup\sigma(A_{s}%
)$. As $A_{u,s}$ are finite dimensional matrices, $\lambda$ must be an
isolated eigenvalue of $JL$ with finite algebraic multiplicity. Moreover, from
the blockwise forms of $L$ and $JL$ and $J^{\ast}=-J$, it is easy to compute
\[
J\longleftrightarrow%
\begin{pmatrix}
J_{X_{c}} & 0 & 0\\
0 & 0 & A_{u}(B^{\ast})^{-1}\\
0 & A_{s}B^{-1} & 0
\end{pmatrix}
.
\]
Again since $J^{\ast}=-J$, we have $A_{s}=-B^{-1}A_{u}^{\ast}B$. As $A_{u,s}$
are finite dimensional matrices and eigenvalues of $JL$ with positive (or
negative) real parts coincide with eigenvalues of $A_{u}$ (or $A_{s}$), the
statement in the corollary follows from this similarity immediately.

Since by Proposition \ref{P:decomposition2} $L|_{X_{u}\oplus X_{s}}$ is
non-degenerate, formula (\ref{formula-negative-dimension-hyperbolic}) follows
from Lemma \ref{lemma-negative-unstable-space} in the finite dimensional case.
\end{proof}

\noindent\textbf{Proof of Theorem \ref{T:decomposition}.} Let $X_{5,6}%
=X_{u,s}$ and $J_{X_{c}}=P_{c}JP_{c}^{\ast}$, where $X_{u,s,c}$ are obtained
in Proposition \ref{P:decomposition2} and $P_{c}:X\rightarrow X_{c}$ be the
projection associated to $X=X_{c}\oplus X_{u}\oplus X_{s}$. According to Lemma
\ref{L:decomJ}, $(X_{c},L_{X_{c}},J_{X_{c}})$ satisfy assumption
(\textbf{H1-3)} as well. Since Proposition \ref{P:decomposition2} also ensures
the non-degeneracy of $L_{X_{5}\oplus X_{6}}$ and $\dim X_{5,6}\leq n^{-}(L)$,
the finite dimensional results in Section \ref{S:finiteD} (Lemma
\ref{L:symmetry} and \ref{lemma-negative-unstable-space}) imply the symmetry
between the spectra $\sigma(A_{5})$ and $\sigma(A_{6})$ and $n^{-}%
(L|_{X_{5}\oplus X_{6}})=\dim X_{5}$. Therefore, we obtain, from the
$L$-orthogonality between $X_{c}$ and $X_{u}\oplus X_{s}$,
\[
n^{-}(L_{X_{c}})=n^{-}(L)-\dim X_{5}.
\]
Recall $X_{0}=\ker L=\ker L_{X_{c}}\subset X_{c}$. Let $X_{\pm}$ be given by
Lemma \ref{L:decom1} applied to $(X_{c},L_{X_{c}},J_{X_{c}})$, $Y=X_{+}\oplus
X_{-}$, $P_{Y}:X\rightarrow Y$ be the associated projection, and $J_{Y}%
=P_{Y}JP_{Y}^{\ast}$. Again Lemma \ref{L:decomJ} implies $(Y,L_{Y},J_{Y})$
satisfy (\textbf{H1--3}) with $L_{Y}$ being an isomorphism. Applying
Proposition \ref{P:decomposition1} to $Y$ and we obtain subspaces $\tilde
{X}_{j}$, $j=1,2,3,4$. To ensure the orthogonality between $X_{0}=\ker L$ and
$X_{j}$, $j=1,2,3,4$, we modify the definition of $X_{j}$ as
\[
X_{j}=\{u\in X_{0}\oplus\tilde{X}_{j}\mid(u,v)=0,\;\forall v\in\ker L\},\quad
j=1,2,3,4.
\]
It is straightforward to verify the desired properties of the decomposition
$X=\oplus_{j=0}^{6}X_{j}$ by using Propositions \ref{P:decomposition1} and
\ref{P:decomposition2}. The proof of Theorem \ref{T:decomposition} is
complete. \hfill$\square$

To finish this section, we give the following lemma on the $L$-orthogonality
between certain eigenspaces defined by spectral integrals.

\begin{lemma}
\label{L:L-orth-eS} Let $\Omega\subset\mathbf{C}$ be an open subset symmetric
about $i\mathbf{R}$ with smooth boundary $\Gamma=\partial\Omega$ and compact
closure such that $\Gamma\cap\sigma(JL)=\emptyset$. Let
\[
P=\frac{1}{2\pi i}\oint_{\Gamma}(z-JL)^{-1}dz.
\]
and then it holds that $\langle L(I-P)u,Pv\rangle=0$, for any $u,v\in X$.
\end{lemma}

The above $P$ is simply the standard spectral projection operator.

\begin{proof}
We first observe for any $w,w^{\prime}\in X$, \eqref{E:complexify} and
\eqref{E:AntiHermitian} imply
\begin{align}
&  \frac{1}{2\pi i}\oint_{\Gamma}\langle Lw,(z-JL)^{-1}w^{\prime}\rangle
dz=\langle Lw,Pw^{\prime}\rangle,\label{E:temp2.1}\\
&  \frac{1}{2\pi i}\oint_{\Gamma}\langle L(z-JL)^{-1}w,w^{\prime}\rangle
d\bar{z}=-\langle LPw,w^{\prime}\rangle, \label{E:temp2.2}%
\end{align}
where the first equality is used in the derivation of the second equality.
Here the $d\bar{z}$ and the minus sign in the second equality are due to the
anti-linear nature of $L$ in \eqref{E:AntiHermitian}.

Let $\Omega_{1}\subset\Omega$ be an open subset symmetric about $i\mathbf{R}$
such that $\Gamma_{1}=\partial\Omega_{1}\subset\Omega$ is smooth and
$\sigma(JL)\cap(\Omega\backslash\Omega_{1})=\emptyset$. Clearly,
\[
P=\frac{1}{2\pi i}\oint_{\Gamma_{1}}(z-JL)^{-1}dz,
\]
due to the analyticity of $(z-JL)^{-1}$. Denote
\[
\tilde{u}(z)=(z-JL)^{-1}u,\;\tilde{v}(z)=(z-JL)^{-1}v,\quad\forall
z\notin\sigma(JL).
\]
For $z_{1},z_{2}\notin\sigma(JL)$ satisfying $\bar{z}_{1}+z_{2}\neq0$, one may
compute using \eqref{E:complexify} and \eqref{E:AntiHermitian}
\begin{align*}
&  \frac{1}{\bar{z}_{1}+z_{2}}\big(\langle L(z_{1}-JL)^{-1}u,v\rangle+\langle
Lu,(z_{2}-JL)^{-1}v\rangle\big)\\
= &  \frac{1}{\bar{z}_{1}+z_{2}}\big(\langle L\tilde{u}(z_{1}),(z_{2}%
-JL)\tilde{v}(z_{2})\rangle+\langle L(z_{1}-JL)\tilde{u}(z_{1}),\tilde
{v}(z_{2})\rangle\big)\\
= &  \langle L\tilde{u}(z_{1}),\tilde{v}(z_{2})\rangle=\langle L(z_{1}%
-JL)^{-1}u,(z_{2}-JL)^{-1}v\rangle.
\end{align*}
Due to the definition of $\Gamma_{1}$ and its symmetry about the imaginary
axis, $\bar{z}_{1}+z_{2}\neq0$ for any $z_{1}\in\Gamma$ and $z_{2}\in
\Gamma_{1}$. Integrating the above equality along these curves, where
$\Gamma_{1}$ is enclosed in $\Gamma$, we obtain from the Cauchy integral
theorem and \eqref{E:temp2.1} and \eqref{E:temp2.2}
\begin{align*}
\langle LPu,Pv\rangle= &  \frac{-1}{(2\pi i)^{2}}\oint_{\Gamma}\oint%
_{\Gamma_{1}}\langle L(z_{1}-JL)^{-1}u,(z_{2}-JL)^{-1}v\rangle dz_{2}d\bar
{z}_{1}\\
= &  \frac{-1}{(2\pi i)^{2}}\oint_{\Gamma}\oint_{\Gamma_{1}}\frac{1}{\bar
{z}_{1}+z_{2}}\langle L(z_{1}-JL)^{-1}u,v\rangle dz_{2}d\bar{z}_{1}\\
&  +\frac{-1}{(2\pi i)^{2}}\oint_{\Gamma_{1}}\oint_{\Gamma}\frac{1}{\bar
{z}_{1}+z_{2}}\langle Lu,(z_{2}-JL)^{-1}v\rangle d\bar{z}_{1}dz_{2}.
\end{align*}
Since $-\bar{z}_{1}$ is not enclosed in $\Gamma_{1}$ while $-\bar{z}_{1}$ is
enclosed in $\Gamma$, the above first integral vanishes and the we obtain from
\eqref{E:temp2.1} and the Cauchy integral theorem
\[
\langle LPu,Pv\rangle=\langle Lu,Pv\rangle.
\]
This proves the lemma.
\end{proof}

The above lemma implies that $\langle Lu, v\rangle=0$ for any $u \in\ker P$
and $v \in PX$, where $X = PX \oplus\ker P$ is a spectral decomposition of $X$
invariant under $JL$. As a corollary, we give the following extension of Lemma
\ref{lemma-orthogonal-eigenspace}.

Let $\tilde{\sigma}\subset\sigma(JL)$ be compact and also open in the relative
topology of $\sigma(JL)$, namely $\tilde{\sigma}$ is isolated in $\sigma(JL)$.
There exists an open domain $\Omega\subset\mathbf{C}$ with compact closure and
smooth boundary such that $\Omega\cap\sigma(JL)=\tilde{\sigma}$. Let
\[
P_{\tilde{\sigma}}=\frac{1}{2\pi i}\oint_{\partial\Omega}(z-JL)^{-1}dz,\quad
X_{\tilde{\sigma}}=P_{\tilde{\sigma}}X,\quad X_{\tilde{\sigma}^{c}}=\ker
P_{\tilde{\sigma}}.
\]
According to the Cauchy integral theorem, the projection operator
$P_{\tilde{\sigma}}$ as well as the above subspaces, which are invariant under
$JL$, are independent of the choice of $\Omega$ and $JLP_{\tilde{\sigma}%
}=P_{\tilde{\sigma}}JL$. Moreover,
\[
\sigma(JL|_{X_{\tilde{\sigma}}})=\tilde{\sigma},\quad\sigma(JL|_{X_{\tilde
{\sigma}^{c}}})=\sigma(JL)\backslash\tilde{\sigma}.
\]

\begin{corollary}
\label{C:L-orth-eS} Suppose $\sigma_{j}\subset\sigma(JL)$, $j=1,2$, are
compact and also open in the relative topology of $\sigma(JL)$. In addition,
assume
\[
\sigma_{1}\cap\tilde{\sigma}_{2}=\emptyset,\;\text{ where }\;\tilde{\sigma
}_{2}=\{\lambda\in\mathbf{C}\mid\lambda\in\sigma_{2}\text{ or }\bar{\lambda
}\in\sigma_{2}\}.
\]
Then $\langle Lu,v\rangle=0$ for any $u\in X_{\sigma_{1}}$ and $v\in
X_{\sigma_{2}}$ where $X_{\sigma_{1,2}}$ are defined as in the above.
\end{corollary}

\begin{proof}
According to our assumptions, there exists an open domain $\Omega
\subset\mathbf{C}$, symmetric about $i\mathbf{R}$ with smooth boundary and
compact closure such that $\Omega\cap\sigma(JL)=\tilde{\sigma}_{2}$ and
$\partial\Omega\cap\sigma(JL)=\emptyset$. The corollary follows from Lemma
\ref{L:L-orth-eS} and the facts $X_{\sigma_{1}}\subset\ker P_{\tilde{\sigma
}_{2}}$ and $X_{\sigma_{2}}\subset P_{\tilde{\sigma}_{2}}X$.
\end{proof}

\section{Exponential trichotomy}

\label{S:ET}

We prove Theorem \ref{theorem-dichotomy} on the exponential trichotomy in this
section. The proof is based on the decomposition Theorem \ref{T:decomposition}
and we follow the notations there.

Let
\[
E^{u} = X_{5}, \quad E^{s} = X_{6}, \quad E^{c} = \oplus_{j=0}^{4} X_{j},
\]
where $X_{j}$, $j=0, \ldots, 6$, are given by Theorem \ref{T:decomposition}.
Based on Theorem \ref{T:decomposition}, it only remains to prove the growth estimates.

Since $A_{2,3}$ are anti-self-adjoint with respect to the equivalent inner
product $\mp\langle L_{X_{2,3}}\cdot,\cdot\rangle$, there exists a constant
$C>0$ such that
\begin{equation}
|e^{tA_{2}}|,\ |e^{tA_{3}}|\leq C,\;\forall t\in\mathbf{R}. \label{E:ET1}%
\end{equation}
Since $\dim X_{5}=\dim X_{6}<\infty$ and $\sigma(A_{5})=-\sigma(A_{6})$, it is
clear
\begin{equation}%
\begin{split}
&  |e^{tA_{5}}|\leq C(1+|t|^{\dim X_{5}-1})e^{\lambda_{u}t},\;\forall t<0,\\
&  |e^{tA_{6}}|\leq C(1+|t|^{\dim X_{6}-1})e^{-\lambda_{u}t},\;\forall t>0
\end{split}
\label{E:ET2}%
\end{equation}
for some $C>0$ and $\lambda_{u}=\min\{{\operatorname{Re}}\lambda\mid\lambda
\in\sigma(A_{5})\}$. Finally, as $\dim X_{1}=\dim X_{4}<\infty$ and
$\sigma(A_{1,4})\subset i\mathbf{R}$, we also have
\begin{equation}
|e^{tA_{1,4}}|\leq C(1+|t|^{\dim X_{1}-1}),\;\forall t\in\mathbf{R}.
\label{E:ET3}%
\end{equation}

For any $x\in X$, write
\[
e^{tJL}x=\sum_{j=0}^{6}x_{j}(t),\quad x_{j}(t)\in X_{j},
\]
where $X_{j}$, $j=0,\ldots,6$, are given by Theorem \ref{T:decomposition}. One
can write down the equations explicitly:
\begin{equation}%
\begin{cases}
\partial_{t}x_{0}=A_{01}x_{1}+A_{02}x_{2}+A_{03}x_{3}+A_{04}x_{4}\\
\partial_{t}x_{1}=A_{1}x_{1}+A_{12}x_{2}+A_{13}x_{3}+A_{14}x_{4}\\
\partial_{t}x_{2}=A_{2}x_{2}+A_{24}x_{4}\\
\partial_{t}x_{3}=A_{3}x_{3}+A_{34}x_{4}\\
x_{j}(t)=e^{tA_{j}}x_{j}(0),\;j=4,5,6.
\end{cases}
\label{E:ET4}%
\end{equation}
For $j=2,3$, we obtain from Theorem \ref{T:decomposition} and inequalities
\eqref{E:ET1} and \eqref{E:ET3} that
\begin{equation}%
\begin{split}
\Vert x_{j}(t)\Vert=  &  \Vert e^{tA_{j}}x_{j}(0)+\int_{0}^{t}e^{(t-\tau
)A_{j}}A_{j4}e^{\tau A_{4}}x_{4}(0)d\tau\Vert\\
\leq &  C\big(\Vert x_{j}(0)\Vert+(1+|t|^{\dim X_{1}})\Vert x_{4}(0)\Vert\big)
\end{split}
\label{E:ET5}%
\end{equation}
for some $C>0$. Regrading $x_{1}(t)$, we have from \eqref{E:ET1},
\eqref{E:ET3}, and \eqref{E:ET5}
\begin{align}
\Vert x_{1}(t)\Vert\leq &  \Vert e^{tA_{1}}x_{1}(0)+\int_{0}^{t}%
e^{(t-\tau)A_{1}}\big(A_{12}x_{2}(\tau)+A_{13}x_{3}(\tau)+A_{14}e^{\tau A_{4}%
}x_{4}(0)\big)d\tau\Vert\nonumber\\
\leq &  C\Big(1+|t|^{\dim X_{1}-1}+\int_{0}^{|t|}1+|t-\tau|^{\dim X_{1}%
-1}|\tau|^{\dim X_{1}}d\tau\Big)\Vert x(0)\Vert\nonumber\\
\leq &  C(1+|t|^{2\dim X_{1}})\Vert x(0)\Vert. \label{E:ET6}%
\end{align}
Much as on the above we also have
\begin{equation}
\Vert x_{0}(t)\Vert\leq C(1+|t|^{2\dim X_{1}+1})\Vert x(0)\Vert, \label{E:ET7}%
\end{equation}
The above inequalities prove the desired exponential trichotomy estimates.

Finally, repeatedly applying $JL$ to equation \eqref{eqn-hamiltonian} and
using the above inequalities yield the trichotomy estimates in the graph norms
on $D\big((JL)^{k}\big)$.

\section{The index theorems and the structure of $E_{i\mu}$}

\label{S:Index}

Our goal in this section is to complete the proof of the index theorems and
related properties.

\subsection{Proof of Theorem \ref{theorem-counting}: the index counting
formula}

\label{SS:Index-counting}

The symmetry of $\sigma(JL)$, the eigenvalues of $JL$, and the dimensions of
the spaces of generalized eigenvectors have been proved in Lemma
\ref{L:symmetry} and Corollary \ref{C:symmetry}. The index formula
\eqref{counting-formula} will be proved in the next two lemmas. Recall the
notations $n^{-}(L|_{Y})$ and $n^{\leq0}(L|_{Y})$ for a subspace $Y\subset X$
and indices $k_{r}$, $k_{c}$, $k^{\leq0}(i\mu)$, $k_{i}^{\leq0}$, $k_{0}%
^{\leq0}$ \textit{etc.} defined in Section \ref{section-index theorem}.

\begin{lemma}
\label{L:index1} Under hypotheses (\textbf{H1-3}), it holds
\[
k_{r}+2k_{c}+2k_{i}^{\le0}+k_{0}^{\le0}\ge n^{-}\left(  L\right)  .
\]

\end{lemma}

\begin{proof}
Let $X_{j}$, $j=0,\ldots,6$, be the closed subspaces constructed in Theorem
\ref{T:decomposition} and $Z=\oplus_{j=0}^{2}X_{j}$. From Theorem
\ref{T:decomposition}, $Z$ is an invariant subspace of $JL$ containing $\ker
L$ satisfying $\sigma(JL|_{Z})\subset i\mathbf{R}$. For any eigenvalue
$i\mu\in\sigma(JL|_{Z})$, let $E_{i\mu}(Z)=E_{i\mu}\cap Z$ be the subspace of
generalized eigenvectors of $i\mu$ in $Z$, and denote the corresponding
non-positive index of $L|_{Z}$ by
\[
k_{i}^{\leq0}(Z)=\Sigma_{i\mu\in\sigma(JL|_{Z})\cap i\mathbf{R^{+}}}k^{\leq
0}(i\mu,Z),
\]
where $k^{\leq0}(i\mu,Z)=n^{\leq0}(L|_{E_{i\mu}(Z)}).$

On the one hand, for any eigenvalue $i\mu\neq0$, it clearly holds $E_{i\mu
}(Z)\subset E_{i\mu}$ and thus $k^{\leq0}(i\mu,Z)\leq k^{\leq0}(i\mu)$.
Therefore, we have $k_{i}^{\leq0}(Z)\leq k_{i}^{\leq0}$. For the same reason,
we also have $k_{0}^{\leq0}(Z)\leq k_{0}^{\leq0}$ as $\ker L\subset E_{0}(Z)$,
where $k_{0}^{\leq0}(Z)$ has a similar definition as $k_{0}^{\leq0}$ (defined
in (\ref{defn-k-0})) except applied to $E_{0}(Z)$ instead of $E_{0}$. From
Theorem \ref{T:decomposition} and the finite dimensionality of $X_{5}$, it is
clear $k_{r}+2k_{c}=\dim X_{5}$. Consequently, we obtain
\begin{equation}
k_{r}+2k_{c}+2k_{i}^{\leq0}+k_{0}^{\leq0}\geq\dim X_{5}+2k_{i}^{\leq
0}(Z)+k_{0}^{\leq0}(Z).\label{inequality-index}%
\end{equation}

On the other hand, due to the finite dimensionality of $X_{j}$, $j=1,2$, and
the blockwise upper triangular form of $JL$, we have $Z=\oplus_{i\mu\in
\sigma(JL|_{Z})\cap i\mathbf{R}}E_{i\mu}(Z)$. Moreover, since $L$ is
non-positive on $Z$ according to Theorem \ref{T:decomposition}, we have
\[
2k_{i}^{\leq0}(Z)+k_{0}^{\leq0}(Z)=\dim X_{1}+\dim X_{2}=n^{-}(L)-\dim X_{5}.
\]
Combining it with (\ref{inequality-index}), we obtain the conclusion of the lemma.
\end{proof}

\begin{lemma}
\label{L:index2} Under hypotheses (\textbf{H1-3}), it holds
\[
k_{r}+2k_{c}+2k_{i}^{\le0}+k_{0}^{\le0}\le n^{-}\left(  L\right)  .
\]

\end{lemma}

\begin{proof}
Let $X_{j}$, $j=0, \ldots, 6$, be the closed subspaces constructed in Theorem
\ref{T:decomposition} and $Y=\oplus_{j=1}^{6} X_{j}$. Let $P_{Y}$ be the
projection associated to $X= \ker L \oplus Y$. Lemma \ref{L:decomJ} implies
that $(Y, L_{Y}, J_{Y})$ satisfies assumptions (\textbf{H1-3}), where
$n^{-}(L_{Y}) = n^{-}(L)$. The definitions of $J_{Y}$ and $L_{Y}$ also imply
$J_{Y} L_{Y} = P_{Y} (JL)$.

Let $i\mu\in\sigma(JL)\cap i\mathbf{R^{+}}$. By the definition of $k^{\leq
0}(i\mu)$, there exists a subspace $E_{i\mu}^{\leq0}\subset E_{i\mu}$ such
that $\dim E_{i\mu}^{\leq0}=k^{\leq0}(i\mu)$ and $\langle Lu,u\rangle\leq0$,
for all $u\in E_{i\mu}^{\leq0}$. Since $\mu\neq0$ and thus $E_{i\mu}\cap\ker
L=\{0\}$, we have $\dim P_{Y}E_{i\mu}^{\leq0}=\dim E_{i\mu}^{\leq0}$. For
$\mu<0$, let $E_{i\mu}^{\leq0}=\{\bar{u}\mid u\in E_{-i\mu}^{\leq0}\}$. For
$\mu=0$, let $\tilde{E}_{0}=E_{0}\cap Y$ where clearly $E_{0}=\ker
L\oplus\tilde{E}_{0}$. There exists a subspace $E_{0}^{\leq0}\subset\tilde
{E}_{0}$ such that $\dim E_{0}^{\leq0}=k_{0}^{\leq0}$ and $\langle
Lu,u\rangle\leq0$, for all $u\in E_{0}^{\leq0}$. Let
\[
W=X_{5}\oplus E_{0}^{\leq0}\oplus(\oplus_{i\mu\in\sigma{JL}\cap i\mathbf{R}%
}P_{Y}E_{i\mu}^{\leq0})\subset Y.
\]
It is clearly (the complexification of) a real subspace of $Y$ satisfying
$\bar{u}\in W$ for all $u\in W$. Theorem \ref{T:decomposition} implies
\[
\dim W=\dim X_{5}+k_{0}^{\leq0}+2k_{i}^{\leq0}=k_{r}+2k_{c}+k_{0}^{\leq
0}+2k_{i}^{\leq0}.
\]
From Lemma \ref{lemma-orthogonal-eigenspace}, we have $X_{5}$ and
$P_{Y}E_{i\mu}^{\leq0}$ $\left(  i\mu\in\sigma{JL}\cap i\mathbf{R}\right)  $
are mutually $L$-orthogonal. Therefore, our construction of $W$ yields that
$\langle Lu,u\rangle\leq0$ for all $u\in W\subset Y$. Applying Theorem
\ref{T:Pontryagin} to $(Y,L_{Y},J_{Y})$ implies $\dim W\leq n^{-}(L_{Y}%
)=n^{-}(L)$ and thus the lemma is proved.
\end{proof}

\subsection{Structures of subspaces $E_{i\mu}$ of generalized eigenvectors}

\label{SS:e-space}

In this subsection, we will prove Propositions \ref{P:finite-dim-E} and
\ref{P:basis}. We complete the proof in several steps.

\begin{lemma}
\label{L:e-space-2} Let $i\mu\in\sigma(JL)\cap i\mathbf{R}$ and $E\subset
E_{i\mu}$ be a closed subspace such that $JL(E)\subset E$. In addition to
(\textbf{H1-3}), assume $\langle L\cdot,\cdot\rangle$ is non-degenerate (in
the sense of \eqref{E:non-degeneracy-def}) on both $X$ and $E$. Then there
exist closed subspaces $E^{1},\tilde{E}\subset E$ such that $E=E^{1}%
\oplus\tilde{E}$ and $L,JL$ take the following forms on $E$
\[
\langle L\cdot,\cdot\rangle\longleftrightarrow%
\begin{pmatrix}
L_{E^{1}} & 0\\
0 & L_{\tilde{E}}%
\end{pmatrix}
,\quad JL\longleftrightarrow%
\begin{pmatrix}
i\mu & 0\\
0 & \tilde{A}%
\end{pmatrix}
,
\]
and $\ker(JL-i\mu)\cap\tilde{E}\subset(JL-i\mu)\tilde{E}$ with non-degenerate
$L_{E^{1}}$ and $L_{\tilde{E}}$ and
\[
\dim\tilde{E}\leq3\big(n^{-}(L|_{E})-n^{-}(L|_{E^{1}})\big),\;\dim
\big((JL-i\mu)E\big)\leq2\big(n^{-}(L|_{E})-n^{-}(L|_{E^{1}})\big).
\]

\end{lemma}

\begin{remark}
The property $\ker(JL-i\mu) \cap\tilde E \subset(JL-i\mu)\tilde E$, or
equivalently $\ker(\tilde A -i\mu) \subset(\tilde A -i\mu) \tilde E$, is
equivalent to that the Jordan canonical form of $\tilde A$ contains only
nontrivial Jordan blocks.
\end{remark}

\begin{proof}
From Lemma \ref{L:e-space-1}, $E_{i\mu}=\ker(JL-i\mu)^{K}$ for some $K>0$ and
$JL:E_{i\mu}\rightarrow E_{i\mu}$ is a bounded operator. Let
\[%
\begin{split}
&  E^{0}=\{u\in E\cap\ker(JL-i\mu)\mid\langle Lu,v\rangle=0,\ \forall v\in
E\cap\ker(JL-i\mu)\},\\
&  E^{1}=\{u\in E\cap\ker(JL-i\mu)\mid(u,v)=0,\ \forall v\in E^{0}\}.
\end{split}
\]
Obviously, $\ker(JL-i\mu)\cap E=E^{0}\oplus E^{1}$. Moreover, for any $u\in
E^{1}\backslash\{0\}$, there must exist $v\in E^{1}$ such that $\langle
Lu,v\rangle\neq0$, otherwise it would lead to $u\in E^{0}$, a contradiction.
Applying statement 2 of Lemma \ref{L:non-degeneracy} to $Y=E^{1}$, we obtain
that $\langle L\cdot,\cdot\rangle$ is non-degenerate on $E^{1}$. Since
$\langle L\cdot,\cdot\rangle$ is assumed to be non-degenerate on both $X$ and
$E$, we apply statement 1 of Lemma \ref{L:non-degeneracy} to obtain
\begin{equation}
X=E^{1}\oplus(E^{1})^{\perp_{L}}\text{ and }E=E^{1}\oplus\tilde{E},\text{
where }\tilde{E}=E\cap(E^{1})^{\perp_{L}}.\label{E:splitting-E}%
\end{equation}
Here $(E^{1})^{\perp_{L}}\subset X$ is the subspace $L$-perpendicular to
$E^{1}$. Clearly, $E^{0}\subset\tilde{E}$.

\textit{Claim. 1.) $\dim\tilde E< \infty$, 2.) $\langle L\cdot, \cdot\rangle$
is non-degenerate on $\tilde E$, and 3.) $JL (\tilde E) \subset\tilde E$.}

The invariance of $\tilde{E}$ under $JL$ follows directly from the invariance
of $E$ and $E^{1}$ and Lemma \ref{L:InvariantSubS}. The non-degeneracy of
$\langle L\cdot,\cdot\rangle$ on both $E$ and $E^{1}$ implies that $\langle
L\cdot,\cdot\rangle$ is non-degenerate on $\tilde{E}$ as well. To complete the
proof of the claim, we only need to prove $\dim\tilde{E}<\infty$.

On the one hand, from the above definitions, $\langle Lu,v\rangle=0$ for any
$u,v\in E^{0}$. The non-degeneracy of $\langle L\cdot,\cdot\rangle$ on $E$ and
Theorem \ref{T:Pontryagin} along with Remark \ref{R:Pontryagin} imply
\begin{equation}
\dim E^{0}\leq n^{-}(L|_{\tilde{E}}).\label{E:dim1'}%
\end{equation}
On the other hand, it is clear from the definitions of $\tilde{E}$ and the
non-degeneracy of $\langle L\cdot,\cdot\rangle$ on $E^{1}$ that
\begin{equation}
\tilde{E}\cap\ker(JL-i\mu)=E^{0}.\label{E:E_0-1}%
\end{equation}
Moreover, from Lemma \ref{L:e-space-1}, $E\subset E_{i\mu}=\ker(JL-i\mu)^{K}$
for some $K>0$, and each Jordan chain in $\tilde{E}$ contains a vector in
$E^{0}$, we obtain
\[
\dim\tilde{E}\leq K\dim E^{0}\leq Kn^{-}(L|_{E}),
\]
from the invariance of $\tilde{E}$ under $JL$. The claim is proved.

Now we complete the proof of the lemma by reducing it to a finite dimensional
problem satisfying our framework. Firstly, to replace $\tilde{E}$ by the
complexification of some real Hilbert space, let
\[
\tilde{E}^{R}=\{u+\bar{v}\mid u,v\in\tilde{E}\}
\]
which satisfies $\bar{u}\in\tilde{E}^{R}$ for any $u\in\tilde{E}^{R}$. Since
$u\in E_{i\mu}$ implies $\bar{u}\in E_{-i\mu}$, we have $\tilde{E}^{R}%
=\tilde{E}$ if $\mu=0$. If $\mu\neq0$, from \eqref{E:real} and Lemma
\ref{lemma-orthogonal-eigenspace} we obtain
\[
\langle L\bar{u},v\rangle=0,\;\langle L\bar{u},\bar{v}\rangle=\overline
{\langle Lu,v\rangle},\quad\forall u,v\in\tilde{E}.
\]
Therefore, $\tilde{E}^{R}$ satisfies the same properties as in the above claim
whether $\mu=0$ or not. Using the non-degeneracy of $\langle L\cdot
,\cdot\rangle$ on $X$ and $\tilde{E}^{R}$, and applying Lemma \ref{L:decomJ}
to the splitting $X=\tilde{E}^{R}\oplus(\tilde{E}^{R})^{\perp_{L}}$ with the
associated projections $P_{\tilde{E}^{R}}$ and $I-P_{\tilde{E}^{R}}$, we have
that the combination $(\tilde{E}^{R},L_{\tilde{E}^{R}},J_{\tilde{E}^{R}})$
satisfies assumptions (\textbf{H1-3}), where $J_{\tilde{E}^{R}}=P_{\tilde
{E}^{R}}JP_{\tilde{E}^{R}}^{\ast}$. We may apply Proposition \ref{P:basis},
whose finite dimensional case under the non-degeneracy assumption on $\langle
L\cdot,\cdot\rangle$ has been proved in Section \ref{S:finiteD}. As $\tilde
{E}=\tilde{E}^{R}\cap\ker(JL-i\mu)^{K}$, that canonical form implies
\[
\dim\big((JL-i\mu)\tilde{E}\big)\leq2n^{-}(L|_{\tilde{E}}),\quad\ker
(JL-i\mu)\cap\tilde{E}\subset(JL-i\mu)\tilde{E},
\]
where \eqref{E:E_0-1} is also used along with the canonical form. We notice
$(JL-i\mu)E=(JL-i\mu)\tilde{E}$, as $E^{1}\subset\ker(JL-i\mu)$, and thus
\[
\dim\big((JL-i\mu)E\big)\leq2n^{-}(L|_{\tilde{E}})=2\big(n^{-}(L|_{E}%
)-n^{-}(L_{E^{1}})\big).
\]
The block forms of $L$ and $JL$ follow from the $L$-orthogonality and the
invariance of the splitting $E=E^{1}\oplus\tilde{E}$. Finally, the estimate on
$\dim\tilde{E}$ follows from the above inequality and (\ref{E:dim1'}) and
\eqref{E:E_0-1}. The proof is complete.
\end{proof}

Next we study $E_{i\mu}$ by assuming the non-degeneracy of $L$.

\begin{lemma}
\label{L:e-space-3} In addition to (\textbf{H1-3}), assume $\langle
L\cdot,\cdot\rangle$ is non-degenerate. Let $i\mu\in\sigma(JL)\cap
i\mathbf{R}$. There exist subspaces $E^{D,1,G}\subset E_{i\mu}$ such that
\[
E_{i\mu}=E^{D}\oplus E^{1}\oplus E^{G},\quad\dim\big((JL-i\mu)E_{i\mu
}\big)\leq2\big(k^{\leq0}(i\mu)-n^{-}(L|_{E^{1}})\big),
\]%
\[
\dim E^{G}\leq3\big(k^{\leq0}(i\mu)-\dim E^{D}-n^{-}(L|_{E^{1}})\big),
\]
and $L$ and $JL$ take the block forms on $E$
\[
\langle L\cdot,\cdot\rangle\longleftrightarrow%
\begin{pmatrix}
0 & 0 & 0\\
0 & L_{1} & 0\\
0 & 0 & L_{G}%
\end{pmatrix}
,\quad JL\longleftrightarrow%
\begin{pmatrix}
A_{D} & A_{D1} & A_{DG}\\
0 & i\mu & 0\\
0 & 0 & A_{G}%
\end{pmatrix}
\]
where all blocks are bounded operators and $L_{1}$ and $L_{G}$ are
non-degenerate. Moreover, $\ker(A_{G}-i\mu)\subset(A_{G}-i\mu)E_{G}$.
\end{lemma}

\begin{proof}
Again, to apply previous results directly it would be easier to consider the
complexifications of real Hilbert spaces
\[
I_{i\mu}\triangleq E_{i\mu}+E_{-i\mu}=\{u+\bar{v}\mid u,v\in E_{i\mu}\}.
\]
Due to Lemma \ref{L:e-space-1}, $JL|_{I_{i\mu}}$ is bounded with
\[
L(I_{i\mu})\subset D(J),\;JL(I_{i\mu})\subset I_{i\mu},\;\sigma(JL|_{I_{i\mu}%
})=\{\pm i\mu\}.
\]
We split the spaces by starting with
\[%
\begin{split}
&  E^{D}=\{u\in E_{i\mu}\mid\langle Lu,v\rangle=0,\ \forall v\in E_{i\mu
}\},\;I^{D}=\{u+\bar{v}\mid u,v\in E^{D}\},\\
&  E_{i\mu}^{ND}=\{u\in E_{i\mu}\mid(u,v)=0,\ \forall v\in E^{D}%
\},\;I^{ND}=\{u+\bar{v}\mid u,v\in E_{i\mu}^{ND}\}.
\end{split}
\]
From the anti-symmetry of $JL$ with respect to $\langle L\cdot,\cdot\rangle$
and the invariance of $E_{i\mu}$ along with \eqref{E:real}, we have
$JL(I^{D})\subset I^{D}$. In the splitting $I_{i\mu}=I^{D}\oplus I^{ND}$,
$\langle L\cdot,\cdot\rangle$ and $JL$ can be represented in the following
block forms
\[
\langle L\cdot,\cdot\rangle\longleftrightarrow%
\begin{pmatrix}
0 & 0\\
0 & L_{ND}%
\end{pmatrix}
,\quad JL\longleftrightarrow%
\begin{pmatrix}
A_{D} & A_{D,ND}\\
0 & A_{ND}%
\end{pmatrix}
,
\]
where all blocks are bounded real (satisfying \eqref{E:real}) operators. In
particular, $\ker(L|_{I_{i\mu}})=I^{D}$ and $I_{i\mu}=I^{D}\oplus I^{ND}$ and
thus Lemma \ref{L:non-degeneracy} implies that $L_{ND}:I^{ND}\rightarrow
(I^{ND})^{\ast}$ is an isomorphism. The anti-symmetry of $JL$ with respect to
$\langle L\cdot,\cdot\rangle$ yields $L_{ND}A_{ND}+A_{ND}^{\ast}L_{ND}=0$.
Therefore,
\[
J_{ND}=A_{ND}L_{ND}^{-1}:(I^{ND})^{\ast}\rightarrow I_{ND}%
\]
is an anti-symmetric bounded operator satisfying $A_{ND}=J_{ND}L_{ND}$.
Clearly, the combination $(I^{ND},L_{ND},J_{ND})$ satisfies (\textbf{H1-3})
with the non-degenerate $L_{ND}$. Moreover, $\sigma(J_{ND}L_{ND})=\{\pm
i\mu\}$ with the eigenspace of $i\mu$ given by $E_{i\mu}^{ND}$ where $\langle
L_{ND}\cdot,\cdot\rangle$ is also non-degenerate. Therefore, we may apply
Lemma \ref{L:e-space-2} (with $X$ and $E$ replaced by $I^{ND}$ and $E_{i\mu
}^{ND}$, respectively) to obtain the splitting $E_{i\mu}^{ND}=E^{1}\oplus
E^{G}$ and the desired block forms of $L$ and $JL$ follow. The desired
estimate on $\dim E^{G}$ is obtained by noting
\begin{equation}
k^{\leq0}(i\mu)=n^{-}(L|_{E_{i\mu}^{ND}})+\dim E^{D}.\label{E:dim2}%
\end{equation}
Moreover, according to Lemma \ref{L:e-space-2}, we have
\[
\dim(A_{ND}-i\mu)E_{i\mu}^{ND}\leq2\big(n^{-}(L|_{E_{i\mu}^{ND}}%
)-n^{-}(L|_{E^{1}})\big).
\]
Along with \eqref{E:dim2} and the block form of $JL$, it implies
\[%
\begin{split}
&  \dim(JL-i\mu)E_{i\mu}\leq\dim E^{D}+\dim(A_{ND}-i\mu)E_{i\mu}^{ND}\\
\leq &  \dim E^{D}+2\big(n^{-}(L|_{E_{i\mu}^{ND}})-n^{-}(L|_{E^{1}}%
)\big)\leq2\big(k^{\leq0}(i\mu)-n^{-}(L|_{E^{1}})\big)
\end{split}
\]
which finishes the proof.
\end{proof}

\textbf{Proof of Proposition \ref{P:finite-dim-E} and Proposition
\ref{P:basis}.} What remains to be proved in these two propositions can be
obtained in a similar framework and we complete their proofs together here.

Let $X_{\pm}$ be given by Lemma \ref{L:decom1} and $X_{1}=X_{-}\oplus X_{+}$.
Clearly, $X=X_{0}\oplus X_{1}$, where $X_{0}=\ker L$, with the associated
projections $P_{X_{0,1}}$. According to Lemma \ref{L:decomJ}, $\langle
L\cdot,\cdot\rangle$ and $JL$ take the following block forms
\[
\langle L\cdot,\cdot\rangle\longleftrightarrow%
\begin{pmatrix}
0 & 0\\
0 & L_{X_{1}}%
\end{pmatrix}
,\quad JL\longleftrightarrow%
\begin{pmatrix}
0 & A_{1}\\
0 & J_{X_{1}}L_{X_{1}}%
\end{pmatrix}
,
\]
where $A_{1}:X_{1}\rightarrow\ker L$ is bounded and $L_{X_{1}}=i_{X_{1}}%
^{\ast}Li_{X_{1}}:X_{1}\rightarrow X_{1}^{\ast}$ and $J_{X_{1}}=P_{X_{1}%
}JP_{X_{1}}^{\ast}$. Moreover, Lemmas \ref{L:decomJ} and \ref{L:decom1} imply
that $(X_{1},L_{X_{1}},J_{X_{1}})$ satisfies assumptions (\textbf{H1-3}) with
the isomorphic $L_{X_{1}}$ and $n^{-}(L_{X_{1}})=n^{-}(L)$. For any eigenvalue
$i\mu\in i\mathbf{R}$, let $E_{i\mu}^{1}$ be the subspace of generalized
eigenvectors of $i\mu$ for $J_{X_{1}}L_{X_{1}}$, possibly $\{0\}$ if $\mu=0$.
From Lemma \ref{L:e-space-1} and \ref{L:e-space-3}, for some $K>0$,
\[
E_{i\mu}^{1}=\ker(J_{X_{1}}L_{X_{1}}-i\mu)^{K},\quad\dim(J_{X_{1}}L_{X_{1}%
}-i\mu)E_{i\mu}^{1}\leq2n^{\leq0}(L_{X_{1}}|_{E_{i\mu}^{1}}).
\]
For any integer $k>0$, $(JL-i\mu)^{k}$ takes the block form
\[
(JL-i\mu)^{k}\longleftrightarrow%
\begin{pmatrix}
(-i\mu)^{k} & A_{k}\\
0 & (J_{X_{1}}L_{X_{1}}-i\mu)^{k}%
\end{pmatrix}
,
\]
where the linear operator $A_{k}:X_{1}\rightarrow\ker L$ can be computed
inductively
\[
A_{k+1}=(-i\mu)^{k}A_{1}+A_{k}(J_{X_{1}}L_{X_{1}}-i\mu),\quad D\big((J_{X_{1}%
}L_{X_{1}}-i\mu)^{k}\big)\subset D(A_{k+1}).
\]
It is straightforward to show
\begin{equation}
u\in E_{i\mu}\ \Longleftrightarrow\ P_{X_{1}}u\in E_{i\mu}^{1}\;\text{ and
}\;(-i\mu)^{K}P_{X_{0}}u+A_{K}P_{X_{1}}u=0.\label{E:temp3}%
\end{equation}

We first consider $\mu\neq0$. We obtain from \eqref{E:temp3}
\[
E_{i\mu}=\{u-(-i\mu)^{-K}A_{K}u\mid u\in E_{i\mu}^{1}\},
\]
i.e. vectors in $E_{i\mu}$ are determined only by their $X_{1}$-component.
From Lemma \ref{L:e-space-1} and Remark \ref{R:e-space-1}, $E_{i\mu}^{1}$ and
$E_{i\mu}$ are both subspaces. Therefore, $A_{K}$ is a bounded operator.
Since
\[
\langle L\big(u-(-i\mu)^{-K}A_{K}u\big),v-(-i\mu)^{-K}A_{K}v\rangle=\langle
Lu,v\rangle,\;\forall u,v\in E_{i\mu}^{1},
\]
we obtain from Lemma \ref{L:e-space-3}
\[%
\begin{split}
\dim(JL-i\mu)E_{i\mu}= &  \dim(J_{X_{1}}L_{X_{1}}-i\mu)E_{i\mu}^{1}\\
\leq &  2n^{\leq0}(L_{X_{1}}|_{E_{i\mu}^{1}})=2n^{\leq0}(L_{E_{i\mu}%
})=2k^{\leq0}(i\mu).
\end{split}
\]
This proves the desired estimate on $\dim(JL-i\mu)E_{i\mu}$ in Proposition
\ref{P:finite-dim-E}. Along with Lemma \ref{L:e-space-1}, it completes the
proof of Proposition \ref{P:finite-dim-E} in the case of $\mu\neq0$.

To prove Proposition \ref{P:basis}, let $E_{i\mu}^{1}=\tilde{E}^{D}%
\oplus\tilde{E}^{1}\oplus\tilde{E}^{G}$ where these subspaces are given by
Lemma \ref{L:e-space-3} for $J_{X_{1}}L_{X_{1}}$. Let
\[
E^{D,1,G}=\{u-(-i\mu)^{-K}A_{K}u\mid u\in\tilde{E}^{D,1,G}\}.
\]
It is easy to verify that they satisfy the properties in Proposition
\ref{P:basis}. Since $\dim E^{G}<\infty$, the `good' basis of $E^{G}$ has been
constructed in the finite dimensional cases in Section \ref{S:finiteD} and the
proof of Proposition \ref{P:basis} is complete.

For $\mu=0$, it is easy to see from the above block forms
\[
E_{0}=X_{0}\oplus E_{0}^{1}.
\]
Therefore, we have
\[
(JL)^{2}E_{0}=(JL)^{2}E_{0}^{1}=JL(P_{X_{1}}JLE_{0}^{1})=JL(J_{X_{1}}L_{X_{1}%
}E_{0}^{1}),
\]
which along with Lemma \ref{L:e-space-3} implies
\[
\dim\ (JL)^{2}E_{0}\leq\dim\ J_{X_{1}}L_{X_{1}}E_{0}^{1}\leq2k_{0}^{\leq0}.
\]
This completes the proof of Proposition \ref{P:finite-dim-E} in the case of
$\mu=0$.

To prove Proposition \ref{P:basis}, let $E_{0}^{1}=E^{D}\oplus E^{1}\oplus
E^{G}$ where these subspaces are given by Lemma \ref{L:e-space-3} for
$J_{X_{1}}L_{X_{1}}$ and $\mu=0$. It is easy to verify that they satisfy the
properties in Proposition \ref{P:basis}. Again since $\dim E^{G}<\infty$, the
`good' basis of $E^{G}$ has been constructed in the finite dimensional cases
in Section \ref{S:finiteD} and the proof of Proposition \ref{P:basis} is
complete. \hfill$\square$

\begin{remark}
\label{R:(JL)^2} In the case of $\mu=0$, we can not replace $(JL)^{2}E_{0}$ by
$JLE_{0}$, as seen from the following counterexample. Consider $X=Y\oplus
Y\oplus\mathbf{R^{2}}$ where $Y$ is any Hilbert space. Let
\[
J=%
\begin{pmatrix}
0 & I & 0 & 0\\
-I & 0 & 0 & 0\\
0 & 0 & 0 & -1\\
0 & 0 & 1 & 0
\end{pmatrix}
,\quad L=%
\begin{pmatrix}
0 & 0 & 0 & 0\\
0 & I & 0 & 0\\
0 & 0 & 1 & 0\\
0 & 0 & 0 & -1
\end{pmatrix}
,\quad JL=%
\begin{pmatrix}
0 & I & 0 & 0\\
0 & 0 & 0 & 0\\
0 & 0 & 0 & 1\\
0 & 0 & 1 & 0
\end{pmatrix}
.
\]
It is clear that $k_{0}^{\leq0}=0$, $\ker L=X_{0}=Y\oplus\{0\}\oplus
\{(0,0)^{T}\}$, $E_{0}=Y\oplus Y\oplus\{(0,0)^{T}\}$, and $\dim JLE_{0}%
=\dim\ker L=\dim Y$.
\end{remark}

\subsection{Subspace of generalized eigenvectors $E_{0}$ and index $k_{0}%
^{\le0}$}

\label{SS:E_0}

In this Subsection we prove Propositions \ref{prop-counting-k-0-1},
\ref{prop-counting-k-0-2}, Lemma \ref{L:counting-k-0-2} and Corollary
\ref{C:counting-k-0-2-a}, \ref{C:counting-k-0-2-b} on the subspace $E_{0}$ and
the non-positive index $k_{0}^{\leq0}$ for the eigenvalue $0$.

\textbf{Proof of Proposition \ref{prop-counting-k-0-1}}.
According to Corollary \ref{C:decomJ}, $LJ:D(J)\rightarrow X$ is closed and
thus $LJL$ is also closed. Therefore, $(JL)^{-1}(\ker L)=\ker(LJL)$ is also closed.

Since $(JL)^{-1}(\ker L)\subset E_{0}$, due to the hyperbolicity of $JL$ on
$X_{5,6}$, we have $(JL)^{-1}(\ker L)\subset\oplus_{j=0}^{4}X_{j}$, where the
decomposition of $X=\ker L\oplus\oplus_{j=1}^{6}X_{j}$ is given in Theorem
\ref{T:decomposition}. Let
\[
S=(JL)^{-1}(\ker L)\cap\oplus_{j=1}^{4}X_{j}.
\]
Since $X_{0}=\ker L\subset(JL)^{-1}(\ker L)$, we have $\ker L\oplus
S=(JL)^{-1}(\ker L)\subset E_{0}$. Therefore, from the definition of
$k_{0}^{\leq0}$ it is clear $k_{0}^{\leq0}\geq n_{0}=n^{\leq0}(L|_{S})$ and we
only need to prove (ii) of Proposition \ref{prop-counting-k-0-1}.

Assume in addition that $\langle L\cdot,\cdot\rangle$ is non-degenerate on
$(JL)^{-1}(\ker L)\slash\ker L$. We claim
\begin{equation}
E_{0}=(JL)^{-1}(\ker L).\label{E:E_0-2}%
\end{equation}
In fact, suppose $u\in E_{0}\backslash\big((JL)^{-1}(\ker L)\big)$. There
exists $m>0$ such that
\[%
\begin{split}
&  u_{1}\triangleq(JL)^{m-1}u\notin(JL)^{-1}(\ker L)\\
&  u_{0}\triangleq JLu_{1}=(JL)^{m}u\in(JL)^{-1}(\ker L)\backslash\ker L.
\end{split}
\]
It follows that, for any $v\in(JL)^{-1}(\ker L)$,
\[
JLv\in\ker L\Longrightarrow\langle Lu_{0},v\rangle=\langle L(JL)u_{1}%
,v\rangle=-\langle Lu_{1},JLv\rangle=0.
\]
The existence of such $u_{0}$ would imply $\langle L\cdot,\cdot\rangle$ is
degenerate on $(JL)^{-1}(\ker L)\slash\ker L$, contradictory to our
assumption. Therefore, \eqref{E:E_0-2} is proved and consequently we obtain
from the definition of $k_{0}^{\leq0}$ that
\[
k_{0}^{\leq0}=n^{\leq0}(\langle L\cdot,\cdot\rangle|_{(JL)^{-1}(\ker
L)\slash\ker L}\big)=n^{-}(\langle L\cdot,\cdot\rangle|_{(JL)^{-1}(\ker
L)\slash\ker L}\big)
\]
due to the non-degeneracy assumption. This completes the proof of the
proposition. \hfill$\square$

We will prove Lemma \ref{L:counting-k-0-2}, Proposition
\ref{prop-counting-k-0-2}, and Corollary \ref{C:counting-k-0-2-a} and
\ref{C:counting-k-0-2-b} in the rest of the subsection. We first observe that
it is straightforward to show $\langle Lu,v\rangle=0$, for any $u\in\ker(JL)$
and $v\in R(J)$. Through a density argument, we obtain
\begin{equation}
\langle Lu,v\rangle=0,\quad\forall u\in\ker(JL),\;v\in\overline{R(J)}.
\label{E:k-0-1}%
\end{equation}
Throughout the rest of this subsection, let $S_{1},S_{2},S^{\#}$ be defined as
in Corollaries \ref{C:counting-k-0-2-a} and \ref{C:counting-k-0-2-b}, i.e.
\[
\ker(JL)=\ker L\oplus S_{1},\qquad\overline{R(J)}=\big(\overline{R(J)}\cap\ker
L\big)\oplus S^{\#}%
\]
and
\[
\overline{R(J)}\cap(JL)^{-1}(\ker L)=S_{2}\oplus\big(\overline{R(J)}\cap\ker
L\big).
\]

\begin{lemma}
\label{L:k-0-2-a} Suppose $\langle L\cdot, \cdot\rangle$ is non-degenerate on
$S^{\#}$, then it is also non-degenerate on $S_{1}$ and moreover,
\begin{equation}
\label{E:k-0-2}X=\ker(JL) \oplus S^{\#} = \ker L \oplus S_{1} \oplus S^{\#}.
\end{equation}

\end{lemma}

\begin{proof}
The non-degeneracy of $\langle L\cdot,\cdot\rangle$ on $S^{\#}$ implies the
non-degeneracy of $L_{S^{\#}}:S^{\#}\rightarrow(S^{\#})^{\ast}$, which is
defined in \eqref{E:L_y1}. For any $u\in X$, as in the proof of Lemma
\ref{L:non-degeneracy}, let
\[
u^{\#}=L_{S^{\#}}^{-1}i_{S^{\#}}^{\ast}Lu\in S^{\#}%
\]
which satisfies
\[
\langle Lu_{1},v\rangle=0,\;\forall v\in S^{\#},\;\text{ where }u_{1}%
=u-u^{\#}.
\]
By the definition of $S^{\#}$, we also have
\[
\langle Lu_{1},v\rangle=0,\;\forall v\in\overline{R(J)}.
\]
Since $J^{\ast}=-J$, we obtain
\[
Lu_{1}\in\ker J^{\ast}=\ker J\,\Longrightarrow u_{1}\in\ker(JL)=\ker L\oplus
S_{1}.
\]
Therefore, $u=u_{1}+u^{\#}\in\ker(JL)+S^{\#}$ and thus $X=\ker(JL)+S^{\#}$.

For any $u\in S^{\#}\cap ker(JL)$, from \eqref{E:k-0-1} we obtain $\langle
Lu,v\rangle=0$, for any $v\in\ker(JL)+\overline{R(J)}\supset\ker(JL)+S^{\#}%
=X$. Therefore, $u\in\ker L$. Since $u\in S^{\#}\cap\ker L=\{0\}$, we have
$u=0$ and thus $X=\ker(JL)\oplus S^{\#}=\ker L\oplus S_{1}\oplus S^{\#}$.

From Lemma \ref{L:non-degeneracy}, $\langle L\cdot,\cdot\rangle$ is
non-degenerate on $S^{\#}\oplus S_{1}$. Since it is also assumed to be
non-degenerate on $S^{\#}$, the non-degeneracy of $\langle L\cdot,\cdot
\rangle$ on $S_{1}\ $follows from the $L$-orthogonality \eqref{E:k-0-1}
between $S_{1}$ and $S^{\#}$.
\end{proof}

\begin{lemma}
\label{L:k-0-2-b} Suppose $\langle L\cdot, \cdot\rangle$ is non-degenerate on
$S_{1}$, then it is also non-degenerate on $S^{\#}$.
\end{lemma}

\begin{proof}
Like in the proof of the previous lemma, the non-degeneracy of $\langle
L\cdot, \cdot\rangle$ on $S_{1}$ implies the non-degeneracy of $L_{S_{1}}:
S_{1} \to S_{1}^{*}$. For any $u\in X$, as in the proof of Lemma
\ref{L:non-degeneracy}, let
\[
u_{1} = L_{S_{1}}^{-1} i_{S_{1}}^{*} L u\in S_{1}
\]
which satisfies
\[
\langle Lu_{*}, v \rangle=0, \; \forall v \in\ker(JL) = \ker L \oplus S_{1},
\; \text{ where } u_{*} = u - u_{1}.
\]
Since $JL = - (LJ)^{*}$, we obtain $Lu_{*} \in\overline{R(LJ)}$.

\textit{Claim: $\overline{R(LJ)}=L(S^{\#})$.} In fact, it is easy to see
$L(S^{\#})\subset L\big(\overline{R(J)}\big)\subset\overline{R(LJ)}$ due to
the boundedness of $L$. In the following we will prove that $\overline
{R(LJ)}\subset L(S^{\#})$. Let $y\in\overline{R(LJ)}$, there exists a sequence
$y_{n}=LJx_{n}$ such that $y_{n}\rightarrow y$ as $n\rightarrow+\infty$. Since
$\overline{R(J)}=\ker L\oplus S^{\#}$, let $Jx_{n}=z_{n,0}+z_{0,\#}$ where
$z_{n,0}\in\ker L$ and $z_{n,\#}\in S^{\#}$. As $y_{n}=LJx_{n}=Lz_{n,\#}%
\rightarrow y$ and the non-degeneracy assumption of $\langle L\cdot
,\cdot\rangle$ on $S^{\#}$ implies that $L|_{S^{\#}}:S^{\#}\rightarrow
L(S^{\#})$ is an isomorphism, we obtain that $\{z_{n,\#}\}$ is a Cauchy
sequence. Let $z_{n,\#}\rightarrow z_{\#}\in S^{\#}$ and then $y=Lz_{\#}\in
L(S^{\#})$. The claim is proved.

We can now finish the proof of the lemma. Since we have proved
\[
L(u-u_{1})=Lu_{\ast}\in\overline{R(LJ)}=L(S^{\#}),
\]
there exists $u_{\#}\in S^{\#}$ such that $L(u-u_{1})=Lu_{\#}$. Let
$u_{0}=u-u_{1}-u_{\#}$. Clearly, $u_{0}\in\ker L$. Therefore, $u=u_{0}%
+u_{1}+u_{\#}$ and thus $X=\ker L\oplus S_{1}\oplus S^{\#}=\ker(JL)\oplus
S^{\#}$. The proof of $\ker(JL)\cap S^{\#}=\{0\}$ and consequently the
non-degeneracy of $\langle L\cdot,\cdot\rangle$ on $S_{1}$ is the same as in
the proof of the last lemma.
\end{proof}

The conclusion in Lemma \ref{L:counting-k-0-2} is already contained in the
above lemmas.

\textbf{Proof of Proposition \ref{prop-counting-k-0-2} and equivalently
Corollary \ref{C:counting-k-0-2-b}}. The property $X= \ker(JL) +
\overline{R(J)}$ is a direct consequence of \eqref{E:k-0-2}. Along with
\eqref{E:k-0-1}, it also implies $\overline{R(J)} \cap\ker(JL)= \overline
{R(J)} \cap\ker L \subset\ker L$.

From the $L$-orthogonality \eqref{E:k-0-1}, the decomposition \eqref{E:k-0-2},
and the non-degeneracy of $\langle L\cdot, \cdot\rangle$ on $S_{1}, S^{\#}$,
and $S_{1} \oplus S^{\#}$, we immediately obtain $n^{-} = n^{-}(L|_{S_{1}}) +
n^{-}(L|_{S^{\#}})$.

From the decomposition \eqref{E:k-0-2} and the definitions of $S_{1}$ and
$S_{2}$, we have
\[
(JL)^{-1}\ker L=\ker L\oplus S_{1}\oplus S_{2}.
\]
Therefore, $k_{0}^{\leq0}\geq n^{-}(L|_{S_{1}})+n^{\leq0}(L|_{S_{2}})$ follows
from Proposition \ref{prop-counting-k-0-1}.

Finally, let us assume, in addition, that $\langle L\cdot, \cdot\rangle$ is
non-degenerate on $S_{2}$. Immediately we have the non-degeneracy of $\langle
L\cdot, \cdot\rangle$ on $(JL)^{-1} (\ker L) \slash \ker L$ and Proposition
\ref{prop-counting-k-0-1} implies $k_{0}^{\le0} = n^{-} (L|_{S_{1}}) +
n^{\le0} (L|_{S_{2}})$. The proof is complete. \hfill$\square$

\subsection{Non-degeneracy of $\langle L\cdot, \cdot\rangle$ on $E_{i\mu}$ and
isolated purely imaginary spectral points}

\label{SS:non-deg}

In Proposition \ref{P:basis}, the presence of the subspace $E^{D} \subset
E_{i\mu}$ is due to the possible degeneracy of $\langle L\cdot, \cdot\rangle$
on $E_{i\mu}$. Otherwise the statement of the proposition would be much more
clean and some results can be improved. However, in case when $i\mu$ is not
isolated in $\sigma(JL)$, it is indeed possible that $\langle L\cdot,
\cdot\rangle$ degenerates on $E_{i\mu}$ even if it is non-degenerate on $X$.
\newline

\noindent\textbf{Example of degenerate $\langle L\cdot, \cdot\rangle$ on
$E_{i\mu}$.} Consider $X= \mathbf{R}^{2n} \oplus\mathbf{R}^{2n} \oplus X_{1}$,
where $X_{1}$ is a Hilbert space. Here we identify Hilbert spaces and their
dual spaces via Riesz Representation Theorem. Let $\mu\in\mathbf{R}$ and
\newline$\bullet$ $A: X_{1} \supset D(A) \to X_{1}$ be an anti-self-adjoint
operator such that $i\mu\in\sigma(A)$ is \textit{not} an eigenvalue;\newline%
$\bullet$ $A_{1}: \mathbf{R}^{2n} \to X_{1}$ such that $\ker A_{1} =\{0\}$
and, after the complexification of $A$ and $A_{1}$ into complex linear
operators, $R(A_{1}) \cap R(A\pm i\mu) = \{0\}$, which is possible due to the
spectral assumption on $A$; and \newline$\bullet$ $J=%
\begin{pmatrix}
0 & J_{2n} & 0\\
J_{2n} & J_{2n} & - B^{-1} A_{1}^{*}\\
0 & A_{1} B^{-1} & A
\end{pmatrix}
, \; L =
\begin{pmatrix}
0 & B & 0\\
B & 0 & 0\\
0 & 0 & I_{X_{1}}%
\end{pmatrix}
$, where $B_{2n\times2n}$ is any symmetric matrix and $J_{2n} =
\begin{pmatrix}
0 & -I_{n\times n}\\
I_{n\times n} & 0
\end{pmatrix}
$. \newline One may compute
\[
JL=
\begin{pmatrix}
J_{2n} B & 0 & 0\\
J_{2n}B & J_{2n}B & - B^{-1} A_{1}^{*}\\
A_{1} & 0 & A
\end{pmatrix}
.
\]

\begin{lemma}
\label{L:CExample} For any integer $k>0$,
\[
\ker(JL-i\mu)^{k}=\{(0,x,0)^{T}\mid x\in\ker(J_{2n}B-i\mu)^{k}\}\subset
\mathbf{R}^{2n}\times\mathbf{R}^{2n}\times X_{1}.
\]
Consequently, $\langle L\cdot,\cdot\rangle$ vanishes on $E_{\pm i\mu}$.
\end{lemma}

\begin{remark}
The embedding from $\mathbf{R}^{2n}$ to $\{0\}\times\mathbf{R}^{2n}%
\times\{0\}\subset X$-- an invariant subspace under $JL$, serves as a
similarity transformation between the $2n$-dim Hamiltonian operator $J_{2n}B$
and the restriction of the infinite dimensional one $JL$. If $i\mu\in
\sigma(J_{2n}B)$, then $J_{2n}B$ and $JL$ have exactly the same structures on
the subspaces $E_{i\mu}(J_{2n}B)$ and $E_{i\mu}$ of generalized eigenvectors
of $i\mu$. However, the energy structure is completely destroyed. Namely the $2n$-dim Hamiltonian operator $J_{2n}B$ has a non-trivial energy $\langle B\cdot, \cdot \rangle$ while the energy $\langle L \cdot, \cdot \rangle$ of $JL$ vanishes completely on $\mathbf{R}^{2n}$ to $\{0\}\times\mathbf{R}^{2n}\times\{0\}\subset X$.
\end{remark}

\begin{proof}
Using the invariance under $JL$ of $\{0\}\times\mathbf{R}^{2n}\times\{0\}$ and
$\{0\}\times\mathbf{R}^{2n}\times X_{1}$, it is easy to compute inductively
\[
(JL-i\mu)^{k}=%
\begin{pmatrix}
(J_{2n}B-i\mu)^{k} & 0 & 0\\
A_{21} & (J_{2n}B-i\mu)^{k} & A_{23}\\
A_{31} & 0 & (A-i\mu)^{k}%
\end{pmatrix}
,
\]
where
\[
A_{31}=\Sigma_{l=0}^{k-1}(A-i\mu)^{l}A_{1}(J_{2n}B-i\mu)^{k-1-l}.
\]
Let $P_{1,2,3}$ denote the projections from $X$ to its components. For any
$u=(x_{1},x_{2},v)^{T}\in X$, we have
\begin{align*}
P_{3}(JL-i\mu)^{k}u  &  =A_{31}x_{1}+(A-i\mu)^{k}v\\
&  =A_{1}(J_{2n}B-i\mu)^{k-1}x_{1}+(A-i\mu)\Big((A-i\mu)^{k-1}v\\
&  \qquad\qquad\quad+\Sigma_{l=1}^{k-1}(A-i\mu)^{l-1}A_{1}(J_{2n}%
B-i\mu)^{k-1-l}x_{1}\Big).
\end{align*}

Suppose $P_{3}(JL-i\mu)^{k}u=0$. Since $A_{1}$ and $A-i\mu$ are both
one-to-one and $R(A_{1})\cap R(A-i\mu)=\{0\}$, we obtain
\[
(J_{2n}B-i\mu)^{k-1}x_{1}=0,\;(A-i\mu)^{k-1}v+\Sigma_{l=1}^{k-1}(A-i\mu
)^{l-1}A_{1}(J_{2n}B-i\mu)^{k-1-l}x_{1}=0.
\]
Let $m\in\lbrack0,k-1]$ be the minimal non-negative integer satisfying
$(J_{2n}B-i\mu)^{m}x_{1}=0$. If $m\geq1$, from the definition of $m$, the
above second equality and the injectivity of $(A-i\mu)^{k-1-m}$ imply
\begin{align*}
0 &  =(A-i\mu)^{m}u+\Sigma_{l=k-m}^{k-1}(A-i\mu)^{l+m-k}A_{1}(J_{2n}%
B-i\mu)^{k-1-l}x_{1}\\
&  =(A-i\mu)\Big((A-i\mu)^{m-1}v+\Sigma_{l=k-m+1}^{k-1}(A-i\mu)^{l+m-k-1}%
A_{1}(J_{2n}B-i\mu)^{k-1-l}x_{1}\Big)\\
&  \ \ \ \ \ \ +A_{1}(J_{2n}B-i\mu)^{m-1}x_{1}.
\end{align*}
Again since $A_{1}$ and $A-i\mu$ are both one-to-one and $R(A_{1})\cap
R(A-i\mu)=\{0\}$, we derive $(J_{2n}B-i\mu)^{m-1}x_{1}=0$ which contradicts
the definition of $m$. Therefore, $m=0$, that is, $x_{1}=0$. Due to the
injectivity of $(A-i\mu)^{k}$, it implies $v=0$ as well.

Suppose $u \in\ker(JL-i\mu)^{k}$, the above arguments imply $u =(0, x, 0)^{T}$
and the lemma follows immediately.
\end{proof}

In the rest of this subsection we will prove that the degeneracy of $\langle
L\cdot, \cdot\rangle$ may occur on $E_{i\mu}$ only if $i\mu\in\sigma(JL)$ is
not an isolated spectral point.

\begin{lemma}
\label{L:isolated-2} Assume (\textbf{H1-3}) and $i\mu\in\sigma(JL) \cap
i\mathbf{R}$ is isolated in $\sigma(JL)$, then there exist closed subspaces
$I^{i\mu}, E_{\&} \subset X$ such that

\begin{enumerate}
\item[(i)] $I^{i\mu}$ and $E_{\&}$ are complexifications of real subspaces of
$X$, namely $u\in I^{i\mu}$ (or $E_{\&}$) if and only if $\overline{u}\in
I^{i\mu}$ (or $E_{\&}$). Moreover, they are invariant under $JL$ and
\[
X=I^{i\mu}\oplus E_{\&},\quad\sigma(JL|_{I^{i\mu}})=\{\pm i\mu\},\quad
\sigma(JL|_{E_{\&}})=\sigma(JL)\backslash\{\pm i\mu\}.
\]

\item[(ii)] $\ker L \subset E_{\&}$ if $\mu\ne0$ or $\ker L \subset I^{i\mu}$
if $\mu=0$.

\item[(iii)] $\langle Lu,v\rangle=0$ for all $u\in I^{i\mu}$ and $v\in E_{\&}%
$. Moreover, $\langle L\cdot,\cdot\rangle$ is non-degenerate on quotient
spaces $I^{i\mu}\slash(\ker L\cap I^{i\mu})$ and $E_{\&}\slash(\ker L\cap
E_{\&})$.
\end{enumerate}
\end{lemma}

\begin{proof}
Let $\Gamma\subset\mathbf{C}\backslash\sigma(JL)$ be a small circle, oriented
counterclockwisely, enclosing $i\mu$ but no other elements in $\sigma(JL)$.
Define the spectral projection and the eigenspaces
\[
P_{i\mu}=\frac{1}{2\pi i}\oint_{\Gamma}(\lambda-JL)^{-1}d\lambda,\quad
E^{i\mu}=P_{i\mu}X.
\]
It is standard to verify that $P_{i\mu}$ is a bounded projection on $X$
satisfying
\[
JLP_{i\mu}=P_{i\mu}JL;\;\sigma\big((JL)|_{E^{i\mu}}\big)=\{i\mu\};\;E^{i\mu
}\subset D(JL);\;e^{tJL}E^{i\mu}=E^{i\mu},\,\forall t\in\mathbf{R}.
\]
By Lemma \ref{L:symmetry}, $-i\mu\in\sigma(JL)$ is also an isolated point of
$\sigma(JL)$. Let $P_{-i\mu}$ and $E^{-i\mu}$ be defined similarly. It is
standard that $P_{i\mu}P_{-i\mu}=P_{-i\mu}P_{i\mu}=0$ and thus $P_{i\mu
}+P_{-i\mu}$ is also a projection (or $P_{i\mu}$ instead if $\mu=0$). Define
\[
E_{\&}=\ker(P_{i\mu}+P_{-i\mu}),\quad I^{i\mu}=E^{i\mu}+E^{-i\mu}%
\]
and we have
\begin{equation}
e^{tJL}E_{\&}=E_{\&},\;\forall t\in\mathbf{R};\;\sigma\big((JL)|_{E_{\&}%
}\big)=\sigma(JL)\backslash\{\pm i\mu\};\;X=I^{i\mu}\oplus E_{\&}%
.\label{E:E-and}%
\end{equation}
Therefore, statements (i) and (ii) in the lemma follow from the standard
spectral theory. The $L$-orthogonality between $I^{i\mu}$ and $E_{\&}$ follows
from Lemma \ref{L:L-orth-eS} where $\Omega$ can be taken as the union of the
two small disks centered at $\pm i\mu$.

To complete the proof of the lemma, it suffices to prove the non-degeneracy of
$\langle L\cdot,\cdot\rangle$ on $I^{i\mu}\slash(\ker L\cap I^{i\mu})$ and
$E_{\&}\slash(\ker L\cap E_{\&})$. According to Lemma \ref{L:decomJ},
$L_{I^{i\mu}}$ and $L_{E_{\&}}$ satisfy (\textbf{H2}). Therefore, either they
are non-degenerate or have non-trivial kernels. Suppose there exists $u\in
I^{i\mu}$ such that $\langle Lu,v\rangle=0$ for all $v\in I^{i\mu}$. From
$X=I^{i\mu}\oplus E_{\&}$ and the $L$-orthogonality between $I^{i\mu}$ and
$E_{\&}$, we obtain $\langle Lu,v\rangle=0$ for all $v\in X$, which implies
$u\in\ker L$. Therefore, $\langle L\cdot,\cdot\rangle$ is non-degenerate on
$I^{i\mu}\slash(\ker L\cap I^{i\mu})$. The proof of the non-degeneracy of
$\langle L\cdot,\cdot\rangle$ on $E_{\&}\slash(\ker L\cap E_{\&})$ is similar
and thus we complete the proof of the lemma.
\end{proof}

Notice that $I^{i\mu}$ is given not in terms of $E_{i\mu}$, but of $E^{i\mu}$
defined using spectral integrals. In the following we establish the
relationship between $I^{i\mu}$ and the subspace $E_{i\mu}$ of generalized eigenvectors.

\begin{lemma}
\label{L:isolated-3} It holds $I^{i\mu} = E_{i\mu} + E_{-i\mu}$.
\end{lemma}

\begin{proof}
Let $P^{\mu}:X\rightarrow I^{i\mu}$ be the projection associated to the
$L$-orthogonal decomposition $X=I^{i\mu}\oplus E_{\&}$. Let $J^{\mu}=P^{\mu
}J(P^{\mu})^{\ast}$. As $I^{i\mu}\subset D(JL)$, Lemma \ref{L:decomJ} implies
that $(I^{i\mu},L_{I^{i\mu}},J^{\mu})$ satisfies assumptions (\textbf{H1-3}).
The invariance of $I^{i\mu}$ under $JL$ implies $JL|_{I^{i\mu}}=J^{\mu
}L_{I^{i\mu}}$ and $\sigma(J^{\mu}L_{I^{i\mu}})=\{\pm i\mu\}$. Since $\pm
i\mu\notin\sigma(JL|_{E_{\&}})$, we have $E_{\pm i\mu}\subset I^{i\mu}$.

We apply Theorem \ref{T:decomposition} to $JL$ on $I^{i\mu}$, where there is
no hyperbolic subspace, and obtain the decomposition of $I^{i\mu}$ into closed
subspaces $I^{i\mu}=\Sigma_{j=0}^{4}X_{j}$, where $X_{0}=\ker L$ if $\mu=0$ or
$X_{0}=\{0\}$ if $\mu\neq0$. In this decomposition, $L_{I^{i\mu}}$ and $JL$
take the block forms
\[
JL\leftrightarrow%
\begin{pmatrix}
0 & A_{01} & A_{02} & A_{03} & A_{04}\\
0 & A_{1} & A_{12} & A_{13} & A_{14}\\
0 & 0 & A_{2} & 0 & A_{24}\\
0 & 0 & 0 & A_{3} & A_{34}\\
0 & 0 & 0 & 0 & A_{4}%
\end{pmatrix}
,\;L_{I^{i\mu}}\leftrightarrow%
\begin{pmatrix}
0 & 0 & 0 & 0 & 0\\
0 & 0 & 0 & 0 & B_{14}\\
0 & 0 & L_{X_{2}} & 0 & 0\\
0 & 0 & 0 & L_{X_{3}} & 0\\
0 & B_{14}^{\ast} & 0 & 0 & 0
\end{pmatrix}
.
\]
Note $L_{X_{3}}\geq\delta$ for some $\delta>0$ and $A_{2,3}$ are
anti-self-adjoint with respect to the equivalent inner product $\mp\langle
L_{X_{2,3}}\cdot,\cdot\rangle$ with $\sigma(A_{1,2,3,4})=\{i\mu,-i\mu\}$.

In the case of $\mu=0$, the anti-self-adjoint operator $A_{2,3}$ must be
$A_{2,3}=0$. Meanwhile all other finite dimensional diagonal blocks are also
nilpotent. Therefore, it is straightforward to compute that $(JL|_{I^{i\mu}%
})^{k}=0$ for some integer $k>0$. Therefore, $I^{i\mu}$ consists of
generalized eigenvectors only and $I^{i\mu}=E_{i\mu}$ in the case of $\mu=0$.

In the case of $\mu\neq0$, $X_{0}=\{0\}$. Moreover, as $A_{3}$ is
anti-self-adjoint with respect to the inner product $\langle L_{X_{3}}%
\cdot,\cdot\rangle$, we can further decompose $X_{3}$ into closed subspaces
$X_{3}=X_{3+}\oplus X_{3-}$, where $X_{3\pm}=\ker(A_{3}\pm i\mu)$, with
associated projections $Q_{\pm}:X_{3}\rightarrow X_{3\pm}$. Accordingly
$A_{3}=i\mu Q_{+}-i\mu Q_{-}$, which implies $A_{3}^{2}+\mu^{2}=0$. As
$A_{1,2,4}$ are finite dimensional with the only eigenvalues $\pm i\mu$, we
obtain that $\big((JL|_{I^{i\mu}})^{2}+\mu^{2}\big)^{k}=0$ for some integer
$k>0$. Rewrite it as
\[
(JL-i\mu)^{k}(JL+i\mu)^{k}=(JL+i\mu)^{k}(JL-i\mu)^{k}=0\;\text{ on }I^{i\mu}.
\]
Let $X_{\pm}$ be the invariant eigenspace of $\pm i\mu$ of $JL|_{I^{i\mu}}$
defined via spectral integrals. We have $I^{i\mu}=X_{+}\oplus X_{-}$. As
$JL\pm i\mu$ is an isomorphism from $X_{\pm}$ to itself, we obtain from the
above identity that $X_{\pm}=\ker(JL\mp i\mu)^{k}$. Therefore, $X_{\pm}$ are
the subspaces of generalized eigenvectors of $\pm i\mu$ of $JL$, that is,
\[
I^{i\mu}=\ker(JL-i\mu)^{k}\oplus\ker(JL+i\mu)^{k}=E_{i\mu}\oplus E_{-i\mu}.
\]

\end{proof}

Finally, let $E_{\#}=E_{\&}$ if $\mu=0$ or $E_{\#}=E_{-i\mu}\oplus E_{\&}$ if
$\mu\neq0$. In the case of $\mu\neq0$, Lemmas \ref{lemma-L-form-preserve},
\ref{L:non-degeneracy}, \ref{L:decomJ} and the non-degeneracy of $\langle
L\cdot,\cdot\rangle$ on $I^{i\mu}$ imply that $\langle L\cdot,\cdot\rangle$ is
non-degenerate on $E_{i\mu}$ and $E_{\&}\slash\ker L$. This along with the
above lemmas completes the proof of Proposition \ref{P:non-deg}.

Based on Proposition \ref{P:non-deg}, we are ready to prove Proposition
\ref{P:direct decomposition}. \newline

\textbf{Proof of Proposition \ref{P:direct decomposition}:} Let
\[
\Lambda= \{ 0 \ne i\mu\in\sigma(JL) \cap i\mathbf{R} \mid k^{\le0} (i\mu)
>0\},
\]
which is a finite set according to \eqref{counting-formula} of Theorem
\ref{theorem-counting}.

Let $i\mu\in\Lambda$. We have that $\langle L\cdot,\cdot\rangle$ is
non-degenerate on $E_{i\mu}$ either by our assumption if $i\mu$ is not
isolated in $\sigma(JL)$ or by Proposition \ref{P:non-deg} if $i\mu$ is
isolated. From Proposition \ref{P:basis}, we have the $L$-orthogonal and
$JL$-invariant decomposition $E_{i\mu}=E_{i\mu}^{1}\oplus E_{i\mu}^{G}$, where
$E_{i\mu}^{1}\subset\ker(JL-i\mu)$, $\dim E_{i\mu}^{G}<\infty$, and $\langle
L\cdot,\cdot\rangle$ is non-degenerate on both $E_{i\mu}^{1}$ and $E_{i\mu
}^{G}$. Let $E_{i\mu}^{1,-}\subset E_{i\mu}^{1}$ be a subspace such that $\dim
E_{i\mu}^{1,-}=n^{-}(L|_{E_{i\mu}^{-}})$ and $\langle L\cdot,\cdot\rangle$ is
negative definite on $E_{i\mu}^{1,-}$. Let
\[
E_{i\mu}^{finite}=E_{i\mu}^{G}\oplus E_{i\mu}^{1,-},\text{ which satisfies
}\dim E_{i\mu}^{finite}<\infty,\;n^{-}(L|_{E_{i\mu}^{finite}})=k^{\leq0}%
(i\mu).
\]
Moreover, $JL(E_{i\mu}^{finite})=E_{i\mu}^{finite}$ according to its construction.

If $0\notin\sigma(JL)$, let $E_{0}^{finite}=\{0\}$ and we may skip to the next
step to define $N$ and $M$. Otherwise, our assumption and Propositions
\ref{P:non-deg}, \ref{P:basis} imply an $L$-orthogonal decomposition
$E_{0}=\ker L\oplus E_{0}^{1}\oplus E_{0}^{G}$, where
\[
JL(E_{0}^{1})\subset\ker L,\ \dim E_{0}^{G}<\infty,\ JL(E_{0}^{G})\subset
E_{0}^{G}\oplus\ker L,
\]
and $\langle L\cdot,\cdot\rangle$ is non-degenerate on both $E_{0}^{1}$ and
$E_{0}^{G}$. Let $E_{0}^{1,-}\subset E_{0}^{1}$ be such that $\dim E_{0}%
^{1,-}=n^{-}(L|_{E_{0}^{1}})$ and $\langle L\cdot,\cdot\rangle$ is negative
definite on $E_{0}^{1,-}$. Let $E_{0}^{finite}=E_{0}^{1,-}\oplus E_{0}^{G}$,
which satisfies
\[
\dim E_{0}^{finite}<\infty,\;n^{-}(L|_{E_{0}^{finite}})=k_{0}^{\leq0}.
\]

Let
\[
N=(\oplus_{\operatorname{Re}\lambda\neq0}E_{\lambda})\oplus(\oplus_{i\mu
\in\Lambda}E_{i\mu}^{finite})\oplus E_{0}^{finite},\quad\tilde{M}=N^{\perp
_{L}}=(N\oplus\ker L)^{\perp_{L}}.
\]
Clearly, $\dim N<\infty$, $n^{-}(L|_{N})=n^{-}(L)$ (due to
\eqref{counting-formula} of Theorem \ref{theorem-counting}), $\langle
L\cdot,\cdot\rangle$ is non-degenerate on $N$, and $N\oplus\ker L$ is
invariant under $JL$. Therefore, $\tilde{M}$ is also invariant under $JL$,
$X=N\oplus\tilde{M}$ (due to Lemma \ref{L:non-degeneracy}), $\ker
L\subset\tilde{M}$, and $n^{-}(L|_{\tilde{M}})=0$. Moreover, $N$ and $M$ are
complexifications of real subspaces as $E_{\lambda}$ and $E_{\bar{\lambda}}$
have exactly the same structure. Let $M\subset\tilde{M}$ be any closed
subspace such that $\tilde{M}=M\oplus\ker L$ and this completes the proof of
proposition. \hfill$\square$

To end this section, we prove the decomposition result Proposition
\ref{P:pego} for $L$-self-adjoint operators. \newline

\textbf{Proof of Proposition \ref{P:pego}}: In order to apply the previous
results, which have been given in the framework of real Hilbert spaces, to
prove this proposition, we first convert it into a problem on real Hilbert
spaces. Recall $(\cdot,\cdot)$ and $\langle\cdot,\cdot\rangle$ denote the
complex inner product and the complex duality pair between $X^{\ast}$ and $X$,
respectively. Let $X_{r}$ be the same set as $X$ but equipped with the real
inner product $(u,v)_{r}=\operatorname{Re}(u,v)$. On $X_{r}$, the
$i-$multiplication $i:X\rightarrow X$ becomes a real linear isometry $i_{r}:$
$X_{r}\rightarrow X_{r}$ with $i_{r}^{2}=-I$. Let $L_{r}:X_{r}\rightarrow
X_{r}^{\ast}$ be the linear symmetry bounded operator defined as $\langle
L_{r}u,v\rangle_{r}=\operatorname{Re}\langle Lu,v\rangle$ where $\langle
\cdot,\cdot\rangle_{r}$ denote the real duality pair between $X_{r}^{\ast}$
and $X_{r}$. Subsequently, the non-degeneracy of $L$ yields the non-degeneracy
of $L_{r}$. Accordingly, $A$ becomes a real linear operator $A_{r}%
:X_{r}\supset D(A_{r})\rightarrow X_{r}$. The linearity of $L$ and $A$ implies
that $i_{r}A_{r}=A_{r}i_{r}$ and $L_{r}i_{r}=-i_{r}^{\ast}L_{r}$. Finally,
that $A$ is $L$-self-adjoint is translated to the $L_{r}$-self-adjointness of
$A_{r}$, namely, $L_{r}A_{r}=A_{r}^{\ast}L_{r}$.

Define $J=i_{r}A_{r}L_{r}^{-1}:X_{r}^{\ast}\rightarrow X_{r}$. The $L_{r}%
$-self-adjointness of $A_{r}$ implies $J^{\ast}=-J$ and thus $i_{r}%
A_{r}=JL_{r}$ with $\left(  J,L_{r},X_{r}\right)  $ satisfying (\textbf{H1-3}%
). It is easy to prove
\[
\sigma(A_{r})=\sigma(A)\subset\mathbf{R},\quad\sigma(i_{r}A_{r})=\big(i\sigma
(A_{r})\big)\cup\big(-i\sigma(A_{r})\big),
\]
so the nonzero eigenvalues of $i_{r}A_{r}$ are isolated and $\ker\left(
i_{r}A_{r}\right)  =\ker A$. It is straightforward to deduce the
non-degeneracy of $\langle L_{r}\cdot,\cdot\rangle_{r}$ on $\ker(i_{r}A_{r})$
from our non-degeneracy assumption of $\langle L\cdot,\cdot\rangle$ on $\ker
A$, and thus (\textbf{H4}) is satisfied. Thus by Proposition
\ref{P:direct decomposition}, there exists a decomposition $X_{r}=\tilde
{N}\oplus\tilde{M}$ such that $\tilde{N}$ and $\tilde{M}$ are $L_{r}%
$-orthogonal and invariant under $i_{r}A_{r}$, $\dim\tilde{N}<\infty$ and
$L_{r}|_{\tilde{M}}>0$, which also implies $L$ is uniformly positive on
$\tilde{M}$. Let
\[
N=\tilde{N}+i_{r}\tilde{N},\quad M=N^{\perp_{L_{r}}}=\{u\in X_{r}\mid\langle
L_{r}u,v\rangle_{r}=0,\,\forall v\in N\}\subset\tilde{M}.
\]
Clearly, $\dim N\leq2\dim\tilde{N}<\infty$ and $\langle L_{r}\cdot
,\cdot\rangle$ is uniformly positive on $M$, thus so is $\langle L\cdot
,\cdot\rangle$ on $M$. To complete the proof, we only need to show $N,M\subset
X$ are $L$-orthogonal and invariant under $A$. We first consider the
$L$-orthogonality which also involves the imaginary part of the quadratic form
of $L$. Suppose there exist $u_{1},u_{2},\in\tilde{N}$ and $v\in M$ such that
$\langle L(u_{1}+iu_{2}),v\rangle=Re^{i\theta}\neq0$. It implies $\langle
Lu,v\rangle=\langle L_{r}u,v\rangle_{r}=R\in\mathbf{R}\backslash\{0\}$, where
\[
u=(\cos\theta)u_{1}+(\sin\theta)u_{2}+i_{r}\big((\cos\theta)u_{2}+(\sin
\theta)u_{1}\big)\in N
\]
which is a contradiction to the definitions of $N$ and $M$ and thus they are
$L$-orthogonal. Secondly, $i_{r}A_{r}(\tilde{N})\subset\tilde{N}$, $i_{r}%
A_{r}=A_{r}i_{r}$, and $i_{r}^{2}=-I$ imply $A_{r}(i_{r}\tilde{N}%
)\subset\tilde{N}$ and $A_{r}\tilde{N}\subset i_{r}\tilde{N}$. Therefore, $N$
is invariant under $A$. It along with the $L$-self-adjointness of $A$ also
implies the invariance of $M$ and the proof of Proposition \ref{P:pego} is
complete. \hfill$\square$

\section{Perturbations}

\label{S:perturbation}

In this section we study the robustness of the spectral properties of the
Hamiltonian operator $JL$ under small perturbations preserving Hamiltonian
structures. Consider
\[
u_{t}=J_{\#}L_{\#}u,\qquad J_{\#}=J+J_{1},\quad L_{\#}=L+L_{1},\quad u\in X.
\]
Unless otherwise specified, assumptions (\textbf{A1-3}) given in Subsection
\ref{SS:SS} are assumed throughout this section. We first prove

\begin{lemma}
\label{L:P-H3} Assumptions (\textbf{A1-3}) imply that there exists
$\epsilon>0$ depending on $J$ and $L$ such that, if $|L_{1}|\leq\epsilon$,
then (\textbf{H1-3}) is satisfied by $J_{\#}$ and $L_{\#}$ and
\[
\dim\ker L_{\#}\leq\dim\ker L<\infty,\quad D(J_{\#}L_{\#})=D(JL).
\]

\end{lemma}

\begin{proof}
It is obvious that (\textbf{H1}) is satisfied by $J_{\#}$. Let $X_{\pm}$ be
the subspaces provided in (\textbf{H2}) satisfied by $L$. Clearly, we still
have, for $\epsilon<<1$,
\[
\pm\langle L_{\#}u,u\rangle\geq\delta||u||^{2},\quad\forall u\in X_{\pm}%
\]
for some $\delta>0$ independent of $\epsilon$. Let $X_{1}=X_{+}\oplus X_{-}$.
Assumption (\textbf{H2}) for $L$ implies that$\langle L\cdot,\cdot\rangle$
restricted to $X_{1}$ is non-degenerate, i.e.
\[
L_{X_{1}}=i_{X_{1}}^{\ast}Li_{X_{1}}:X_{1}\rightarrow X_{1}^{\ast},
\]
defined as in \eqref{E:L_y1}, is an isomorphism. Therefore,
\[
L_{\#,X_{1}}=i_{X_{1}}^{\ast}L_{\#}i_{X_{1}}:X_{1}\rightarrow X_{1}^{\ast},
\]
as a small bounded perturbation of $L_{X_{1}}$, is also an isomorphism.
Suppose $u=u_{0}+u_{1}\in\ker L_{\#}$, where $u_{0}\in\ker L$ and $u_{1}\in
X_{1}$, then we have
\[
0=L_{\#}u=L_{\#}u_{1}+L_{1}u_{0}\ \Longrightarrow\ u_{1}=-L_{\#,X_{1}}%
^{-1}i_{X_{1}}^{\ast}L_{1}u_{0},
\]
that is,
\[
\ker L_{\#}\subset Y,\;\text{ where }Y=\graph(S)\text{ and }S=-L_{\#,X_{1}%
}^{-1}i_{X_{1}}^{\ast}L_{1}:\ker L\rightarrow X_{1}.
\]
Moreover, from the \eqref{E:L_Y2} type identity, it also holds that, for any
$v\in\ker L$ and $u_{1}\in X_{1}$,
\begin{align*}
\langle L_{\#}(v+Sv),u_{1}\rangle= &  \langle L_{1}v,u_{1}\rangle-\langle
L_{\#}L_{\#,X_{1}}^{-1}i_{X_{1}}^{\ast}L_{1}v,u_{1}\rangle\\
= &  \langle L_{1}v,u_{1}\rangle-\langle i_{X_{1}}^{\ast}L_{1}v,u_{1}%
\rangle=0,
\end{align*}
that is, $Y$ and $X_{1}$ are $L_{\#}$-orthogonal.

Since $\dim Y=\dim\ker L<\infty$ due to (\textbf{A2}), the quadratic form
$\langle L_{\#}\cdot,\cdot\rangle$ restricted to $Y$ leads to a decomposition
of $Y$
\[
Y=Y_{+}\oplus\ker L_{\#}\oplus Y_{-},
\]
where $\pm L_{\#}$ is positive on $Y_{\pm}$. Let $X_{\#\pm}=X_{\pm}\oplus
Y_{\pm}$, then
\[
X=X_{\#+}\oplus\ker L_{\#}\oplus X_{\#-}.
\]
Due to the $L_{\#}$-orthogonality between $Y$ and $X_{1}$, it is easy to
derive that $\pm\langle L_{\#}\cdot,\cdot\rangle$ are positive definite on
$X_{\#\pm}$. Therefore, (\textbf{H2}) is satisfied.

Finally we prove (\textbf{H3}). Suppose $\gamma\in X^{*}$ and $\langle\gamma,
u \rangle=0$ for all $u \in X_{\#+} \oplus X_{\#-} \supset X_{+} \oplus X_{-}%
$. From (\textbf{A1}) which requires that (\textbf{H1-3}) being satisfied by
$J$ and $L$, we have $\gamma\in D(J) = D(J_{\#})$ as $J_{1}$ is assumed to be bounded.
\end{proof}

Much as in Remark \ref{R:assumptions} by composing with the Riesz
representation, we may treat $L_{\#}$ as a bounded symmetric operator on $X$
and then apply its spectral decomposition, a decomposition satisfying
(\textbf{H2}) can be obtained much more easily. However, that decomposition
may not satisfy (\textbf{H3}).

In Subsection \ref{SS:PET}, we will obtain the persistence of exponential
trichotomy of the perturbed system. In Subsection \ref{SS:ImSpec}, we will
focus on purely imaginary spectral points of $\sigma(JL)$ and the possibility
of bifurcation of unstable eigenvalues of $J_{\#} L_{\#}$. To start, following
the standard procedure we show that assumption (\textbf{A3}) implies the
convergence of the resolvents. Recall $\vert|\cdot\vert|_{G}$ denote the graph
norm on $D(JL)$ and $|\cdot|_{G}$ the corresponding operator norm.

\begin{lemma}
\label{L:resolvent} Let $K \subset\mathbf{C}\backslash\sigma(JL)$ be compact,
then there exist $C, \epsilon>0$ depending on $K$, $J$, and $L$, such that,
for any $\lambda\in K$ and
\[
|J_{1}|, |JL_{1}|_{G} \le\epsilon, \quad|L_{1}| \le1,
\]
it holds that the densely defined closed operator $\lambda- J_{\#} L_{\#}:
D(JL) \to X$ has a bounded inverse and
\[
| (\lambda- J_{\#} L_{\#})^{-1} - (\lambda- JL)^{-1} | \le C (|J_{1}|+ |J
L_{1}|_{G} ).
\]

\end{lemma}

\begin{proof}
It is straightforward to compute
\[
\lambda-J_{\#}L_{\#}=\big(I-(JL_{1}+J_{1}L_{\#})(\lambda-JL)^{-1}%
\big)(\lambda-JL).
\]
According to assumption (\textbf{A3}), $JL_{1}(\lambda-JL)^{-1}$ is a closed
operator with the domain $X$. The closed graph theorem implies that it is
actually bounded with
\begin{align*}
|JL_{1}(\lambda-JL)^{-1}|\leq &  |JL_{1}|_{G}\big(|(\lambda-JL)^{-1}%
|+|JL(\lambda-JL)^{-1}|\big)\\
\leq &  |JL_{1}|_{G}\big(1+(1+|\lambda|)|(\lambda-JL)^{-1}|\big),
\end{align*}
where $JL=\lambda-(\lambda-JL)$ was used in the last step. The conclusion of
the lemma follows from this along with the boundedness of $J_{1}$, $L$, and
$L_{1}$.
\end{proof}

\subsection{Persistent exponential trichotomy and stability: Theorem
\ref{T:PET} and Proposition \ref{P:PSubS}}

\label{SS:PET}

In this subsection, our main task is to prove Theorem \ref{T:PET} as well as
Proposition \ref{P:PSubS}. With the help of Lemmas \ref{L:resolvent} and
\ref{L:L-orth-eS}, we are able to prove most of Theorem \ref{T:PET} and
Proposition \ref{P:PSubS} by standard arguments in the spectral theory.
However, proving \eqref{E:P-center} requires more elaborated arguments as one
of the perturbation term $JL_{1}$ is not necessarily a small bounded operator.\\

\noindent\textbf{Proof of Theorem \ref{T:PET} except \eqref{E:P-center}.}
Adopt the notation used in \eqref{E:ep}
\begin{equation}
|J_{1}|+|L_{1}|+|JL_{1}|_{G} \le \ep. \label{E:ep1}
\end{equation}
Let
\[
\sigma_{u,s}=\{\lambda\in\sigma(JL)\mid\pm\operatorname{Re}\lambda
>0\}\subset\sigma(JL)
\]
and $\Omega_{u}\subset\mathbf{C}$ be open and bounded with smooth boundary
$\Gamma_{u}=\partial\Omega_{u}\subset\mathbf{C}\backslash\sigma(JL)$ such that
$\sigma(JL)\cap\Omega_{u}=\sigma_{u}$. According to Lemma \ref{L:symmetry},
$\sigma_{s}$ is symmetric to $\sigma_{u}$ about $i\mathbf{R}$ and thus we let
$\Omega_{s}$ be the domain symmetric to $\Omega_{u}$ and $\Gamma_{s}%
=\partial\Omega_{s}$. For small $\epsilon$, Lemma \ref{L:resolvent} allows us
to define the following objects via standard contour integrals
\begin{align*}
&  \tilde{P}_{\#}^{u,s}=\frac{1}{2\pi i}\oint_{\Gamma_{u,s}}(z-J_{\#}%
L_{\#})^{-1}dz,\quad E_{\#}^{u,s}=\tilde{P}_{\#}^{u,s}X,\\
&  A_{\#}^{u,s}=(J_{\#}L_{\#})|_{E_{\#}^{u,s}}=\frac{1}{2\pi i}\oint%
_{\Gamma_{u,s}}z(z-J_{\#}L_{\#})^{-1}dz.
\end{align*}
Let
\[
\tilde{P}_{\#}^{c}=I-\tilde{P}_{\#}^{u}-\tilde{P}_{\#}^{s},\quad E_{\#}%
^{c}=\tilde{P}_{\#}^{c}X,\quad A_{\#}^{c}=(J_{\#}L_{\#})|_{E_{\#}^{c}},
\]
and $\tilde{P}^{u,s},E^{u,s,c},A^{u,s,c}$ denote the corresponding unperturbed objects.

From the standard spectral theory, subspaces $E_{\#}^{u,s,c}$ are invariant
under $J_{\#}L_{\#}$. Therefore, $A_{\#}^{u,s,c}$ are operators on
$E_{\#}^{u,s,c}$ with $\sigma(A_{\#}^{u,s})\subset\Omega_{u,s}$ and
$\sigma(A_{\#}^{c})\subset\big(\mathbf{C}\backslash(\Omega_{u}\cup\Omega
_{s})\big)$. Lemma \ref{L:resolvent} implies
\[
|\tilde{P}_{\#}^{u,s}-\tilde{P}^{u,s}|\leq C\epsilon,
\]
and thus $E_{\#}^{c}$ is $O(\epsilon)$ close to $E^{c}$, too. Along with the
non-degeneracy of $\langle L\cdot,\cdot\rangle$ on $E^{u}\oplus E^{s}$ and
$|L_{1}|\leq\epsilon$, above implies the non-degeneracy of $\langle
L_{\#}\cdot,\cdot\rangle$ on $E_{\#}^{u}\oplus E_{\#}^{s}$. Therefore, we
obtain from $X=E_{\#}^{u}\oplus E_{\#}^{s}\oplus E_{\#}^{c}$ and Lemma
\ref{L:L-orth-eS}
\[
E_{\#}^{c}=\{u\in X\mid\langle L_{\#}u,v\rangle=0,\,\forall u\in E_{\#}%
^{u}\oplus E_{\#}^{s}\}.
\]
As $O(\epsilon)$ perturbations, it is clear that subspaces $E_{\#}^{u,s,c}$
can be written as graphs of $O(\epsilon)$ bounded operators $S_{\#}^{u,s,c}$
in the coordinate frame $X=E^{u}\oplus E^{s}\oplus E^{c}$. Moreover, from the
above integral forms, $A_{\#}^{u,s}$ are only $O(\epsilon)$ bounded
perturbations to $JL$ on finite dimensional subspaces $E_{\#}^{u,s}$ which are
$O(\epsilon)$ perturbations to $E^{u,s}$, and thus inequality
\eqref{E:P-stable-unstable} follows as well. Since the subspace $E_{\#}%
^{u}\oplus E_{\#}^{s}$, invariant under $J_{\#}L_{\#}$, is finite dimensional,
the vanishness of $\langle L_{\#}\cdot,\cdot\rangle$ on $E_{\#}^{u,s}$ follows
from Lemma \ref{lemma-orthogonal-eigenspace}. Through this point we complete
the proof of parts (a) and (b), except \eqref{E:P-center}, of Theorem
\ref{T:PET}.

Suppose, as in part (c) in Theorem \ref{T:PET}, there exists $\delta>0$ such
that $\langle Lu,u\rangle\geq\delta||u||^{2}$ for all $u\in E^{c}$. Since $L$
and $L_{1}$ are bounded and $E_{\#}^{c}$ is $O(\epsilon)$ perturbation of
$E^{c}$, we have $\langle L_{\#}u,u\rangle\geq\frac{\delta}{2}||u||^{2}$ for
all $u\in E_{\#}^{c}$. Therefore, the conservation of $\langle L_{\#}%
\cdot,\cdot\rangle$ by $e^{tJ_{\#}L_{\#}}$ and the invariance of $E_{\#}^{c}$
under $e^{tJ_{\#}L_{\#}}$ imply the boundedness of $e^{tJ_{\#}L_{\#}}%
|_{E_{\#}^{c}}$ uniformly in $t\in\mathbf{R}$ which proves part (b) of Theorem
\ref{T:PET}. \hfill$\square$\newline

To complete the proof of Theorem \ref{T:PET}, we shall prove the weak
exponential growth estimate \eqref{E:P-center} in the perturbed center
subspace $E_{\#}^{c}$, which involves much more than simple applications of
the standard operator calculus and the conservation of energy. We first
consider a special case where $JL$ has no hyperbolic directions.

\begin{lemma}
\label{L:P-center-R} Assume $E^{c} = X$, then \eqref{E:P-center} holds for
some $C, \epsilon_{0}>0$ depending on $J$, $L$.
\end{lemma}

\begin{proof}
From the construction of $E_{\#}^{u,s,c}$ and the additional assumption
$E^{c}=X$, it is clear $E_{\#}^{c} =X$ and \eqref{E:P-center} is reduced to
\[
|e^{tJ_{\#} L_{\#}}| \le C e^{C\epsilon|t|}, \quad\forall t\in\mathbf{R}, \;
\text{ where } \epsilon\triangleq|J_{1}|+ |L_{1}|+ |JL_{1}|_{G}.
\]
Since
\[
J_{\#} L_{\#} = J L_{\#} + J_{1} (L + L_{1})
\]
and $J_{1}$, $L$, and $L_{1}$ are bounded with $|J_{1}|\le\epsilon$, 
we have 
\begin{equation}
\label{E:P-center-R}
|J_{\#} L_{\#} - J L_{\#}| \le C \ep.  
\end{equation}

Let $X=\oplus_{j=0}^{4}X_{j}$ be the decomposition, associated with
projections $P_{j}$, given by Theorem \ref{T:decomposition} for $J$ and $L$,
where $X_{0}=\ker L$ and $X_{5}=X_{6}=\{0\}$ due to the assumption $E^{c}=X$.
Much as in \eqref{E:L_y1}, let
\[
L_{1,jk}=i_{j}^{\ast}L_{1}i_{k}:X_{k}\rightarrow X_{j}^{\ast},\quad
j,k=0,\ldots,4,
\]
which satisfy $L_{1,jk}=L_{1,kj}^{\ast}$ and
\[
L_{1}=\Sigma_{j,k=0}^{4}P_{j}^{\ast}L_{1,jk}P_{k},\quad\quad\langle
L_{1,jk}u,v\rangle=\langle L_{1}u,v\rangle,\quad\forall u\in X_{k},\;v\in
X_{j}.
\]
Let $J_{jk}=P_{j}JP_{k}^{\ast}$ be the blocks of $J$ associated to this
decomposition, which have the forms given in Corollary \ref{C:decomposition}
and satisfy
\[
J=\Sigma_{j,k=0}^{4}i_{X_{j}}J_{jk}i_{X_{k}}^{\ast},\quad|J-i_{X_{3}}%
J_{33}i_{X_{3}}^{\ast}|\leq C.
\]
We write
\begin{align*}
JL_{\#}= &  JL+JL_{1}P_{3}+\Sigma_{k\in\{0,1,2,4\}}JL_{1}P_{k}\\
= &  JL+(J-i_{X_{3}}J_{33}i_{X_{3}}^{\ast})L_{1}P_{3}\\
&  \qquad+i_{X_{3}}J_{33}i_{X_{3}}^{\ast}\Sigma_{j,k=0}^{4}P_{j}^{\ast
}L_{1,jk}P_{k}P_{3}+\Sigma_{k\in\{0,1,2,4\}}JL_{1}P_{k}\\
= &  JL+i_{X_{3}}J_{33}L_{1,33}P_{3}+(J-i_{X_{3}}J_{33}i_{X_{3}}^{\ast}%
)L_{1}P_{3}+\Sigma_{k\in\{0,1,2,4\}}JL_{1}P_{k},
\end{align*}
where $P_{X_{j}}i_{X_{k}}=\delta_{jk}I_{X_{k}}$ is used. Since, for $k\neq3$,
$X_{k}\subset D(JL)\subset D(JL_{1})$, we have that $JL_{1}P_{k}$ is a bounded
operator with the norm bounded in terms of $|JL_{1}|_{G}$, $|P_{k}|$, and
$|JLP_{k}|$. Along with the boundedness of $J-i_{X_{3}}J_{33}i_{X_{3}}^{\ast
},$ we obtain
\begin{equation}
|JL_{\#}-(JL+i_{X_{3}}J_{33}L_{1,33}P_{3})|<C\epsilon.\label{E:P-center-1}%
\end{equation}
From Theorem \ref{T:decomposition} we have
\[
JL+i_{X_{3}}J_{33}L_{1,33}P_{3}\longleftrightarrow%
\begin{pmatrix}
0 & A_{01} & A_{02} & A_{03} & A_{04}\\
0 & A_{1} & A_{12} & A_{13} & A_{14}\\
0 & 0 & A_{2} & 0 & A_{24}\\
0 & 0 & 0 & A_{3}+J_{33}L_{1,33} & A_{34}\\
0 & 0 & 0 & 0 & A_{4}%
\end{pmatrix}
.
\]
Note all blocks of $JL+i_{X_{3}}J_{33}L_{1,33}P_{3}$ are identical to those of
$JL$ except its $(4,4)$-block
\[
A_{3}+J_{33}L_{1,33}=J_{33}(L_{X_{3}}+L_{1,33}).
\]
Since $\langle L_{X_{3}}\cdot,\cdot\rangle$ is uniformly positive on $X_{3}$,
so is $\langle(L_{X_{3}}+L_{1,33})\cdot,\cdot\rangle$. Therefore, the group
$e^{t(A_{3}+J_{33}L_{1,33})}$, conserving $\langle(L_{X_{3}}+L_{1,33}%
)\cdot,\cdot\rangle$, satisfies
\[
|e^{t(A_{3}+J_{33}L_{1,33})}|\leq C,\quad\forall t\in\mathbf{R}.
\]
By using the upper triangular form of $JL$, this inequality, assumption
$E^{c}=X$ that $JL$ has no hyperbolic eigenvalues, and the finite
dimensionality $\dim X_{1}=\dim X_{4}=n^{-}(L)-\dim X_{2}$, it is easy to
prove
\[
|e^{t(JL+i_{X_{3}}J_{33}L_{1,33}P_{3})}|\leq C(1+|t|^{1+2n^{-}(L)}),\quad\forall t\in\mathbf{R}.
\]
Along with \eqref{E:P-center-R}, (\ref{E:P-center-1}), and the above estimate, 
from the following lemma we obtain \eqref{E:P-center} assuming $E^{c}=X$.
\end{proof}

The following lemma follows from standard argument and we include a sketch of the proof for the sake of completeness. 

\begin{lemma} \label{L:pert-tech-1}
Let $\BFX$ be a Banach space space, $\omega\in \BFR$, $C_0, k\ge 0$, and $\BFA: D(\BFA) \to X$ the generator of a $C^0$ semigroup on $\BFX$, such that 
\[
|e^{t\BFA}|\le C_0 (1+t^k) e^{\omega t}, \quad \forall t>0.
\]
Suppose $\BFA_1 \in L(\BFX)$ and $|\BFA_1| \le \ep\in (0, 1]$, then there exists $C>0$ depending only on $C_0$ and $k$ such that 
\[
|e^{t(\BFA+ \BFA_1)}|\le C \ep^{-\frac k{k+1}} e^{(\omega + C \ep^{\frac 1{k+1}}) t}, \quad \forall t>0.
\] 
\end{lemma}

\begin{proof} 
Without loss of generality, we may assume $\omega =0$ (otherwise $\BFA$ can be replaced by $\BFA-\omega$). Since $t^k e^{-t}$ is bounded for $t>0$, there exists $C>0$ depending only on $k$ such that $(1+t^k) \le C \ep^{-\frac k{k+1}} e^{\ep^{\frac 1{k+1}} t}$ for all $t>0$. From the variation of parameter formula we have  
\begin{align*}
\big|e^{t(\BFA+ \BFA_1)}\big| = & \Big| e^{t\BFA} + \int_0^t e^{(t-\tau) \BFA} \BFA_1e^{\tau (\BFA+ \BFA_1)} d\tau\Big| \\
\le & C \ep^{-\frac k{k+1}} e^{\ep^{\frac 1{k+1}} t} + C \ep^{\frac 1{k+1}} \int_0^t e^{\ep^{\frac 1{k+1}} (t-\tau)}\big|e^{\tau(\BFA+ \BFA_1)}\big| d\tau,    
\end{align*}
and the desired estimate follows from the Gronwall inequality. 
\end{proof}

By using the invariance of $E_{\#}^{c}$ under $J_{\#}L_{\#}$, in the following
we convert $e^{tJ_{\#}L_{\#}}|_{E_{\#}^{c}}$ to a flow $e^{t\tilde{J}%
_{\#}\tilde{L}_{\#}}$ on $E^{c}$ via a similarity transformation and then
apply Lemma \ref{L:P-center-R} to obtain \eqref{E:P-center}. \newline

\noindent\textbf{Proof of \eqref{E:P-center} in Theorem \ref{T:PET}.} In the
general case, let $E_{\#}^{u,s,c}$ be the invariant unstable/stable/center
subspaces and $P_{\#}^{u,s,c}$ be the projections associated to the
decomposition $X=E_{\#}^{u}\oplus E_{\#}^{s}\oplus E_{\#}^{c}$. We also adopt
the notations $E_{\#}^{us}=E_{\#}^{u}\oplus E_{\#}^{s}$ and $P_{\#}%
^{us}=P_{\#}^{u}+P_{\#}^{s}$. Correspondingly, let $E^{u,s,c,us}$ and
$P^{u,s,c,us}$ denoted the unperturbed invariant subspaces and projections.
Recall $E_{\#}^{c}$ can be written as the graph of a bounded operator
$S_{\#}^{c}:E^{c}\rightarrow E^{us}$ with $|S_{\#}^{c}|=O(\epsilon)$. Let
$\tilde{S}_{\#}^{c}=i_{E^{c}}+S_{\#}^{c}:E^{c}\rightarrow E_{\#}^{c}\subset X$
so that $E_{\#}^{c}=\tilde{S}_{\#}^{c}(X^{c})$. Clearly, $P^{c}i_{E_{\#}^{c}%
}=(\tilde{S}_{\#}^{c})^{-1}$.

Let
\[
J^{c}=P^{c}J(P^{c})^{\ast},\quad J_{\#}^{c}=P_{\#}^{c}J_{\#}(P_{\#}^{c}%
)^{\ast},\quad L_{\#}^{c}=i_{E_{\#}^{c}}^{\ast}L_{\#}i_{E_{\#}^{c}},\quad
L^{c}=i_{E^{c}}^{\ast}Li_{E^{c}}.
\]
From the invariance of $E_{\#}^{c}$ under $J_{\#}L_{\#}$, the $L_{\#}%
$-orthogonality between $E_{\#}^{c}$ and $E_{\#}^{us}$, and Lemma
\ref{L:decomJ} applied to the decomposition $X=E_{\#}^{c}\oplus E_{\#}^{us}$,
we have
\[
J_{\#}^{c}L_{\#}^{c}=J_{\#}L_{\#}|_{E_{\#}^{c}},\quad e^{t(J_{\#}L_{\#}%
)}|_{E_{\#}^{c}}=e^{tJ_{\#}^{c}L_{\#}^{c}},
\]
and the combination $(E_{\#}^{c},J_{\#}^{c},L_{\#}^{c})$ satisfies ({H1-3}).
Using the mapping $\tilde{S}_{\#}^{c}$, we may just consider its conjugate
flow on $E^{c}$
\[
P^{c}i_{E_{\#}^{c}}e^{tJ_{\#}^{c}L_{\#}^{c}}\tilde{S}_{\#}^{c},\;\text{ with
the generator }P^{c}i_{E_{\#}^{c}}J_{\#}^{c}L_{\#}^{c}\tilde{S}_{\#}%
^{c}\;\text{ on }E^{c}.
\]
Let
\[
\tilde{L}_{\#}=(\tilde{S}_{\#}^{c})^{\ast}L_{\#}^{c}\tilde{S}_{\#}^{c}%
=(\tilde{S}_{\#}^{c})^{\ast}i_{E_{\#}^{c}}^{\ast}L_{\#}i_{E_{\#}^{c}}\tilde
{S}_{\#}^{c}:E^{c}\rightarrow(E^{c})^{\ast},
\]
and
\[
\tilde{J}_{\#}=P^{c}i_{E_{\#}^{c}}J_{\#}^{c}(P^{c}i_{E_{\#}^{c}})^{\ast}%
=P^{c}P_{\#}^{c}J_{\#}(P^{c}P_{\#}^{c})^{\ast}:(E^{c})^{\ast}\supset
D(\tilde{J}_{\#})\rightarrow E^{c}.
\]
Clearly,
\begin{equation}
\tilde{J}_{\#}\tilde{L}_{\#}=P^{c}i_{E_{\#}^{c}}J_{\#}^{c}L_{\#}^{c}\tilde
{S}_{\#}^{c}=P^{c}i_{E_{\#}^{c}}J_{\#}L_{\#}\tilde{S}_{\#}^{c}%
.\label{E:P-center-0.5}%
\end{equation}
Since $|P^{c}i_{E_{\#}^{c}}|,|\tilde{S}_{\#}^{c}|\leq2$, in order to prove
\eqref{E:P-center}, it suffices to prove on $E^{c}$
\begin{equation}
|e^{\tilde{J}_{\#}\tilde{L}_{\#}}|\leq 
C\epsilon^{\frac 1{2(1+ n^-(L) - \dim E^u)}-1} e^{C\epsilon^{\frac 1{2(1+ n^-(L) - \dim E^u)}}|t|},
\quad \forall t\in\mathbf{R}\label{E:P-center-1'}%
\end{equation}
for some $C$ depending only on $J$ and $L$. Our strategy is to verify that
$(E^{c},\tilde{J}_{\#},\tilde{L}_{\#})$ as a perturbation to $(E^{c}%
,J^{c},L^{c})$ satisfies (\textbf{A1-3}) and then apply Lemma
\ref{L:P-center-R}.

When $\epsilon=0$, Lemma \ref{L:decomJ} ensures that the unperturbed
\[
(E^{c},\ \tilde{J}_{\#}=J^{c},\ \tilde{L}_{\#}=L^{c}=i_{E^{c}}^{\ast}%
Li_{E^{c}})
\]
satisfies (\textbf{H1-3}). Moreover, since $\langle L\cdot,\cdot\rangle$ is
non-degenerate on $E^{us}$, we have
\[
\dim\ker L^{c}=\dim\ker(i_{E^{c}}^{\ast}Li_{E^{c}})=\dim\ker L<\infty
\]
due to the $L$-orthogonality between $E^{c}$ and $E^{us}$ and thus
(\textbf{A2}) is satisfied by $\tilde{L}_{\#}$ for $\epsilon=0$. From the
definitions, $\tilde{J}_{\#}-J^{c}$ is clearly anti-symmetric. We will show
that it is also bounded. Using the fact $I-P_{\#}^{c}=P_{\#}^{us}$, one may
compute
\begin{align*}
\tilde{J}_{\#}-J^{c}= &  -P^{c}P_{\#}^{us}J_{\#}(P^{c}P_{\#}^{c})^{\ast}%
-P^{c}J_{\#}(P^{c}P_{\#}^{us})^{\ast}+P^{c}J_{1}(P^{c})^{\ast}\\
= &  -P^{c}P_{\#}^{us}P_{\#}^{us}J_{\#}(P^{c}P_{\#}^{c})^{\ast}-P^{c}%
J_{\#}(P_{\#}^{us})^{\ast}(P^{c}P_{\#}^{us})^{\ast}+P^{c}J_{1}(P^{c})^{\ast}.
\end{align*}
Due to the $L_{\#}$-orthogonality between $E_{\#}^{us}$ and $E_{\#}^{c}$ and
the non-degeneracy of $\langle L_{\#}\cdot,\cdot\rangle$ on $E_{\#}^{us}$, it
is straightforward to obtain that $L_{\#}$ is an isomorphism from $E_{\#}%
^{us}$ to $R\big((P_{\#}^{us})^{\ast}\big)=(P_{\#}^{us})^{\ast}(E_{\#}%
^{us})^{\ast}$. Since $E_{\#}^{us}\subset D(J_{\#}L_{\#})$, we have
$R\big((P_{\#}^{us})^{\ast}\big)\subset D(J_{\#})$ and thus $J_{\#}%
|_{R\big((P_{\#}^{us})^{\ast}\big)}$ is a bounded operator. To estimate its
norm, we use the relationship
\[
J_{\#}|_{R\big((P_{\#}^{us})^{\ast}\big)}=(J_{\#}L_{\#}|_{E_{\#}^{us}}%
)(L_{\#}|_{E_{\#}^{us}})^{-1}=(A_{\#}^{u}\oplus A_{\#}^{s})(L_{\#}%
|_{E_{\#}^{us}})^{-1}.
\]
Recall $E_{\#}^{us}$ is $O(\epsilon)$ perturbation to $E^{us}$ and $L_{\#}$ is
$O(\epsilon)$ to $L$. Moreover, the spectral integral representations of
$A_{\#}^{u,s}$ yield that they are $O(\epsilon)$ perturbation to $JL|_{E^{us}%
}$. Therefore, we obtain that
\[
|J_{\#}|_{R\big((P_{\#}^{us})^{\ast}\big)}|\leq C\Longrightarrow|P_{\#}%
^{us}J_{\#}|=|J_{\#}(P_{\#}^{us})^{\ast}|\leq C
\]
for some $C>0$ depending on $J$ and $L$. Since $|P^{c}P_{\#}^{us}|\leq
C\epsilon$, we have
\[
|\tilde{J}_{1,\#}|\leq C\epsilon,\;\text{ where }\tilde{J}_{1,\#}%
\triangleq\tilde{J}_{\#}-J^{c}.
\]
From the definition of $\tilde{L}_{\#}$, it is easy to obtain
\[
\tilde{L}_{1,\#}=\tilde{L}_{1,\#}^{\ast},\quad|\tilde{L}_{1,\#}|\leq
C\epsilon,\;\text{ where }\tilde{L}_{1,\#}\triangleq\tilde{L}_{\#}-L^{c}.
\]
Therefore, we finish verifying (\textbf{A1}) for $(E^{c},\tilde{J}_{\#}%
,\tilde{L}_{\#})$.

We proceed to verify (\textbf{A3}). From Lemma \ref{L:resolvent}, we have
$D(J_{\#}L_{\#})=D(JL)$. Since $E_{\#}^{c}=\tilde{S}_{\#}^{c}(E^{c})$ is the
graph of $S_{\#}^{c}:E^{c}\rightarrow E^{us}$ and $E^{us}\subset
D(JL)=D(J_{\#}L_{\#})$, we obtain
\[
D(J_{\#}L_{\#})\cap E_{\#}^{c}=\tilde{S}_{\#}^{c}\big(E^{c}\cap D(JL)\big).
\]
From the boundedness of $\tilde{J}_{1,\#}$ and \eqref{E:P-center-0.5}, we
further obtain
\[
D(J^{c}\tilde{L}_{\#})=D(\tilde{J}_{\#}\tilde{L}_{\#})=E^{c}\cap
D(JL)=D(J^{c}L^{c})
\]
which along with $\tilde{L}_{\#}=L^{c}+\tilde{L}_{1,\#}$ obviously implies
$D(J^{c}L^{c})\subset D(J^{c}\tilde{L}_{1,\#})$.

In the next we estimate the graph norm of $J^{c}\tilde{L}_{1,\#}$, like the
one defined in \eqref{E:graph-norm}, on the domain $D(J^{c}\tilde{L}%
_{\#})=E^{c}\cap D(JL)$. From \eqref{E:P-center-0.5} one may compute that,
when restricted on $E^{c}\cap D(JL)$,
\begin{equation}
\tilde{J}_{\#}\tilde{L}_{\#}-J^{c}L^{c}=-P^{us}i_{E_{\#}^{c}}J_{\#}%
L_{\#}\tilde{S}_{\#}^{c}+J_{\#}L_{\#}S_{\#}^{c}+J_{\#}L_{\#}-JL.
\label{E:P-center-2}%
\end{equation}
We shall use
\begin{equation}
J_{\#}L_{\#}=JL+JL_{1}+J_{1}L_{\#} \label{E:P-center-3}%
\end{equation}
to estimate the three terms in \eqref{E:P-center-2}. In fact, for any $v\in
D(JL)$,
\[
||J_{\#}L_{\#}v-JLv||\leq|JL_{1}|_{G}||v||_{G}+|J_{1}L_{\#}|||v||\leq
C\epsilon||v||_{G}%
\]
for some $C>0$ depending on $J$ and $L$. A combination of this inequality with
\eqref{E:P-center-2} and the fact $|P^{us}i_{E_{\#}^{c}}|\leq C\epsilon$
implies, for any $u\in E^{c}\cap D(JL)$,
\[
||(\tilde{J}_{\#}\tilde{L}_{\#}-J^{c}L^{c})u||\leq C\big(\epsilon||\tilde
{S}_{\#}^{c}u||_{G}+||S_{\#}^{c}u||_{G}+\epsilon||u||_{G}\big)\leq
C\epsilon||u||_{G},
\]
where we used the fact that $JL$ is bounded on $E^{us}\subset D(JL)$. Since
\[
J^{c}\tilde{L}_{1,\#}=\tilde{J}_{\#}\tilde{L}_{\#}-J^{c}L^{c}-\tilde{J}%
_{1,\#}\tilde{L}_{\#}%
\]
with $|\tilde{J}_{1,\#}|\leq C\epsilon$, the above inequality implies
\[
||J^{c}\tilde{L}_{1,\#}u||\leq C\epsilon||u||_{G},\quad\forall u\in E^{c}\cap
D(JL)
\]
for some $C>0$ depending on $J$ and $L$. The above estimates allow us to apply
Lemma \ref{L:P-center-R} to obtain (\ref{E:P-center-1'}) for $e^{t\tilde
{J}_{\#}\tilde{L}_{\#}}\ $ on $E^{c}$, which in turn implies
\eqref{E:P-center} for $e^{tJ_{\#}L_{\#}}$ on $E_{\#}^{c}$. \hfill$\square$

To complete this subsection we present \newline

\noindent\textbf{Proof of Proposition \ref{P:PSubS}.} Adopt the notations used
in \eqref{E:ep}-\eqref{E:ep1}, and let
\[
\epsilon\triangleq|J_{1}|+|L_{1}|+|JL_{1}|_{G}.
\]
Let $\Omega\subset\mathbf{C}$ be an open domain with the compact closure and
smooth boundary $\Gamma\subset\mathbf{C}\backslash i\mathbf{R}$ such that
$\Omega\cap\sigma(JL)=\sigma_{2}$. For small $\epsilon$, Lemma
\ref{L:resolvent} allows us to define the following objects via standard
contour integrals
\begin{align*}
&  P_{\#}=\frac{1}{2\pi i}\oint_{\Gamma}(z-J_{\#}L_{\#})^{-1}dz,\quad
X_{2}^{\#}=P_{\#}X\subset D(J_{\#}L_{\#}),\quad X_{1}^{\#}=(I-P_{\#})X\\
&  A_{1,2}^{\#}=(J_{\#}L_{\#})|_{X_{1,2}}=\frac{1}{2\pi i}\oint_{\Gamma
}z(z-J_{\#}L_{\#})^{-1}dz.
\end{align*}
Let $P$, $X_{1,2}$, and $A_{1,2}$ denote the corresponding unperturbed objects.

From the standard spectral theory, the decomposition $X = X_{1}^{\#} \oplus
X_{2}^{\#}$ is invariant under $J_{\#} L_{\#}$ and thus $A_{1,2}^{\#}$ are
operators on $X_{1,2}^{\#}$ with $\sigma(A_{2}^{\#}) \subset\Omega$. (In fact
$\sigma(A_{1,2}) = \sigma_{1,2}$.) Since $\sigma(J_{\#}L_{\#}) = \sigma
(A_{1}^{\#}) \cup\sigma(A_{2}^{\#})$, we only need to prove that $\sigma
(A_{1}^{\#}) \subset i\mathbf{R}$.

Since the decomposition $X=X_{1}\oplus X_{2}$ is $L$-orthogonal and
$X_{2}\subset D(JL)$, Lemma \ref{L:decomJ} implies that $(X_{1},J_{X_{1}%
}=(I-P)J(I-P)^{\ast},L_{X_{1}})$ satisfies (\textbf{H1-3}). Therefore, the
index theorem Theorem \ref{theorem-counting} applies to $L_{X_{1}}$ and
$JL|_{X_{1}}$ which along with the second assumption of Proposition
\ref{P:PSubS} implies that $n^{-}(L_{X_{1}})=0$. As $L_{X_{1}}$ satisfies
(\textbf{H2}), we obtain that $L_{X_{1}}$ is positive definite. Lemma
\ref{L:resolvent} implies
\[
|P_{\#}-P|\leq C\epsilon.
\]
and thus $X_{1}^{\#}$ is $O(\epsilon)$ close to $X_{1}$. Namely $X_{1}^{\#}$
can be written as the graph of an $O(\epsilon)$ order bounded operator
$S_{\#}:X_{1}\rightarrow X_{2}$. It immediately implies that $L_{X_{1}^{\#}}$
is uniformly positive on $X_{1}^{\#}$ and the proposition follows from the
invariance of $X_{1}^{\#}$ under $J_{\#}L_{\#}$. \hfill$\square$

\subsection{Perturbations of purely imaginary spectrum and bifurcation to
unstable eigenvalues}

\label{SS:ImSpec}

In this subsection, we consider $\sigma(J_{\#}L_{\#})$ near some $i\mu
\in\sigma(JL)\cap i\mathbf{R}$ and prove Theorems \ref{T:SImSpec} and
\ref{T:USImSpec}. \newline

\noindent\textbf{`Structurally stable' cases}. We still adopt the notation
used in \eqref{E:ep} and let
\begin{equation}
\epsilon\triangleq|J_{1}|+|L_{1}|+|JL_{1}|_{G}. \label{E:ep2}%
\end{equation}

\noindent\textbf{Case 1: $i\mu\in\sigma(JL)\cap i \mathbf{R}$ is isolated with
$\langle L \cdot, \cdot\rangle$ sign definite on $E_{i\mu}$.}

Suppose $\delta>0$ and $\langle Lu,u\rangle\geq\delta||u||^{2}$, for all $u\in
E_{i\mu}$ (the opposite case $\langle L\cdot,\cdot\rangle\leq-\delta<0$ on
$E_{i\mu}$ is similar). Since $i\mu$ is assumed to be isolated in $\sigma
(JL)$, there exists $\alpha>0$ such that the closed disk $\overline
{B(i\mu,\alpha)}\cap\sigma(JL)=\{i\mu\}$. Let $\Gamma=\partial B(i\mu,\alpha)$
and $\Gamma\cap\sigma(J_{\#}L_{\#})=\emptyset$ for small $\epsilon$ due to
Lemma \ref{L:resolvent}. Define
\[
\tilde{P}_{\#}=\frac{1}{2\pi i}\oint_{\Gamma}(z-J_{\#}L_{\#})^{-1}%
dz,\quad\tilde{E}_{\#}=\tilde{P}_{\#}X.
\]
From the standard spectral theory, $\tilde{E}_{\#}$ is invariant under
$J_{\#}L_{\#}$ and
\[
\sigma(J_{\#}L_{\#})\cap\overline{B(i\mu,\alpha)}=\sigma(J_{\#}L_{\#}%
|_{\tilde{E}_{\#}}).
\]
Lemma \ref{L:resolvent}, the isolation of $i\mu$, and Proposition
\ref{P:non-deg} imply that $\tilde{E}_{\#}$ is $O(\epsilon)$ close to
$E_{i\mu}$. The positive definiteness assumption of $\langle L\cdot
,\cdot\rangle$ on $E_{i\mu}$ and the boundedness of $L$ and $L_{1}$ imply that
$\langle L_{\#}\cdot,\cdot\rangle$ is also positive definite on $\tilde
{E}_{\#}$. The stability -- both forward and backward in time -- of
$e^{tJ_{\#}L_{\#}}$ on $\tilde{E}_{\#}$, due to the conservation of energy,
implies
\[
\sigma(J_{\#}L_{\#})\cap\overline{B(i\mu,\alpha)}=\sigma(J_{\#}L_{\#}%
|_{\tilde{E}_{\#}})\subset i\mathbf{R}.
\]

\noindent\textbf{Case 2: $i\mu\in\sigma(JL)\cap i\mathbf{R}$ and $\langle
L\cdot,\cdot\rangle$ is positive definite on $E_{i\mu}$.}

Then
\begin{equation}
E_{i\mu}=\{0\}\text{ or }\langle Lu,u\rangle\geq\delta||u||^{2},\;\forall u\in
E_{i\mu}, \label{E:ImSpec1}%
\end{equation}
for some $\delta>0$.

\begin{remark}
In this case, besides the possibility of an isolated eigenvalue $i\mu\in
\sigma(JL)$ with $L$ positive definite on $E_{i\mu}$, we are mainly concerned
with the scenario that $i\mu$ is embedded in the continuous spectrum, whether
an eigenvalue or not, but without any eigenvector in a non-positive direction
of $L$. Our conclusion is that, under small perturbations, no hyperbolic
eigenvalues (i.e. away from imaginary axis) may bifurcate from $i\mu$.
\end{remark}

We argue by contradiction for Case 2. Suppose Theorem \ref{T:SImSpec} does not
hold in this case, then there exist a sequence
\[
J_{\#n}=J+\tilde{J}_{n},\;L_{\#n}=L+\tilde{L}_{n},\quad n=1,2,\ldots,
\]
satisfying (\textbf{A1-3}) for each $n$ such that
\[
\exists\lambda_{n}\in\sigma(J_{\#n}L_{\#n})\backslash i\mathbf{R}%
;\;\;\epsilon_{n}\triangleq|\tilde{J}_{n}|+|\tilde{L}_{n}|+|J\tilde{L}%
_{n}|_{G}\rightarrow0;\;\;\delta_{n}\triangleq\left\vert \lambda_{n}%
-i\mu\right\vert \rightarrow0.
\]
Since not in $i\mathbf{R}$, $\lambda_{n}$ must be eigenvalues. Let
\[
u_{n}\in X,\quad J_{\#n}L_{\#n}u_{n}=\lambda_{n}u_{n},\quad||u_{n}||=1.
\]

Using the graph norm of $JL_{1}$, one may estimate
\[
||J\tilde{L}_{n}u_{n}||\leq|J\tilde{L}_{n}|_{G}(1+||JLu_{n}||)\leq|J\tilde
{L}_{n}|_{G}\big(1+|\lambda_{n}|+||JLu_{n}-\lambda_{n}u_{n}||\big)
\]
and
\[
||JLu_{n}-\lambda_{n}u_{n}||=||JLu_{n}-J_{\#n}L_{\#n}u_{n}||\leq||\tilde
{J}_{n}L_{\#n}u_{n}||+||J\tilde{L}_{n}u_{n}||.
\]
Therefore, we obtain
\begin{equation}
||JLu_{n}-\lambda_{n}u_{n}||\leq C\epsilon_{n},\;\;||JLu_{n}-i\mu u_{n}%
||\leq(C\epsilon_{n}+\delta_{n})\label{E:SS-1}%
\end{equation}
for some $C>0$ depending on $|L|$ and $\mu$.

Let $X=\oplus_{j=0}^{6}$ be the decomposition given by Theorem
\ref{T:decomposition} for $(L,J)$, with $X_{0}=\ker L$, $P_{j}$ be the
associated projections, and $u_{n,j}=P_{j}u_{n}$. Let $A_{j}$ and $A_{jk}$
denote the blocks of $JL$ in this decomposition as given in Theorem
\ref{T:decomposition}. From the commutativity between $JL$ and $P_{5,6}$, we
obtain from \eqref{E:SS-1}
\[
||A_{5}u_{n,5}-i\mu u_{n,5}||+||A_{6}u_{n,6}-i\mu u_{n,6}||\leq C(\epsilon
_{n}+\delta_{n}).
\]
Since $\sigma(A_{5,6})\cap i\mathbf{R}=\emptyset$, we have
\begin{equation}
||u_{n,5}||+||u_{n,6}||\leq C(\epsilon_{n}+\delta_{n}).\label{E:SS-3}%
\end{equation}
From Lemma \ref{lemma-orthogonal-eigenspace}, we have $\langle L_{\#n}%
u_{n},u_{n}\rangle=0$. Along with \eqref{E:SS-3} this implies that
\begin{equation}
|2\langle Lu_{n,1},u_{n,4}\rangle+\langle L_{2}u_{n,2},u_{n,2}\rangle+\langle
L_{3}u_{n,3},u_{n,3}\rangle|\leq C(\epsilon_{n}+\delta_{n}).\label{E:SS-2}%
\end{equation}
Applying $P_{j}$, $j=0,\ldots,4$ to \eqref{E:SS-1} and using Theorem
\ref{T:decomposition}, we have
\begin{align}
&  ||A_{4}u_{n,4}-i\mu u_{n,4}||\leq C(\epsilon_{n}+\delta_{n});\label{E:SS2}%
\\
&  ||A_{3}u_{n,3}+A_{34}u_{n,4}-i\mu u_{n,3}||\leq C(\epsilon_{n}+\delta
_{n});\nonumber\\
&  ||A_{2}u_{n,2}+A_{24}u_{n,4}-i\mu u_{n,2}||\leq C(\epsilon_{n}+\delta
_{n});\nonumber\\
&  ||A_{1}u_{n,1}+A_{12}u_{n,2}+A_{13}u_{n,3}+A_{14}u_{n,4}-i\mu u_{n,1}||\leq
C(\epsilon_{n}+\delta_{n});\nonumber\\
&  ||A_{01}u_{n,1}+A_{02}u_{n,2}+A_{03}u_{n,3}+A_{04}u_{n,4}-i\mu
u_{n,0}||\leq C(\epsilon_{n}+\delta_{n}).\label{E:SS-4}%
\end{align}
Since $\dim X_{j}<\infty\ $when $j\neq0,3$, subject to a subsequence, we may
assume that as $n\rightarrow\infty$,
\[
u_{n,j}\rightarrow u_{j},\;j=1,2,4;\quad u_{n,5},\ u_{n,6}\rightarrow0;\quad
u_{n,j}\rightharpoonup u_{j},\;j=0,3.
\]
Passing to the limits in the above inequalities and using the boundedness of
$A_{j}$ and $A_{jk}$ except $A_{3}$, we obtain
\begin{align*}
&  A_{4}u_{4}-i\mu u_{4}=0;\quad A_{2}u_{2}+A_{24}u_{4}-i\mu u_{2}=0;\\
&  A_{1}u_{1}+A_{12}u_{2}+A_{13}u_{3}+A_{14}u_{4}-i\mu u_{1}=0;\\
&  A_{01}u_{1}+A_{02}u_{2}+A_{03}u_{3}+A_{04}u_{4}-i\mu u_{0}=0.
\end{align*}
Moreover, the above inequality involving $A_{3}u_{n,3}$ also implies that
\[
u_{n,3}\rightharpoonup u_{3},\quad A_{3}u_{n,3}\rightharpoonup-A_{34}%
u_{4}+i\mu u_{3}.
\]
Since the graph of the closed operator of $A_{3}$ as a closed subspace in
$X_{3}\times X_{3}$ is also closed under the weak topology, we obtain
\[
u_{3}\in D(A_{3})\;\text{ and }\;A_{3}u_{3}+A_{34}u_{4}-i\mu u_{3}=0.
\]
These equalities imply that
\[
JLu=i\mu u,\;\text{ where }u=u_{0}+u_{1}+u_{2}+u_{3}+u_{4}.
\]
In addition, \eqref{E:SS-2} implies
\[
\langle Lu,u\rangle=2\langle Lu_{1},u_{4}\rangle+\langle L_{2}u_{2}%
,u_{2}\rangle+\langle L_{3}u_{3},u_{3}\rangle\leq0.
\]

Due to property \eqref{E:ImSpec1} of $i\mu$, we must have $u=0$, which
immediately yields $u_{n,j}\rightarrow0$, $j=1,2,4,5,6$ and thus
\eqref{E:SS-2} implies $u_{n,3}\rightarrow0$ as well. Then the normalization
$||u_{n}||=1$ implies that we must have $\dim\ker L\geq1$, $||u_{n,0}%
||\rightarrow1$ and $u_{n,0}\rightharpoonup0$. From \eqref{E:SS-4}, we obtain
$\mu=0$. As $\ker L$ is nontrivial, this again contradicts to
\eqref{E:ImSpec1}. Therefore, Theorem \ref{T:SImSpec} holds in this case.

Summarizing the above two cases, \textbf{Theorem \ref{T:SImSpec} is proved.}
\newline

\noindent\textbf{`Structurally unstable' cases.} In the following, we will
consider cases in Theorem \ref{T:USImSpec} for the structural instability. In
many applications the symplectic structure $J$ usually does not vary,
therefore we will fix $J$ and focus on constructing perturbations to the
energy operator $L$ to induce instabilities arising from a purely imaginary
eigenvalue $i\mu$ of $JL$. Recall we have to complexify $X$, $J$, and $L$
accordingly. However, keep in mind that we would like to construct real
perturbations to create unstable eigenvalues near $i\mu$. This would require
the perturbations to also satisfy \eqref{E:real}, see Remark \ref{R:real}.
Recall that while $JL$ is a linear operator, $L$ and $J$ are complexified as
Hermitian forms or anti-linear mappings, see \eqref{E:complexify} and
\eqref{E:anti-linear}. \newline

\noindent\textbf{Case 3: $i\mu\in\sigma(JL)$ and $\exists$ a closed subspace
$\{0\}\neq Y\subset E_{i\mu}$ such that}
\[
JL(Y)\subset Y,\text{ and }\ \langle L\cdot,\cdot\rangle\text{ is
non-degenerate and sign indefinite on }Y.
\]

\begin{remark}
\label{R:Unstable1} Clearly, this includes, but not limited to, the situation
where $\langle L\cdot,\cdot\rangle$ is non-degenerate and indefinite on
$E_{i\mu}$, a special case of which is when $i\mu$ is isolated in $\sigma
(JL)$. It is analyzed in several subcases below.
\end{remark}

We will construct a perturbation $L_{\#}$ such that $\sigma(JL_{\#})$ contains
a hyperbolic eigenvalue near $i\mu$. The proof will basically be carried out
in some finite dimensional subspaces. Such finite dimensional problems had
been well studied in the literature, mostly for the Case 3b below when there
are two eigenvectors of opposite signs of $\langle L\cdot,\cdot\rangle$ (see
e.g. \cite{mackay86, ekeland90}). We could not find a reference for the proof
of structural instability when the indefiniteness of $L|_{E_{i\mu}}$ is caused
by a Jordan chain of $JL$ (Case 3c below). So we give a detailed proof for the
general case, which will also be used in later cases of embedded eigenvalues.
Our proof for the Case 3c uses the special basis constructed in Proposition
\ref{P:basis} for the Jordan blocks of $JL$ on $E_{i\mu}$.

Recall $\bar{u}\in E_{-i\mu}$ for any $u\in E_{i\mu}$. Let
\[
Y_{\mu}=\{u+\bar{v}\mid u,v\in Y\}\subset E_{i\mu}+E_{-i\mu}.
\]
From Lemma \ref{lemma-orthogonal-eigenspace} and the assumption on $Y$,
$\langle L\cdot,\cdot\rangle$ is still non-degenerate on $Y_{\mu}$ which is
also clearly invariant under $JL$. Recall that
\[
Y_{\mu}^{\perp_{L}}=\{u\in X\mid\langle Lu,v\rangle=0,\ \forall v\in Y_{\mu
}\}.
\]
From Lemmas \ref{L:non-degeneracy} and \ref{L:InvariantSubS}, $Y_{\mu}%
^{\perp_{L}}$ is also invariant under $e^{tJL}$ and $X=Y_{\mu}\oplus Y_{\mu
}^{\perp_{L}}$. The definition of $Y_{\mu}$ implies that $Y_{\mu}$ is real, in
the sense $\bar{u}\in Y_{\mu}$ for any $u\in Y_{\mu}$, and thus the
complexification of a real subspace of $X$. According to Lemma \ref{L:decomJ},
$JL|_{Y_{\mu}}$ is a also a Hamiltonian operator satisfying hypotheses
(\textbf{H1-3}) with the non-degenerate energy $L_{Y_{\mu}}$, defined in
\eqref{E:L_y1}. Therefore, we may apply Proposition \ref{P:basis} to $Y_{\mu}$
and $JL|_{Y_{\mu}}$, where $Y\subset Y_{\mu}$ is the subspace of all
generalized eigenvectors of $i\mu$ of $JL|_{Y_{\mu}}$. Since $L$ is
non-degenerate on $Y$, it is clear that \textbf{Case 3} contains the following
three subcases only. \newline

\noindent\textit{Case 3a: $\mu\ne0$ and $\langle L\cdot, \cdot\rangle$ changes
sign on $\ker(JL - i\mu) \cap Y$.}

In this subcase, let $u_{\pm}\in\ker(JL-i\mu)\cap Y$ be such that $\pm\langle
Lu_{\pm},u_{\pm}\rangle>0$. By a Gram-Schmidt process, without loss of
generality, we may assume
\[
\langle Lu_{\pm},u_{\pm}\rangle=\pm1,\quad\langle Lu_{+},u_{-}\rangle=0.
\]
Note that $u_{\pm}$ can not be real for $\mu\neq0$. As we will construct real
perturbations to create instability, we have to consider the complex conjugate
of $u_{\pm}$ as well. Let
\[
X_{1}=span\{u_{+},u_{-},\overline{u_{+}},\overline{u_{-}}\},\quad X_{2}%
=X_{1}^{\perp_{L}}=\{v\in X\mid\langle Lu_{\pm},v\rangle=\langle
L\overline{u_{\pm}},v\rangle=0\}.
\]
It is clear that subspaces $X_{1,2}$ are comlexifications of real subspaces in
the sense
\begin{equation}
\bar{u}\in X_{1,2}\;\text{ if }\;u\in X_{1,2}. \label{E:conjugacy}%
\end{equation}

Note that $\overline{u_{\pm}}$ are eigenvectors of $-i\mu\ \left(  \neq
i\mu\right)  $ and that $E_{i\mu}$ and $E_{-i\mu}$ are $L$-orthogonal.
Therefore, from the complexification process, it is easy to verify that, with
respect to this basis of the invariant subspace $X_{1}$ of $JL$, operators
$L_{X_{1}}$ and $JL|_{X_{1}}$ take the forms
\[
L_{X_{1}}=%
\begin{pmatrix}
\Lambda & 0\\
0 & \Lambda
\end{pmatrix}
,\quad JL|_{X_{1}}\triangleq A_{1}=i\mu%
\begin{pmatrix}
I_{2\times2} & 0\\
0 & -I_{2\times2}%
\end{pmatrix}
,
\]
where
\[
\Lambda=%
\begin{pmatrix}
1 & 0\\
0 & -1
\end{pmatrix}
,\quad I_{2\times2}=%
\begin{pmatrix}
1 & 0\\
0 & 1
\end{pmatrix}
.
\]
From the invariance of $X_{1}$ and Lemmas \ref{L:non-degeneracy} and
\ref{L:InvariantSubS}, $X_{2}=X_{1}^{\perp_{L}}$ is also invariant under $JL$
and $X=X_{1}\oplus X_{2}$. In this decomposition, $L$, $JL$, and $J$ take the
forms
\begin{equation}
L=%
\begin{pmatrix}
L_{X_{1}} & 0\\
0 & L_{X_{2}}%
\end{pmatrix}
,\quad JL=%
\begin{pmatrix}
A_{X_{1}} & 0\\
0 & A_{X_{2}}%
\end{pmatrix}
,\quad J=%
\begin{pmatrix}
J_{X_{1}} & 0\\
0 & J_{X_{2}}%
\end{pmatrix}
,\label{E:decomp1}%
\end{equation}
where, with respect to the basis of $X_{1}^{\ast}$ dual to $\{u_{\pm
},\ \overline{u_{\pm}}\}$,
\[
J_{X_{1}}=i\mu%
\begin{pmatrix}
\Lambda & 0\\
0 & -\Lambda
\end{pmatrix}
,\quad J_{X_{2}}:X_{2}^{\ast}\supset D(J_{X_{2}})\rightarrow X_{2},\quad
J_{X_{2}}^{\ast}=-J_{X_{2}}.
\]
Here, $J^{\ast}=-J$ is used.

Consider a perturbation $L_{1}$ in the form of
\[
L_{1}=%
\begin{pmatrix}
L_{1,X_{1}} & 0\\
0 & 0
\end{pmatrix}
,\;\text{ where }L_{1,X_{1}}=%
\begin{pmatrix}
\epsilon R & 0\\
0 & \epsilon R
\end{pmatrix}
,\quad R=%
\begin{pmatrix}
0 & 1\\
1 & 0
\end{pmatrix}
.
\]
It is straightforward to verify that $L_{1}$ is real, namely, $\langle
L_{1}\bar{u},\bar{v}\rangle=\overline{\langle L_{1}u,v\rangle}$. Let
$L_{\#}=L+L_{1}$. Clearly, the decomposition $X=X_{1}\oplus X_{2}$ is still
invariant under $JL_{\#}$ and orthogonal with respect to $L_{\#}$. Therefore,
by a direct computation on the $4\times4$ matrix $JL_{\#}|_{X_{1}}$ which can
be further reduced to the $2\times2$ matrix $i\mu\Lambda(\Lambda+\epsilon R)$,
we obtain
\[
i\mu\pm\epsilon\mu\in\sigma(JL_{\#}).
\]
Therefore, $\sigma(JL_{\#})$ contains hyperbolic eigenvalues near $i\mu$ for
any $\epsilon\neq0$. \newline

\noindent\textit{Case 3b: $\mu=0$ and $\langle L\cdot, \cdot\rangle$ changes
sign on $\ker(JL - i\mu) \cap Y$.}

In this case one may proceed as in the above through \eqref{E:decomp1},
however, with $J_{X_{1}}=0$. Therefore, no hyperbolic eigenvalue can bifurcate
through such type of perturbations of $L$. \newline

\noindent\textit{Cases 3c: $\mu\ne0$ and $Y$ contains a non-trivial Jordan
chain $u_{j} = (JL-i\mu)^{j-1} u_{1}$, $j=1, \ldots, k>1$, of $JL$ such that
$u_{k} \in\ker(JL-i\mu) \backslash\{0\}$ and $\langle L\cdot, \cdot\rangle$ is
non-degenerate on $span\{u_{1}, \ldots, u_{k}\}$}.

Again in this case let
\[
X_{1}=span\{u_{j},\ \overline{u_{j}}\mid j=1,\ldots,k\},\quad X_{2}%
=X_{1}^{\perp_{L}},
\]
which is a $L$-orthogonal invariant decomposition under $JL$ satisfying
\eqref{E:conjugacy} and thus the forms \eqref{E:decomp1} hold. From
Proposition \ref{P:basis}, without loss of generality, we may assume that,
with respect to the basis $\{u_{1},\ldots,u_{k},\overline{u_{1}}%
,\ldots,\overline{u_{k}}\}$ (as well as its dual basis in $X_{1}^{\ast}$),
$L_{X_{1}}$, $JL|_{X_{1}}\triangleq A_{X_{1}}$, and $J_{X_{1}}$ take the
forms
\[
L_{X_{1}}=%
\begin{pmatrix}
B_{+} & 0\\
0 & \overline{B_{+}}%
\end{pmatrix}
,\quad A_{X_{1}}=%
\begin{pmatrix}
A_{+} & 0\\
0 & \overline{A_{+}}%
\end{pmatrix}
,\quad J_{X_{1}}=%
\begin{pmatrix}
J_{+} & 0\\
0 & \overline{J_{+}}%
\end{pmatrix}
,
\]
where
\[
B_{+}=%
\begin{pmatrix}
0 & \ldots & 0 & b_{k}\\
0 & \ldots & b_{k-1} & 0\\
\ldots &  &  & \\
b_{1} & \ldots & 0 & 0
\end{pmatrix}
,\quad A_{+}=%
\begin{pmatrix}
i\mu & 0 & \ldots & 0 & 0\\
1 & i\mu & \ldots & 0 & 0\\
\ldots &  &  &  & \\
0 & 0 & \ldots & 1 & i\mu
\end{pmatrix}
,
\]
and $b_{j+1}=-b_{j},\ b_{k+1-j}=\overline{b_{j}}$. Therefore, $b_{j}\in\{\pm
i\}$ if $2|k$ or $b_{j}\in\{\pm1\}$ otherwise. Then one may compute
\[
J_{+}=%
\begin{pmatrix}
0 & 0 & \ldots & 0 & i\mu\overline{b_{1}}^{-1}\\
0 & 0 & \ldots & i\mu\overline{b_{2}}^{-1} & \overline{b_{1}}^{-1}\\
\ldots &  &  &  & \\
i\mu\overline{b_{k}}^{-1} & \overline{b_{k-1}}^{-1} & \ldots & 0 & 0
\end{pmatrix}
.
\]
Here, note that $\overline{b_{j}}^{-1}$ instead of $b_{j}^{-1}$ appears in
above $J_{+}$, namely
\[
J_{+}(u_{j}^{\ast})=i\mu\overline{b_{k+1-j}}^{-1}u_{k+1-j}+\overline
{b_{k+1-j}}^{-1}u_{k+2-j}%
\]
where $u_{k+1}=0$ is understood. This is due to the \textit{anti-linear}
complexification of $L$ and $J$, see \eqref{E:anti-linear}. In fact, let
$\{u_{j}^{\ast},\ \overline{u_{j}}^{\ast}\mid j=1,\ldots,k\}$ be the dual
basis in $X_{1}^{\ast}$ which are complex linear functionals. We have
\begin{align*}
\langle u_{l}^{\ast},Ju_{j}^{\ast}\rangle= &  \langle u_{l}^{\ast}%
,J(b_{k+1-j}^{-1}Lu_{k+1-j})\rangle=\langle u_{l}^{\ast},\overline{b_{k+1-j}%
}^{-1}JLu_{k+1-j})\rangle\\
= &  \overline{b_{k+1-j}}^{-1}\langle u_{l}^{\ast},JLu_{k+1-j}\rangle
=\overline{b_{k+1-j}}^{-1}\langle u_{l}^{\ast},i\mu u_{k+1-j}+u_{k+2-j}%
\rangle\\
= &  \overline{b_{k+1-j}}^{-1}\left(  i\mu\delta_{{l,k+1-j}}+\delta
_{l,k+2-j}\right)  .
\end{align*}

Consider perturbations in the form of
\[
L_{1}=%
\begin{pmatrix}
L_{1,X_{1}} & 0\\
0 & 0
\end{pmatrix}
,\;\text{ where }L_{1,X_{1}}=\epsilon%
\begin{pmatrix}
B & 0\\
0 & \overline{B}%
\end{pmatrix}
,\quad B=%
\begin{pmatrix}
0 & \ldots & 0 & 0\\
\ldots &  &  & \\
0 & \ldots & 0 & 0\\
0 & \ldots & 0 & 1
\end{pmatrix}
.
\]
Clearly, $L_{1}$ is real in the sense $\langle L_{1}\bar{u},\bar{v}%
\rangle=\overline{\langle L_{1}u,v\rangle}$. Let $L_{\#}=L+L_{1}$ and the
decomposition $X=X_{1}\oplus X_{2}$ is still invariant under $JL_{\#}$ and
orthogonal with respect to $L_{\#}$. Therefore, $\sigma\big(J_{+}%
(B_{+}+\epsilon B)\big)\subset\sigma(JL_{\#})$. By direct computation, we
obtain the matrix
\[
J_{+}(B_{+}+\epsilon B)=%
\begin{pmatrix}
i\mu & 0 & \ldots & 0 & i\epsilon\mu\overline{b_{1}}^{-1}\\
1 & i\mu & \ldots & 0 & \epsilon\overline{b_{1}}^{-1}\\
\ldots &  &  &  & \\
0 & 0 & \ldots & 1 & i\mu
\end{pmatrix}
\]
and its characteristic polynomial
\[
\det\big(\lambda-J_{+}(B_{+}+\epsilon B)\big)=(-i)^{k}p\big(i(\lambda
-i\mu)\big),\;\text{ }%
\]
where
\[
p(\lambda)=\lambda^{k}-\epsilon b\lambda+\epsilon b\mu,\ \ b=(-i)^{k-1}%
\overline{b_{1}}^{-1}\in\{\pm1\}.
\]
To find hyperbolic eigenvalues of $J_{+}(B_{+}+\epsilon B)$, it is equivalent
to show that $\,p(\lambda)=0$ has a root $\lambda\notin\mathbf{R}$. Choose the
sign of $\epsilon$ such that $\epsilon b\mu>0$. Denote $c_{1},\cdots,c_{k}$ to
be all the $k-$th roots of $-\left\vert b\mu\right\vert $, which are not real
except for at most one. So we can assume $\operatorname{Im}c_{1}\neq0$. Let
$\delta=\left\vert \epsilon\right\vert ^{\frac{1}{k}}$. Then $p\left(
\lambda\right)  =0$ is equivalent to
\begin{equation}
\left(  \frac{\lambda}{\delta}\right)  ^{k}-\delta|b|\frac{\left\vert
\mu\right\vert }{\mu}\left(  \frac{\lambda}{\delta}\right)  +\left\vert
b\mu\right\vert =0.\label{eqn-scaled}%
\end{equation}
When $\delta\ll1$, by the Implicit Function Theorem, (\ref{eqn-scaled}) has
$k$ roots of the form
\[
\frac{\lambda_{j}}{\delta}=c_{j}+O\left(  \delta\right)  ,\ j=1,\cdots,k,
\]
among which $\lambda_{1}=\delta c_{1}+O\left(  \delta^{2}\right)  $ satisfies
$\operatorname{Im}\lambda_{1}\neq0$. This implies that $J_{+}(B_{+}+\epsilon
B)$ has a hyperbolic eigenvalue of the form $i\mu-i\delta c_{1}+O\left(
\delta^{2}\right)  $.\newline

\noindent\textit{Cases 3d: $\mu=0$ and $Y$ contains a non-trivial Jordan chain
of length }$\geq3$. Let \textit{$u_{j}=(JL)^{j-1}u_{1}$, $j=1,\ldots,k,\left(
k\geq3\right)  \ $be a Jordan chain of $JL$ such that $\langle L\cdot
,\cdot\rangle$ is non-degenerate on $span\{u_{1},\ldots,u_{k}\}$}.

In this case one may proceed as in the above with
\[
p(\lambda)=\lambda^{k}-\epsilon b\lambda=\lambda(\lambda^{k-1}-\epsilon b).
\]
Choose $\epsilon$ such that $\epsilon b<0$. Since $k\geq3$, $p(\lambda)$ has a
complex root which implies that $J_{+}(B_{+}+\epsilon B)$ has a hyperbolic
eigenvalue. For $\mu=0$ and $k=2,\ $by straightforward computations, it can be
shown that $J$ must be degenerate and $J_{+}(B_{+}+\epsilon B)$ has only
eigenvalue $0$ and a purely imaginary eigenvalues for any $2\times2$ Hermitian
matrix $B$ and $\epsilon\ll1$.\newline

\noindent\textbf{Case 4: $i\mu\in\sigma(JL) \cap i\mathbf{R} \backslash\{0\}$
and $\langle L\cdot, \cdot\rangle$ is degenerate on $E_{i\mu} \ne\{0\}$. }

In this case, Proposition \ref{P:non-deg} implies that $i\mu$ must be
non-isolated in $\sigma(JL)$ and we start with the following lemma to isolate
$i\mu$ through a perturbation.

\begin{lemma}
\label{L:isolation} Assume (\textbf{H1-3}). Suppose $i\mu\in\sigma(JL)\cap
i\mathbf{R}$ is non-isolated in $\sigma(JL)$. For any $\epsilon>0$, there
exists a symmetric bounded linear operator $L_{1}:X\rightarrow X^{\ast}$
satisfying \eqref{E:real} such that $|L_{1}|<\epsilon$ and $i\mu\in
\sigma(JL_{\#})$ is an isolated eigenvalue, where $L_{\#}=L+L_{1}$, and
$\langle L_{\#}u,u\rangle>0$ for some generalized eigenvector $u$ of the
eigenvalue $i\mu$ of $JL_{\#}$.
\end{lemma}

\begin{proof}
Since $\sigma(JL)$ is symmetric about both real and imaginary axes, without
loss of generality we can assume that $\mu\geq0$.

Let $X= \Sigma_{j=0}^{6} X_{j}$ be the decomposition given in Theorem
\ref{T:decomposition} with associated projections $P_{j}$. We will use the
notations there in the rest of the proof. Recall Theorem \ref{T:decomposition}
is proved without the complexification, i.e. in the framework of real Hilbert
space $X$ and real operators $J$ and $L$, the resulted decomposition and
operators are real. After the complexification, $X_{j}$ are real in the sense
of \eqref{E:conjugacy} and the operators satisfy \eqref{E:real} and the blocks
in $L$ and $J$ are anti-linear.

As $i\mu\in\sigma(JL)$ is assumed to be non-isolated and $\dim X_{j}<\infty$,
$j\neq0,3$, it must hold $i\mu\in\sigma(A_{3})$. Since $A_{3}$ is
anti-self-adjoint with respect to the positive definite Hermitian form
$\langle L_{X_{3}}\cdot,\cdot\rangle$, it induces a resolution of the
identity. Namely there exists a family of projections $\{\Pi_{\lambda
}\}_{\lambda\in\mathbf{R}}$ on $X_{3}$ such that

\begin{enumerate}
\item $\lim_{\lambda\to\lambda_{0}+} \Pi_{\lambda}u = \Pi_{\lambda_{0}}u$, for
all $\lambda_{0} \in\mathbf{R}$ and $u \in X_{3}$;

\item $\Pi_{\lambda_{1}} \Pi_{\lambda_{2}} = \Pi_{\min\{\lambda_{1},
\lambda_{2}\}}$, for all $\lambda_{1,2} \in\mathbf{R}$;

\item $\langle L_{X_{3}} \Pi_{\lambda}u_{1}, u_{2} \rangle= \langle L_{X_{3}}
u_{1}, \Pi_{\lambda}u_{2}\rangle$ for any $u_{1,2} \in X_{3}$ and $\lambda
\in\mathbf{R}$;

\item $u= \int_{-\infty}^{+\infty} d\Pi_{\lambda}u$, $A_{3} u = \int_{-\infty
}^{+\infty} i \lambda\ d\Pi_{\lambda}u$, for any $u \in X_{3}$;

\item $d\Pi_{\lambda}= \overline{d \Pi_{-\lambda}}$ for any $\lambda
\in\mathbf{R}$.
\end{enumerate}

Here the last property is due to the fact that $J$ and $L$ are real satisfying \eqref{E:real}.

For $\mu>0$, define a perturbation $\tilde L_{1} : X_{3} \to X_{3}^{*}$ by
\[
\tilde L_{1} u = L_{X_{3}} \big(\int_{|\lambda-\mu|< \nu} \frac{\lambda-\mu
}\lambda d\Pi_{\lambda}u + \int_{|\lambda+\mu|< \nu} \frac{\lambda+ \mu
}\lambda d\Pi_{\lambda}u \big), \quad\forall u \in X_{3},
\]
where $\nu\in(0, \mu)$ is a small constant to be determined later. As
$\Pi_{\lambda}u$ is continuous from the right, the integrals take the same
values on the open intervals or half open half closed interval like $[\mu-
\nu, \mu+ \nu)$.

For $\mu=0$, define
\[
\tilde{L}_{1}u=L_{X_{3}}\int_{-\nu}^{\nu}\ d\Pi_{\lambda}u,\quad\forall u\in
X_{3},
\]
where again $\nu>0$ is determined later.

For $\mu>0$, like $L_{X_{3}}$, it is clear that $\tilde{L}_{1}$ is anti-linear
satisfying \eqref{E:anti-linear}. We will verify that $\tilde{L}_{1}$ is also
real and symmetric. For any $u\in X$, one may compute using $d\Pi_{\lambda
}=\overline{d\Pi_{-\lambda}}$
\begin{align*}
\overline{\tilde{L}_{1}u}= &  \overline{L_{X_{3}}\big(\int_{|\lambda-\mu|<\nu
}\frac{\lambda-\mu}{\lambda}d\Pi_{\lambda}u+\int_{|\lambda+\mu|<\nu}%
\frac{\lambda+\mu}{\lambda}d\Pi_{\lambda}u\big)}\\
= &  L_{X_{3}}\big(\int_{|\lambda-\mu|<\nu}\frac{\lambda-\mu}{\lambda
}\overline{d\Pi_{\lambda}}\bar{u}+\int_{|\lambda+\mu|<\nu}\frac{\lambda+\mu
}{\lambda}\overline{d\Pi_{\lambda}}\bar{u}\big)\\
= &  L_{X_{3}}\big(\int_{|\lambda-\mu|<\nu}\frac{\lambda-\mu}{\lambda}%
d\Pi_{-\lambda}\bar{u}+\int_{|\lambda+\mu|<\nu}\frac{\lambda+\mu}{\lambda}%
d\Pi_{-\lambda}\bar{u}\big).
\end{align*}
Through a change of variable $\lambda\rightarrow-\lambda$, we obtain
$\overline{\tilde{L}_{1}u}=\tilde{L}_{1}\bar{u}$, namely, $\tilde{L}_{1}$ is
real (the complxification of a real linear operator). Moreover, for any
$u_{1,2}\in X$, we have
\begin{align*}
&  \langle\tilde{L}_{1}u_{1},u_{2}\rangle=\langle L_{X_{3}}\big(\int%
_{|\lambda-\mu|<\nu}\frac{\lambda-\mu}{\lambda}d\Pi_{\lambda}u_{1}%
+\int_{|\lambda+\mu|<\nu}\frac{\lambda+\mu}{\lambda}d\Pi_{\lambda}%
u_{1}\big),u_{2}\rangle\\
= &  \int_{|\lambda-\mu|<\nu}\frac{\lambda-\mu}{\lambda}d\langle L_{X_{3}}%
\Pi_{\lambda}u_{1},u_{2}\rangle+\int_{|\lambda+\mu|<\nu}\frac{\lambda+\mu
}{\lambda}d\langle L_{X_{3}}\Pi_{\lambda}u_{1},u_{2}\rangle.
\end{align*}
Since $L_{X_{3}}$ is Hermitian and $\langle\Pi_{\lambda}\cdot,\cdot
\rangle=\langle\cdot,\Pi_{\lambda}\cdot\rangle$ on $X_{3}$, we obtain that
$\tilde{L}_{1}$ is Hermitian. Therefore, $\left\langle \tilde{L}_{1}%
\cdot,\cdot\right\rangle $ is the complexification of a real bounded symmetric
quadratic form on $X_{3}$. Clearly, in the equivalent norm $\langle L_{X_{3}%
}u,u\rangle^{\frac{1}{2}}$ on $X_{3}$,
\[
|\tilde{L}_{1}|\leq\frac{\nu}{\mu-\nu}\rightarrow0,\quad\text{ as }%
\nu\rightarrow0.
\]
The same properties also hold for $\tilde{L}_{1}$ for $\mu=0$ and we skip the details.

Let
\[
L_{1}=P_{3}^{\ast}\tilde{L}_{1}P_{3},\quad L_{\#}=L-L_{1}.
\]
Accordingly in this decomposition
\[
L_{\#}\longleftrightarrow%
\begin{pmatrix}
0 & 0 & 0 & 0 & 0 & 0 & 0\\
0 & 0 & 0 & 0 & B_{14} & 0 & 0\\
0 & 0 & L_{X_{2}} & 0 & 0 & 0 & 0\\
0 & 0 & 0 & L_{X_{3}}+\tilde{L}_{1} & 0 & 0 & 0\\
0 & B_{14}^{\ast} & 0 & 0 & 0 & 0 & 0\\
0 & 0 & 0 & 0 & 0 & 0 & B_{56}\\
0 & 0 & 0 & 0 & 0 & B_{56}^{\ast} & 0
\end{pmatrix}
.
\]
From Corollary \ref{C:decomposition}, one can compute
\[
JL_{\#}\longleftrightarrow%
\begin{pmatrix}
0 & A_{01} & A_{02} & A_{03}(I-L_{X_{3}}^{-1}\tilde{L}_{1}) & A_{04} & 0 & 0\\
0 & A_{1} & A_{12} & A_{13}(I-L_{X_{3}}^{-1}\tilde{L}_{1}) & A_{14} & 0 & 0\\
0 & 0 & A_{2} & 0 & A_{24} & 0 & 0\\
0 & 0 & 0 & A_{3}(I-L_{X_{3}}^{-1}\tilde{L}_{1}) & A_{34} & 0 & 0\\
0 & 0 & 0 & 0 & A_{4} & 0 & 0\\
0 & 0 & 0 & 0 & 0 & A_{5} & 0\\
0 & 0 & 0 & 0 & 0 & 0 & A_{6}%
\end{pmatrix}
.
\]
Due to the upper triangular structure of $JL_{\#}$ and the finite
dimensionality of $X_{1,2}$, in order to prove that $i\mu$ belongs to and is
isolated in $\sigma(JL_{\#})$, it suffices to show that $i\mu$ belongs to and
is isolated in $\sigma\big(A_{3}(I-L_{X_{3}}^{-1}\tilde{L}_{1})\big)$. In
fact, for any $u\in X_{3}$ ,
\begin{equation}
A_{3}(I-L_{X_{3}}^{-1}\tilde{L}_{1})u=i\int_{S}\mu\ d\Pi_{\lambda}%
u+i\int_{\mathbf{R}\backslash S}\lambda\ d\Pi_{\lambda}u,
\label{formula-perturbation-A3}%
\end{equation}
where $S=(-\mu-\nu,-\mu+\nu)\cup(\mu-\nu,\mu+\nu)$. Since $\Pi_{\lambda}$ is
not constant on $S$ as $i\mu\in\sigma(A_{3})$, we obtain that $i\mu$ is an
isolated eigenvalue of $A_{3}(I-L_{X_{3}}^{-1}\tilde{L}_{1})$ and thus of
$\sigma(JL_{\#})$ as well. Indeed, for any $u\in R\left(  \Pi_{\mu+\nu}%
-\Pi_{\mu-\nu}\right)  $, by (\ref{formula-perturbation-A3}) we have%
\[
A_{3}(I-L_{X_{3}}^{-1}\tilde{L}_{1})u=i\mu u.
\]
So $i\mu$ is an eigenvalues of $A_{3}(I-L_{X_{3}}^{-1}\tilde{L}_{1})$. To show
$i\mu$ is isolated, taking any $\alpha\in\mathbf{C}$ such that $0<\left\vert
\alpha-i\mu\right\vert <\nu,$ then we have
\[
\left(  \alpha-A_{3}(I-L_{X_{3}}^{-1}\tilde{L}_{1})\right)  ^{-1}=\int%
_{S}\left(  \alpha-i\mu\right)  ^{-1}\ d\Pi_{\lambda}+\int_{\mathbf{R}%
\backslash S}\left(  \alpha-i\lambda\right)  ^{-1}\ d\Pi_{\lambda},
\]
which is clearly a bounded operator.

Finally, we prove that there exists a generalized eigenvector $u$ of $i\mu$ of
$JL_{\#}$ such that $\langle L_{\#}u,u\rangle>0$. Since $\dim X_{1}<\infty$,
there exists an integer $K>0$ such that
\[
X_{1}=Y_{\mu}\oplus\tilde{Y},\;\text{ where }Y_{\mu}=\ker(A_{1}-i\mu)^{K}\cap
X_{1},\;\tilde{Y}=(A_{1}-i\mu)^{K}X_{1}.
\]

In the following we proceed in the case of $\mu>0$ first. Let
\[
Z_{\mu}=\{u-(-i\mu)^{-K}P_{0}(JL-i\mu)^{K}u\mid u\in Y_{\mu}\},\quad\tilde
{Z}=\ker L\oplus\tilde{Y}.
\]
Note that the upper triangular structure of $JL$ implies that
\[
(A_{1}-i\mu)^{K}=P_{1}(JL-i\mu)^{K}|_{X_{1}}.
\]
Using this observation and the invariance of $\tilde{X}=\ker L\oplus X_{1}$
under $JL$, we obtain through straightforward computations
\begin{equation}
\tilde{X}=Z_{\mu}\oplus\tilde{Z},\;Z_{\mu}=\ker(JL-i\mu)^{K}\cap\tilde
{X},\;\tilde{Z}=(JL-i\mu)^{K}\tilde{X}, \label{E:decom2}%
\end{equation}
and on the invariant subspaces $Z_{\mu}$ and $\tilde{Z}$
\begin{equation}
\sigma(JL|_{Z_{\mu}})=\{i\mu\}\ \text{ if }\ i\mu\in\sigma(A_{1}),\quad
i\mu\notin\sigma(JL|_{\tilde{Z}}). \label{E:decom3}%
\end{equation}
Let $P:\tilde{X}\rightarrow\tilde{Z}$ be the projection associated to the
above decomposition and $u_{3}\in X_{3}$ be such that
\[
A_{3}(I-L_{X_{3}}^{-1}\tilde{L}_{1})u_{3}=i\mu u_{3}.
\]
The structure of $JL_{\#}$ implies $(JL_{\#}-i\mu)u_{3}\in\tilde{X}$. Let
\[
\tilde{u}=\big((JL-i\mu)|_{\tilde{Z}}\big)^{-1}\tilde{P}(JL_{\#}-i\mu)u_{3}%
\in\tilde{Z}\subset\tilde{X},\quad u=u_{3}-\tilde{u}.
\]
By using $(L_{\#}-L)|_{\tilde{X}}=0$, it is easy to verify that
\[
(JL_{\#}-i\mu)u\in Z_{\mu},
\]
which implies
\[
(JL_{\#}-i\mu)^{K+1}u=0.
\]
From the structure of $L_{\#}$, straightforward computation leads to
\[
\langle L_{\#}u,u\rangle=\langle(L_{3}+\tilde{L}_{1})u_{3},u_{3}\rangle>0,
\]
for $0<\nu<<1$.

The case of $\mu=0$ is largely similar. Let
\[
Z_{0} = Y_{0} \oplus\ker L, \quad\tilde Z = \tilde Y
\]
and \eqref{E:decom2} and \eqref{E:decom3} still hold. The rest of the argument
follows in exactly the same procedure.
\end{proof}

We return to construct a perturbation $L_{1}$ to $L$ to create unstable
eigenvalues. In Case 4, $E^{D}$ in Proposition \ref{P:basis} is non-trivial
and finite dimensional, therefore
\begin{equation}
\exists\ 0\neq u_{0}\in\ker(JL-i\mu)\;\text{ such that }\;\langle Lu_{0}%
,u_{0}\rangle=0, \label{E:temp4}%
\end{equation}
where $u_{0}\in E^{D}$. Since $\mu\neq0$ implies $\overline{u_{0}}\in
E_{-i\mu}$ with $\langle Lu_{0},\overline{u_{0}}\rangle=0$, let
\begin{equation}
Y_{0}=span\{u_{0},\overline{u_{0}}\}\subset\ker(JL-i\mu)\oplus\ker(JL+i\mu).
\label{defn-Y-0}%
\end{equation}
The following decomposition lemma is our first step in the construction of a
hyperbolically generating perturbation.

\begin{lemma}
\label{L:decom2} Suppose $0\neq i\mu\in\sigma(JL)\cap i\mathbf{R}$ satisfying
\eqref{E:temp4}. Let $Y_{0}$ be defined in (\ref{defn-Y-0}). Then there exists
$w\in D(JL)$ with $\overline{w}\neq w\ $and a codim-4 closed subspace
$Y_{1}\subset X$ satisfying \eqref{E:conjugacy} such that $X=Y_{0}\oplus
Y_{1}\oplus Y_{2}$, where $Y_{2}=span\{w,\overline{w}\}$. Moreover, in this
decomposition and the bases $\{u_{0},\overline{u_{0}}\},\ \{w,\bar{w}\}$ on
$Y_{0,2}$ respectively, $L$ and $JL$ take the forms
\[
L\longleftrightarrow%
\begin{pmatrix}
0 & 0 & I_{2\times2}\\
0 & L_{Y_{1}} & 0\\
I_{2\times2} & 0 & 0
\end{pmatrix}
,\;JL\longleftrightarrow%
\begin{pmatrix}
i\mu\Lambda & A_{01} & A_{02}\\
0 & A_{1} & A_{12}\\
0 & 0 & i\mu\Lambda
\end{pmatrix}
,\;\Lambda=%
\begin{pmatrix}
1 & 0\\
0 & -1
\end{pmatrix}
.
\]
Here, all blocks are bounded operators except $A_{1}=J_{Y_{1}}L_{Y_{1}}$ and
$(Y_{1},J_{Y_{1}},L_{Y_{1}})$ satisfies (\textbf{H1-3}).
\end{lemma}

\begin{proof}
Let
\[
\tilde{Y}=\{u\in X\mid\langle Lu,u_{0}\rangle=0=\langle Lu,\overline{u_{0}%
}\rangle\}\supset\{u_{0},\overline{u_{0}}\}.
\]
Clearly, $\tilde{Y}$, satisfying \eqref{E:conjugacy}, is the complexification
of some real codim-2 subspace. Lemma \ref{L:InvariantSubS} implies that
$\tilde{Y}$ is invariant under $JL$. Let
\[
\tilde{Y}_{1}=\{u\in\tilde{Y}\mid(u,u_{0})=(u,\overline{u_{0}})=0\}.
\]
Since $Y_{0}\cap\ker L=\{0\}$ and $D(JL)$ is dense in $X$, there exists a
2-dim subspace $\tilde{Y}_{2}\subset D(JL)$ such that $\langle Lu,v\rangle$,
$u\in Y_{0}$ and $v\in\tilde{Y}_{2}$, defines a non-degenerate bilinear form
on $Y_{0}\otimes\tilde{Y}_{2}$. Clearly, we have $X=Y_{0}\oplus\tilde{Y}%
_{1}\oplus\tilde{Y}_{2}$ and in this decomposition $L$ takes the form
\[
L\longleftrightarrow%
\begin{pmatrix}
0 & 0 & B_{02}\\
0 & L_{Y_{1}} & B_{12}\\
B_{02}^{\ast} & B_{12}^{\ast} & B_{22}%
\end{pmatrix}
,
\]
where $B_{02}:\tilde{Y}_{2}\rightarrow Y_{0}^{\ast}$ is non-degenerate and
$B_{22}^{\ast}=B_{22}$. Through exactly the same procedure as in the proof of
Proposition \ref{P:decomposition1}, we may obtain subspaces $Y_{1}$ and
$\tilde{X}_{2}$ as graphs of bounded linear operators from $\tilde{Y}_{1,2}$
to $Y_{0}$ such that $X=Y_{0}\oplus Y_{1}\oplus X_{2}$ and in this
decomposition $L$ takes the form
\[
L\longleftrightarrow%
\begin{pmatrix}
0 & 0 & B\\
0 & L_{Y_{1}} & 0\\
B^{\ast} & 0 & 0
\end{pmatrix}
,
\]
where $B:X_{2}\rightarrow Y_{0}^{\ast}$ is non-degenerate. There exists $w\in
X_{2}\subset D(JL)$ such that $\langle Lu_{0},w\rangle=1$ and $\langle
Lu_{0},\bar{w}\rangle=\langle L\overline{u_{0}},w\rangle=0$, which also
implies $\langle L\overline{u_{0}},\bar{w}\rangle=0$, where \eqref{E:real} is
used. Let $Y_{2}=span\{w,\bar{w}\}$. From the definition of $w$, $\tilde{Y}$,
and $Y_{1}$, we have $X=Y_{0}\oplus Y_{1}\oplus Y_{2}$, associated with
projections $P_{0,1,2}$, and in this decomposition, the desired block form of
$L$ is achieved. Applying Lemma \ref{L:decomJ} to $X=(Y_{0}\oplus Y_{2})\oplus
Y_{1}$, we obtain that $(Y_{1},J_{Y_{1}},L_{Y_{1}})$ satisfies (\textbf{H1-3}%
), where $J_{Y_{1}}=P_{1}JP_{1}^{\ast}$. The upper triangular block form of
$JL$ is due to the invariance of $Y_{0}$ and $\tilde{Y}=Y_{0}\oplus Y_{2}$.

To complete the proof of the lemma, we are left to show $P_{2} JLw = i\mu w$,
which along with the facts that $\tilde Y$ satisfies \eqref{E:conjugacy} and
$JL$ satisfies \eqref{E:real} also implies $P_{2} JL \bar w= -i \mu\bar w$.
From $(JL)^{*} = -LJ$ (Corollary \ref{C:decomJ}), we have
\[
\langle LJLw, u_{0} \rangle= - \langle L w, JL u_{0}\rangle= i\mu\langle L w,
u_{0}\rangle= i\mu
\]
and similarly $\langle LJLw, \overline{u_{0}} \rangle= - i\mu\langle L w,
\overline{u_{0}}\rangle=0$. According to the definitions of $\tilde Y$ and
$w$, we obtain $P_{2} JLw = i\mu w$ and the lemma is proved.
\end{proof}

With the above lemmas, we are ready to construct a perturbed energy operator
$L_{\#}$ to create unstable eigenvalues of $JL_{\#}$ near $i\mu$ in the Case
4. We start with the decomposition given in Lemma \ref{L:decom2}. Since $i\mu$
is an eigenvalue of $JL$ non-isolated in $\sigma(JL)$, we have $i\mu\in
\sigma(J_{Y_{1}}L_{Y_{1}})$ and is non-isolated in $\sigma(J_{Y_{1}}L_{Y_{1}%
})$. From Lemma \ref{L:isolation}, there exists a sufficiently small symmetric
bounded linear operator $\tilde{L}_{2}:Y_{1}\rightarrow Y_{1}^{\ast}$ such
that $i\mu\in\sigma\big(J_{Y_{1}}(L_{Y_{1}}+\tilde{L}_{2})\big)$ and is
isolated with an eigenvector $u_{1}\in Y_{1}$ satisfying $\langle(L_{Y_{1}%
}+\tilde{L}_{2})u_{1},u_{1}\rangle>0$. Let $L_{2}=P_{1}^{\ast}\tilde{L}%
_{2}P_{1}$ and $\tilde{L}_{\#}=L+L_{2}$, then the block forms of $L_{\#}$ and
$J\tilde{L}_{\#}$ imply that $i\mu\in\sigma(J\tilde{L}_{\#})$ is isolated and
\begin{equation}
u_{0},u_{1}\in\ker(J\tilde{L}_{\#}-i\mu),\quad\langle\tilde{L}_{\#}u_{0}%
,u_{0}\rangle=0,\;\text{ and }\langle\tilde{L}_{\#}u_{1},u_{1}\rangle
>0.\label{eqn-sign-quadratic}%
\end{equation}
Since $i\mu$ is isolated in $\sigma(J\tilde{L}_{\#})$, Proposition
\ref{P:non-deg} implies that $\langle\tilde{L}_{\#}\cdot,\cdot\rangle$ is
non-degenerate on $E_{i\mu}(J\tilde{L}_{\#})$, the subspace of generalized
eigenvectors of $i\mu$ for $J\tilde{L}_{\#}$. Moreover, by
(\ref{eqn-sign-quadratic}), $\langle\tilde{L}_{\#}\cdot,\cdot\rangle$ is sign
indefinite on $E_{i\mu}(J\tilde{L}_{\#})$. This situation has been covered in
\textbf{Case 3}. Therefore, there exists a sufficient small symmetric bounded
linear operator $L_{3}:X\rightarrow X^{\ast}$ such that there exists
$\lambda\in\sigma(JL_{\#})\backslash i\mathbf{R}$ sufficiently close to $i\mu
$, where $L_{\#}=L+L_{2}+L_{3}$. \newline

\noindent\textbf{Case 5: $i\mu\in\sigma(JL)\cap i\mathbf{R}\backslash\{0\}$ is
non-isolated and $\langle L\cdot,\cdot\rangle$ is negative definite on
$E_{i\mu}\neq\{0\}$. }

Much as in \textbf{Case 4} (but more easily), we can construct sufficiently
small symmetric bounded perturbations to the energy operator $L$ to create
unstable eigenvalues. In fact, Proposition \ref{P:basis} implies that in Case
5, it holds $\ker(JL-i\mu)=E_{i\mu}$. Let
\[
Y_{0}=E_{i\mu}\oplus E_{-i\mu},\quad Y=Y_{0}^{\perp_{L}}=\{v\in X\mid\langle
Lv,u\rangle=\langle Lv,\overline{u}\rangle=0,\ \forall u\in E_{i\mu}\}.
\]
Since $\langle L\cdot,\cdot\rangle$ is negative on $Y_{0}$, Lemma
\ref{L:non-degeneracy} implies that $X=Y_{0}\oplus Y$ associated with
projections $P_{Y_{0},Y}$. In this decomposition $L$ and $JL$ take the forms
\[
L\longleftrightarrow%
\begin{pmatrix}
L_{Y_{0}} & 0\\
0 & L_{Y}%
\end{pmatrix}
,\quad JL\longleftrightarrow%
\begin{pmatrix}
A_{0} & 0\\
0 & A
\end{pmatrix}
,
\]
where $A_{0}$ is a bounded operator satisfying $A^{2}+\mu^{2}=0$. Lemma
\ref{L:decomJ} implies that $A=J_{Y}L_{Y}$ and $(Y,J_{Y},L_{Y})$ satisfies
(\textbf{H1-3}). Clearly, it still holds that $i\mu\in\sigma(J_{Y}L_{Y})$ and
is non-isolated there. Applying Lemma \ref{L:isolation} to $L_{Y},$ we obtain
a perturbation $\tilde{L}:Y\rightarrow Y^{\ast}$ such that $i\mu$ is an
isolated point in $\sigma\big(J_{Y}(L_{Y}+\tilde{L})\big)$. Let $\tilde
{L}_{\#}=L+P_{Y}^{\ast}\tilde{L}P_{Y}$ and we obtain that $i\mu$ is an
isolated point in $\sigma(J\tilde{L}_{\#})$ with $\langle\tilde{L}_{\#}%
\cdot,\cdot\rangle$ sign indefinite on its eigenspace. This is a case covered
in \textbf{Case 3} and thus there exists a sufficient small symmetric bounded
linear perturbation $L_{\#}$ to $L$ so that $JL_{\#}$ has an unstable
eigenvalue close to $i\mu$.\newline

\noindent\textbf{Proof of Theorem \ref{T:USImSpec}.} It suffices to show that
\textbf{Cases 3, 4, 5} cover all the cases in Theorem \ref{T:USImSpec}. In
fact, if $\langle L\cdot,\cdot\rangle$ is degenerate on $E_{i\mu}\neq\{0\}$
and $\mu\neq0$, this is precisely \textbf{Case 4}. Let us consider the case
when $\langle L\cdot,\cdot\rangle$ is non-degenerate on $E_{i\mu}$ and
satisfies the assumptions in Theorem \ref{T:USImSpec}. Then $\langle
L\cdot,\cdot\rangle$ is either sign indefinite on $E_{i\mu}$ (\textbf{Case 3})
or is negative definite on $E_{i\mu}$ for an eigenvalue $i\mu\neq0$
non-isolated in $\sigma(JL)$ (\textbf{Case 5}). \hfill$\square$

\section{Proof of Theorem \ref{T:degenerate} where (\textbf{H2.b}) is
weakened}

\label{S:degenerate}

In this section, we consider the case when (\textbf{H2.b}) is weakened,
namely, $L$ is only assumed to be positive on $X_{+}$, but not necessarily
uniformly positive. More precisely, we will prove Theorem \ref{T:degenerate}
under hypotheses (\textbf{B1-5}) given in Subsection \ref{SS:degenerate}. In
Subsection \ref{SS:2dGGP}, as an example we will consider the stability of
traveling waves of a nonlinear Schr\"{o}dinger equation with non-vanishing
condition at infinity in two dimensions. \newline

\noindent\textbf{Initial decomposition of the phase space.} We adopt the
notations as in Section \ref{S:Preliminary}. Let $P_{\pm,0}:X\rightarrow
X_{\pm,0}$ be the projections associated to the decomposition $X=X_{-}%
\oplus\ker L\oplus X_{+}$, where $X_{0}=\ker L$, and
\[
\tilde{X}_{\pm,0}^{\ast}=P_{\pm,0}^{\ast}X_{\pm,0}^{\ast}\subset X^{\ast}.
\]
We also let
\[
X_{\leq0}=X_{-}\oplus\ker L,\quad P_{\leq0}=P_{0}+P_{-}=I-P_{+},\quad\tilde
{X}_{\leq0}^{\ast}=\tilde{X}_{-}^{\ast}\oplus\tilde{X}_{0}^{\ast}.
\]
Clearly, we have
\begin{equation}
\tilde{X}_{+}^{\ast}=\ker i_{X_{\leq0}}^{\ast},\quad\tilde{X}_{\leq0}^{\ast
}=\ker i_{X_{+}}^{\ast}\subset Q_{0}(X),\quad X^{\ast}=\tilde{X}_{\leq0}%
^{\ast}\oplus\tilde{X}_{+}^{\ast},\label{E:tX^*}%
\end{equation}
where assumption (\textbf{B5}) is used. Since $\langle Lu,u\rangle<0$ on
$X_{-}\backslash\{0\}$ and $\dim X_{-}=n^{-}(L)<\infty$, there exists
$\delta>0$ such that
\[
\langle Lu,u\rangle\leq-\delta\Vert u\Vert^{2},\quad\forall\ u\in X_{-}.
\]
From (\textbf{B4}), we also have
\[
LX_{+}\subset\tilde{X}_{+}^{\ast},\quad LX_{\leq0}=\tilde{X}_{-}^{\ast}%
\subset\tilde{X}_{\leq0}^{\ast}.
\]
Denote
\[
L_{+}=i_{X_{+}}^{\ast}Li_{X_{+}}:X_{+}\rightarrow X_{+}^{\ast},\quad L_{\leq
0}=i_{X_{\leq0}}^{\ast}Li_{X_{\leq0}}:X_{\leq0}\rightarrow X_{\leq0}^{\ast},
\]
which along with the $L$-orthogonality in (\textbf{B4}) implies
\[
L=P_{+}^{\ast}L_{+}P_{+}+P_{\leq0}^{\ast}L_{\leq0}P_{\leq0}.
\]

While the decomposition is not necessarily $Q_{0}$-orthogonal, we have the
following lemma. Let
\begin{align*}
&  Q_{0}^{\leq0,+}=i_{\leq0}^{\ast}Q_{0}i_{X_{+}}:X_{+}\rightarrow X_{\leq
0}^{\ast},\quad Q_{0}^{+,\leq0}=i_{X_{+}}^{\ast}Q_{0}i_{\leq0}:X_{\leq
0}\rightarrow X_{+}^{\ast},\\
&  Q_{0}^{\leq0}=i_{X_{\leq0}}^{\ast}Q_{0}i_{X_{\leq0}}:X_{\leq0}\rightarrow
X_{\leq0}^{\ast},\quad Q_{0}^{+}=i_{X_{+}}^{\ast}Q_{0}i_{X_{+}}:X_{+}%
\rightarrow X_{+}^{\ast}.
\end{align*}
Clearly, $Q_{0}^{\leq0,+}=(Q_{0}^{+,\leq0})^{\ast}$ and in the decomposition
$X=X_{\leq0}\oplus X_{+}$ and $X^{\ast}=P_{\leq0}^{\ast}X_{\leq0}^{\ast}\oplus
P_{+}^{\ast}X_{+}^{\ast}$, operator $Q_{0}$ takes the form $%
\begin{pmatrix}
Q_{0}^{\leq0} & Q_{0}^{\leq0,+}\\
Q_{0}^{+,\leq0} & Q_{0}^{+}%
\end{pmatrix}
$. Since $\langle Q_{0}u,u\rangle>0$ for all $0\neq u\in X$, $Q_{0}^{+}$ and
$Q_{0}^{\leq0}$, as well as $L_{+}$, are bounded, symmetric, and positive.
Therefore, $Q_{0}^{\leq0}:X_{\leq0}\rightarrow X_{\leq0}^{\ast}$ and
$Q_{0}^{+},L_{+}:X_{+}\rightarrow X_{+}^{\ast}$ are injective with dense
ranges. Consequently, $(Q_{0}^{\leq0})^{-1}:X_{\leq0}^{\ast}\rightarrow
X_{\leq0}$ and $(Q_{0}^{+})^{-1},L_{+}^{-1}:X_{+}^{\ast}\rightarrow X_{+}$ are
densely defined, closed, and positive operators with
\[
\big((Q_{0}^{\leq0})^{-1}\big)^{\ast}=(Q_{0}^{\leq0})^{-1},\big((Q_{0}%
^{+})^{-1}\big)^{\ast}=(Q_{0}^{+})^{-1},
\]
and $(L_{+}^{-1})^{\ast}=L_{+}^{-1}$.

\begin{lemma}
\label{L:Q0+} It holds that $P_{+}^{\ast}Q_{0}^{+}(X_{+})\subset\tilde{X}%
_{+}^{\ast}$ is dense in $\tilde{X}_{+}^{\ast}$ and
\[
Q_{0}(X)=\tilde{X}_{\leq0}^{\ast}\oplus P_{+}^{\ast}Q_{0}^{+}(X_{+}),\quad
P_{+}^{\ast}Q_{0}^{+}(X_{+})=Q_{0}(X)\cap\tilde{X}_{+}^{\ast},
\]
with $(Q_{0}^{\leq0})^{-1}$ and $(Q_{0}^{+})^{-1}Q_{0}^{+,\leq0}$ being
bounded operators. Moreover,
\[
A\triangleq Q_{0}^{-1}P_{+}^{\ast}Q_{0}^{+}:X_{+}\rightarrow X_{2},\;\text{
where }X_{2}=Q_{0}^{-1} (\tilde{X}_{+}^{\ast})\subset X,
\]
is an isomorphism.
\end{lemma}

This lemma makes the natural connection between $Q_{0}(X)$ and $Q_{0}^{+}
(X_{+})$.

\begin{proof}
Since the quadratic form $\langle Q_{0}u,u\rangle$ is positive on $X$, we have
that $Q_{0}:X\rightarrow X^{\ast}$ is injective with dense $Q_{0}(X)\subset
X^{\ast}$. As $\tilde{X}_{\leq0}^{\ast}=\ker i_{X_{+}}^{\ast}\subset Q_{0}(X)$
due to (\textbf{B5}) and $X^{\ast}=\tilde{X}_{\leq0}^{\ast}\oplus\tilde{X}%
_{+}^{\ast}$, we obtain that $\tilde{X}_{+}^{\ast}\cap Q_{0}(X)$ is dense in
$\tilde{X}_{+}^{\ast}$ and $Q_{0}(X)=\tilde{X}_{\leq0}^{\ast}\oplus
\big(Q_{0}(X)\cap\tilde{X}_{+}^{\ast}\big)$. In the rest of the proof, we
study $Q_{0}(X)\cap\tilde{X}_{+}^{\ast}$ and its associated properties.

Let $X_{1}=Q_{0}^{-1}(\tilde{X}_{\leq0}^{\ast})\subset X$, which is a closed
subspace. Since $\tilde{X}_{\leq0}^{\ast}\subset Q_{0}(X)$ and $Q_{0}$ is
injective, $Q_{0}:X_{1}\rightarrow\tilde{X}_{\leq0}^{\ast}$ is bounded,
injective, and surjective and thus an isomorphism. Let
\[
\phi=(Q_{0}|_{X_{1}})^{-1}P_{\leq0}^{\ast}:X_{\leq0}^{\ast}\rightarrow
X_{1},\ \ \phi_{\leq0}=P_{\leq0}\phi,\ \ \ \ \phi_{+}=P_{+}\phi,
\]
which are bounded operators. For any $f,g\in X_{\leq0}^{\ast}$, since
\[
\langle g,\phi_{\leq0}f\rangle=\langle P_{\leq0}^{\ast}g,\phi f\rangle=\langle
Q_{0}\phi g,\phi f\rangle,
\]
we obtain that $\phi_{\leq0}:X_{\leq0}^{\ast}\rightarrow X_{\leq0}$ is
symmetric and $\langle f,\phi_{\leq0}f\rangle>0$ for any $0\neq f\in X_{\leq
0}^*$. Therefore $\phi_{\le 0}^{-1}$ is a densely defined closed operator satisfying $(\phi_{\le 0}^{-1})^* = \phi_{\le 0}^{-1}>0$. 

For any $f\in X_{\leq0}^{\ast}$, let
\[
\phi f=u_{\leq0}+u_{+},\quad u_{+}=\phi_{+}f,\quad u_{\leq0}=\phi_{\leq0}f,
\]
then we have
\[
Q_{0}^{\leq0}u_{\leq0}+Q_{0}^{\leq0,+}u_{+}=f,\quad Q_{0}^{+,\leq0}u_{\leq
0}+Q_{0}^{+}u_{+}=0.
\]
It implies that $Q_{0}^{+,\leq0}u_{\leq0}\in Q_{0}^{+}(X_{+})$ and
$u_{+}=-(Q_{0}^{+})^{-1}Q_{0}^{+,\leq0}u_{\leq0}$. Therefore,
\[
\big(Q_{0}^{\leq0}-Q_{0}^{\leq0,+}(Q_{0}^{+})^{-1}Q_{0}^{+,\leq0}%
\big)u_{\leq0}=f,
\]
which implies that the closed positive symmetric operator$\ \phi_{\leq0}^{-1}$
satisfies
\[
0<\phi_{\leq0}^{-1}=Q_{0}^{\leq0}-Q_{0}^{\leq0,+}(Q_{0}^{+})^{-1}Q_{0}%
^{+,\leq0}\leq Q_{0}^{\leq0}.
\]
Here we also used $Q_{0}^{\leq0,+}=(Q_{0}^{+,\leq0})^{\ast}$ and the
positivity of the symmetric closed operator $(Q_{0}^{+})^{-1}$. Therefore,
$\phi_{\leq0}$ is an isomorphism and
\[
(Q_{0}^{+})^{-1}Q_{0}^{+,\leq0}=-\phi_{+}\phi_{\leq0}^{-1}%
\]
is bounded. The above inequality also implies the boundedness of $(Q_{0}%
^{\leq0})^{-1}\leq\phi_{\leq0}$.

On the one hand, for any $u\in X_{+}$, using $I=i_{X_{\leq0}}P_{\leq
0}+i_{X_{+}}P_{+}$ we can write
\[
P_{+}^{\ast}Q_{0}^{+}u=Q_{0}u-P_{\leq0}^{\ast}i_{X_{\leq0}}^{\ast}Q_{0}%
u=Q_{0}(I-\phi i_{X_{\leq0}}^{\ast}Q_{0})u.
\]
Therefore, $P_{+}^{\ast}Q_{0}^{+}(X_{+})\subset Q_{0}(X) \cap\tilde{X}_{+}^{\ast}$
and
\[
A\triangleq Q_{0}^{-1}P_{+}^{\ast}Q_{0}^{+}=I-\phi i_{X_{\leq0}}^{\ast}%
Q_{0}:X_{+}\rightarrow X_{2}%
\]
is bounded, where $X_{2}=Q_{0}^{-1} (\tilde{X}_{+}^{\ast})$ is a closed subspace
of $X$ and $Q_{0}(X)\cap\tilde{X}_{+}^{\ast} = Q_0(X_2)$. 

On the other hand, suppose $u=u_{\leq0}+u_{+}\in X_{2}$, let $f= i_{X_{+}}^{\ast}Q_{0}u\in X_{+}^{\ast}$ and $f_+= P_+^* f  = Q_0u \in \tilde X_+^*$. We have
\[
Q_{0}^{\leq0}u_{\leq0}+Q_{0}^{\leq0,+}u_{+}=0,\quad Q_{0}^{+,\leq0}u_{\leq
0}+Q_{0}^{+}u_{+}=f,
\]
and thus $u_{\leq0}=-(Q_{0}^{\leq0})^{-1}Q_{0}^{\leq0,+}u_{+}$. Substituting
it into the second equation in the above, we obtain
\[
f=\big(Q_{0}^{+}-Q_{0}^{+,\leq0}(Q_{0}^{\leq0})^{-1}Q_{0}^{\leq0,+}%
\big)u_{+}=Q_{0}^{+}\tilde{u}_{+},
\]
where, from the above boundedness of $(Q_{0}^{+})^{-1}Q_{0}^{+,\leq0}$,
\[
\tilde{u}_{+}=\big(I-(Q_{0}^{+})^{-1}Q_{0}^{+,\leq0}(Q_{0}^{\leq0})^{-1}%
Q_{0}^{\leq0,+}\big)u_{+} \in X_+.
\]
It implies $f\in Q_{0}^{+}(X_{+})$ and thus $f_+ \in P_{+}^{\ast}Q_{0}^{+}(X_{+})$. Therefore 
\[
Q_{0}(X)\cap\tilde{X}_{+}^{\ast}\subset P_{+}^{\ast}Q_{0}^{+}(X_{+}).
\]
Moreover, the above equality on $\tilde u_+$ also implies
\[
A\big(I-(Q_{0}^{+})^{-1}Q_{0}^{+,\leq0}(Q_{0}^{\leq0})^{-1}Q_{0}^{\leq0,+}\big)P_{+} u = Q_{0}^{-1}P_{+}^{\ast}Q_{0}^{+} \tilde u_+ = Q_{0}^{-1} f_+= u.
\]
Therefore we obtain
\[
A^{-1}=\big(I-(Q_{0}^{+})^{-1}Q_{0}^{+,\leq0}(Q_{0}^{\leq0})^{-1}Q_{0}%
^{\leq0,+}\big)P_{+}%
\]
is bounded and the proof of the lemma is complete.
\end{proof}

\noindent\textbf{Construction of $Y$.} As our main concern is that $L_{+}$ is
not uniformly positive definite on $X_{+}$, we will actually work on the
completion $Y_{+}$ of $X_{+}$ under the positive quadratic form $\langle
L_{+}\cdot,\cdot\rangle$.

We start with a resolution of identity to rewrite $L_{+}$ on $X_{+}$. From
(\textbf{B3}), there exists $a>0$ such that
\begin{equation}
\frac{1}{C}\Vert u\Vert^{2}\leq\Vert u\Vert_{L_{+},a}^{2}\leq C\Vert
u\Vert^{2},\quad\forall u\in X_{+},\label{E:La}%
\end{equation}
for some $C>0$, where, for $u,v\in X_{+},\Vert u\Vert_{L_{+},a}^{2}%
=(u,u)_{L_{+},a}\;$and
\[
(u,v)_{L_{+},a}\triangleq\langle(L_{+}+aQ_{0}^{+})u,v\rangle=\langle
(L+aQ_{0})u,v\rangle.
\]
For $u,v\in X_{+}$, let
\[
\mathbb{L}=(L_{+}+aQ_{0}^{+})^{-1}L_{+}:X_{+}\rightarrow X_{+},
\]
which implies $(\mathbb{L}u,v)_{L_{+},a}=\langle Lu,v\rangle\ $and
\begin{equation}
\mathbb{D}=(Q_{0}^{+})^{-1}(L_{+}+aQ_{0}^{+}):X_{+}\supset D(\mathbb{D}%
)=(L_{+}+aQ_{0}^{+})^{-1}Q_{0}^{+}(X_{+})\rightarrow X_{+}.\label{E:BD1}%
\end{equation}
Clearly, the Riesz representation $\mathbb{L}$ of $L_{+}$ with respect to the
equivalent metric $(\cdot,\cdot)_{L_{+},a}$ is a bounded symmetric linear
operator. Since
\begin{equation}
\mathbb{D}^{-1}=(L_{+}+aQ_{0}^{+})^{-1}Q_{0}^{+}=a^{-1}(I-\mathbb{L}%
)\label{E:BD2}%
\end{equation}
is a bounded linear operator symmetric (and positive) with respect to
$(\cdot,\cdot)_{L_{+},a}$, $\mathbb{D}$ is self-adjoint with respect to
$(\cdot,\cdot)_{L_{+},a}$. In applications, if $Q_{1}$ is a uniformly positive
elliptic operator and $Q_{0}$ corresponds to the $L^{2}$ duality, the operator
$\mathbb{D}$ is basically a differential operator on $X_{+}$ of the same order
as $Q_{1}$. The symmetric operator $\mathbb{L}$ admits a resolution of
identity consisting of bounded projections $\Pi_{\lambda}:X_{+}\rightarrow
X_{+}$, $\lambda\in\lbrack0,1]$, where

\begin{enumerate}
\item $\lim_{\lambda\to\lambda_{0}+} \Pi_{\lambda}u = \Pi_{\lambda_{0}}u$, for
all $\lambda_{0} \in[0, 1)$ and $u \in X_{+}$;

\item $\Pi_{\lambda_{1}} \Pi_{\lambda_{2}} = \Pi_{\min\{\lambda_{1},
\lambda_{2}\}}$, for all $\lambda_{1,2} \in[0, 1]$;

\item $\langle(L_{+} + aQ_{0}^{+}) \Pi_{\lambda}u_{1}, u_{2} \rangle=
\langle(L_{+} + aQ_{0}^{+}) u_{1}, \Pi_{\lambda}u_{2}\rangle$ for any $u\in
X_{+}$ and $\lambda\in[0, 1]$;

\item $u= \int_{0}^{1} d\Pi_{\lambda}u$, $\mathbb{L} u = \int_{0}^{1}
\lambda\ d\Pi_{\lambda}u$, for any $u \in X_{+}$.
\end{enumerate}

Here, $\Pi_{1}=I$ and $\Pi_{0}=0$ since $L_{+}$ is bounded and $0<L_{+}%
<L_{+}+aQ_{0}^{+}$ as a quadratic form. Using this resolution of identity, we
have the representations of $L_{+}$ and $\Vert\cdot\Vert_{L_{+},a}$
\[
\langle L_{+}u,v\rangle=\int_{0}^{1}\lambda\,d(\Pi_{\lambda}u,v)_{L_{+}%
,a},\;\Vert u\Vert_{L_{+},a}^{2}=\int_{0}^{1}d\,\Vert\Pi_{\lambda}%
u\Vert_{L_{+},a}^{2},\;u,v\in X_{+}.
\]

Let $(Y_{+},\Vert\cdot\Vert_{L_{+}})$ be the Hilbert space of the completion
of $X_{+}$ with respect to the inner product
\[
(u,v)_{L_{+}}=(\mathbb{L}u,v)_{L_{+},a}=\langle L_{+}u,v\rangle=\langle
Lu,v\rangle=\int_{0}^{1}\lambda\,d(\Pi_{\lambda}u,v)_{L_{+},a},\;u,v\in X_{+}.
\]
Therefore, $X_{+}$ is densely embedded into $Y_{+}$ through the embedding
$i_{X_{+}}$. Using the above spectral integral representation of $\mathbb{L}$,
one may extend $\Pi_{\lambda}$ to be bounded linear projections on $Y$
orthogonal with respect to $(\cdot,\cdot)_{L_{+}}$ as well, satisfying
$|\Pi_{\lambda}|_{Y}\leq1$. Moreover, for $\lambda\in(0,1]$, $(I-\Pi_{\lambda
})Y_{+}\subset X_{+}$ and
\begin{equation}%
\begin{split}
&  \forall\,u\in X_{+},\;\Vert\Pi_{\lambda}u\Vert_{L_{+}}\leq\lambda\Vert
\Pi_{\lambda}u\Vert_{L_{+},a},\\
&  \forall\,u\in Y_{+},\;\lambda\Vert(I-\Pi_{\lambda})u\Vert_{L_{+},a}%
\leq\Vert(I-\Pi_{\lambda})u\Vert_{L_{+}}\leq\Vert(I-\Pi_{\lambda}%
)u\Vert_{L_{+},a},
\end{split}
\label{E:Pi2}%
\end{equation}
where $I-\Pi_{\lambda}=\int_{(\lambda,1]}d\Pi_{\lambda}$ is used.

As $Y_{+}$ is defined as the completion of $X_{+}$ with respect to the metric
$(\mathbb{L}u,u)_{L_{+},a}$, elements in $Y_{+}$ are defined via Cauchy
sequences in $X_{+}$ with respect to this metric. This is rather inconvenient
technically. Instead, we give an integral representation of elements in
$Y_{+}$ and some linear quantities on $Y_{+}$ using $\Pi_{\lambda}$ and the
following lemma.

\begin{lemma}
\label{L:Space-Y+} $\lim_{\lambda\to0+} \Vert\Pi_{\lambda}u \Vert_{L_{+}} =0$
for any $u \in Y_{+}$.
\end{lemma}

\begin{proof}
For any $\epsilon>0$, there exists $v \in X_{+}$ such that $\Vert
u-v\Vert_{L_{+}} < \frac\epsilon2$. Since $\lim_{\lambda\to0+} \Pi_{\lambda}v
= \Pi_{0} v =0$ in $X_{+}$, there exists $\lambda_{0} > 0$ such that $\Vert
\Pi_{\lambda}v \Vert_{L_{+},a} < \frac\epsilon2$ for any $\lambda\in(0,
\lambda_{0})$. Therefore, for any $\lambda\in(0, \lambda_{0})$,
\[
\Vert\Pi_{\lambda}u \Vert_{L_{+}} \le\Vert\Pi_{\lambda}(u-v) \Vert_{L_{+}} +
\Vert\Pi_{\lambda}v \Vert_{L_{+}} \le\Vert u-v\Vert_{L_{+}} + \lambda\Vert
\Pi_{\lambda}v\Vert_{L_{+},a} \le\epsilon.
\]
The lemma is proved.
\end{proof}

\begin{corollary}
\label{C:Space-Y+} For any $u,v\in Y_{+}$, we have
\begin{align*}
&  u=\int_{0}^{1}d\Pi_{\lambda}u=-\int_{0}^{1}d(I-\Pi_{\lambda})u=-\lim
_{\lambda\rightarrow0+}\int_{\lambda}^{1}d(I-\Pi_{\lambda})u,\\
&  \mathbb{L}u=-\int_{0}^{1}\lambda d(I-\Pi_{\lambda})u,\;\Vert u\Vert_{L_{+}%
}^{2}=-\int_{0}^{1}\lambda d\Vert(I-\Pi_{\lambda})u\Vert_{L_{+},a}^{2},\\
&  \langle L_{+}u,v\rangle=-\int_{0}^{1}\lambda d\big((I-\Pi_{\lambda
})u,v\big)_{L_{+},a}=\lim_{\lambda\rightarrow0+}\langle L_{+}(I-\Pi_{\lambda
})u,(I-\Pi_{\lambda})v\rangle,\\
&  Y_{+}^{\ast}=\{f=(L_{+}+aQ_{0}^{+})u\mid u\in X_{+},\\
&  \Vert f\Vert_{Y_{+}^{\ast}}^{2}=-\int_{0}^{1}\lambda^{-1}d\Vert
(I-\Pi_{\lambda})u\Vert_{L_{+},a}^{2}<\infty\}\subset X_{+}^{\ast}.
\end{align*}

\end{corollary}

Here, the first integral converges in the $\Vert\cdot\Vert_{L_{+}}$ norm and
the minus signs are due to the non-increasing monotonicity of $\Vert
(I-\Pi_{\lambda})u\Vert_{L_{+},a}^{2}$. With $(I-\Pi_{\lambda})u\in X_{+}$ for
$\lambda\in(0,1]$, these integral representations are more convenient than the
Cauchy sequence representations of elements in $Y_{+}$. In particular, for
$f=(L_{+}+aQ_{0}^{+})u\in Y_{+}^{\ast}$ and $v\in Y_{+}$,
\[%
\begin{split}
\langle f,v\rangle= &  -\int_{0}^{1}d\big((I-\Pi_{\lambda})u,v\big)_{L_{+}%
,a}=\lim_{\lambda\rightarrow0+}\langle(L_{+}+aQ_{0}^{+})(I-\Pi_{\lambda
})u,(I-\Pi_{\lambda})v\rangle\\
\leq &  \Vert f\Vert_{Y_{+}^{\ast}}\Vert v\Vert_{L_{+}}.
\end{split}
\]

Let
\begin{equation}
Y=X_{\leq0}\oplus Y_{+},\quad(u,v)_{Y}=(P_{\leq0}u,P_{\leq0}v)+\big((I-P_{\leq
0})u,(I-P_{\leq0})v\big)_{L_{+}},\label{E:space-Y}%
\end{equation}
where, with slight abuse of notations, $P_{\leq0}:Y\rightarrow X_{\leq0}$
represents the projection operator with kernel $Y_{+}$. Clearly, $X$ is
densely embedded into $Y$ and let $i_{X}$ denote the embedding.

The dual space $Y^{\ast}$ is densely embedded into $X^{\ast}$ through
$i_{X}^{\ast}$ and thus can be viewed as a dense subspace of $X^{\ast}$. It is
straightforward to see that $i_{X}^{\ast}Y^{\ast}=\tilde{X}_{\leq0}^{\ast
}\oplus\tilde{Y}_{+}^{\ast},\quad$
\[
\langle f,v\rangle=\langle g,u\rangle=0,\;\forall\,u\in X_{\leq0},\ v\in
Y_{+},\ \,f\in\tilde{X}_{\leq0}^{\ast},\ \,g\in\tilde{Y}_{+}^{\ast},
\]
and
\begin{equation}%
\begin{split}
\tilde{Y}_{+}^{\ast} &  =\tilde{X}_{+}^{\ast}\cap i_{X}^{\ast}(Y^{\ast}%
)=P_{+}^{\ast}\{f=(L_{+}+aQ_{0}^{+})u\mid u\in X_{+},\\
&  \qquad\qquad\qquad\Vert f\Vert_{Y_{+}^{\ast}}^{2}=-\int_{0}^{1}\lambda
^{-1}d\Vert(I-\Pi_{\lambda})u\Vert_{L_{+},a}^{2}<\infty\}\subset X^{\ast}.
\end{split}
\label{E:Y_+^*}%
\end{equation}
Operator $L$ is naturally extended as a bounded symmetric linear operator
$L_{Y}:Y\rightarrow Y^{\ast}$ by
\begin{equation}
\langle L_{Y}u,v\rangle=\langle L_{\leq0}P_{\leq0}u,P_{\leq0}v\rangle+\langle
L_{+}(I-P_{\leq0})u,(I-P_{\leq0})v\rangle,\label{E:L_Y}%
\end{equation}
where $L_{+}$ on $Y_{+}$ is computed by the formula given in Corollary
\ref{C:Space-Y+}.

From assumption (\textbf{B4}) on the $L$-orthogonality of the decomposition
\[
X=X_{-}\oplus\ker L\oplus X_{+}=X_{\leq0}\oplus X_{+}%
\]
and Corollary \ref{C:Space-Y+}, the operator $L_{Y}$ defined in the above
satisfies (\textbf{H2}) on $Y$ with $\delta=1$ in (\textbf{H2.b}). \newline

\noindent\textbf{Operator $J_{Y}$.} We define $J_{Y}:Y^{\ast}\supset
D(J_{Y})\rightarrow Y$ essentially as the restriction of $J$ on $Y^{\ast}$,
namely,
\begin{equation}
J_{Y}\triangleq i_{X}\mathbb{J}Q_{0}^{-1}i_{X}^{\ast}:Y^{\ast}\supset
D(J_{Y})\rightarrow Y,\quad D(J_{Y})=(i_{X}^{\ast})^{-1}Q_{0}(X)\subset
Y^{\ast},\label{E:J-ext}%
\end{equation}
where we recall that $i_{X}:X\rightarrow Y$ is the embedding. Assumption
(\textbf{H3}) is satisfied due to (\textbf{B5}) and \eqref{E:tX^*}. Therefore,
to complete the proof of Theorem \ref{T:degenerate}, it suffice to prove
$J_{Y}^{\ast}=-J_{Y}$.

\begin{lemma}
\label{L:DJ_Y} It holds that $i_{X}^{\ast}D(J_{Y})$ is dense in $X^{\ast}$
and
\[
i_{X}^{\ast}D(J_{Y})=Q_{0}(X)\cap i_{X}^{\ast}Y^{\ast}=\tilde{X}_{\leq0}%
^{\ast}\oplus P_{+}^{\ast}(L_{+}+aQ_{0}^{+})X_{1+},
\]
where
\[
X_{1+}=\{u\in X_{+}\mid\int_{0}^{1}\frac{-1}{\lambda(1-\lambda)^{2}}%
d\Vert(I-\Pi_{\lambda})u\Vert^{2}<\infty\}\subset X_{+}.
\]

\end{lemma}

\begin{proof}
From $i_{X}^{*} Y^{*} = \tilde X_{\le0}^{*} \oplus\tilde Y_{+}^{*}$ and
(\textbf{B5}), we can decompose
\[
i_{X}^{*} D(J_{Y}) = Q_{0}(X) \cap i_{X}^{*} Y^{*} = \tilde X_{\le0}^{*}
\oplus\big(\tilde Y_{+}^{*}\cap Q_{0}(X)\big).
\]
As $\tilde Y_{+}^{*} \subset\tilde X_{+}^{*}$, we obtain from Lemma
\ref{L:Q0+}
\begin{equation}
\label{E:DJ_Y}i_{X}^{*} D(J_{Y}) = \tilde X_{\le0}^{*} \oplus\big( \tilde
Y_{+}^{*}\cap P_{+}^{*} Q_{0}^{+} (X_{+})\big).
\end{equation}

Recall \eqref{E:BD1} and we have $(L_{+}+aQ_{0}^{+})u\in Q_{0}^{+}(X_{+})$,
$u\in X_{+}$, if and only if $u\in D(\mathbb{D})$, which is equivalent to
$u\in(I-\mathbb{L})(X_{+})$ according to \eqref{E:BD2}. Therefore, we obtain
that
\[
(L_{+}+aQ_{0}^{+})u\in Q_{0}^{+}(X_{+}),\ u\in X_{+},
\]
if and only if
\[
-\int_{0}^{1}(1-\lambda)^{-2}d\Vert(I-\Pi_{\lambda})u\Vert_{L_{+},a}%
^{2}<\infty,
\]
which can be seen from
\begin{equation}
\mathbb{D}u=(Q_{0}^{+})^{-1}(L_{+}+aQ_{0}^{+})u=-\frac{1}{a}\int_{0}%
^{1}(1-\lambda)^{-1}d(I-\Pi_{\lambda})u.\label{E:BD3}%
\end{equation}
The lemma follows immediately from this property and the characterization
\eqref{E:Y_+^*} of $\tilde{Y}_{+}^{\ast}$.
\end{proof}

To prove $J_{Y}^{\ast}=-J_{Y}$, suppose $f\in D(J_{Y}^{\ast})$ and
$u=J_{Y}^{\ast}f$, namely, $f\in Y^{\ast}$ and $u\in Y$ satisfies
\begin{equation}
\langle f,J_{Y}g\rangle=\langle g,u\rangle,\quad\forall g\in D(J_{Y}%
).\label{E:J_Y*1}%
\end{equation}
Firstly, for any $\epsilon\in(0,\frac{1}{2})$, take
\[
g=-(i_{X}^{\ast})^{-1}P_{+}^{\ast}(L_{+}+aQ_{0}^{+})\int_{\epsilon}^{\frac
{1}{2}}d(I-\Pi_{\lambda})P_{+}u\in D(J_{Y})\cap(i_{X}^{\ast})^{-1}\tilde
{X}_{+}^{\ast},
\]
where Lemma \ref{L:DJ_Y} is used. Equalities \eqref{E:J_Y*1} and \eqref{E:BD3}
imply
\begin{align*}
&  -\int_{\epsilon}^{\frac{1}{2}}d\Vert(I-\Pi_{\lambda})P_{+}u\Vert_{L_{+}%
,a}^{2}=\langle g,u\rangle=\langle f,J_{Y}g\rangle=\langle\mathbb{J}^{\ast
}i_{X}^{\ast}f,Q_{0}^{-1}i_{X}^{\ast}g\rangle\\
= &  -\langle\mathbb{J}^{\ast}i_{X}^{\ast}f,A\mathbb{D}\int_{\epsilon}%
^{\frac{1}{2}}d(I-\Pi_{\lambda})P_{+}u\rangle=-\frac{1}{a}\langle
\mathbb{J}^{\ast}i_{X}^{\ast}f,A\int_{\epsilon}^{\frac{1}{2}}(1-\lambda
)^{-1}d(I-\Pi_{\lambda})P_{+}u\rangle,
\end{align*}
where $A$ is defined in Lemma \ref{L:Q0+} and proved to be bounded. Since
$\lambda\in(0,\frac{1}{2}]$, there exists $C>0$ such that
\[
-\int_{\epsilon}^{\frac{1}{2}}d\Vert(I-\Pi_{\lambda})P_{+}u\Vert_{L_{+},a}%
^{2}\leq C\Vert f\Vert_{Y^{\ast}}^{2},\quad\forall\epsilon\in(0,\frac{1}{2}).
\]
Therefore, we obtain $u\in X$, or more precisely,
\[
u=i_{X}\tilde{u},\quad\tilde{u}\in X.
\]
For any $g\in D(J_{Y})$, from (\textbf{B1-2}) we can compute
\begin{align*}
&  \langle i_{X}^{\ast}f,\mathbb{J}Q_{0}^{-1}i_{X}^{\ast}g\rangle=\langle
f,J_{Y}g\rangle=\langle g,u\rangle=\langle g,i_{X}\tilde{u}\rangle\\
= &  \langle Q_{0}Q_{0}^{-1}i_{X}^{\ast}g,\tilde{u}\rangle=\langle
Q_{0}\mathbb{J}Q_{0}^{-1}i_{X}^{\ast}g,\mathbb{J}\tilde{u}\rangle=\langle
Q_{0}\mathbb{J}\tilde{u},\mathbb{J}Q_{0}^{-1}i_{X}^{\ast}g\rangle.
\end{align*}
Since $\mathbb{J}$ is assumed to be isomorphic in (\textbf{B2}), $Q_{0}^{-1}$
is surjective, and $i_{X}^{\ast}D(J_{Y})$ is dense in $X^{\ast}$ (Lemma
\ref{L:DJ_Y}), we obtain
\[
i_{X}^{\ast}f=Q_{0}\mathbb{J}\tilde{u}\in Q_{0}(X).
\]
Thus it follows from (\textbf{B2}) that
\[
J_{Y}f=i_{X}\mathbb{J}^{2}\tilde{u}=-u,
\]
which implies $J_{Y}^{\ast}\subset-J_{Y}$. Again from (\textbf{B2}), it is
easy to see that $J_{Y}$ is symmetric, namely, $J_{Y}\subset-J_{Y}^{\ast}$.
Therefore, we complete the proof of Theorem \ref{T:degenerate}.

\section{Hamiltonian PDE models}

\label{SS:example}

In this section, based on the above general theory, we study the stability
issues of examples of Hamiltonian PDEs including several dispersive wave
models, the 2D Euler equation for inviscid flows and a 2D nonlinear
Schr\"{o}dinger equations with nonzero conditions at infinity.

First, in Subsections 11.1 to 11.3, we study the stability/instability of
traveling solitary and periodic wave solutions of several classes of equations
modeling weakly nonlinear dispersive long waves. They include BBM, KDV, and
good Boussinesq type equations. These equations respectively have the forms:

1. BBM type%

\begin{equation}
\partial_{t}u+\partial_{x}u+\partial_{x}f\left(  u\right)  +\partial
_{t}\mathcal{M}u=0; \label{BBM}%
\end{equation}

2. KDV type
\begin{equation}
\partial_{t}u+\partial_{x}f\left(  u\right)  -\partial_{x}\mathcal{M}u=0;
\label{kdv}%
\end{equation}

3. good Boussinesq (gBou) type%
\begin{equation}
\partial_{t}^{2}u-\partial_{x}^{2}u+\partial_{x}^{2}f\left(  u\right)
-\partial_{x}^{2}\mathcal{M}u=0. \label{gBOU}%
\end{equation}
We follow the notations in \cite{lin-jfa09}. Here, the pseudo-differential
operator $\mathcal{M}$ is defined as
\[
\widehat{\mathcal{M}g}(\xi)=\alpha(\xi)\widehat{g}(\xi),
\]
where $\hat{g}$ is the Fourier transformation of $g$. We assume: i) $f$ is
$C^{1}$ with $f\left(  0\right)  =f^{\prime}\left(  0\right)  =0,\ $and
$f\left(  u\right)  /u\rightarrow\infty.\ $ii) $a\left\vert \xi\right\vert
^{m}\leq\alpha\left(  \xi\right)  \leq b\left\vert \xi\right\vert ^{m}$ for
large $\xi$, where $m>0$ and $a,b>0$. If $f\left(  u\right)  =u^{2}$ and
$\mathcal{M}=-\partial_{x}^{2}$, the above equations recover the original BBM,
KDV, and good Boussinesq equations, which have been used to model the
propagation of water waves of long wavelengths and small amplitude.

\subsection{Stability of Solitary waves of Long wave models}

\label{SS:solitary}

Consider the equations (\ref{BBM})-(\ref{gBOU}) with $\left(  x,t\right)
\in\mathbf{R}\times\mathbf{R}$. Up to a shift of a constant of the wave
speed/symbol $\alpha\left(  \xi\right)  $, we can assume that $\sigma
_{ess}\left(  \mathcal{M}\right)  \subset\lbrack0,\infty)$. Each of the
equations (\ref{BBM})-(\ref{gBOU}) admits solitary-wave solutions of the form
$u\left(  x,t\right)  =u_{c}\left(  x-ct\right)  $ for $c>1,c>0,c^{2}<1$
respectively, where $u_{c}\left(  x\right)  \rightarrow0$ as $\left\vert
x\right\vert \rightarrow\infty$. They satisfy the equations
\begin{equation}
\mathcal{M}u_{c}+\left(  1-\frac{1}{c}\right)  u_{c}-\frac{1}{c}f\left(
u_{c}\right)  =0,\ \text{(BBM)} \label{steady-BBM}%
\end{equation}%
\[
\mathcal{M}u_{c}+cu_{c}-f\left(  u_{c}\right)  =0,\ \text{(KDV)}%
\]
and
\[
\mathcal{M}u_{c}+\left(  1-c^{2}\right)  u_{c}-f\left(  u_{c}\right)
=0,\ \text{(gBou)}%
\]
respectively. We refer to the introduction of \cite{lin-jfa09} and the book
\cite{pava-book} for the literature on the existence of such solitary waves.
Before stating the results, we introduce some notations. For BBM type
equations (\ref{BBM}), define the operator
\begin{equation}
\mathcal{L}_{0}=\mathcal{M}+\left(  1-\frac{1}{c}\right)  -\frac{1}%
{c}f^{\prime}\left(  u_{c}\right)  :H^{m}\rightarrow L^{2}, \label{L-0-BBM}%
\end{equation}
and the momentum
\begin{equation}
P\left(  c\right)  =\frac{1}{2}\int u_{c}\left(  \mathcal{M+}1\right)  u_{c}.
\label{momentum-BBM}%
\end{equation}
For KDV type equations (\ref{kdv}), define
\begin{equation}
\mathcal{L}_{0}:=\mathcal{M}+c-f^{\prime}\left(  u_{c}\right)  ,\ \ P\left(
c\right)  =\frac{1}{2}\int u_{c}^{2}. \label{L-0-KDV}%
\end{equation}
For good Boussinesq type equations (\ref{gBOU}), define
\begin{equation}
\mathcal{L}_{0}:=\mathcal{M}+1-c^{2}-f^{\prime}\left(  u_{c}\right)
,\ P\left(  c\right)  =-c\int u_{c}^{2}. \label{L-0-gbou}%
\end{equation}
Denote by $n^{-}\left(  \mathcal{L}_{0}\right)  $ the number (counting
multiplicity) of negative eigenvalues of the operators $\mathcal{L}_{0}$.

The linearizations of (\ref{BBM})-(\ref{gBOU}) in the traveling frame $\left(
x-ct,t\right)  $ are
\begin{equation}
\left(  \partial_{t}-c\partial_{x}\right)  \left(  u+\mathcal{M}u\right)
+\partial_{x}\left(  u+f^{\prime}\left(  u_{c}\right)  u\right)
=0,\ \text{(BBM)} \label{L-BBM}%
\end{equation}%
\begin{equation}
\left(  \partial_{t}-c\partial_{x}\right)  u+\partial_{x}\left(  f^{\prime
}\left(  u_{c}\right)  u-\mathcal{M}u\right)  =0,\ \text{(KDV)} \label{L-KDV}%
\end{equation}
and
\begin{equation}
\left(  \partial_{t}-c\partial_{x}\right)  ^{2}u-\partial_{x}^{2}\left(
u-f^{\prime}\left(  u_{c}\right)  u+\mathcal{M}u\right)  =0,\ \text{(gBou)}
\label{L-gBou}%
\end{equation}
respectively. We consider the Hamiltonian structures of these equations.

For BBM type equations, (\ref{L-BBM}) can be written as $\partial_{t}u=JLu$,
where $J=c\partial_{x}\left(  1+\mathcal{M}\right)  ^{-1}$ and $L=\mathcal{L}%
_{0}$ is defined in (\ref{L-0-BBM}). By differentiating (\ref{steady-BBM}) in
$x$ and $c$, we have $\mathcal{L}_{0}u_{c,x}=0$ and
\[
\mathcal{L}_{0}\partial_{c}u_{c}=-\frac{1}{c}\left(  1+\mathcal{M}\right)
u_{c}%
\]
which implies that $J\mathcal{L}_{0}\partial_{c}u_{c}=-u_{c,x}$ and
$\left\langle \mathcal{L}_{0}\partial_{c}u_{c},\partial_{c}u_{c}\right\rangle
=-\frac{1}{c}dP/dc$.

For KDV type equations, (\ref{L-KDV}) is written as $\partial_{t}%
u=J\mathcal{L}_{0}u$, where $J=\partial_{x}$ and $L=\mathcal{L}_{0}$ is
defined in (\ref{L-0-KDV}). Similarly, $\mathcal{L}_{0}u_{c,x}=0$,
$\mathcal{L}_{0}\partial_{c}u_{c}=-u_{c},$ and
\begin{equation}
J\mathcal{L}_{0}\partial_{c}u_{c}=-u_{c,x},\ \left\langle \mathcal{L}%
_{0}\partial_{c}u_{c},\partial_{c}u_{c}\right\rangle =-dP/dc.
\label{KDV-index}%
\end{equation}

For good Boussinesq type equations, we write (\ref{L-gBou}) as a first order
system. Let $\left(  \partial_{t}-c\partial_{x}\right)  u=v_{x}$, then
\[
\left(  \partial_{t}-c\partial_{x}\right)  v=\partial_{x}\left(
\mathcal{M}+1-f^{\prime}\left(  u_{c}\right)  \right)  u=\partial_{x}\left(
\mathcal{L}_{0}+c^{2}\right)  u.
\]
Thus
\[
\partial_{t}\left(
\begin{array}
[c]{c}%
u\\
v
\end{array}
\right)  =JL\left(
\begin{array}
[c]{c}%
u\\
v
\end{array}
\right)  ,
\]
with
\begin{equation}
J=\left(
\begin{array}
[c]{cc}%
0 & \partial_{x}\\
\partial_{x} & 0
\end{array}
\right)  ,\ \ \ L=\left(
\begin{array}
[c]{cc}%
\mathcal{L}_{0}+c^{2} & c\\
c & 1
\end{array}
\right)  . \label{hamiltonian-gBou}%
\end{equation}
We have
\[
\ker L=\left\{  \left(  u,-cu\right)  \ |\ u\in\ker\mathcal{L}_{0}\right\}  .
\]
Since
\begin{equation}
\left\langle L\left(
\begin{array}
[c]{c}%
u\\
v
\end{array}
\right)  ,\left(
\begin{array}
[c]{c}%
u\\
v
\end{array}
\right)  \right\rangle =\left\langle \mathcal{L}_{0}u,u\right\rangle
+\int(v+cu)^{2}, \label{quadratic-L-gBou}%
\end{equation}
so $n^{-}\left(  L\right)  =n^{-}\left(  \mathcal{L}_{0}\right)  $. Similarly
as in BBM and KDV types, we have $\ \ $
\[
L\left(
\begin{array}
[c]{c}%
u_{c,x}\\
-cu_{c,x}%
\end{array}
\right)  =0,\ L\left(
\begin{array}
[c]{c}%
-\partial_{c}u_{c}\\
c\partial_{c}u_{c}+u_{c}%
\end{array}
\right)  =\left(
\begin{array}
[c]{c}%
-cu_{c}\\
u_{c}%
\end{array}
\right)  ,\
\]%
\[
JL\left(
\begin{array}
[c]{c}%
-\partial_{c}u_{c}\\
c\partial_{c}u_{c}+u_{c}%
\end{array}
\right)  =\left(
\begin{array}
[c]{c}%
u_{c,x}\\
-cu_{c,x}%
\end{array}
\right)  ,
\]
and%
\[
\left\langle \ L\left(
\begin{array}
[c]{c}%
-\partial_{c}u_{c}\\
c\partial_{c}u_{c}+u_{c}%
\end{array}
\right)  ,\left(
\begin{array}
[c]{c}%
-\partial_{c}u_{c}\\
c\partial_{c}u_{c}+u_{c}%
\end{array}
\right)  \right\rangle =-dP/dc,
\]
where $P$ is defined in (\ref{L-0-gbou}). For all three cases, we have
$\sigma_{ess}\left(  \mathcal{L}_{0}\right)  \subset\lbrack\delta_{0},\infty)$
for some $\delta_{0}>0$. So the quadratic form $\left\langle \mathcal{L}%
_{0}\cdot,\cdot\right\rangle $ is positive definite on $H^{\frac{m}{2}}\ $in a
finite codimensional space. This along with Remark \ref{R:H3} shows that the
quadratic form $\left\langle \mathcal{L}_{0}\cdot,\cdot\right\rangle $ in the
Hamiltonian formulation of BBM and KDV type equations satisfies the assumption
(\textbf{H1-3}) in the general framework with $X=H^{\frac{m}{2}}$. By
(\ref{quadratic-L-gBou}), the quadratic form $\left\langle L\cdot
,\cdot\right\rangle $ in the Hamiltonian formulation (\ref{hamiltonian-gBou})
of good Boussinesq type equations also satisfies (\textbf{H1-3}) in the space
$\left(  u,\partial_{t}u\right)  \in X=H^{\frac{m}{2}}\times L^{2}$. Thus by
Theorems \ref{theorem-dichotomy}, \ref{theorem-counting}, and Corollary
\ref{cor-instability-index}, we get the following results.

\begin{theorem}
\label{THM: solitary-long-wave}Consider the linearized equations
(\ref{L-BBM})-(\ref{L-gBou}) at solitary waves $u_{c}\left(  x-ct\right)
\ $of equations (\ref{BBM})-(\ref{gBOU}). Then: (i) The following index
formula holds
\begin{equation}
k_{r}+2k_{c}+2k_{i}^{\leq0}+k_{0}^{\leq0}=n^{-}\left(  \mathcal{L}_{0}\right)
. \label{index-KDV}%
\end{equation}
(ii) The linear exponential trichotomy holds in the space $H^{\frac{m}{2}}$
for the linearized equations (\ref{L-BBM}) and (\ref{L-KDV}), and in
$H^{\frac{m}{2}}\times L^{2}$ for (\ref{L-gBou}).\newline(iii) When
$dP/dc\geq0$, we have $k_{0}^{\leq0}\geq1$. Moreover, if $\ker\mathcal{L}%
_{0}=span\left\{  u_{c,x}\right\}  $, then
\[
k_{0}^{\leq0}=\left\{
\begin{array}
[c]{cc}%
1 & \text{if }dP/dc>0\\
0 & \text{if }dP/dc<0
\end{array}
\right.  .
\]

\end{theorem}

\begin{corollary}
\label{cor-instability-parity}(i) When $dP/dc\geq0$ and $n^{-}\left(
\mathcal{L}_{0}\right)  \leq1$, the spectral stability holds true.\newline(ii)
If $\ker\mathcal{L}_{0}=span\left\{  u_{c,x}\right\}  $, then there is linear
instability when $n^{-}\left(  \mathcal{L}_{0}\right)  $ is even and $dP/dc>0$
or $n^{-}\left(  \mathcal{L}_{0}\right)  $ is odd and $dP/dc<0.$
\end{corollary}

In particular, when $\mathcal{M=-\partial}_{x}^{2}$, by the fact that
$u_{c,x}$ changes sign exactly once and the Sturm-Liouville theory, we have
$\ker\mathcal{L}_{0}=span \left\{  u_{c,x}\right\}  $ and $n^{-}\left(
\mathcal{L}_{0}\right)  =1$. Thus, we have

\begin{corollary}
When $\mathcal{M=-\partial}_{x}^{2}$ and $dP/dc<0$, for the linearized
equations (\ref{BBM})-(\ref{gBOU}), we have $k_{r}=1$ and $k_{c}=k_{i}^{-}=0.$
In particular, on the center space $E^{c}$ as given in Theorem
\ref{theorem-dichotomy}, we have
\begin{equation}
\left\langle L\cdot,\cdot\right\rangle |_{E^{c}\cap\left\{  u_{c,x}\right\}
^{\perp}}\geq\delta_{0}>0. \label{center-positive}%
\end{equation}

\end{corollary}

The stability and instability of solitary waves of dispersive models had been
studied a lot in the literature. Assume $\ker\mathcal{L}_{0}=\left\{
u_{c,x}\right\}  ,\ n^{-}\left(  \mathcal{L}_{0}\right)  =1$, then when
$dP/dc>0$, the orbital stability of traveling solitary waves was proved (e.g.
\cite{benjamin72} \cite{gss-87} \cite{bona-sacks-88}) by using the method of
Lyapunov functionals. When $dP/dc<0$, the nonlinear instability was proved in
\cite{bss87} \cite{ss90} for generalized BBM and KDV equations, and in
\cite{liu-bousinesq93} for good Boussinesq equation. The instability proof in
these papers was by contradiction argument which bypassed the linearized
equation. The existence of unstable eigenvalues when $dP/dc>0$ was proved in
\cite{pego-weinstein92} for KDV and BBM equations. In \cite{lin-jfa09}, an
instability criterion as in Corollary \ref{cor-instability-parity} (ii) was
proved for KDV and BBM type equations. In \cite{pego-weinstein92} and
\cite{lin-jfa09}, an instability criterion was also given for the regularized
Boussinesq equation which takes an indefinite Hamiltonian form (i.e.
$n^{-}\left(  L\right)  =\infty$) and is therefore not included in the
framework of this paper. Recently, in \cite{kapitula-stefanov-kdv} and
\cite{pelinovsky-KDV}, an instability index theorem similar to
(\ref{index-KDV}) was given for KDV and BBM type equations under the
assumption that $\dim\ker\mathcal{L}_{0}=1$ and $dP/dc\neq0$. The proof of
\cite{kapitula-stefanov-kdv} \cite{pelinovsky-KDV} was by using ad-hoc
arguments to transform the eigenvalue problem $\partial_{x}\mathcal{L}%
_{0}u=\lambda u$ to another Hamiltonian form with a symplectic operator which
has a bounded inverse. The linear instability of solitary waves of good
Boussinesq equation ($\mathcal{M=-\partial}_{x}^{2}$) was studied in
\cite{alexander-sacks} by Evans function and in \cite{stefnov-2nd-order} by
using quadratic operator pencils. The index formula (\ref{index-KDV}) for the
good Boussinesq type equations appears to be new.

Besides giving a more unified and general index formula for linear
instability, Theorem \ref{THM: solitary-long-wave} also gives the exponential
trichotomy for $e^{tJ\mathcal{L}_{0}}$, which is an important step for
constructing invariant (stable, unstable and center) manifolds near the
translation orbits of $u_{c}$. Moreover, when $\ker\mathcal{L}_{0}%
=span\left\{  u_{c,x}\right\}  $ and $n^{-}\left(  \mathcal{L}_{0}\right)
=1$, there exists a pair of stable and unstable eigenvalues and $\mathcal{L}%
_{0}$ is positive on the codimension two center space modulo the translation
kernel. This positivity property has an important implication for the center
manifolds once constructed. For example, in \cite{jin-et-kdv}, the invariant
(stable, unstable and center) manifolds were constructed near the orbits of
unstable solitary waves of generalized KDV equation in the energy space. More
precisely, there exist 1-d stable and unstable manifolds and co-dimension two
center manifold near the translation orbits of unstable solitary waves. These
invariant manifolds give a complete description of local dynamics near
unstable traveling wave orbits. The positivity estimate (\ref{center-positive}%
) on the center subspace implies that on the codimension two center manifold,
the solitary wave $u_{c}$ is orbitally stable, which in turn also leads to the
local uniqueness of the center manifold. Any initial data not lying on the
center manifold will leave the orbit neighborhood of unstable traveling waves
exponentially fast.

\subsection{\label{subsection-periodic}Stability of periodic traveling waves}

Consider the equations (\ref{BBM})-(\ref{gBOU}) in the periodic case. For
convenience, we assume the period is $2\pi\,,\ $that is, $\left(  x,t\right)
\in\mathbf{S}^{1}\times\mathbf{R}$. A periodic traveling wave is of the form
$u\left(  x,t\right)  =u_{c,a}\left(  x-ct\right)  $, where $u_{c,a}$
satisfies the equations
\begin{equation}
\mathcal{M}u_{c,a}+\left(  1-\frac{1}{c}\right)  u_{c,a}-\frac{1}{c}f\left(
u_{c,a}\right)  =a,\ \text{(BBM)}\label{steady-BBM-P}%
\end{equation}%
\begin{equation}
\mathcal{M}u_{c,a}+cu_{c,a}-f\left(  u_{c,a}\right)  =a,\ \text{(KDV)}%
\label{steady-KDV-P}%
\end{equation}
and
\begin{equation}
\mathcal{M}u_{c,a}+\left(  1-c^{2}\right)  u_{c,a}-f\left(  u_{c,a}\right)
=a,\ \text{(gBou)}\label{steady-gBou-P}%
\end{equation}
for some constant $a$. In this subsection, we consider the perturbations of
the same period $2\pi$ (i.e. co-periodic perturbations) and leave the case of
different periods to the next subsection. The linearized equations in the
traveling frame $\left(  x-ct,t\right)  $ near traveling waves $u_{c,a}$ take
the same form (\ref{L-BBM})-(\ref{L-gBou}). Their Hamiltonian structures are
formally the same as in the case of solitary waves. However, the operator $J$
has rather different spectral properties in the periodic case. More precisely,
for solitary waves the symplectic operators $J,\ $which is $c\partial
_{x}\left(  1+\mathcal{M}\right)  ^{-1}~$for BBM, $\partial_{x}$ for KDV and
$\left(
\begin{array}
[c]{cc}%
0 & \partial_{x}\\
\partial_{x} & 0
\end{array}
\right)  $ for good Boussinesq, has no kernel in $L^{2}\left(  \mathbf{R}%
\right)  $. But for the periodic case, $J$ has nontrivial kernel in $X^{\ast}%
$. Indeed, $\ker J=span\left\{  1\right\}  $ for BBM and KDV, and
\[
\ker J=span\left\{  \vec{e}_{1},\vec{e}_{2}\right\}  =span\left\{  \left(
\begin{array}
[c]{c}%
1\\
0
\end{array}
\right)  ,\left(
\begin{array}
[c]{c}%
0\\
1
\end{array}
\right)  \right\}
\]
for good Boussinesq. This degeneracy of $J$ leads to the extra free parameter
$a$ in traveling waves.

We now discuss the consequential changes in the index formula induced by the
nontrivial kernel of $J$. For BBM type equations, define the operator
$\mathcal{L}_{0}:H^{m}\left(  \mathbf{S}^{1}\right)  \rightarrow L^{2}\left(
\mathbf{S}^{1}\right)  $ and the momentum $P$ as in (\ref{L-0-BBM}) and
(\ref{momentum-BBM}). Differentiating \eqref{steady-BBM-P}, we obtain
\[
R(\mathcal{L}_{0})\ni\mathcal{L}_{0}\partial_{a}u_{c,a}=1.
\]
Let
\begin{equation}
U_{c,a}=\partial_{a}u_{c,a},\ d_{1}=\int_{\mathbf{S}^{1}}U_{c,a}%
\ dx,\ N=\int_{\mathbf{S}^{1}}u_{c,a}dx~\left(  \text{total mass}\right)
.\label{defn-U-c-d-1}%
\end{equation}
We have $\mathcal{L}_{0}\partial_{x}u_{c,a}=0$ and from differentiating
\eqref{steady-BBM-P}
\[
\mathcal{L}_{0}\partial_{c}u_{c,a}=-\frac{1}{c}\left(  1+\mathcal{M}\right)
u_{c,a}+\frac{a}{c},
\]
and thus $J\mathcal{L}_{0}\partial_{c}u_{c,a}=-\partial_{x}u_{c,a}$. Denote
\begin{equation}
D=%
\begin{pmatrix}
\left\langle \mathcal{L}_{0}U_{c,a},U_{c,a}\right\rangle  & \left\langle
\mathcal{L}_{0}U_{c,a},\partial_{c}u_{c,a}\right\rangle \\
\left\langle \mathcal{L}_{0}U_{c,a},\partial_{c}u_{c,a}\right\rangle  &
\left\langle \mathcal{L}_{0}\partial_{c}u_{c,a},\partial_{c}u_{c,a}%
\right\rangle
\end{pmatrix}
=%
\begin{pmatrix}
d_{1} & N^{\prime}\left(  c\right)  \\
N^{\prime}\left(  c\right)   & -\frac{1}{c}dP/dc+\frac{a}{c}N^{\prime}\left(
c\right)
\end{pmatrix}
,\label{D-BBM}%
\end{equation}
that is, the matrix for $\left\langle L\cdot,\cdot\right\rangle $ on
$span\left\{  U_{c,a},\partial_{c}u_{c,a}\right\}  \subset g\ker\left(
J\mathcal{L}_{0}\right)  $. Denote $n^{\leq0}\left(  D\right)  $ to be the
number of non-positive eigenvalues of $D$.

For KDV type equations, similarly, $\mathcal{L}_{0}\partial_{x}u_{c,a}=0,$
\[
\mathcal{L}_{0}\partial_{c}u_{c,a}=-u_{c,a},\ J\mathcal{L}_{0}\partial
_{c}u_{c,a}=-\partial_{x}u_{c,a},\ \mathcal{L}_{0}\partial_{a}u_{c,a}=1,\
\]
and we define $U_{c,a},d_{1},N,D,n^{\leq0}\left(  D\right)  $ etc. as in
(\ref{defn-U-c-d-1}) and \eqref{D-BBM}.

For good Boussinesq type equations, still define $U_{c,a},d_{1},N$ as in
(\ref{defn-U-c-d-1}). Let
\begin{equation}
\vec{U}_{1}=\left(
\begin{array}
[c]{c}%
U_{c,a}\\
-cU_{c,a}%
\end{array}
\right)  ,\ \vec{U}_{2}=\left(
\begin{array}
[c]{c}%
-cU_{c,a}\\
1+c^{2}U_{c,a}%
\end{array}
\right)  ,\vec{U}_{3}=\left(
\begin{array}
[c]{c}%
-\partial_{c}u_{c,a}\\
c\partial_{c}u_{c,a}+u_{c,a}%
\end{array}
\right)  , \label{defn-U}%
\end{equation}
then
\[
L\vec{U}_{1}=\vec{e}_{1},\ \ L\vec{U}_{2}=\vec{e}_{2},\ L\vec{U}_{3}=\left(
\begin{array}
[c]{c}%
-cu_{c,a}\\
u_{c,a}%
\end{array}
\right)  .
\]
Define the matrix $D$ of $\left\langle L\cdot,\cdot\right\rangle \ $on the
space spanned by$\left\{  \vec{U}_{1},\ \vec{U}_{2},\vec{U}_{3}\right\}  $,
that is,
\begin{align}
D  &  =\left(
\begin{array}
[c]{ccc}%
\left\langle L\vec{U}_{1},\vec{U}_{1}\right\rangle  & \left\langle L\vec
{U}_{1},\vec{U}_{2}\right\rangle  & \left\langle L\vec{U}_{1},\vec{U}%
_{3}\right\rangle \\
\left\langle L\vec{U}_{1},\vec{U}_{2}\right\rangle  & \left\langle L\vec
{U}_{2},\vec{U}_{2}\right\rangle  & \left\langle L\vec{U}_{3},\vec{U}%
_{2}\right\rangle \\
\left\langle L\vec{U}_{1},\vec{U}_{3}\right\rangle  & \left\langle L\vec
{U}_{3},\vec{U}_{2}\right\rangle  & \left\langle L\vec{U}_{3},\vec{U}%
_{3}\right\rangle
\end{array}
\right) \label{D-gBou}\\
&  =\left(
\begin{array}
[c]{ccc}%
d_{1} & -cd_{1} & -N^{\prime}\left(  c\right) \\
-cd_{1} & \int_{\mathbf{S}^{1}}\left(  1+c^{2}U_{c,a}\right)  dx & cN^{\prime
}\left(  c\right)  +N\left(  c\right) \\
-N^{\prime}\left(  c\right)  & cN^{\prime}\left(  c\right)  +N\left(  c\right)
& -P^{\prime}\left(  c\right)
\end{array}
\right)  .\nonumber
\end{align}
Again $n^{\leq0}\left(  D\right)  $ denotes the number of non-positive
eigenvalues of $D$.

Since in the periodic case, the operator $\mathcal{L}_{0}$ has only discrete
spectrum which tends to $+\infty$, it is easy to verify that assumptions
(\textbf{H1-3}) are satisfied in $X= H^{\frac m2}$ for BBM and KdV type
equations and $X=H^{\frac m2} \times L^{2}$ for good Boussinesq type
equations. Thus similar to Theorem \ref{THM: solitary-long-wave}, we have

\begin{theorem}
\label{thm-periodic waves}Consider the linearized equations (\ref{L-BBM}%
)-(\ref{L-gBou}) near periodic waves $u_{c}\left(  x-ct\right)  \ $of
equations (\ref{BBM})-(\ref{gBOU}). Then: (i) the following index formula
holds
\[
k_{r}+2k_{c}+2k_{i}^{-}+k_{0}^{-}=n^{-}\left(  \mathcal{L}_{0}\right)  .
\]
(ii) the linear exponential trichotomy is true in the space $H^{\frac{m}{2}%
}\left(  \mathbf{S}^{1}\right)  $ for the linearized equations (\ref{L-BBM})
and (\ref{L-KDV}), and in $H^{\frac{m}{2}}\left(  \mathbf{S}^{1}\right)
\times L^{2}\left(  \mathbf{S}^{1}\right)  $ for (\ref{L-gBou}). (iii)
$k_{0}^{-}\geq n^{\leq0}\left(  D\right)  $, the number of non-positive
eigenvalues of the matrix $D$ defined in (\ref{D-BBM}), (\ref{defn-U}) and
(\ref{D-gBou}). Moreover, when $\ker\mathcal{L}_{0}=\left\{  \partial
_{x}u_{c,a}\right\}  $ and $D$ is nonsingular, $k_{0}^{-}=n^{-}\left(
D\right)  $ (the number of negative eigenvalues of $D$) and we have
\begin{equation}
k_{r}+2k_{c}+2k_{i}^{-}=n^{-}\left(  \mathcal{L}_{0}\right)  -n^{-}\left(
D\right)  . \label{index-nonsingular-D}%
\end{equation}

\end{theorem}

As corollaries, we have from Proposition \ref{prop-counting-k-0-1} and Remark
\ref{R:counting-k-0-1} the following linear stability/instability conditions.

\begin{corollary}
\label{cor-parity-periodic}(i) If $n^{\leq0}\left(  D\right)  \geq
n^{-}\left(  \mathcal{L}_{0}\right)  $, then the spectral stability holds.

(ii) If $\ker\mathcal{L}_{0}=span\left\{  \partial_{x} u_{c,a}\right\}  $, $D$
is nonsingular and $n^{-}\left(  \mathcal{L}_{0}\right)  -n^{-} \left(
D\right)  $ is odd, then there is linear instability.
\end{corollary}

When $n^{-}\left(  \mathcal{L}_{0}\right)  =n^{-}\left(  D\right)  $,
nonlinear orbital stability holds for (\ref{BBM})-(\ref{gBOU}) as well. More
precisely, we have

\begin{proposition}
\label{Propo-stability-periodic} When $\ker\mathcal{L}_{0}=span\left\{
\partial_{x}u_{c,a}\right\}  $, $D$ is nonsingular and $n^{-}\left(
\mathcal{L}_{0}\right)  =n^{-}\left(  D\right)  $, then there is orbital
stability in $X$ of the traveling waves $u_{c,a}\left(  x-ct\right)  $ of
equations (\ref{BBM})-(\ref{gBOU}) for perturbations of the same period.
\end{proposition}

\begin{proof}
Here we sketch the proof based on the standard Lyapunov functional method
(e.g. \cite{gss-87}, \cite{gss-90}). Consider the KDV type equation
(\ref{kdv}). It has three invariants: (1) energy $E\left(  u\right)
=\int\left[  \frac{1}{2}u\mathcal{M}u-F\left(  u\right)  \right]  dx,\ $with
$F(u)=\int_{0}^{u}f(u^{\prime})du^{\prime};$ (2) momentum $P\left(  u\right)
=\frac{1}{2}\int u^{2}dx\ $and (3) total mass $N\left(  u\right)  =\int udx$.
Define the invariant
\[
I\left(  u\right)  =E\left(  u\right)  +cP\left(  u\right)  -aN\left(
u\right)  ,
\]
then $I^{\prime}\left(  u_{c,a}\right)  =0$ if and only if $u_{c,a}$ is a
traveling wave solution satisfying (\ref{steady-KDV-P}). So
\begin{equation}
I\left(  u\right)  -I\left(  u_{c,a}\right)  =\left\langle \mathcal{L}%
_{0}\delta u,\delta u\right\rangle +O\left(  \left\Vert \delta u\right\Vert
^{3}\right)  \text{, where }\delta u=u-u_{c,a}. \label{Liapunov-KDV}%
\end{equation}
Denote
\[
X_{1}=\left\{  u\in H^{\frac{m}{2}}\ |\ \left\langle u,\mathcal{L}_{0}%
U_{c,a}\right\rangle =\left\langle u,\mathcal{L}_{0}\partial_{c}%
u_{c,a}\right\rangle =0\right\}
\]
to be the orthogonal complement of $X_{2}=span\left\{  U_{c,a},\partial
_{c}u_{c,a}\right\}  $ in $\left\langle \mathcal{L}_{0}\cdot,\cdot
\right\rangle .$ Since
\[
\mathcal{L}_{0}U_{c,a}=N^{\prime}(u_{c,a}),\quad\mathcal{L}_{0}\partial
_{c}u_{c,a}=P^{\prime}(u_{c,a}),
\]
$X_{1}$ is the tangent space of the intersection of the level surfaces of the
conserved momentum $P$ and mass $N$. With $D$ assumed to be nonsingular,
$X_{2}$ roughly represents the gradient directions of $P$ and $N$ and thus
$X=X_{1}\oplus X_{2}$. Moreover, we have $\left\langle \mathcal{L}_{0}%
\cdot,\cdot\right\rangle |_{X_{1}}\geq0$ since
\[
n^{-}\left(  \mathcal{L}_{0}|_{X_{1}}\right)  =n^{-}\left(  \mathcal{L}%
_{0}\right)  -n^{-}\left(  \mathcal{L}_{0}|_{X_{2}}\right)  =n^{-}\left(
\mathcal{L}_{0}\right)  -n^{-}\left(  D\right)  =0\text{. }%
\]
We further decompose
\[
X_{1}=Y\oplus span\{\partial_{x}u_{c,a}\},\;\text{ where }Y=\{u\in X_{1}%
\mid(u,\partial_{x}u_{c,a})_{X}=0\}.
\]
Since $\ker\mathcal{L}_{0}=span\{\partial_{x}u_{c,a}\}$, there exists
$c_{0}>0$ such that $\langle\mathcal{L}_{0}\delta u,\delta u\rangle\geq
c_{0}\left\Vert \delta u\right\Vert ^{2}$ for any $\delta u\in Y$.

Suppose $u(t)$ is solution with $u(0)$ close to $u_{c,a}$ and $h(t)\in
\mathbf{S}^{1}$ satisfies
\[
\left\Vert u-u_{c,a}\left(  \cdot-h\right)  \right\Vert =\min_{y\in
\mathbf{S}^{1}}\left\Vert u-u_{c,a}\left(  \cdot-y\right)  \right\Vert ,
\]
then $w(t)=u(t)-u_{c,a}\big(\cdot-h(t)\big)\in Y\oplus X_{2}$. By using the
conservation of $P$ and $N$ to control the $X_{2}$ components of $w(t)$, and
the uniform positivity of $\mathcal{L}_{0}$ on $Y$ and \eqref{Liapunov-KDV} to
control the $Y$ component, we obtain the orbital stability. More details of
such arguments can be found for example in \cite{gss-87} \cite{lin-wang-zeng}.

For BBM type equations, the Lyapunov functional is
\[
I\left(  u\right)  =cP\left(  u\right)  -E\left(  u\right)  -caN\left(
u\right)  ,
\]
where the energy functional $E\left(  u\right)  =\int\left(  \frac{1}{2}%
u^{2}+F\left(  u\right)  \right)  \ dx$ and $P\left(  u\right)  $ is defined
in (\ref{momentum-BBM}). The rest of the proof is the same as in the KDV case.
For good Boussinesq type equations (\ref{gBOU}), we write it as a first order
Hamiltonian system
\[
\partial_{t}\left(
\begin{array}
[c]{c}%
u\\
v
\end{array}
\right)  =J\ \nabla E\left(  u,v\right)  ,\
\]
where $J=\left(
\begin{array}
[c]{cc}%
0 & \partial_{x}\\
\partial_{x} & 0
\end{array}
\right)  $ and the energy functional
\[
E\left(  u,v\right)  =\frac{1}{2}\left(  \mathcal{M}u,u\right)  +\int\left(
\frac{1}{2}v^{2}+\frac{1}{2}u^{2}-F\left(  u\right)  \right)  dx.
\]
For the traveling wave solution $\left(  u_{c,a}\left(  x-ct\right)
,v_{c,a}\left(  x-ct\right)  \right)  $, $u_{c,a}$ satisfies
(\ref{steady-BBM-P}) and $v_{c,a}=-cu_{c,a}$. Let $\vec{u}=\left(  u,v\right)
^{T}$ and construct the Lyapunov functional
\[
I\left(  \vec{u}\right)  =E\left(  \vec{u}\right)  +cP\left(  \vec{u}\right)
-aN_{1}\left(  \vec{u}\right)  ,
\]
where
\[
P\left(  \vec{u}\right)  =\int uv\ dx,\ N_{1}\left(  \vec{u}\right)  =\int
udx,\ N_{2}\left(  \vec{u}\right)  =\int vdx.
\]
Then $I^{\prime}\left(  \vec{u}_{c,a}\right)  =0$. The rest of the proof is
the same.
\end{proof}

Compared with solitary waves, the periodic traveling waves have richer
structures. They consist of a three parameter (period $T$, speed $c$, and
integration constant $a$) family of solutions and different type of
perturbations (co-periodic, multiple periodic, localized etc.) can be
considered. In recent years, there have been lots of works on
stability/instability of periodic traveling waves of dispersive PDEs. For
co-periodic perturbations (i.e. of the same period), the nonlinear orbital
stability were proved for various dispersive models (e.g. \cite{pava-book}
\cite{bona-et-cnoidal} \cite{johnson13} \cite{hur-johnson-stability}
\cite{bronski-et-quadratic-pencils} \cite{gavage16}) by using Liapunov
functionals. These stability results were proved for the cases when $\dim
\ker\left(  \mathcal{L}_{0}\right)  =1$ and $n^{-}\left(  \mathcal{L}%
_{0}\right)  =n^{-}\left(  D\right)  $ as in Proposition
\ref{Propo-stability-periodic}. An instability index formula similar to
(\ref{index-nonsingular-D}) was proved for KDV type equations
(\cite{kappitula-haragus} \cite{deconinck-kapitula}
\cite{bronski-et-index-KDV}). In these papers, some conditions (e.g.
Assumption 2.1 in \cite{kappitula-haragus} and Assumption 3 in
\cite{deconinck-kapitula}) were imposed to ensure that the generalized
eigenvectors of $J\mathcal{L}_{0}$ form a basis of $X$. These assumptions can
be checked for the case $\mathcal{M=-\partial}_{x}^{2}$. In Theorem
\ref{thm-periodic waves}, we do not need such assumptions on the completion of
generalized eigenspaces of $J\mathcal{L}_{0}\ $and therefore we can get the
index formula for very general nonlocal operators $\mathcal{M}$. In
\cite{bronski-et-quadratic-pencils}, an index formula was proved for periodic
traveling waves of good Boussinesq equation ($\mathcal{M=-\partial}_{x}^{2}%
$)\ by using the theory of quadratic operator pencils. In \cite{gavage16}, a
parity instability criterion (as in Corollary \ref{cor-parity-periodic} (ii))
was proved for periodic waves of several Hamiltonian PDEs including
generalized KDV equations by using Evans functions.

Besides providing a unified way to get instability index formula and the
stability criterion, we could also use the exponential trichotomy of
$e^{J\mathcal{L}_{0}}$ in Theorem \ref{thm-periodic waves} to construct
invariant manifolds near the orbit of unstable periodic traveling waves.
Moreover, as in the case of solitary waves, when $\dim\ker\left(
\mathcal{L}_{0}\right)  =1,\ D$ is nonsingular and $k_{i}^{-}=0$, we have
orbital stability and local uniqueness of the center manifolds once constructed.

\subsection{Modulational Instability of periodic traveling waves}

\label{SS:modulational}

Consider periodic traveling waves $u_{c,a}\left(  x-ct\right)  $ studied in
the Subsection \ref{subsection-periodic}. Assume the conditions in Proposition
\ref{Propo-stability-periodic}, so that $u_{c,a}$ is orbitally stable under
perturbations of the same period. In this subsection, we consider modulational
instability of periodic traveling waves, under perturbations of different
period or even localized perturbations. The modulational instability, also
called Benjamin-Feir or side-band instability in the literature, is a very
important instability mechanism in lots of dispersive and fluid models. Again,
we assume the minimal period of the traveling wave $u_{c,a}$ is $2\pi$. We
focus on KDV type equations (\ref{kdv}), and the consideration for BBM and
good-Boussinesq type equations is similar. We assume the Fourier symbol
$\alpha\left(  \xi\right)  $ of the operator $\mathcal{M}$ is even, so that
$\mathcal{M}$ is a real operator. Based on the standard Floquet-Bloch theory,
we seek bounded eigenfunction $\phi(x)$ of the linearized operator
$J\mathcal{L}_{0}$ in the form of $\phi(x)=e^{ikx}v_{k}(x)$, where
$k\in\mathbf{R}$ is a parameter and $v_{k}\in L^{2}(\mathbf{S}^{1})$. Recall
that $J=\partial_{x}$ and $\mathcal{L}_{0}:=\mathcal{M}+c-f^{\prime}\left(
u_{c,a}\right)  $. It leads us to the one-parameter family of eigenvalue
problems
\[
J\mathcal{L}_{0}e^{ikx}v_{k}(x)=\lambda(k)e^{ikx}v_{k}(x),
\]
or equivalently $\mathcal{J}_{k}\mathcal{L}_{k}v_{k}=\lambda\left(  k\right)
v_{k}$, where
\begin{equation}
\mathcal{J}_{k}=\partial_{x}+ik,\ \mathcal{L}_{k}=\mathcal{M}_{k}%
\mathcal{+}c-f^{\prime}(u_{c,a}). \label{defn-Jk-Lk}%
\end{equation}
Here, $\mathcal{M}_{k}$ is the Fourier multiplier operator with the symbol
$\alpha(\xi+k)$. We say that $u_{c,a}$ is linearly modulationally unstable if
there exists $k\notin\mathbf{Z}$ such that the operator $\mathcal{J}%
_{k}\mathcal{L}_{k}$ has an unstable eigenvalue $\lambda(k)$ with
$\operatorname{Re}\lambda(k)>0$ in the space $L^{2}(\mathbf{S}^{1})$.

Since $\mathcal{J}_{k}$ and $\mathcal{L}_{k}$ are complex operators, we first
reformulate the problem in terms of real operators to use the general theory
in this paper. Consider
\begin{equation}
\phi(x)=\cos\left(  kx\right)  u_{1}\left(  x\right)  +\sin\left(  kx\right)
u_{2}\left(  x\right)  , \label{E:phi-real-form}%
\end{equation}
where $u_{1},u_{2}\in L^{2}(\mathbf{S}^{1})$ are real functions. By
definition,%
\[
\mathcal{M}\left(  e^{ikx}u\left(  x\right)  \right)  =e^{ikx}\mathcal{M}%
_{k}u.
\]
We decompose
\[
\mathcal{M}_{k}=\mathcal{M}_{k}^{e}+i\mathcal{M}_{k}^{o},\quad\mathcal{M}%
_{-k}=\mathcal{M}_{k}^{e}-i\mathcal{M}_{k}^{o}%
\]
where $\mathcal{M}_{k}^{e},\mathcal{M}_{k}^{o}$ are operators with Fourier
multipliers
\[
\alpha_{k}^{e}\left(  \xi\right)  =\frac{1}{2}\left(  \alpha(\xi+k)+\alpha
(\xi-k)\right)
\]
and
\[
\alpha_{k}^{o}\left(  \xi\right)  =-\frac{i}{2}\left(  \alpha(\xi
+k)-\alpha(\xi-k)\right)  .
\]
Then $\mathcal{M}_{k}^{e},\ \mathcal{M}_{k}^{o}$ are self-adjoint and
skew-adjoint respectively. Since $\overline{\alpha_{k}^{e,o}\left(
\xi\right)  }=\alpha_{k}^{e,o}\left(  -\xi\right)  $, $\mathcal{M}_{k}^{e}$
and $\mathcal{M}_{k}^{o}$ map real functions to real. In particular, for
$\mathcal{M=-\partial}_{x}^{2}$, we have $\mathcal{M}_{k}^{e}=-\partial
_{x}^{2}+k^{2}$ and $\mathcal{M}_{k}^{o}=-2k\partial_{x}$. By using
\begin{equation}
\phi(x)=\frac{e^{ikx}}{2}\left(  u_{1}-iu_{2}\right)  +\frac{e^{-ikx}}%
{2}\left(  u_{1}+iu_{2}\right)  \label{E:complex-exp}%
\end{equation}
and via simple computations, we obtain
\[
\mathcal{M}\phi=\cos\left(  kx\right)  \left(  \mathcal{M}_{k}^{e}%
u_{1}+\mathcal{M}_{k}^{o}u_{2}\right)  +\sin\left(  kx\right)  \left(
-\mathcal{M}_{k}^{o}u_{1}+\mathcal{M}_{k}^{e}u_{2}\right)  ,
\]
and
\[
J\phi=\cos\left(  kx\right)  \left(  \partial_{x}u_{1}+ku_{2}\right)
+\sin\left(  kx\right)  \left(  \partial_{x}u_{2}-ku_{1}\right)  .
\]
Define the operators
\begin{equation}
J_{k}=%
\begin{pmatrix}
\partial_{x} & k\\
-k & \partial_{x}%
\end{pmatrix}
,\quad L_{k}=%
\begin{pmatrix}
\mathcal{M}_{k}^{e}\mathcal{+}c-f^{\prime}(u_{c,a}) & \mathcal{M}_{k}^{o}\\
-\mathcal{M}_{k}^{o} & \mathcal{M}_{k}^{e}\mathcal{+}c-f^{\prime}(u_{c,a})
\end{pmatrix}
. \label{E:JkLk}%
\end{equation}
Then $J_{k},\ L_{k}$ are skew-adjoint and self-adjoint real operators and
\[
\mathcal{J}\mathcal{L}_{0}\phi=\big(\cos(kx),\sin(kx)\big)J_{k}L_{k}%
\begin{pmatrix}
u_{1}\\
u_{2}%
\end{pmatrix}
.
\]

As always in the spectral analysis, $u_{1}$ and $u_{2}$, as well as operator
$J_{k}L_{k}$ and quadratic forms $\langle L_{k}\cdot,\cdot\rangle$ and
$\langle\cdot,J_{k}\cdot\rangle$, need to be complexified. By using operators
$\mathcal{J}_{k}$ and $\mathcal{L}_{k}$ we can diagonalize $J_{k}L_{k}$ and
$L_{k}$ blockwisely. In fact, let
\[
w_{1}=\frac{1}{2}(u_{1}-iu_{2}),\quad w_{2}=\frac{1}{2}(u_{1}+iu_{2}),\quad S%
\begin{pmatrix}
u_{1}\\
u_{2}%
\end{pmatrix}
=%
\begin{pmatrix}
w_{1}\\
w_{2}%
\end{pmatrix}
.
\]
One may compute using \eqref{E:complex-exp} and the definition of
$\mathcal{J}_{k}$ and $\mathcal{L}_{k}$
\begin{equation}
L_{k}=S^{-1}%
\begin{pmatrix}
\mathcal{L}_{k} & 0\\
0 & \mathcal{L}_{-k}%
\end{pmatrix}
S,\quad J_{k}L_{k}=S^{-1}%
\begin{pmatrix}
\mathcal{J}_{k}\mathcal{L}_{k} & 0\\
0 & \mathcal{J}_{-k}\mathcal{L}_{-k}%
\end{pmatrix}
S.\label{E:conjugacy1}%
\end{equation}
Moreover, $\mathcal{L}_{-k}$ and $\mathcal{J}_{-k}\mathcal{L}_{-k}$ are the
complex conjugates of $L_{k}$ and $J_{k}L_{k}$ respectively, namely,
\begin{equation}
\mathcal{L}_{-k}w=\overline{\mathcal{L}_{k}\bar{w}},\quad\mathcal{J}%
_{-k}\mathcal{L}_{-k}w=\overline{\mathcal{J}_{k}\mathcal{L}_{k}\bar{w}%
}.\label{E:conjugacy2}%
\end{equation}
From the above relations, we obtain
\[
n^{-}(L_{k})=n^{-}(\mathcal{L}_{k})+n^{-}(\mathcal{L}_{-k})=2n^{-}%
(\mathcal{L}_{k}),
\]
where $n^{-}(\mathcal{L}_{k})$ is understood as the negative index of the
complex Hermitian form $\langle\mathcal{L}_{k}\cdot,\cdot\rangle$. Moreover,
\eqref{E:conjugacy2} implies that $\lambda\in\sigma(\mathcal{J}_{k}%
\mathcal{L}_{k})$ if and only if $\bar{\lambda}\in\sigma(\mathcal{J}%
_{-k}\mathcal{L}_{-k})$, with $\ker(\bar{\lambda}-\mathcal{J}_{-k}%
\mathcal{L}_{-k})^{n}$ consisting exactly of the complex conjugates of the
functions in $\ker(\lambda-\mathcal{J}_{k}\mathcal{L}_{k})^{n}$ for any $n>0$.
Next, it is easy to see that $J_{k}L_{k}$, as well as $\mathcal{J}%
_{k}\mathcal{L}_{k}$, has compact resolvents and thus $\sigma\left(
J_{k}L_{k}\right)  $, as well as $\sigma(\mathcal{J}_{k}\mathcal{L}_{k})$,
consists of only discrete eigenvalues of finite algebraic multiplicity.
Therefore, by Proposition \ref{P:non-deg}, for any purely imaginary eigenvalue
$i\mu\in$ $\sigma\left(  J_{k}L_{k}\right)  $, $L_{k}$ is non-degenerate on
the finite dimensional eigenspace $E_{i\mu}$, and thus $n^{\leq0}%
(L_{k}|_{E_{i\mu}})=n^{-}(L_{k}|_{E_{i\mu}})$. Let $(k_{r},k_{c},k_{i}%
^{-},k_{0}^{-})$ be the indices defined in (\ref{E:kr-kc}), (\ref{E:ki}), and
\eqref{defn-k-0} for $J_{k}L_{k}$, and $(\tilde{k}_{r},\tilde{k}_{0}^{-})$ be
the corresponding indices for the positive and zero eigenvalues of
$\mathcal{J}_{k}\mathcal{L}_{k}$. Let $\tilde{k}_{c}$ be the sum of algebraic
multiplicities of eigenvalues of $\mathcal{J}_{k}\mathcal{L}_{k}$ in the first
and the fourth quadrants, $\tilde{k}_{i}^{-}$ be the total number of negative
dimensions of $\langle\mathcal{L}_{k}\cdot,\cdot\rangle$ restricted to the
subspaces of generalized eigenvectors of nonzero purely imaginary eigenvalues
of $\mathcal{J}_{k}\mathcal{L}_{k}$. On the one hand, \eqref{E:conjugacy1} and
\eqref{E:conjugacy2} imply
\[
k_{r}=2\tilde{k}_{r},\quad k_{c}=\tilde{k}_{c},\quad k_{i}^{-}=\tilde{k}%
_{i}^{-},\quad k_{0}^{-}=2\tilde{k}_{0}^{-}.
\]
On the other hand, Theorem \ref{theorem-counting} implies
\begin{equation}
k_{r}+2k_{c}+2k_{i}^{-}+k_{0}^{-}=2n^{-}\left(  \mathcal{L}_{k}\right)
.\label{index-modulational-double}%
\end{equation}
Therefore, we obtain

\begin{proposition}
\label{P:modu-index} For any $k\in\left(  0,1\right)  $,
\begin{equation}
\tilde{k}_{r}+\tilde{k}_{c}+\tilde{k}_{i}^{-}+\tilde{k}_{0}^{-}=n^{-}\left(
\mathcal{L}_{k}\right)  \text{. } \label{index-modulational}%
\end{equation}

\end{proposition}

The modulational instability occurs if $\tilde k_{r}\ne0$ or $\tilde k_{c}
\ne0$.

\begin{remark}
Note that $\mathcal{J}_{k}$ is invertible for any $k\notin\mathbf{Z}$. With a
more concrete form of $\mathcal{M}$, it is possible to determine $\tilde
{k}_{0}^{-}$. \newline$\bullet$ Firstly, if $\ker\mathcal{L}_{0}$ is known
(recall $\partial_{x}u_{c,a}\in\ker\mathcal{L}_{0}$), then one may study
$\ker\mathcal{L}_{k}$, as well as $\tilde{k}_{0}^{-}$, for $0<|k|<<1$ through
asymptotic analysis.\newline$\bullet$ If $\mathcal{M}=-\partial_{xx}$, then
$\ker\mathcal{L}_{k}=\{0\}$ for any $k\in(0,1)$ (and thus for any
$k\notin\mathbf{Z}$). In fact, in this case,
\[
\mathcal{L}_{0}=-\partial_{xx}+c-f^{\prime}(u_{c,a}),\quad v\in\ker
\mathcal{L}_{k}\,\Longleftrightarrow\,e^{ikx}v\in\ker\mathcal{L}_{0}%
\]
and $\ker\mathcal{L}_{0}=span\{\partial_{x}u_{c,a}\}$. Suppose $\mathcal{L}%
_{k}\,$\ has nontrivial kernel for some $k\in\left(  0,1\right)  $ and $0\neq
v\in\ker\mathcal{L}_{k}$. Denote $v_{0}\triangleq\partial_{x}u_{c,a}$, then
the Wronskian of $v_{0}$ and $\,e^{ikx}v$ satisfies
\[
W(x)=e^{ikx}(v_{x}v_{0}-vv_{0x}+ikvv_{0})=const.
\]
Since $v$ and $v_{0}$ are $2\pi$-periodic and $k\in(0,1)$, it must hold that
\begin{equation}
v_{x}v_{0}-vv_{0x}+ikvv_{0}=0. \label{E:Wronskian}%
\end{equation}
We claim $v(x)\neq0$ for any $x\in\mathbf{S}^{1}$. In fact, if $v(x_{0})=0$,
then $v_{x}(x_{0})\neq0$ and \eqref{E:Wronskian} imply $v_{0}(x_{0})=0$. The
uniqueness of the solution to the ODE $\mathcal{L}_{0}u=0$ leads to the
proportionality between $v_{0}$ and $e^{ikx}v$, a contradiction to $k\in(0,1)$
and the $2\pi$-periodicity of $v(x)$. Now that $v(x)\neq0$,
\eqref{E:Wronskian} implies $\frac{v_{0}}{v}=Ce^{-ikx}$, which is again a contradiction.
\end{remark}

\begin{remark}
The above index formula (\ref{index-modulational}) was proved in
\cite{kappitula-haragus} for the case when $\ker\mathcal{L}_{k}=\{0\}$, with
additional assumptions to ensure that the generalized eigenfunctions of
$\mathcal{J}_{k}\mathcal{L}_{k}$ form a complete basis of $L^{2}\left(
\mathbf{S}^{1}\right)  $ as assumed in the case of co-periodic perturbations.
Proposition \ref{P:modu-index} is proved without such assumptions.
\end{remark}

\begin{remark}
We can also consider the case when the operator $\mathcal{M}$ is a smoothing
operator, that is, $\Vert\mathcal{M}(\cdot)\Vert_{H^{r}}\sim\Vert\cdot
\Vert_{L^{2}}$ for some $r>0$. One example is the Whitham equation which is a
KDV type equation (\ref{kdv}) with the symbol of $\mathcal{M}$ being
$\sqrt{\frac{\tanh\xi}{\xi}}$ and thus $r=\frac{1}{2}$. In this case, if we
assume that
\begin{equation}
-c-\Vert f^{\prime}(u_{c})\Vert_{L^{\infty}(\mathbb{T}_{2\pi})}\geqslant
\epsilon>0, \label{assumption-whitham}%
\end{equation}
then $\mathcal{L}_{0}$ and $\mathcal{L}_{k}$ are compact perturbations of the
positive operator $-c+f^{\prime}\left(  u_{c,a}\right)  $ so that
$n^{-}\left(  -\mathcal{L}_{0}\right)  ,n^{-}\left(  -\mathcal{L}_{k}\right)
<\infty$. Then the index formula
\[
\bar{k}_{r}+\bar{k}_{c}+\ k_{i}^{-}+k_{0}^{-}=n^{-}\left(  -\mathcal{L}%
_{k}\right)
\]
is still true for the operator $\mathcal{J}_{k}\mathcal{L}_{k}$ $,k\in\left(
0,1\right)  $. The assumption (\ref{assumption-whitham}) can be verified
(\cite{lin-liao-jin-modulational}) for small amplitude periodic traveling
waves of Whitham equation with $f\left(  u\right)  =u^{2}$.
\end{remark}

Under the conditions of orbital stability in Proposition
\ref{Propo-stability-periodic}, the spectra of the operator $J\mathcal{L}_{0}$
in $L^{2}\left(  \mathbf{S}^{1}\right)  $ lie on the imaginary axis and are
all discrete. Moreover, the non-degeneracy of the matrix $D$ (defined by
(\ref{D-BBM})) implies that the generalized kernel of $J\mathcal{L}_{0}$ is
spanned by $\left\{  \partial_{x}u_{c,a},\partial_{c}u_{c,a},U_{c,a}\right\}
$. For $k\in\left(  0,1\right)  $ small, it is natural to study the spectra of
$\mathcal{J}_{k}\mathcal{L}_{k}$ by the perturbation theory. Even though the
results in Subsection \ref{SS:SS} and Section \ref{S:perturbation} do not
apply directly as $\mathcal{J}_{k}-J:X^{\ast}\rightarrow X$ is not bounded,
the ideas there and the property that $\mathcal{J}_{k}\mathcal{L}_{k}$ has
only isolated eigenvalues still yield the desired results. We start with the
following lemma on the resolvent of $\mathcal{J}_{k}\mathcal{L}_{k}$.

\begin{lemma}
\label{L:C0-resolvent} Assume that the symbol $\alpha(\xi)$ of $\mathcal{M}$
satisfies $a |\xi|^{m} \le\alpha(\xi) \le b |\xi|^{m}$, $a, b>0$, $m>0$, for
large $\xi$ and
\begin{equation}
\label{E:M-assump-1}\lim_{\rho\to0} \sup_{\xi\in\mathbf{Z}} \frac{|\alpha(\xi+
\rho) - \alpha(\xi)|}{1+ |\xi|^{m}} \to0,
\end{equation}
then the resolvent $(\lambda- \mathcal{J}_{k}\mathcal{L}_{k})^{-1}$ is
continuous in $k\in[0,1]$.
\end{lemma}

\begin{proof}
Fix $k\in\lbrack0,1]$. From \eqref{defn-Jk-Lk}, one can compute
\[
\mathcal{J}_{k^{\prime}}\mathcal{L}_{k^{\prime}}-\mathcal{J}_{k}%
\mathcal{L}_{k}=(\partial_{x}+ik)(\mathcal{M}_{k^{\prime}}-\mathcal{M}%
_{k})+i(k^{\prime}-k)\big(\mathcal{M}_{k^{\prime}}+c-f^{\prime}(u_{c,a})\big).
\]
On the one hand, there exists $a_{0}\neq0$ such that $a_{0}+(\partial
_{x}+ik)\mathcal{M}_{k}$ has a compact inverse on $X$. We obtain from
\eqref{E:M-assump-1}
\begin{equation}
|\big(a_{0}+(\partial_{x}+ik)\mathcal{M}_{k}\big)^{-1}(\mathcal{J}_{k^{\prime
}}\mathcal{L}_{k^{\prime}}-\mathcal{J}_{k}\mathcal{L}_{k})|\rightarrow
0\;\text{ as }k^{\prime}\rightarrow k.\label{E:C0-resolvent-1}%
\end{equation}
On the other hand, \eqref{defn-Jk-Lk} and $m>0$ imply that
\begin{align*}
&  I+\big(a_{0}+(\partial_{x}+ik)\mathcal{M}_{k}\big)^{-1}(\lambda
-\mathcal{J}_{k}\mathcal{L}_{k})\\
= &  \big(I+(\partial_{x}+ik)\mathcal{M}_{k}\big)^{-1}(\lambda+a_{0}%
-(\partial_{x}+ik)(c-f^{\prime}(u_{c,a})\big)
\end{align*}
is compact. Therefore, $A=\big(a_{0}+(\partial_{x}+ik)\mathcal{M}%
_{k}\big)^{-1}(\lambda-\mathcal{J}_{k}\mathcal{L}_{k})$ is a Fredholm operator
of index $0$. Suppose $\lambda\notin\sigma(\mathcal{J}_{k}\mathcal{L}_{k})$,
then $A$ is injective and thus $A^{-1}$ is bounded on $X$. Along with
\eqref{E:C0-resolvent-1}, we obtain
\[
|(\lambda-\mathcal{J}_{k}\mathcal{L}_{k})^{-1}(\mathcal{J}_{k^{\prime}%
}\mathcal{L}_{k^{\prime}}-\mathcal{J}_{k}\mathcal{L}_{k})|=|A^{-1}%
\big(a_{0}+(\partial_{x}+ik)\mathcal{M}_{k}\big)^{-1}(\mathcal{J}_{k^{\prime}%
}\mathcal{L}_{k^{\prime}}-\mathcal{J}_{k}\mathcal{L}_{k})|\rightarrow0
\]
as $k^{\prime}\rightarrow k$. From
\[
\lambda-\mathcal{J}_{k^{\prime}}\mathcal{L}_{k^{\prime}}=(\lambda
-\mathcal{J}_{k}\mathcal{L}_{k})\big(I-(\lambda-\mathcal{J}_{k}\mathcal{L}%
_{k})^{-1}(\mathcal{J}_{k^{\prime}}\mathcal{L}_{k^{\prime}}-\mathcal{J}%
_{k}\mathcal{L}_{k})\big),
\]
we obtain the continuity of the resolvent $(\lambda-\mathcal{J}_{k}%
\mathcal{L}_{k})^{-1}$ in $k\in\lbrack0,1]$.
\end{proof}

\begin{remark}
The assumption (\ref{E:M-assump-1}) is clearly satisfied if $\alpha(\xi)\in
C^{1}\left(  \mathbf{R}\right)  $ and
\[
\limsup_{\left\vert \xi\right\vert \rightarrow\infty}\frac{\alpha^{\prime}%
(\xi)}{\left\vert \xi\right\vert ^{m}}< \infty.
\]

\end{remark}

Next we show that when $k$ is small enough, the unstable modes of
$\mathcal{J}_{k}\mathcal{L}_{k}$ can only bifurcate from the zero eigenvalue
of $J\mathcal{L}_{0}$.

\begin{proposition}
Suppose $\ker\mathcal{L}_{0}=span \left\{  \partial_{x} u_{c,a}\right\}  $,
$D$ is nonsingular, $n^{-}\left(  \mathcal{L}_{0}\right)  =n^{-} \left(
D\right)  $ and \eqref{E:M-assump-1} holds. Then for any $\delta>0$, there
exists $\varepsilon_{0}>0$ such that if $\left\vert k\right\vert
<\varepsilon_{0}$, then $\sigma\left(  \mathcal{J}_{k}\mathcal{L}_{k}\right)
\cap\left\{  \left\vert z\right\vert \geq\delta\right\}  \subset i\mathbf{R}$.
\end{proposition}

\begin{proof}
Since $0$ is an isolated spectral point of $J\mathcal{L}_{0}$, there exists
$\delta_{0}>0$ such that $\lambda\notin\sigma(J\mathcal{L}_{0})$ as long as
$0<|\lambda|\leq\delta_{0}$. Without loss of generality, assume $0<\delta
<\delta_{0}$. Lemma \ref{L:C0-resolvent} implies $\lambda\notin\sigma
(\mathcal{J}_{k}\mathcal{L}_{k})$ for $0<|k|<<1$. Let
\[
P(k)=\frac{1}{2\pi i}\oint_{|\lambda|=\delta}(\lambda-\mathcal{J}%
_{k}\mathcal{L}_{k})^{-1}d\lambda,\quad Z_{k}=P(k)X,\quad Y_{k}%
=\big(I-P(k)\big)X.
\]
The standard spectral theory implies that $P(k)$ is continuous in $k$, $Y_{k}$
and $Z_{k}$ are invariant under $\mathcal{J}_{k}\mathcal{L}_{k}$, and
\[
|\lambda|<\delta,\;\forall\lambda\in\sigma(\mathcal{J}_{k}\mathcal{L}%
_{k}|_{Z_{k}})\;\text{ and }\;|\lambda|>\delta,\;\forall\lambda\in
\sigma(\mathcal{J}_{k}\mathcal{L}_{k}|_{Y_{k}}).
\]
For $k=0$, our assumptions imply that $Z_{0}=span\left\{  \partial_{x}%
u_{c,a},\partial_{c}u_{c,a},U_{c,a}\right\}  $. Therefore, $Z_{k}$ close to
$Z_{0}$ is a 3-dim invariant subspace of $\mathcal{J}_{k}\mathcal{L}_{k}$ with
small eigenvalues containing $\ker\mathcal{L}_{k}$. Moreover, the assumption
\[
n^{-}(\mathcal{L}_{0})=n^{-}(D)=n^{-}(\mathcal{L}_{0}|_{Z_{0}})
\]
and the $\mathcal{L}_{0}$-orthogonality between $Z_{0}$ and $Y_{0}$ imply that
$\mathcal{L}_{0}$ is uniformly positive definite on $Y_{0}$. As $\mathcal{L}%
_{k}:X=H^{\frac{m}{2}}\rightarrow X^{\ast}=H^{-\frac{m}{2}}$ is continuous in
$k$, there exists $\alpha>0$ such that $\langle\mathcal{L}_{k}u,u\rangle
>\alpha\Vert u\Vert^{2}$ for all $u\in Y_{k}$. Clearly, $\mathcal{J}%
_{k}\mathcal{L}_{k}|_{Y_{k}}$ is skew-adjoint with respect to the equivalent
inner product given by $\langle\mathcal{L}_{k}\cdot,\cdot\rangle$ on $Y_{k}$,
therefore $\sigma(\mathcal{J}_{k}\mathcal{L}_{k}|_{Y_{k}})\subset i\mathbf{R}$
and the proposition follows.
\end{proof}

Since $\dim\ker\left(  J\mathcal{L}_{0}\right)  =3$, the perturbation of zero
eigenvalue of $J\mathcal{L}_{0}\ $for $\mathcal{J}_{k}\mathcal{L}_{k}$
$\left(  0<k\ll1\right)  $ can be reduced to the eigenvalue perturbation of a
$3$ by $3$ matrix. This had been studied extensively in the literature and
instability conditions were obtained for various dispersive models. See the
survey \cite{hur-et-survey-15} and the references therein.

Recently, it was proved in \cite{lin-liao-jin-modulational} that linear
modulational instability of the traveling wave $u_{c}\left(  x-ct\right)
\ $also implies the nonlinear instability for both multi-periodic and
localized perturbations. The semigroup estimates of $e^{tJ\mathcal{L}_{0}}$
play an important role on this proof of nonlinear instability. We sketch these
estimates below, as an example of the application of Theorem
\ref{theorem-dichotomy} on the exponential trichotomy of linear Hamiltonian
PDE. First, if $u_{c}$ is linearly modulationally unstable, then there exists
a rational $k_{0}=\frac{p}{q}\in\left(  0,1\right)  $ such that $\mathcal{J}%
_{k_{0}}\mathcal{L}_{k_{0}}$ has an unstable eigenvalue. By the definition of
$\mathcal{J}_{k_{0}}\mathcal{L}_{k_{0}}$, this implies that the operator
$J\mathcal{L}_{0}$ has an unstable eigenvalue on the $2\pi q$ periodic space
$L^{2}\left(  \mathbf{S}_{2\pi q}^{1}\right)  $ with an eigenfunction of the
form $e^{ik_{0}x}u\left(  x\right)  $ $\left(  u\in L^{2}\left(
\mathbf{S}^{1}\right)  \right)  $. The exponential trichotomy of the semigroup
$e^{tJ\mathcal{L}_{0}}$ on the space $H^{s}\left(  \mathbf{S}_{2\pi q}%
^{1}\right)  $ $\left(  s\geq\frac{m}{2}\right)  $ follows directly by Theorem
\ref{theorem-dichotomy}. This is used in \cite{lin-liao-jin-modulational} to
prove nonlinear orbital instability of $u_{c}$ for $2\pi q$ periodic
perturbations or even to construct stable and unstable manifolds. To prove
nonlinear instability for localized perturbations, we study the semigroup
$e^{tJ\mathcal{L}_{0}}$ on the space $H^{s}\left(  \mathbf{R}\right)  $
$\left(  s\geq\frac{m}{2}\right)  $. The operator $\mathcal{L}_{0}$ might have
negative continuous spectrum in $H^{s}\left(  \mathbf{R}\right)  $. For
example, when $\mathcal{M=-\partial}_{x}^{2}$, the spectrum of $\mathcal{L}%
_{0}=$ $\mathcal{-\partial}_{x}^{2}$ $+V\left(  x\right)  $ with periodic
$V\left(  x\right)  $ is well studied in the literature and is known to have
bands of continuous spectrum. So Theorem \ref{theorem-dichotomy} does not
apply. However, we have the following upper bound estimate of
$e^{tJ\mathcal{L}_{0}}$ on $H^{s}\left(  \mathbf{R}\right)  $, which suffices
to prove nonlinear localized instability.

\begin{lemma}
\label{lemma-modulational-localized-semigroup}Assume \eqref{E:M-assump-1}. Let
$\lambda_{0}\geq0$ be such that
\[
\operatorname{Re}\lambda\leq\lambda_{0},\quad\forall\xi\in\lbrack
0,1],\;\lambda\in\sigma(\mathcal{J}_{\xi}\mathcal{L}_{\xi}).
\]
For every $s\geq\frac{m}{2}$, there exist $C(s)>0$ such that
\begin{align}
&  \Vert e^{t\mathcal{J}_{\xi}\mathcal{L}_{\xi}}v(x)\Vert_{H^{s}%
(\mathbf{S}^{1})}\leq C(s)(1+t^{2n^{-}(\mathcal{L}_{\xi})+1})e^{\lambda_{0}%
t}\Vert v(x)\Vert_{H^{s}(\mathbf{S}^{1})},\label{E:SG-esti-1}\\
&  \Vert e^{tJ\mathcal{L}_{0}}u(x)\Vert_{H^{s}{(\mathbf{R})}}\leqslant
C(s)(1+t^{2n^{-}(\mathcal{L}_{\xi})+1})e^{\lambda_{0}t}\Vert u(x)\Vert
_{H^{s}{(\mathbf{R})}}, \label{E:SG-esti-2}%
\end{align}
for any $\xi\in\lbrack0,1]$, $v\in H^{s}(\mathcal{S}^{1})$, and $u\in
H^{s}{(\mathbf{R})}$.
\end{lemma}

\begin{proof}
It suffices to prove the lemma for $s=\frac{m}{2}$. The estimates for general
$s\geq\frac{m}{2}$ can be obtained by applying $\mathcal{J}_{\xi}%
\mathcal{L}_{\xi}$ and $J\mathcal{L}_{0}$ repeatedly to the estimates for
$s=\frac{m}{2}$ (and interpolation for the case when $\frac{2s}{m}$ is not an
integer). We start with the first estimate in the $2\pi$-periodic case. Due to
the compactness of $[0,1]$, it suffices to prove that for any $\xi_{0}%
\in\lbrack0,1]$, there exist $C,\epsilon>0$ and an integer $K\geq0$ such that
\eqref{E:SG-esti-1} holds for $\xi\in(\xi_{0}-\epsilon,\xi_{0}+\epsilon)$. We
first note that each $\lambda\in\sigma(\mathcal{J}_{\xi_{0}}\mathcal{L}%
_{\xi_{0}})$ is an isolated eigenvalue with finite algebraic multiplicity and
$\mathcal{L}_{\xi_{0}}$ is non-degenerate on $E_{\lambda}\slash(E_{\lambda
}\cap\ker\mathcal{L}_{\xi_{0}})$. Let
\[
\Lambda=\{\lambda\in\sigma(\mathcal{J}_{\xi_{0}}\mathcal{L}_{\xi_{0}}%
)\mid\exists\ \delta>0\text{ s.t. }\langle\mathcal{L}_{\xi_{0}}v,v\rangle
\geq\delta\Vert v\Vert^{2}\}.
\]
Due to Proposition \ref{P:modu-index}, $\sigma(\mathcal{J}_{\xi_{0}%
}\mathcal{L}_{\xi_{0}})\backslash\Lambda$ is finite and
\[
n=\Sigma_{\lambda\in\sigma(\mathcal{J}_{\xi_{0}}\mathcal{L}_{\xi_{0}%
})\backslash\Lambda}\dim E_{\lambda}<\infty.
\]
Moreover, there exists $\varepsilon>0$ such that
\[
\Omega\cap\Lambda=\emptyset,\;\text{ where }\Omega=\cup_{\lambda\in
\sigma(\mathcal{J}_{\xi_{0}}\mathcal{L}_{\xi_{0}})\backslash\Lambda}%
\{z\mid\Vert z-\lambda\Vert<\varepsilon\}\subset\mathbf{C}.
\]
From Lemma \ref{L:C0-resolvent}, there exists $\epsilon>0$ such that
$\partial\Omega\cap\sigma(\mathcal{J}_{\xi}\mathcal{L}_{\xi})=\emptyset$ for
any $\xi\in\lbrack\xi_{0}-\epsilon,\xi_{0}+\epsilon]$. For such $\xi$, let
\[
P(\xi)=\frac{1}{2\pi i}\oint_{\partial\Omega}(\lambda-\mathcal{J}%
_{k}\mathcal{L}_{k})^{-1}d\lambda,\quad Z_{\xi}=P(\xi)X,\quad Y_{\xi
}=\big(I-P(\xi)\big)X,
\]
which are continuous in $\xi$ and invariant under $e^{t\mathcal{J}_{\xi
}\mathcal{L}_{\xi}}$. Therefore, $\dim Z_{\xi}=n$ and the continuity of
$\mathcal{L}_{\xi}$ in $\xi$ implies that there exists $\delta>0$ such that
\[
\delta^{-2}\Vert v\Vert^{2}\geq\langle\mathcal{L}_{\xi}v,v\rangle\geq
\delta^{2}\Vert v\Vert^{2},\quad\forall v\in Y_{\xi},\;|\xi-\xi_{0}%
|\leq\epsilon.
\]
Moreover, according to Proposition \ref{P:basis}, for any $\lambda\in
\Omega\cap\sigma(\mathcal{J}_{\xi}\mathcal{L}_{\xi})$, the dimension of its
eigenspace
\[
E_{\lambda}(\mathcal{J}_{\xi}\mathcal{L}_{\xi})=\ker(\lambda-\mathcal{J}_{\xi
}\mathcal{L}_{\xi})^{2\big(1+n^{-}(\mathcal{L}_{\xi})\big)},
\]
namely, the maximal dimension of Jordan blocks of $\mathcal{J}_{\xi
}\mathcal{L}_{\xi}$ on $Y_{\xi}$ is no more than $2(1+n^{-}(\mathcal{L}_{\xi
})\big)$. So for any $\xi\in\lbrack\xi_{0}-\epsilon,\xi_{0}+\epsilon]$, there
exists a generic constant $C>0$ independent of $\xi$, such that
\begin{align*}
&  \Vert e^{t\mathcal{J}_{\xi}\mathcal{L}_{\xi}}v\Vert\leq\Vert
e^{t\mathcal{J}_{\xi}\mathcal{L}_{\xi}}P(\xi)v\Vert+\Vert e^{t\mathcal{J}%
_{\xi}\mathcal{L}_{\xi}}\big(I-P(\xi)\big)v\Vert\\
\leq &  C\Big((1+t^{2n^{-}(\mathcal{L}_{\xi})+1})e^{\lambda_{0}t}\Vert
P(\xi)v\Vert+\langle\mathcal{L}_{\xi}e^{t\mathcal{J}_{\xi}\mathcal{L}_{\xi}%
}\big(I-P(\xi)\big)v,e^{t\mathcal{J}_{\xi}\mathcal{L}_{\xi}}\big(I-P(\xi
)\big)v\rangle^{\frac{1}{2}}\Big)\\
\leq &  C\Big((1+t^{2n^{-}(\mathcal{L}_{\xi})+1})e^{\lambda_{0}t}\Vert
P(\xi)v\Vert+\langle\mathcal{L}_{\xi}\big(I-P(\xi)\big)v,\big(I-P(\xi
)\big)v\rangle^{\frac{1}{2}}\Big)\\
\leq &  C(1+t^{2n^{-}(\mathcal{L}_{\xi})+1})e^{\lambda_{0}t}\Vert v\Vert.
\end{align*}
Along with the compactness of $[0,1]$, it implies \eqref{E:SG-esti-1}.

To prove \eqref{E:SG-esti-2}, we first write, for any $u\in H^{s}%
{(\mathbf{R)}}$,
\[
u(x)=\int_{0}^{1}e^{i\xi x}u_{\xi}(x)d\xi,\;\text{ where }u_{\xi}%
(x)=\Sigma_{n\in\mathbf{Z}}e^{inx}\hat{u}(n+\xi)\in H^{s}(\mathbf{S}^{1}),
\]
and $\hat{u}$ is the Fourier transform of $u$. Clearly, there exists $C>0$
such that
\begin{equation}
\frac{1}{C}\Vert u\Vert_{H^{s}(\mathbb{R})}^{2}\leq\int_{0}^{1}\Vert u_{\xi
}\left(  x\right)  \Vert_{H^{s}(\mathbf{S}^{1})}^{2}\,d\xi\leq C\Vert
u\Vert_{H^{s}(\mathbb{R})}^{2}.\label{norm-equivalence}%
\end{equation}
Note
\[
e^{tJ\mathcal{L}_{0}}u(x)=\int_{0}^{1}e^{i\xi x}e^{t\mathcal{J}_{\xi
}\mathcal{L}_{\xi}}u_{\xi}\left(  x\right)  \,d\xi
\]
and thus
\begin{equation}
\Vert e^{tJ\mathcal{L}_{0}}u(x)\Vert_{H^{s}{(\mathbf{R})}}^{2}\thickapprox
\int_{0}^{1}\Vert e^{t\mathcal{J}_{\xi}\mathcal{L}_{\xi}}u_{\xi}\left(
x\right)  \Vert_{H^{s}(\mathbf{S}^{1})}^{2}\,d\xi.\label{semigroup-frequency}%
\end{equation}
Along with \eqref{E:SG-esti-1}, it immediately implies \eqref{E:SG-esti-2}.
\end{proof}

\begin{remark}
The semigroup estimates of the types (\ref{E:SG-esti-1}) and
(\ref{E:SG-esti-2}) can also be obtained for $s=-1$, that is, in the negative
Sobolev space $H^{-1}\left(  \mathbf{S}_{2\pi q}^{1}\right)  $ and
$H^{-1}{(\mathbf{R})}$ for $e^{tJ\mathcal{L}_{0}}$ (see
\cite{lin-liao-jin-modulational}). Such semigroup estimates were used in
\cite{lin-liao-jin-modulational} to prove nonlinear modulational instability
by a bootstrap argument.
\end{remark}

\subsection{The spectral problem $Lu=\lambda u^{\prime}$}

In this subsection, we consider the eigenvalue problem of the form
\begin{equation}
Lu=\lambda u^{\prime}, \label{eigenvalue-form-prime}%
\end{equation}
where the symmetric operator $L$ is of the form of $\mathcal{L}_{0}\ $in
Subsection \ref{SS:solitary}. As an example, consider the stability of
solitary waves of generalized Bullough--Dodd equation
(\cite{stefnov-L-prime16})%
\begin{equation}
u_{tx}=au-f\left(  u\right)  , \label{eqn-BD}%
\end{equation}
where $a>0$ and $f$ is a smooth function of $u$ satisfying
\begin{equation}
f\left(  u\right)  =O\left(  u^{2}\right)  ,\ f^{\prime}\left(  u\right)
=O\left(  u\right)  \ \text{for\ small\ }u. \label{assumption-f}%
\end{equation}
The traveling wave $u_{c}\left(  x+ct\right)  $ satisfies the ODE
\[
-cu_{c}^{\prime\prime}+au_{c}-f\left(  u_{c}\right)  =0\text{.}%
\]
Then the linearized equation in the traveling frame $\left(  x+ct,t\right)  $
takes the form
\begin{equation}
u_{tx}=-cu_{xx}+au-f^{\prime}\left(  u_{c}\right)  u. \label{eqn-LBD}%
\end{equation}
Thus the eigenvalue problem takes the form (\ref{eigenvalue-form-prime}) with
\begin{equation}
L=-c\frac{d^{2}}{dx^{2}}+a-f^{\prime}\left(  u_{c}\right)  . \label{defn-L-BD}%
\end{equation}

We consider the general problem (\ref{eigenvalue-form-prime}) with $L$ of the
form $L=\mathcal{M+}V\left(  x\right)  $. We assume that: i) $M$ is a Fourier
multiplier operator with the symbol $\alpha\left(  \xi\right)  $ satisfying
\begin{equation}
\alpha\left(  \xi\right)  \geq0\text{ and}\ \alpha\left(  \xi\right)
\thickapprox\left\vert \xi\right\vert ^{2s}\ \left(  s>0\right)  \text{,
when\ }\left\vert \xi\right\vert \text{\ is large, } \label{assumption-symbol}%
\end{equation}
$\ $ and ii) the real potential $V\left(  x\right)  $ satisfies
\begin{equation}
V\left(  x\right)  \rightarrow\delta_{0}>0\text{ when\ }\left\vert
x\right\vert \rightarrow\infty. \label{assumption-potential}%
\end{equation}
Let $X=H^{s}\left(  R\right)  $ $\left(  s>0\right)  $. Then the assumption
(\textbf{H2}) is satisfied for $L$ on $X$. Namely, $L:X\rightarrow X^{\ast}$
is bounded and symmetric, and there exists a decomposition of $X$
\[
X=X_{-}\oplus\ker L\oplus X_{+},\quad n^{-}(L)\triangleq\dim X_{-}<\infty,
\]
satisfying $L|_{X_{-}}<0$ and $L|_{X_{+}}\geq\delta>0$.

Define $J=\partial_{x}^{-1}$. Now we check that $J:X^{\ast}\rightarrow X$ is
densely defined and $J^{*}=-J$. On $X=H^{s}\left(  R\right)  $ with $s>0$, the
operator $\partial_{x}:X\rightarrow X^{\ast}$ is densely defined and satisfies
$(\partial_{x})^{*} = -\partial_{x}$. Since $\ker\partial_{x}=\left\{
0\right\}  $,
\[
\overline{R\left(  \partial_{x}\right)  }=\left(  \ker\left(  \partial
_{x}^{\ast}\right)  \right)  ^{\perp}=\left(  \ker\left(  -\partial
_{x}\right)  \right)  ^{\perp}=X^{\ast},
\]
so $D\left(  \partial_{x}^{-1}\right)  =R\left(  \partial_{x}\right)  $ is
dense in $X^{\ast}$ and $J=\partial_{x}^{-1}:X^{\ast}\rightarrow X$ satisfies
$J^{*}=-J$.

So the eigenvalue problem $Lu=\lambda u^{\prime}$ can be equivalently written
in the Hamiltonian form $JLu=\lambda u$, where $\left(  J,L,X\right)  $
satisfies the assumptions (\textbf{H1})-(\textbf{H3}). Let $\ker
L=span\left\{  \psi_{1},\cdots,\psi_{l}\right\}  $ and
\[
span\left\{  \psi_{1}^{\prime},\cdots,\psi_{l}^{\prime}\right\}  \cap R\left(
L\right)  =span\left\{  g_{1},\cdots,g_{m}\right\}  ,\ m\leq l\text{. }%
\]
Define the $m$ by $m$ matrix$\ $%
\[
D=\left(  \left\langle L^{-1}g_{i},g_{j}\right\rangle \right)  ,\ 1\leq
i,j\leq m\text{. }\
\]
By Theorem \ref{theorem-counting} and Proposition \ref{prop-counting-k-0-1},
we get the following theorem.

\begin{theorem}
Assume (\ref{assumption-symbol}) and (\ref{assumption-potential}). Then
\[
k_{r}+2k_{c}+2k_{i}^{\leq0}+k_{0}^{\leq0}=n^{-}\left(  L\right)  ,
\]
where $k_{r},k_{c},k_{i}^{\leq0},k_{0}^{\leq0}$ are the indexes for the
eigenvalues of $\partial_{x}^{-1}L$, as defined in Section
\ref{section-index theorem}. In addition, we have $k_{0}^{\leq0}\geq n^{\leq
0}\left(  D\right)  $, where $n^{\leq0}\left(  D\right)  $ is the number of
nonpositive eigenvalues of $D$. If $D$ is nonsingular, then $k_{0}^{\leq
0}=n^{-}\left(  D\right)  $, i.e., the number of negative eigenvalues of $D$.
\end{theorem}

For many applications, particularly the generalized Bullough--Dodd equation
where$\ \mathcal{M=-}c\mathcal{\partial}_{x}^{2}$ $\left(  c>0\right)  $, $L$
has at most one dimensional kernel and negative eigenspace. In this case, we
get a more explicit instability criterion.

\begin{corollary}
\label{corollary-prime}i) Assume $n^{-}\left(  L\right)  =1$ and $\ker
L=\left\{  \psi_{0}\right\}  $. Then there is a positive eigenvalue of
$\partial_{x}^{-1}L$ when $\left\langle L^{-1}\psi_{0}^{\prime},\psi
_{0}^{\prime}\right\rangle >0$.

ii) Assume $n^{-}\left(  L\right)  \le1$ and there exists $0\ne\psi_{0}\in\ker
L$ such that $\left\langle L^{-1}\psi_{0}^{\prime},\psi_{0}^{\prime
}\right\rangle \leq0$, then $\partial_{x}^{-1}L$ has no unstable eigenvalues.
\end{corollary}

\begin{remark}
\label{remark-stefanov-paper}The above Corollary was obtained in
\cite{stefnov-L-prime16} under some additional assumptions. In
\cite{stefnov-L-prime16}, Corollary \ref{corollary-prime} i) was proved under
the following two assumptions:

C1) $\left(  f_{0},g_{0}\right)  \neq0$, where $f_{0}$ is the eigenfunction of
$L$ with the negative eigenvalue and $g_{0}^{\prime}\in\ker L$.

C2) For any $\lambda\in\mathbf{R}$,
\[
\left\Vert P_{+}\left(  L-\lambda\partial_{x}\right)  ^{-1}P_{+}v\right\Vert
_{H^{1}}\leq C\left(  \lambda\right)  \left\Vert v\right\Vert _{L^{2}},
\]
where $P_{+}$ is the projection to the positive space of $L$ and $C\left(
\lambda\right)  $ is bounded on compact sets. \newline The proof in
\cite{stefnov-L-prime16} is by constructing Evans-like functions. Corollary
\ref{corollary-prime} ii) was proved in \cite{stefnov-L-prime16} under the
following additional assumptions:

D1) $\ker L=\left\{  \psi_{0}\right\}  $ and $\left\langle L^{-1}\psi
_{0}^{\prime},\psi_{0}^{\prime}\right\rangle <0$;

D2) For any $\lambda\notin i\mathbf{R}$, the operator $L-\lambda\partial_{x}$
has zero index and the equation $\left(  L-\lambda\partial_{x}\right)
f=g\ $satisfies certain Fredholm alternative properties (see (12)(13)(14) in
\cite{stefnov-L-prime16});

D3) The symbol $\alpha\left(  \xi\right)  \ $ of the leading order part
$\mathcal{M}$ of $L$ satisfies
\[
\alpha\left(  \xi\right)  \thickapprox\left\vert \xi\right\vert ^{2s}\ \left(
s>\frac12\right)  \text{, when\ }\left\vert \xi\right\vert \text{\ is large.}%
\]
The proof in \cite{stefnov-L-prime16} is by Lyapunov--Schmidt reduction
arguments and the index theorem in \cite{kapitula-et-04}.
\end{remark}

For the Bullough--Dodd equation (\ref{eqn-BD}), $\ker L=\left\{
u_{c,x}\right\}  $ where $L$ is defined by (\ref{defn-L-BD}). Since the
momentum of the problem is $\frac{1}{2}\int\left(  u_{c}^{\prime}\right)
^{2}dx$, by similar computation as in (\ref{KDV-index}), it was shown in
\cite{stefnov-L-prime16} that
\[
\left\langle L^{-1}u_{c}^{\prime\prime},u_{c}^{\prime\prime}\right\rangle
=-\frac{1}{2}\partial_{c}\int\left(  u_{c}^{\prime}\right)  ^{2}dx=-\frac
{1}{2}\partial_{c}\left[  c^{-\frac{1}{2}}\int\left(  u_{1}^{\prime}\right)
^{2}dx\right]  >0\text{,}%
\]
where $u_{c}=u_{1}\left(  x/\sqrt{c}\right)  $ and $-u_{1}^{\prime\prime
}+au_{1}-f\left(  u_{1}\right)  =0$. So we get the following

\begin{theorem}
Assume $f\left(  u\right)  $ is a smooth function satisfying
(\ref{assumption-f}) and the traveling wave solution $u_{c}\left(
x-ct\right)  \ $to (\ref{eqn-BD}) exists with $c>0$ and $u_{c}(x)\rightarrow0$
as $|x|\rightarrow\infty$, then $u_{c}$ is linearly unstable.
\end{theorem}

In \cite{stefnov-L-prime16}, the above Theorem was proved for smooth and
convex function $f$. Their additional convexity assumption on $f$ was used to
verify the condition C1) in Remark \ref{remark-stefanov-paper}.

Besides the above linear instability result, Theorem \ref{theorem-dichotomy}
can be applied to give the exponential trichotomy for the linearized equation
(\ref{eqn-LBD}). This will be useful for the construction of invariant
manifolds of (\ref{eqn-BD}) near the unstable traveling wave orbit.

\subsection{Stability of steady flows of 2D Euler equation}

\label{SS:Euler}

We consider the 2D Euler equations
\begin{equation}
\partial_{t}u+\left(  u\cdot\nabla u\right)  +\nabla p=0, \label{eqn-Euler}%
\end{equation}

\begin{equation}
\nabla\cdot u=0, \label{eqn-div-free}%
\end{equation}
in a bounded domain $\Omega\subset\mathbf{R}^{2}$ with smooth boundary
$\partial\Omega$ composed of a finite number of connected components
$\Gamma_{i}$ . The boundary condition is
\[
u\cdot n=0\text{ \ \ \ on\ \ }\partial\Omega,
\]
For simplicity, first we consider $\Omega$ to be simply connected and
$\partial\Omega=\Gamma$. The vorticity form of (\ref{eqn-Euler}%
)-(\ref{eqn-div-free}) is given by%

\begin{equation}
\partial_{t}\omega+\psi_{y}\partial_{x}\omega-\psi_{x}\partial_{y}\omega=0,
\label{vorticity}%
\end{equation}
where $\psi$ is the stream function, then $\omega\equiv-\Delta\psi
\equiv-\left(  \partial_{x}^{2}+\partial_{y}^{2}\right)  \psi$ is the
vorticity and $u=\nabla^{\perp}\psi=\left(  \psi_{y},-\psi_{x}\right)  $ is
the velocity. The boundary condition associated with (\ref{vorticity}) is
given by $\psi=0$ on $\partial\Omega$. A stationary solution of
(\ref{vorticity}) is given by a stream function $\psi_{0}$ satisfying%
\begin{equation}
-\psi_{0_{y}}\partial_{x}\omega_{0}+\psi_{0_{x}}\partial_{y}\omega_{0}=0,
\label{steady}%
\end{equation}
here $\omega_{0}\equiv-\Delta\psi_{0}$ and $u_{0}=\nabla^{\perp}\psi_{0}\ $are
the associated vorticity and velocity. Suppose $\psi_{0}$ satisfy the
following elliptic equation
\[
-\Delta\psi_{0}=g\left(  \psi_{0}\right)
\]
with boundary condition $\psi_{0}=0$ on $\partial\Omega$, where $g$ is some
differentiable function. Then $\omega_{0}\equiv-\Delta\psi_{0}=g\left(
\psi_{0}\right)  $ is a steady solution of equation (\ref{vorticity}). The
linearized equation near $\omega_{0}$ is
\begin{equation}
\partial_{t}\omega+\psi_{0_{y}}\partial_{x}\omega-\psi_{0_{x}}\partial
_{y}\omega=-\psi_{y}\partial_{x}\omega_{0}+\psi_{x}\partial_{y}\omega_{0},
\label{1stlinear}%
\end{equation}
with $\omega=-\Delta\psi$ and the boundary condition $\psi\ |_{\partial\Omega
}=0$. The above equation can be written as
\begin{equation}
\partial_{t}\omega+u_{0}\cdot\nabla\omega-g^{\prime}\left(  \psi_{0}\right)
u_{0}\cdot\nabla\psi=0. \label{eqn-linearized-vorticity}%
\end{equation}

Below we consider the case when $g^{\prime}>0$ which appeared in many
interesting cases such as mean field equations (e.g. \cite{caglioti-et-1}
\cite{caglioti-et-2}). Then (\ref{eqn-linearized-vorticity}) has the following
Hamiltonian structure
\begin{equation}
\partial_{t}\omega=JL\omega,\;\text{where }J=-g^{\prime}\left(  \psi
_{0}\right)  u_{0}\cdot\nabla,\quad L=\frac{1}{g^{\prime}\left(  \psi
_{0}\right)  }-\left(  -\Delta\right)  ^{-1}.\label{E:LEuler}%
\end{equation}
We take the energy space of the linearized Euler \eqref{E:LEuler} as the
weighted space
\[
X=\left\{  \omega\ |\Vert\omega\Vert_{X}<\infty\right\}  ,\;\text{ where
}\left\Vert \omega\right\Vert _{X}=\left(  \int\int_{\Omega}\frac{\left\vert
\omega\right\vert ^{2}}{g^{\prime}\left(  \psi_{0}\right)  }\ dxdy\right)
^{\frac{1}{2}}.
\]
If $g^{\prime}$ has a positive lower bound, $X$ is equivalent to $L^{2}%
(\Omega)$. In general, $\omega\in X$ implies $\omega\in L^{2}$ and $\nabla
\psi\in L^{2}$. Therefore, $\langle L\cdot,\cdot\rangle$ defines a bounded
symmetric quadratic form on $X$ and $L:X\rightarrow X^{\ast}$ is a bounded
symmetric operator. Moreover, it is easy to see that
\[
S:L^{2}\rightarrow X,\quad S\omega=g^{\prime}(\psi_{0})^{\frac{1}{2}}\omega
\]
defines an isometry. As $f(\psi_{0})\cdot$ and $u_{0}\cdot\nabla$ are
commutative for any $f$, we have
\[
\tilde{J}\triangleq S^{-1}J(S^{\ast})^{-1}=u_{0}\cdot\nabla:(L^{2})^{\ast
}\rightarrow L^{2}%
\]
is anti-self-dual due to $\nabla\cdot u_{0}=0$, from which we obtain $J^{\ast
}=-J$ and thus (\textbf{H1}) is satisfied by $J$ and $X$. Moreover, since
$\frac{1}{g^{\prime}\left(  \psi_{0}\right)  }\cdot:X\rightarrow X^{\ast}$ is
an isomorphism and $(-\Delta)^{-1}$ is compact, we have $\dim\ker L<\infty$
and thus (\textbf{H3}) is satisfied. Note that the closed subspace $\ker
J\subset X^{\ast}$ is infinite dimensional since
\[
\ker J\supset\left\{  h\left(  \psi_{0}\right)  ,\text{ }h\in C^{1}\right\}  .
\]
Let $\tilde{P}:(L^{2})^{\ast}\rightarrow\ker\tilde{J}$ be the orthogonal
projection and define
\[
P=(S^{\ast})^{-1}\tilde{P}S^{\ast}:X^{\ast}\rightarrow\ker J.
\]
Clearly, $P$ is a bounded linear operator on $X^{\ast}$ and it defines a
projection on $X^{\ast}$, but orthogonal in the $L^{2}$ sense. In fact, due to
the commutativity between $f(\psi_{0})\cdot$ and $u_{0}\cdot\nabla$ for any
$f$, operators $P$ and $\tilde{P}$ take the same form shown in
(\cite{lin-cmp-04})
\[
P\phi\ |_{\gamma_{i}\left(  c\right)  }=\frac{\oint_{\gamma_{i}\left(
c\right)  }\frac{\phi\left(  x,y\right)  }{\left\vert \nabla\psi
_{0}\right\vert }dl}{\oint_{\gamma_{i}\left(  c\right)  }\frac{1}{\left\vert
\nabla\psi_{0}\right\vert }dl},
\]
where $c$ is in the range of $\psi_{0}$ and $\gamma_{i}\left(  c\right)  $ is
a branch of $\left\{  \psi_{0}=c\right\}  $. As in (\cite{lin-cmp-04}), define
operator $A:H_{0}^{1}\cap H^{2}\left(  \Omega\right)  \rightarrow L^{2}\left(
\Omega\right)  $ by
\[
A\phi=-\Delta\phi-g^{\prime}\left(  \psi_{0}\right)  \phi+g^{\prime}\left(
\psi_{0}\right)  P\phi.
\]
We also denote the operator
\[
A_{0}=-\Delta-g^{\prime}\left(  \psi_{0}\right)  :H_{0}^{1}\cap
H^{2}\left(  \Omega\right)  \rightarrow L^{2}\left(  \Omega\right)  \ .
\]
Clearly, $A,A_{0}$ are self-adjoint with compact resolvents and thus with only
discrete spectra. The next lemma studies the spectral information of $L$ on
the weighted space $X$.

Recall that for any subspace $Y \in X$, $\langle L\cdot, \cdot\rangle$ also
defines a bounded symmetric quadratic form on the quotient space $Y/(Y
\cap\ker L)$.

\begin{lemma}
\label{lemma-euler-spectral}i) The assumption (\textbf{H2}) is satisfied by
$\left\langle L\cdot,\cdot\right\rangle $ on $X$, with $n^{-}\left(  L\right)
=n^{-}\left(  A_{0}\right)  $ and $\dim\ker L=\dim\ker A_{0}$.

ii) The quadratic form $\left\langle L\cdot,\cdot\right\rangle $ is
non-degenerate on $\overline{R(J)} / \big(\overline{R(J)} \cap\ker L\big)$ if
and only if $\ker A\subset\ker A_{0}$. Moreover,
\begin{equation}
n^{-}\Big( L|_{\overline{R(J)} / \big(\overline{R(J)} \cap\ker L\big)}
\Big) =n^{-}\big( L|_{\overline{R\left(  J\right)  } }\big) =n^{-}\left(
A\right)  \text{. } \label{equality-negative-modes-Euler}%
\end{equation}

\end{lemma}

\begin{proof}
i) For any $\omega\in X,$ we have
\begin{align}
\left\langle L\omega,\omega\right\rangle  &  =\int\int_{\Omega}\left\{
\frac{\omega^{2}}{g^{\prime}\left(  \psi_{0}\right)  }-\left\vert \nabla
\psi\right\vert ^{2}\right\}  dxdy\label{estimate-inequality-L>A-0}\\
&  =\int\int_{\Omega}\left\{  \frac{\omega^{2}}{g^{\prime}\left(  \psi
_{0}\right)  }-2\psi\omega+\left\vert \nabla\psi\right\vert ^{2}\right\}
dxdy\nonumber\\
&  =\int\int_{\Omega}\left\{  \left(  \frac{\omega}{\sqrt{g^{\prime}\left(
\psi_{0}\right)  }}-\psi\sqrt{g^{\prime}\left(  \psi_{0}\right)  }\right)
^{2}-g^{\prime}\left(  \psi_{0}\right)  \psi^{2}+\left\vert \nabla
\psi\right\vert ^{2}\right\}  dxdy\nonumber\\
&  \geq\int\int_{\Omega}\left[  \left\vert \nabla\psi\right\vert
^{2}-g^{\prime}\left(  \psi_{0}\right)  \psi^{2}\ \right]  dxdy=\left(
A_{0}\psi,\psi\right)  ,\nonumber
\end{align}
where $\psi=\left(  -\Delta\right)  ^{-1}\omega$. Recall that $n^{\leq
0}\left(  L\right)  $ and $n^{\leq0}\left(  A_{0}\right)  $ denote the maximal
dimensions of subspaces where the quadratic forms $\langle L\cdot,\cdot
\rangle$ and $(A_{0}\cdot,\cdot)$ are nonpositive. Let
\[
\left\{  \psi_{1},\cdots,\psi_{l}\right\}  ,\ \ \ l=n^{\leq0}\left(
A_{0}\right)  ,
\]
be linearly independent  eigenfunctions associated to nonpositive eigenvalues of $A_{0}$.
Define the space $Y_{1}\subset X$ by
\[
Y_{1}=\left\{  \omega\in X\ |\ \int_{\Omega}\psi_{j}\left(  -\Delta\right)
^{-1}\omega=0,\;1\leq j\leq l\right\}  .
\]
Then for any $\omega\in Y_{1}$, we have
\[
\left(  A_{0}\psi,\psi\right)  \geq\delta\left\Vert \psi\right\Vert _{H^{1}%
}^{2},\text{ for some }\delta>0.
\]
So by (\ref{estimate-inequality-L>A-0}), for any $\omega\in Y_{1}$,
\begin{align*}
\left\langle L\omega,\omega\right\rangle  &  =\varepsilon\int\int_{\Omega
}\left\{  \frac{\omega^{2}}{g^{\prime}\left(  \psi_{0}\right)  }-\left\vert
\nabla\psi\right\vert ^{2}\right\}  dxdy+\left(  1-\varepsilon\right)
\left\langle L\omega,\omega\right\rangle \\
&  \geq\varepsilon\int\int_{\Omega}\left\{  \frac{\omega^{2}}{g^{\prime
}\left(  \psi_{0}\right)  }-\left\vert \nabla\psi\right\vert ^{2}\right\}
dxdy+\left(  1-\varepsilon\right)  \delta\left\Vert \psi\right\Vert _{H^{1}%
}^{2}\\
&  \geq\varepsilon\int\int_{\Omega}\left\{  \frac{\omega^{2}}{g^{\prime
}\left(  \psi_{0}\right)  }+\left\vert \nabla\psi\right\vert ^{2}\right\}
dxdy,
\end{align*}
by choosing $\varepsilon>0$ such that $\left(  1-\varepsilon\right)
\delta>2\varepsilon$. Since the positive subspace $Y_{1}$ has co-dimension
$n^{\leq0}\left(  A_{0}\right)  $, this shows that the assumption
(\textbf{H2}) for $L\ $on $X\ $is satisfied and $n^{\leq0}\left(  L\right)
\leq n^{\leq0}\left(  A_{0}\right)  $.

To prove $n^{\leq0}\left(  L\right)  \geq n^{\leq0}\left(  A_{0}\right)  $,
let $\tilde{\omega}_{j}=g^{\prime}\left(  \psi_{0}\right)  \psi_{j}\in X$ and
$\tilde{\psi}_{j}=\left(  -\Delta\right)  ^{-1}\tilde{\omega}_{j}$,
$j=1,\ldots,l$ and then
\begin{align*}
\left(  A_{0}\psi_{j},\psi_{j}\right)   &  =\int\int_{\Omega}\left[
\left\vert \nabla\psi_{j}\right\vert ^{2}-g^{\prime}\left(  \psi_{0}\right)
\psi_{j}^{2}\ \right]  dxdy=\int\int_{\Omega}\left[  \left\vert \nabla\psi
_{j}\right\vert ^{2}-\frac{\tilde{\omega}_{j}^{2}}{g^{\prime}\left(  \psi
_{0}\right)  }\right]  dxdy\\
&  =\int\int_{\Omega}\left[  \left\vert \nabla\psi_{j}\right\vert ^{2}%
-2\tilde{\omega}_{j}\psi_{j}+\frac{\tilde{\omega}_{j}^{2}}{g^{\prime}\left(
\psi_{0}\right)  }\right]  dxdy\\
&  =\int\int_{\Omega}\left[  \left\vert \nabla\psi_{j}\right\vert ^{2}%
-2\nabla\psi_{j}\cdot\nabla\tilde{\psi}_{j}+\frac{\tilde{\omega}_{j}^{2}%
}{g^{\prime}\left(  \psi_{0}\right)  }\right]  dxdy\\
&  \geq\int\int_{\Omega}\left[  \frac{\tilde{\omega}_{j}^{2}}{g^{\prime
}\left(  \psi_{0}\right)  }-\left\vert \nabla\tilde{\psi}_{j}\right\vert
^{2}\right]  dxdy=\left\langle L\tilde{\omega}_{j},\tilde{\omega}%
_{j}\right\rangle ,
\end{align*}
and thus $n^{\leq0}\left(  L\right)  \geq n^{\leq0}\left(  A_{0}\right)  $.
Combined with above, this implies that $n^{\leq0}\left(  L\right)  =n^{\leq
0}\left(  A_{0}\right)  $. Since $\omega\in\ker L$ if and only if
$\psi=\left(  -\Delta\right)  ^{-1}\omega\in\ker A_{0}$, we obtain $\dim\ker
L=\dim\ker A_{0}$ and thus
\[
n^{-}\left(  L\right)  =n^{\leq0}\left(  L\right)  -\dim\ker L=n^{\leq
0}\left(  A_{0}\right)  -\dim\ker A_{0}=n^{-}\left(  A_{0}\right)  .
\]

ii) Note that, like $J$, the projection $P$ also commutes with $f(\psi
_{0})\cdot$ for any $f$. Therefore, $\omega\in\overline{R(J)}$ if and only if
$P\frac{\omega}{g^{\prime}(\psi_{0})}=0$. It implies that
\begin{equation}
(I-P)L\omega=\frac{\omega}{g^{\prime}\left(  \psi_{0}\right)  }-\left(
I-P\right)  \psi=\frac{1}{g^{\prime}(\psi_{0})}A\psi,\quad\forall\,\omega
\in\overline{R(J)},\label{E:A(omega)}%
\end{equation}
where $\psi=\left(  -\Delta\right)  ^{-1}\omega$, and thus
\begin{equation}
(-\Delta)\ker A=\overline{R(J)}\cap\ker\big((I-P)L\big)=\overline{R(J)}%
\cap\ker JL.\label{eqn-ker-L-R}%
\end{equation}
Since
\[
\ker\left(  \langle L\cdot,\cdot\rangle|_{\overline{R(J)}}\right)
=\overline{R(J)}\cap\ker JL,
\]
it immediately implies
\begin{equation}
\dim\ker\left(  \langle L\cdot,\cdot\rangle|_{\overline{R(J)}}\right)
=\dim\ker A.\label{equality-kernel-L-R(J)}%
\end{equation}

Suppose $\left\langle L\cdot,\cdot\right\rangle $ is degenerate on
$\overline{R(J)}/\big(\overline{R(J)}\cap\ker L\big)$, namely
\[
\exists\,\omega_{1}\in\overline{R(J)}\backslash\ker L\;\text{ such that
}\left\langle L\omega_{1},\omega\right\rangle =0\text{,\ }\forall\ \omega
\in\overline{R\left(  J\right)  }\text{. }%
\]
Such $\omega_{1}$ satisfies $0\neq L\omega_{1}\in\ker J$, or equivalently
$\left(  I-P\right)  L\omega_{1}=0$. Therefore, \eqref{eqn-ker-L-R} implies
$A\psi_{1}=0$. Since $A_{0}\psi_{1}\neq0$ due to $L\omega_{1}\neq0,$ we obtain
$\ker A\subsetneqq\ker A_{0}$. The converse can be proved similarly and the
first statement follows.

To prove (\ref{equality-negative-modes-Euler}), first we notice that for any
$\omega\in\overline{R\left(  J\right)  },$
\begin{align}
&  \left\langle L\omega,\omega\right\rangle =\int\int_{\Omega}\left\{  \left(
\frac{\omega}{\sqrt{g^{\prime}\left(  \psi_{0}\right)  }}-\psi\sqrt{g^{\prime
}\left(  \psi_{0}\right)  }\right)  ^{2}-g^{\prime}\left(  \psi_{0}\right)
\psi^{2}+\left\vert \nabla\psi\right\vert ^{2}\right\}  dxdy\nonumber\\
=  &  \int\int_{\Omega}[\left(  \frac{\omega}{\sqrt{g^{\prime}\left(  \psi
_{0}\right)  }}-\sqrt{g^{\prime}\left(  \psi_{0}\right)  }\left(  I-P\right)
\psi\right)  ^{2}+g^{\prime}\left(  \psi_{0}\right)  \left(  P\psi\right)
^{2}\nonumber\\
&  \ \ \ \ \ \ \ \ \ \ \ \ \ -g^{\prime}\left(  \psi_{0}\right)  \psi
^{2}+\left\vert \nabla\psi\right\vert ^{2}]\ dxdy\nonumber\\
\geq &  \int\int_{\Omega}\left\vert \nabla\psi\right\vert ^{2}-g^{\prime
}\left(  \psi_{0}\right)  \psi^{2}+g^{\prime}\left(  \psi_{0}\right)  \left(
P\psi\right)  ^{2}dxdy=\left(  A\psi,\psi\right)  . \label{inequality-L>A}%
\end{align}
Next, for any $\psi\in H_{0}^{1}$, let
\[
\tilde{\omega}=g^{\prime}\left(  \psi_{0}\right)  \left(  I-P\right)  \psi
\in\overline{R\left(  J\right)  }, \quad\tilde\psi= (-\Delta)^{-1}
\tilde\omega,
\]
then
\begin{align}
\left(  A\psi,\psi\right)   &  =\int\int_{\Omega}\left\vert \nabla
\psi\right\vert ^{2}-g^{\prime}\left(  \psi_{0}\right)  \left(  \left(
I-P\right)  \psi\right)  ^{2}dxdy\label{inequality-A-L}\\
&  =\int\int_{\Omega}\left[  \left\vert \nabla\psi\right\vert ^{2}%
-\frac{\tilde{\omega}^{2}}{g^{\prime}\left(  \psi_{0}\right)  }\right]
dxdy\nonumber\\
&  =\int\int_{\Omega}\left[  \left\vert \nabla\psi\right\vert ^{2}%
-2\tilde{\omega}\psi+\frac{\tilde{\omega}^{2}}{g^{\prime}\left(  \psi
_{0}\right)  }\right]  dxdy\nonumber\\
&  \geq\int\int_{\Omega}\left[  \frac{\tilde{\omega}^{2}}{g^{\prime}\left(
\psi_{0}\right)  }-\left\vert \nabla\tilde{\psi}\right\vert ^{2}\right]
dxdy=\left\langle L\tilde{\omega},\tilde{\omega}\right\rangle .\nonumber
\end{align}
From (\ref{inequality-L>A}), (\ref{inequality-A-L}) and
(\ref{equality-kernel-L-R(J)}), we get (\ref{equality-negative-modes-Euler})
as in the proof of i).
\end{proof}

By Lemma \ref{lemma-euler-spectral} and (iii) of Proposition
\ref{prop-counting-k-0-2}, we have

\begin{theorem}
\label{thm-stability-euler-1}Assume $g^{\prime}\left(  \psi_{0}\right)  >0$
and $\ker A=\left\{  0\right\}  $, then the index formula
\begin{equation}
k_{r}+2k_{c}+2k_{i}^{\leq0}=n^{-}\left(  A\right)  .
\label{Euler-index-formula}%
\end{equation}
holds. In particular, when $n^{-}\left(  A\right)  $ is odd, there is linear
instability; when $A>0$, there is linear stability.
\end{theorem}

\begin{proof}
To apply (iii) of Proposition \ref{prop-counting-k-0-2} to obtain
\eqref{Euler-index-formula}, it suffices to verify that a.) $\langle
L\cdot,\cdot\rangle$ is non-degenerate on $\ker(JL)/\ker L$, which is
satisfied due to Lemma \ref{L:counting-k-0-2}, Lemma
\ref{lemma-euler-spectral}, and $\{0\}=\ker A\subset\ker A_{0}$; and b.)
\[
\tilde{S}\triangleq\overline{R(J)}\cap(JL)^{-1}(\ker L)=\{0\}.
\]
To see the latter, we first note that $\ker A=\{0\}$ and \eqref{eqn-ker-L-R}
imply
\[
\overline{R(J)}\cap\ker\big((I-P)L\big)=\{0\}.
\]
Consequently, if $\omega\in\overline{R(J)}\cap(JL)^{-1}(\ker L)$, then
$JL\omega\in R(J)\cap\ker L$ must vanish, namely, $L\omega\in\ker J$, and thus
$(I-P)L\omega=0$. Again, since $\omega\in\overline{R(J)}$, we obtain
$\omega=0$. Therefore, $\tilde{S}=\{0\}$ and \eqref{Euler-index-formula} follows.

The instability of $e^{tJL}$ under the assumption of $n^{-}(A)$ being odd is
straightforward from \eqref{Euler-index-formula}. Finally suppose $A>0$,
\eqref{equality-negative-modes-Euler} and \eqref{equality-kernel-L-R(J)} imply
that $\langle L\cdot,\cdot\rangle$ is uniformly positive definite on
$\overline{R(J)}/\big(\overline{R(J)}\cap\ker L\big)=\overline{R(J)}$.
Therefore, $e^{tJL}$ is stable on the closed invariant subspace $\overline
{R(J)}$ and thus its stability follows from the decomposition $X=\ker
(JL)+\overline{R(J)}$, which is proved in Proposition
\ref{prop-counting-k-0-2}.
\end{proof}

By Theorem \ref{T:USImSpec}, the index formula (\ref{Euler-index-formula}) and
the fact $i\mathbf{R} \subset\sigma(JL)$ for the linearized Euler equation
imply the following.

\begin{corollary}
\label{cor-structure-insta-Euler}Under the assumption of Theorem
\ref{thm-stability-euler-1}, when $n^{-}\left(  A\right)  >0$, then there is
linear instability or structural instability for $JL$ (in the sense of Theorem
\ref{T:USImSpec}).
\end{corollary}

In the sense of Theorem \ref{T:USImSpec}, the structural instability of the
linearized Euler equation $\omega_{t} = JL \omega$ means that there exist
arbitrarily small bounded perturbations $L_{\#}$ to $L$ such that $JL_{\#}$
has unstable eigenvalues. However, it is not clear that such perturbations can
be realized in the context of the Euler equation, such as by considering
neighboring steady states along with possible small domain variation.

\begin{remark}
\label{remark-nonlinear-stability-Euler}In \cite{lin-cmp-04}, it was shown
that for general $g\in C^{1}$, when $\ker A=\left\{  0\right\}  $ and
$n^{-}\left(  A\right)  $ is odd, there is linear instability. Here, the index
formula (\ref{Euler-index-formula}) gives more detailed information about the
spectrum of the linearized Euler operator.

We give one example satisfying the stability condition $A>0$. Let $\lambda
_{0}>0$ be the lowest eigenvalue of $-\Delta$ in $\Omega$ with Dirichlet
boundary condition and $\psi_{0}$ be the corresponding eigenfunction. Then
$g^{\prime}\left(  \psi_{0}\right)  =\lambda_{0}$ and it is easy to show that
$A>0$.
\end{remark}

\begin{remark}
When the domain $\Omega$ is not simply connected, let $\partial\Omega
=\cup_{i=0}^{n}\Gamma_{i}$ consist of outer boundary $\Gamma_{0}$ and $n$
interior boundaries $\Gamma_{1},\cdots,\Gamma_{n}$. Then the operators
$A_{0},\ A,\ -\Delta$ should be defined by using the boundary conditions:
\begin{equation}
\phi|_{\Gamma_{i}}\ \text{is\ constant},\ \oint_{\Gamma_{i}}\frac{\partial
\phi}{\partial n}=0\ \text{and }\int\int_{\Omega}\phi\ dxdy=0.
\label{bc-non-simply-connected}%
\end{equation}
The same formula (\ref{Euler-index-formula}) is still true. The linearized
stream functions satisfying \eqref{bc-non-simply-connected} represent
perturbations preserving the circulations along each $\Gamma_{i}$, which are
conserved in the nonlinear evolution.
\end{remark}

Below, we consider the case when $\ker A$ is nontrivial. This usually happens
when the problem has some symmetry. As an example, we consider the case when
$\Omega$ is a channel, that is,
\[
\Omega=\left\{  y_{1}\leq y\leq y_{2},\ x\text{ is }T-\text{periodic}\right\}
.
\]
The steady stream function $\psi_{0}$ satisfies
\begin{equation}
-\Delta\psi_{0}=g\left(  \psi_{0}\right)  \text{ in }\Omega
,\label{eqn-steady-channel}%
\end{equation}
with boundary conditions $\psi_{0}$ being constants on $\left\{
y=y_{i}\right\}  $, $i=1,2,$ where $g\in C^{1}$. Define the operators
$L,\ A_{0},\ A$ as before with the boundary conditions
\begin{equation}
\phi\ \text{is\ constant on }\left\{  y=y_{i}\right\}  ,\ \int_{\left\{
y=y_{i}\right\}  }\frac{\partial\phi}{\partial y}%
dx=0,\ i=1,2,\label{bc-channel-eigenfunction}%
\end{equation}
and $\int\int_{\Omega}\phi\ dxdy=0$. Taking $x-$derivative of equation
(\ref{eqn-steady-channel}), we get
\[
-\Delta\psi_{0,x}=g^{\prime}\left(  \psi_{0}\right)  \psi_{0,x}\ \text{in
}\Omega,
\]
and $\psi_{0,x}$ satisfies the boundary condition
(\ref{bc-channel-eigenfunction}). Thus we have $A_{0}\psi_{0,x}=0$ and
\[
L\omega_{0,x}=L\left(  g^{\prime}\left(  \psi_{0}\right)  \psi_{0,x}\right)
=0.
\]
Since $\psi_{0,x}=u_{0}\cdot\nabla\left(  -y\right)  $, so $P\psi_{0,x}=0$ and
thus $A\psi_{0,x}=A_{0}\psi_{0,x}=0$.

\begin{theorem}
\label{thm-Euler-index-2}Assume $g^{\prime}\left(  \psi_{0}\right)  >0$ and
$\ker A=span\left\{  \psi_{0,x}\right\}  ,$ then
\[
\exists\,\omega_{1}\in\overline{R(J)}\;\text{ such that }JL\omega_{1}%
=\omega_{0,x}.
\]
Moreover, if
\[
d=\left\langle L\omega_{1},\omega_{1}\right\rangle =-\int\int_{\Omega}%
y\omega_{1}dxdy=T(\psi_{1}|_{y=y_{1}}-\psi_{1}|_{y=y_{2}})\neq0,
\]
where $\psi_{1}$ satisfies $-\Delta\psi_{1}=\omega_{1}$ with the boundary
condition \eqref{bc-channel-eigenfunction}, or more explicitly
\[
\psi_{1}=-A^{-1}\left(  g^{\prime}\left(  \psi_{0}\right)  \left(  I-P\right)
y\right)  ,
\]
then we have the index formula
\begin{equation}
k_{r}+2k_{c}+2k_{i}^{\leq0}=n^{-}\left(  A\right)  -n^{-}\left(  d\right)
,\label{index-euler-2}%
\end{equation}
where $n^{-}(d)=1$ if $d<0$ and $n^{-}(d)=0$ if $d>0$.
\end{theorem}

\begin{proof}
Our assumption implies $\ker A\subset\ker A_{0}$ and thus $\langle
L\cdot,\cdot\rangle$ is non-degenerate on $\overline{R(J)}/\big(\overline
{R(J)}\cap\ker L\big)$ by Proposition \ref{prop-counting-k-0-2}. To apply
index formula \eqref{index-formula-D}, we need to obtain the non-degeneracy of
$\langle L\cdot,\cdot\rangle$ on $\tilde{S}/(\tilde{S}\cap\ker L)$ and compute
$n^{\leq0}(L|_{\tilde{S}/(\tilde{S}\cap\ker L)})$, where
\[
\tilde{S}=\overline{R(J)}\cap(JL)^{-1}(\ker L)=\overline{R(J)}\cap
(JL)^{-1}\big(\ker L\cap\overline{R(J)}\big).
\]
Since $P\psi_{0,x}=0$ implies $\omega_{0,x}=\frac{\psi_{0,x}}{g^{\prime}%
(\psi_{0})}\in\overline{R(J)}$. From \eqref{eqn-ker-L-R} and our assumption on
$\ker A$, we have
\[
span\{\omega_{0,x}\}\subset\overline{R(J)}\cap\ker L\subset\overline{R(J)}%
\cap\ker\big((I-P)L\big)=\Delta\ker A=span\{\omega_{0,x}\},
\]
which yields
\[
\overline{R(J)}\cap\ker L=span\{\omega_{0,x}\}.
\]
By the definition of $\tilde{S}$, $\omega\in\tilde{S}$ if and only if there
exist $\omega\in\overline{R(J)}$ and $a\in\mathbf{R}$ such that
\[
a\omega_{0,x}=JL\omega=J(I-P)L\omega.
\]
Since $\omega_{0,x}=-Jy=-J(I-P)y$, we obtain equivalently $(I-P)(L\omega
+ay)=0$. From \eqref{E:A(omega)}, it follows that
\[
\omega\in\tilde{S}\Longleftrightarrow\omega\in\overline{R(J)}\,\text{ and
}A\psi=-ag^{\prime}(\psi_{0})(I-P)y,\,\text{ where }-\Delta\psi=\omega.
\]

Note $\ker A=span\{\psi_{0,x}\}$ and
\[
\left\langle g^{\prime}\left(  \psi_{0}\right)  \left(  I-P\right)
y,\psi_{0,x}\right\rangle =\int\int_{\Omega}yg^{\prime}\left(  \psi
_{0}\right)  \psi_{0,x}\ dxdy=0.
\]
There exists a stream function $\psi_{1}$ satisfying
\[
A\psi_{1}=-g^{\prime}\left(  \psi_{0}\right)  \left(  I-P\right)  y,
\]
which implies $\omega_{1}=-\Delta\psi_{1}\in\overline{R(J)}$ and $JL\omega
_{1}=\omega_{0,x}$. Namely $\omega_{1}\in\tilde{S}$ and $\tilde{S}%
=span\{\omega_{0,x},\omega_{1}\}$. One may compute
\begin{align*}
d  &  =\left\langle L\omega_{1},\omega_{1}\right\rangle =\left\langle \left(
I-P\right)  L\omega_{1},\omega_{1}\right\rangle =\left\langle -\left(
I-P\right)  y,\omega_{1}\right\rangle \\
&  =-\left\langle y,\omega_{1}\right\rangle =T(\psi_{1}|_{y=y_{1}}-\psi
_{1}|_{y=y_{2}})
\end{align*}
where the last equal sign follows from integration by parts. If $d\neq0$, then
the desired index formula follows from (iii) of Proposition
\ref{prop-counting-k-0-2}.
\end{proof}

Similar to Corollary \ref{cor-structure-insta-Euler} (and the comments
immediately thereafter), we have

\begin{corollary}
Under the assumption of Theorem \ref{thm-Euler-index-2}, when $n^{-}\left(
A\right)  -n^{-}\left(  d\right)  >0$, then there is linear instability or
structural instability for $JL$.
\end{corollary}

As another application of the Hamiltonian structure of the linearized Euler
equation, we consider the inviscid damping of a stable steady flow. Assume
$g^{\prime}\left(  \psi_{0}\right)  >0$ and $A>0$, then by Theorem
\ref{thm-stability-euler-1}, the steady flow is linearly stable in the $L^{2}$
norm of vorticity. There is no time decay in $\left\Vert \omega\right\Vert
_{L^{2}}$. However, the linear decay in the velocity norm $\left\Vert u
\right\Vert $ $_{L^{2}}$ is possible due to the mixing of the vorticity. For
example, see \cite{lin-zeng-couette} for the linear damping near Couette flow
$\left(  y,0\right)  $ in a channel. Here, we give a weak form of the linear
decay for general stable steady flows.

\begin{theorem}
\label{thm-RAGE-Euler}Assume $g^{\prime}\left(  \psi_{0}\right)  >0$ and
$A>0$. For $\omega\left(  0\right)  \in\overline{R\left(  J\right)  }$, let
$\omega\left(  t\right)  \in\overline{R\left(  J\right)  }$ be the solution of
the linearized Euler equation (\ref{eqn-linearized-vorticity}). Then

(i) When $T\rightarrow\infty$, $\frac{1}{T}\int_{0}^{T}\omega\left(  t\right)
dt\rightarrow0$ strongly in $L^{2}$.

(ii) If there is no embedded imaginary eigenvalue of $JL$ on $\overline
{R\left(  J\right)  }$, then for any compact operator $C$ in $L^{2}$, we have
\begin{equation}
\frac{1}{T}\int_{0}^{T}\left\Vert C\omega\left(  t\right)  \right\Vert
_{L^{2}}^{2}dt\rightarrow0\text{, when }T\rightarrow\infty. \label{Rage-Euler}%
\end{equation}
In particular, for the velocity $u=\operatorname{curl}^{-1}\omega$,
\begin{equation}
\frac{1}{T}\int_{0}^{T}\left\Vert u \left(  t\right)  \right\Vert _{L^{2}}%
^{2}dt\rightarrow0\text{, when }T\rightarrow\infty.
\label{rage-Euler-velocity}%
\end{equation}

\end{theorem}

\begin{proof}
By Lemma \ref{lemma-euler-spectral}, $L|_{\overline{R\left(  J\right)  } }>0$.
Since $\overline{R\left(  J\right)  }$ is an invariant subspace of $JL$, we
can consider the operator $JL$ in $\overline{R\left(  J\right)  }$. Define the
inner product $\left[  \cdot,\cdot\right]  =\left\langle L\cdot,\cdot
\right\rangle $ on $\overline{R\left(  J\right)  }$, then the norm in $\left[
\cdot,\cdot\right]  \ $is equivalent to the $L^{2}$ norm. As noted before, the
operator $JL|_{\overline{R\left(  J\right)  }}$ is anti-self-adjoint with
respect to the inner product $\left[  \cdot,\cdot\right]  $.

(i) By the mean ergodic convergence of unitary operators (\cite{yosida-book})
\[
\lim_{T\rightarrow\infty}\frac{1}{T}\int_{0}^{T}\omega\left(  t\right)
dt=\lim_{T\rightarrow\infty}\frac{1}{T}\int_{0}^{T}e^{tJL|_{\overline{R\left(
J\right)  }}}\omega\left(  0\right)  dt=P_{0}\omega\left(  0\right)
\]
in $L^{2}$, where $P_{0}$ is the projection operator from $\overline{R\left(
J\right)  }$ to $\ker JL|_{\overline{R\left(  J\right)  }}$ orthogonal with
respect to $[\cdot,\cdot]$. Since $\ker A=\left\{  0\right\}  $, by Lemma
\ref{lemma-euler-spectral} and \eqref{E:A(omega)} in particular, $\ker
JL|_{\overline{R\left(  J\right)  }}=\left\{  0\right\}  $ and thus
$P_{0}\omega\left(  0\right)  =0$.

(ii) If $JL$ has no embedded imaginary eigenvalue, then (\ref{Rage-Euler})
follows directly by the RAGE theorem (\cite{CFKS}), again by using the
anti-self-adjoint property of $JL|_{\overline{R\left(  J\right)  }}$. The
conclusion (\ref{rage-Euler-velocity}) follows by choosing the compact
operator $C=\operatorname{curl}^{-1}$.
\end{proof}

\begin{remark}
Assuming $A>0$, from the proof of Theorem \ref{thm-stability-euler-1}, the
subspace $\tilde{S}$ defined in Proposition \ref{prop-counting-k-0-2} is
trivial. By Proposition \ref{prop-counting-k-0-2}, there is a direct sum
decomposition $L^{2}=\ker\left(  JL\right)  \oplus\overline{R\left(  J\right)
}$ invariant under $JL$. In fact $\ker\left(  JL\right)  $ corresponds to the
steady solution of the linearized Euler equation. So above Lemma shows that
for any initial data in $L^{2}$, in the time averaged limit, the solution of
the linearized Euler equation converges to a steady solution. This is a weak
form of inviscid damping.
\end{remark}

A stable example satisfying the assumption $A>0$ in Theorem
\ref{thm-RAGE-Euler} is given in Remark \ref{remark-nonlinear-stability-Euler}%
. Below, we consider two examples of stable shear flows. First, we consider
the Poisseulle flow $U\left(  y\right)  =y^{2}$ in a $2\pi$-periodic channel
$\left\{  -1<y<1\right\}  $. The linearized Euler equation becomes
\[
\partial_{t}\omega+y^{2}\partial_{x}\omega+2\partial_{x}\psi=0.
\]
Consider the subspace of non-shear vorticities with a weighted $L^{2}$ norm
\[
X_{1}=\left\{  \omega=\sum_{k\in\mathbf{Z},\ k\neq0}e^{ikx}\omega_{k}\left(
y\right)  ,\ \Vert\omega\Vert_{X_{1}}^{2}=\sum_{k\in\mathbf{Z},\ k\neq0}\Vert
y\omega_{k}\Vert_{L^{2}}^{2}<\infty\right\}  .
\]
Define $J=-\partial_{x}$ and $L=y^{2}+2\left(  -\Delta\right)  ^{-1}$. Then
$L$ is uniformly positive on $X_{1}$.

Second, consider the Kolmogorov flow $U\left(  y\right)  =\sin y$ in a torus
$T^{2}=S_{\frac{2\pi}{\alpha}}\times S_{2\pi}$ with $\alpha>1$. Here
$\alpha>1$ is the sharp stability condition since the shear flow is unstable
when $\alpha<1$. The linearized equation is
\[
\partial_{t}\omega+\sin y\partial_{x}\left(  \omega-\psi\right)  =0.
\]
Let $J=\sin y\partial_{x}$ and $L=1-\left(  -\Delta\right)  ^{-1}$. Then $L$
is uniformly positive on
\[
X_{2}=\left\{  \omega=\sum_{k\in\mathbf{Z},\ k\neq0}e^{ikx}\omega_{k}\left(
y\right)  ,\ \omega\in L^{2}\right\}
\]
when $\alpha>1$. It can be shown (\cite{lin-Xu}) that for above two examples,
the linearized Euler operator has no embedded eigenvalues. Therefore, Theorem
\ref{thm-RAGE-Euler} (ii) is true for the above two shear flows in $X_{1}$ and
$X_{2}$ respectively. In particular, if we choose $C$ to be $P_{N}$, the
projection operator to the first $N$ Fourier modes (in $x$), then
\begin{equation}
\frac{1}{T}\int_{0}^{T}\left\Vert P_{N}\omega\left(  t\right)  \right\Vert
_{L^{2}}^{2}dt\rightarrow0,\text{when }T\rightarrow\infty.
\label{enstrophy-cascade}%
\end{equation}
This shows that in the time averaged sense, the low frequency parts of
$\omega$ tends to zero. This observation was used to prove (\cite{lin-Xu}) the
metastability of Kolmogorov flows. In the fluid literature (see e.g.
\cite{tableling}), for 2D turbulence a dual cascade was known that energy
moves to low frequency end and the enstrophy ($\int\omega^{2}dx$) moves to the
high frequency end. The result (\ref{enstrophy-cascade}) can be seen as a
justification of such physical intuition in a weak sense.

\begin{remark}
Two classes of shear flows generalizing the above two examples are studied in
\cite{lin-Xu}. The linear inviscid damping in the sense of
(\ref{rage-Euler-velocity}) is proved for stable shear flows and on the center
space for the unstable shear flows, when $\omega\left(  0\right)  \in L^{2}$
is non-shear. Recently, for monotone and certain symmetric shear flows, more
explicit linear decay estimates of the velocity were obtained in
\cite{zillinger, Wei-Zhang-Zhao2017, Wei-Zhang-Zhao1704} for more regular
initial data (e.g. $\omega\left(  0\right)  \in H^{1}$ or $H^{2}$).

In \cite{lin-zhu}, the stability of shear flows under Coriolis forces is
studied. By using the instability index Theorem \ref{theorem-counting}, the
sharp stability condition for a class of shear flows can be obtained. Then the
linear damping as in the above sense is proved for non-shear $\omega\left(
0\right)  \in L^{2}$.
\end{remark}

\subsection{Stability of traveling waves of 2-dim nonlinear Schr\"{o}dinger
equations with nonzero conditions at infinity}

\label{SS:2dGGP}

In this subsection, we consider the nonlinear Schr\"{o}dinger equation (NLS)
\begin{equation}
i\frac{\partial u}{\partial t}+\Delta u+F(|u|^{2})u=0,\quad u=u_{1}%
+iu_{2}:\mathbf{R}\times\mathbf{R}^{2}\rightarrow\mathbf{C}. \label{E:GGP}%
\end{equation}
In particular, we assume that the nonlinearity $F(s)$ satisfies
\begin{equation}
F\in C^{2},\quad F(1)=0,\quad F^{\prime}(1)<0. \label{E:F1}%
\end{equation}
Important well-known equations of this type are Gross-Pitaevskii (GP) equation
with $F\left(  s\right)  =1-s\ $and the cubic-quintic NLS with $F\left(
s\right)  =-\alpha_{1}+\alpha_{3}s-\alpha_{5}s^{2}$, where $\alpha_{1}%
,\alpha_{3}$ and $\alpha_{5}$ are positive constants. Assume $s=1$ is a local
minimal point of $F$, it is natural to consider solutions $u(t,x)$ satisfying
the following boundary condition in some appropriate sense
\begin{equation}
|u|\rightarrow1\;\text{ as }|x|\rightarrow\infty. \label{E:GGP-BC}%
\end{equation}
After normalization, we can assume that $u\rightarrow1$ when $\left\vert
x\right\vert \rightarrow\infty$ in some weak sense such as $u-1$ being
approximable by Schwartz class functions in certain Sobolev norms. The
equation (\ref{E:GGP}) has the conserved energy and momentum functionals
\begin{align*}
&  E\left(  u\right)  =\frac{1}{2}\int_{\mathbf{R}^{2}}\big(\left\vert \nabla
u\right\vert ^{2}+V(\left\vert u\right\vert ^{2})\big)dx,\\
\vec{P}\left(  u\right)   &  =\left(  P_{1}\left(  u\right)  ,P_{2}\left(
u\right)  \right)  =\frac{1}{2}\int_{\mathbf{R}^{2}}\langle\nabla u,i\left(
u-1\right)  \rangle\ dx=\int_{\mathbf{R}^{2}}\left(  u_{1}-1\right)  \nabla
u_{2}dx,
\end{align*}
\textrm{where}$\ V(s)=\int_{s}^{1}F(\tau)d\tau$. We also denote the first
component of $\vec{P}\left(  u\right)  $ by
\[
P\left(  u\right)  =\frac{1}{2}\int_{\mathbf{R}^{2}}\langle\partial_{x_{1}%
}u,i\left(  u-1\right)  \rangle\ dx=\int_{\mathbf{R}^{2}}\left(
u_{1}-1\right)  \partial_{x_{1}}u_{2}dx.
\]

A traveling wave (without loss of generality, in $x_{1}$-direction) of
\eqref{E:GGP} with wave speed $c\in\left(  0,\sqrt{2}\right)  $ is a solution
in the form of $u=U_{c}(x_{1}-ct,x_{2})$, where $U_{c}$ satisfies the elliptic
equation
\begin{equation}
-ic\partial_{x_{1}}U_{c}+\Delta U_{c}+F(|U_{c}|^{2})U_{c}=0, \label{TW-GGP}%
\end{equation}
with the boundary condition $U_{c}\rightarrow1$ when $|x|\rightarrow\infty$ in
the sense $U_{c}-1\in\dot{H}^{1}$. Here, $\sqrt{2}$ is the sound speed and
when $c\geq\sqrt{2}$, in general the traveling waves do not exist (see e.g.
\cite{Maris-non-existence}). Formally, $U_{c}$ is a critical point of $E-cP$.
Our goal is to understand the linear stability/instability of such a traveling
wave, namely, the evolution of the linearized equation of \eqref{E:GGP} at
$U_{c}=u_{c}+iv_{c}$ put in the moving frame $x_{1}\rightarrow x_{1}%
-ct,x_{2}\rightarrow x_{2}$:
\begin{equation}
u_{t}=JL_{c}u,\quad u=(u_{1},u_{2})^{T}\rightarrow0\;\text{ as }%
|x|\rightarrow\infty, \label{E:LGGP}%
\end{equation}
where $J=%
\begin{pmatrix}
0 & 1\\
-1 & 0
\end{pmatrix}
$ and
\[
L_{c}:=\left(
\begin{array}
[c]{cc}%
-\Delta-F\left(  \left\vert U_{c}\right\vert ^{2}\right)  -2F^{\prime}\left(
\left\vert U_{c}\right\vert ^{2}\right)  u_{c}^{2} & -c\partial_{x_{1}%
}-2F^{\prime}\left(  \left\vert U_{c}\right\vert ^{2}\right)  u_{c}v_{c}\\
c\partial_{x_{1}}-2F^{\prime}\left(  \left\vert U_{c}\right\vert ^{2}\right)
u_{c}v_{c} & -\Delta-F\left(  \left\vert U_{c}\right\vert ^{2}\right)
-2F^{\prime}\left(  \left\vert U_{c}\right\vert ^{2}\right)  v_{c}^{2}%
\end{array}
\right)  .
\]
Through $L^{2}$ duality, $L_{c}$ generates the quadratic form
\begin{align}
\langle L_{c}u,v\rangle=\int_{\mathbf{R}^{2}}  &  {\large \{}\nabla
u\cdot\nabla v+c(v_{1x_{1}}u_{2}-u_{1}v_{2x_{1}})-F(|U_{c}|^{2})u\cdot
v\nonumber\\
&  \qquad\qquad\qquad\qquad-2F^{\prime}(|U_{c}|^{2})(U_{c}\cdot u)(U_{c}\cdot
v)\,{\large \}}\ dx, \label{E:GGP-L1}%
\end{align}
where $u\cdot v=\operatorname{Re}(u\bar{v})$.

For the purpose of studying the linearized equation (\ref{E:LGGP}), we make
the following assumptions:

\begin{enumerate}
\item[(\textbf{NLS-1})] $U_{c}-1\in H^{1}\times\dot{H}^{1}$ satisfies
\eqref{E:GGP-BC} and $|U_{c}|_{C^{1}(\mathbf{R}^{2})}<\infty$.

\item[(\textbf{NLS-2})] Let $\Gamma$ be the collection of subspaces $S\subset
H^{1}(\mathbf{R}^{2})\times H^{1}(\mathbf{R}^{2})$ such that $\langle
L_{c}u,u\rangle<0$ for all $0\neq u\in S$, then
\[
\max\{\dim S\mid S\in\Gamma\}=n^{-}(L_{c})<\infty.
\]

\end{enumerate}

The above (\textbf{NLS-1}) is a natural regularity assumption. For any given
traveling wave of \eqref{E:GGP}, it is probably not so straightforward to
verify (\textbf{NLS-2}). This, however, would be a direct consequence if
$U_{c}$ is obtained through a constrained variational approach related to
energy and momentum, which is often the case. For example, in
\cite{chiron-maris-13} \cite{chiron-scheid}, the 2D traveling waves of
(\ref{TW-GGP}) were constructed by minimizing the functional $E\left(
u\right)  -cP\left(  u\right)  $ subject to a constraint $P\left(  u\right)
=p$ or $E_{kin}\left(  u\right)  =\int\left\vert \nabla u\right\vert ^{2}%
dx=k$, for general nonlinearity $F$. The variational problem of minimizing
$E\left(  u\right)  -cP\left(  u\right)  $ subject to fixed $P\left(
u\right)  $ was also studied in \cite{betheul-et-minimizer-cmp} to construct
2D traveling waves of GP equation. Since these 2D traveling waves $U_{c}%
\ $were minimizers of $E\left(  u\right)  -cP\left(  u\right)  $ subject to
one constraint, it can be shown that $n^{-}\left(  L_{c}\right)  \leq1$ (see
e.g. the proof of Lemma 2.7 of \cite{lin-wang-zeng}). Here, we note that
$U_{c}$ is a critical point of $E\left(  u\right)  -cP\left(  u\right)  $ and
$L_{c}=E^{\prime\prime}\left(  U_{c}\right)  -cP^{\prime\prime}\left(
U_{c}\right)  $.

To study the quadratic form $\left\langle L_{c}\cdot,\cdot\right\rangle $,
obviously one may take $X=H^{1}(\mathbf{R}^{2})\times H^{1}(\mathbf{R}^{2})$.
On the one hand, the above assumptions ensure that $L_{c}:X\rightarrow
X^{\ast}=H^{-1}\times H^{-1}$ is bounded, satisfies $L_{c}^{\ast}=L_{c}$, and
has $n^{-}(L_{c})$ negative dimensions. On the other hand, it is easy to see
that $J:X^{\ast}\rightarrow X$ is unbounded, but has a dense domain
$H^{1}\times H^{1}\subset X^{\ast}=H^{-1}\times H^{-1}$, and satisfies
$J^{\ast}=-J$. However, as the boundary condition \eqref{E:GGP-BC} does not
provide enough control of $|u|^{2}$ near $|x|=\infty$ in $\langle
L_{c}u,u\rangle$, it is not clear that (\textbf{H2.b}) can be satisfied by any decomposition.

For \eqref{E:GGP} considered on $\mathbf{R}^{N}$, $N\geq3$, as in
\cite{lin-wang-zeng}, it would be possible to work on $X=H^{1}\times\dot
{H}^{1}$, where $u_{1}\in H^{1}$ and $u_{2}\in\dot{H}^{1}$, and verify
assumptions (\textbf{H1-3}) for $J$ and $L_{c}$ based on the following two
observations. Firstly, in such higher dimensions, the Gagliardo-Nirenberg
inequality implies that $\dot{H}^{1}$ functions decay at $x=\infty$ in the
$L^{p}$ sense. Therefore, we may reasonably strengthen the boundary condition
\eqref{E:GGP-BC} to $U_{c}\rightarrow1$ as $|x|\rightarrow\infty$.
Consequently the `principle part' in $\langle L_{c}u,u\rangle$ provides the
control on the $H^{1}\times\dot{H}^{1}$ norm of $u$. Secondly, there are
indications that $U_{c}$ decays like $u_{c}-1=O(|x|^{-N})$ and $v_{c}%
=O(|x|^{1-N})$ as in the case proved for the (GP) equation in \cite{B-G-S}.
Along with the Hardy inequality, this allows us to control those terms in
\eqref{E:GGP-L1} with vanishing variable coefficients by the $H^{1}\times
\dot{H}^{1}$ norm of $u$. See \cite{lin-wang-zeng} for more details.

The situation is much worse on $\mathbf{R}^{2}$ unfortunately since both of
the above key observations break down on $\mathbf{R}^{2}$. To overcome these
difficulties, our idea is to study the stability of the linearized equation
\eqref{E:LGGP} on some space roughly between $H^{1}\times H^{1}$ and $\dot
{H}^{1}\times\dot{H}^{1}$ defined according to the properties of $L_{c}$ by
applying Theorem \ref{T:degenerate}.

Let $X=H^{1}\times H^{1}\ $for \eqref{E:LGGP} and define $Q_{0},Q_{1}%
:X\rightarrow X^{\ast}$ as
\[
\langle Q_{0}u,v\rangle=\operatorname{Re}\int_{\mathbf{R}^{2}}u\bar{v}%
dx,\quad\langle Q_{1}u,v\rangle=\operatorname{Re}\int_{\mathbf{R}^{2}}\left(
u_{x_{1}}\bar{v}_{x_{1}}+u_{x_{2}}\bar{v}_{x_{2}}\right)  \ dx,
\]
namely, the $L^{2}$ and $\dot{H}^{1}$ duality, respectively, which satisfy
(\textbf{B1}) in Subsection \ref{SS:degenerate}. Let $\mathbb{J}:X\rightarrow
X$\ be $\mathbb{J}=%
\begin{pmatrix}
0 & 1\\
-1 & 0
\end{pmatrix}
$. Clearly, $\mathbb{J}$ satisfies (\textbf{B2}) and the unbounded operator
$J=\mathbb{J}Q_{0}^{-1}:X^{\ast}\rightarrow X$ has the same matrix
representation through the $L^{2}$ duality. As $L_{c}-Q_{1}$ consists of terms
of at most one order of derivative, it satisfies (\textbf{B3}). From
(\textbf{NLS-2}), there exists a subspace $S\subset X$ such that $\dim
S=n^{-}(L_{c})$ and $L_{c}$ is negative definite on $S$. By a slight
perturbation, e.g. applying the mollifier to a basis of $S$, we obtain a
subspace $X_{-}\subset H^{3}\times H^{3}$ such that $\dim X_{-}=n^{-}(L_{c})$
and $L_{c}$ is negative definite on $X_{-}$. Let
\[
X_{\geq0}=X_{-}^{\perp_{L_{c}}}=\{u\in X\mid\langle L_{c}v,u\rangle
=0\}\supset\ker L_{c},
\]
and
\[
X_{+}=\{u\in X_{\geq0}\mid\int_{\mathbf{R}^{2}}u\cdot vdx=0,\,\forall v\in\ker
L\}.
\]
Since $\dim X_{-}<\infty$ and $L_{c}$ is negative definite on $X_{-}$, from
Lemma \ref{L:non-degeneracy} where (\textbf{H2.b}) is not necessary (see
Remark \ref{R:non-degeneracy}), we have $X=X_{-}\oplus X_{\geq0}$. It is
obvious $X_{\geq0}=X_{+}\oplus\ker L_{c}$ and the decomposition $X=X_{-}%
\oplus\ker L_{c}\oplus X_{+}$ is $L_{c}$-orthogonal. From (\textbf{NLS-2}) and
the definition of $X_{\pm}$, $\langle L_{c}u,u\rangle$ is (not necessarily
uniformly) positive on $X_{+}$ and thus (\textbf{B4}) is satisfied. Finally,
one may compute from the construction that
\[
\ker i_{X_{+}}^{\ast}=Q_{0}(\ker L_{c})\oplus L_{c}(X_{-}).
\]
Since we take $X_{-}\subset H^{3}\times H^{3}$ and $|U_{c}|_{C^{1}}<\infty$,
(\textbf{B5}) is also satisfied. From Theorem \ref{T:degenerate}, there exists
a function space $Y$ roughly between $X=H^{1}\times H^{1}$ and $\dot{H}%
^{1}\times\dot{H}^{1}$, an extension $L_{c,Y}:Y\rightarrow Y^{\ast}$ of
$L_{c}$, and the restriction
\[
J_{Y}:Y^{\ast}\supset D(J_{Y})\rightarrow Y
\]
of $J$, such that $(Y,L_{c,Y},J_{Y})$ satisfies assumption (\textbf{H1-3}).
Therefore, all our main results apply to the linearized NLS \eqref{E:LGGP} on
$Y$.

In the rest of this subsection, we assume, for some $c_{0}>0$,

\begin{enumerate}
\item[(\textbf{NLS})] There exists a $C^{1}$ curve of traveling waves for
$c\ $near$\ c_{0}$ satisfying (\textbf{NLS-1}) such that $n^{-}\left(
L_{c_{0}}\right)  \leq1$ and (\textbf{NLS-2}) is satisfied for $c=c_{0}$.
\end{enumerate}

As mentioned in the above, $n^{-}(L_{c_{0}}) \le1$ is satisfied if $U_{c_{0}}$
is constructed as minimizers of $E-c_{0}P$ subject to one constraint such as
fixed $P\left(  u\right)  $ or $E_{kin}\left(  u\right)  $. We shall apply
Theorem \ref{theorem-counting} to study the linearized equation (\ref{E:LGGP})
on $Y.$ In order to estimate $k_{0}^{\leq0}$ in the counting formula
(\ref{counting-formula}), differentiating (\ref{TW-GGP}) in $x_{i}$ and we get
$\ker L_{c_{0}}\supset\left\{  \partial_{x_{i}}U_{c_{0}}, i=1,2 \right\}  $.
Moreover, differentiating (\ref{TW-GGP}) in $c$, we have
\[
L_{c_{0}}\partial_{c}U_{c}|_{c_{0}}=P^{\prime}(U_{c_{0}})=J^{-1}%
\partial_{x_{1}}U_{c_{0}},
\]
and thus $JL_{c_{0}}\partial_{c}U_{c}|_{c_{0}}\in\ker L_{c_{0}}$. Since
\[
\left\langle L_{c_{0}}\partial_{c}U_{c}|_{c_{0}},\partial_{c}U_{c}|_{c_{0}%
}\right\rangle =\frac{dP\left(  U_{c}\right)  }{dc}|_{c_{0}},
\]
by Proposition \ref{prop-counting-k-0-1}, we have $k_{0}^{\leq0}\geq1$ when
$\frac{dP\left(  U_{c}\right)  }{dc}|_{c_{0}}\leq0$ and in this case $U_{c}$
is spectrally stable by (\ref{counting-formula}).

The traveling waves constructed in the literature
(\cite{betheul-et-minimizer-cmp} \cite{chiron-maris-13} \cite{chiron-scheid})
are even in $x_{2}$, that is, of the form $U_{c}\left(  x_{1},\left\vert
x_{2}\right\vert \right)  $. Thus, we can consider odd and even perturbations
(in $x_{2}$) respectively. We consider the even perturbations, that is, in the
space $Y_{e}=\left\{  u\in Y\ |\ u\text{ is even in }x_{2}\right\}  $. For
traveling waves as constrained minimizers of $E-cP$, in general it can be
shown that there is at least one even negative direction of $\left\langle
L_{c}\cdot,\cdot\right\rangle ,$ which then implies $n^{-}\left(
L_{c}|_{Y_{e}}\right)  =1$. Such a symmetry preserving negative direction of
$L_{c}$ was constructed in \cite{lin-wang-zeng} for the $3$D case. For the 2D
case, an even negative direction could be constructed by refining the Derrick
type arguments used in \cite{jones-et-stability}. More specifically, one can
consider a scaled traveling wave $U^{a,b}=U_{c_{0}}\left(  ax_{1}%
,bx_{2}\right)  $ and choose a family of parameters $a\left(  s\right)
,b\left(  s\right)  $ near $1$ with $a\left(  0\right)  =b\left(  0\right)
=1$ such that
\[
\left(  E-c_{0}P\right)  \left(  U^{a,b}\right)  <\left(  E-c_{0}P\right)
\left(  U_{c_{0}}\right)  ,
\]
from which an even negative direction $\frac{d}{ds}U^{a\left(  s\right)
,b\left(  s\right)  }|_{s=0}$ may be obtained. If in addition to the condition
$n^{-}\left(  L_{c}|_{Y_{e}}\right)  =1$, we assume that $\partial_{x_{1}%
}U_{c}$ is the only even kernel of $L_{c}$, then by Theorem
\ref{theorem-counting} and Proposition \ref{prop-counting-k-0-1}, there is
linear instability in case $\frac{dP\left(  U_{c}\right)  }{dc}|_{c_{0}}>0$.
We summarize above discussions in the following theorem.

\begin{theorem}
\label{thm: stability 2D GGP}(i) Assuming (\textbf{NLS}), the 2D traveling
wave $U_{c_{0}}\ $ is spectrally stable if $\frac{dP\left(  U_{c}\right)
}{dc}|_{c_{0}}\leq0.$

(ii) If we further assume that $U_{c_{0}}$ is even in $x_{2}$ and there exists
$v \in Y_{e}$ in the negative direction of $L_{c_{0}}$ and $\ker L_{c_{0}}
\cap Y_{e} = span\{\partial_{x_{1}}U_{c_{0}}\}$, then $\frac{dP\left(
U_{c}\right)  }{dc}|_{c_{0}}>0$ implies linear instability of $U_{c_{0}}$.
\end{theorem}

For the GP equation, by numerical computations (\cite{jones-et-stability})
$dP/dc$\thinspace$<0$ is true for the whole solitary wave branch. Thus 2D
traveling waves of GP are expected to be linearly stable. In
\cite{chiron-maris-13}, the orbital stability of these GP traveling waves was
obtained by showing concentration compactness of the constrained minimizing
sequence, under the assumption of local uniqueness of minimizers. The
transversal instability of 2D traveling waves of GP to 3D perturbation was
proved in \cite{lin-wang-zeng}. For general nonlinear term $F$ such as
cubic-quintic type, it is possible that there is an unstable branch of 2D
traveling waves with $dP/dc$\thinspace$>0$. See the numerical examples given
in \cite{chiron-scheid}.

Lastly, as a corollary of Theorems \ref{theorem-counting} and
\ref{thm: stability 2D GGP}, we prove that the traveling waves $U_{c_{0}}%
\ $have positive momentum $P\left(  U_{c_{0}}\right)  $.

\begin{corollary}
\label{cor-positive-momentum} Under the assumptions in both (i) and (ii) of
Theorem \ref{thm: stability 2D GGP}, except for the signs of $\frac{dP\left(
U_{c}\right)  }{dc}|_{c_{0}}$, we have $P\left(  U_{c_{0}}\right)  >0$.
\end{corollary}

\begin{proof}
First, we find $v_{2}$ such that $L_{c_{0}}v_{2}=J^{-1}\partial_{x_{2}%
}U_{c_{0}}$. Consider traveling waves $U_{\vec{c}}\ \left(  \vec{x}-\vec
{c}t\right)  $ with velocity vector $\vec{c}=\left(  c_{1},c_{2}\right)
\,\ $and $\left\vert \vec{c}\right\vert =c\in\left(  0,\sqrt{2}\right)  $,
which satisfies
\begin{equation}
-J\vec{c}\cdot\nabla U_{\vec{c}}+\Delta U_{\vec{c}}+F(|U_{\vec{c}}%
|^{2})U_{\vec{c}}=0.\label{tw-general-velocity}%
\end{equation}
Let
\[
Q=\frac{1}{\left\vert \vec{c}\right\vert }\left(
\begin{array}
[c]{cc}%
c_{1} & c_{2}\\
-c_{2} & c_{1}%
\end{array}
\right)
\]
be the rotating matrix which transforms $\vec{c}$ to $\left(  c,0\right)  $,
then it is easy to check that $U_{\vec{c}}\left(  \vec{x}\right)
=U_{c}\left(  Q\vec{x}\right)  $ is a solution of (\ref{tw-general-velocity})
and
\[
\vec{P}\left(  U_{\vec{c}}\right)  =Q^{T}\vec{P}\left(  U_{c}\right)
=P\left(  U_{c}\right)  \frac{\vec{c}}{c},
\]
where we use $\vec{P}\left(  U_{c}\right)  =P\left(  U_{c}\right)  \left(
1,0\right)  ^{T}$ which is due to the evenness of $U_{c}$ in $x_{2}$.
Differentiating (\ref{tw-general-velocity}) in $c_{2}$ and then evaluating at
$\left(  c_{0},0\right)  $, we get
\[
L_{c_{0}}\partial_{c_{2}}U_{\vec{c}}|_{\left(  c_{0},0\right)  }%
=J^{-1}\partial_{x_{2}}U_{c_{0}}.
\]
Thus we can choose $v_{2}=\partial_{c_{2}}U_{\vec{c}}|_{\left(  c_{0}%
,0\right)  }$ and
\[
\left\langle L_{c_{0}}v_{2},v_{2}\right\rangle =\partial_{c_{2}}P_{2}\left(
U_{\vec{c}}\right)  |_{\left(  c_{0},0\right)  }=\partial_{c_{2}}\left(
P\left(  U_{c}\right)  \frac{c_{2}}{c}\right)  |_{\left(  c_{0},0\right)
}=\frac{P\left(  U_{c_{0}}\right)  }{c_{0}}.
\]
Denote $v_{1}=\partial_{c}U_{c}|_{c_{0}}$ and recall that
\[
L_{c_{0}}v_{1}=J^{-1}\partial_{x_{1}}U_{c_{0}},\ \left\langle L_{c_{0}}%
v_{1},v_{1}\right\rangle =\frac{dP\left(  U_{c}\right)  }{dc}|_{c_{0}}.
\]
Also, by using the evenness of $U_{c_{0}}$ in $x_{2}$, we get
\[
\left\langle L_{c_{0}}v_{2},v_{1}\right\rangle =\left\langle J^{-1}%
\partial_{x_{2}}U_{c_{0}},\partial_{c}U_{c}|_{c_{0}}\right\rangle =0,
\]
and thus
\[
\left\langle L_{c_{0}}\cdot,\cdot\right\rangle |_{span\left\{  v_{1}%
,v_{2}\right\}  }=\left(
\begin{array}
[c]{cc}%
\frac{dP\left(  U_{c}\right)  }{dc}|_{c_{0}} & 0\\
0 & \frac{P\left(  U_{c_{0}}\right)  }{c_{0}}%
\end{array}
\right)  .
\]
Since
\[
n^{\leq0}\left(  L_{c_{0}}|_{span\left\{  v_{1},v_{2}\right\}  }\right)  \leq
k_{0}^{\leq0}\left(  L_{c_{0}}\right)  \leq n^{-}\left(  L_{c_{0}}\right)
\leq1,
\]
when $\frac{dP\left(  U_{c}\right)  }{dc}|_{c_{0}}\leq0$, we must have
$P\left(  U_{c_{0}}\right)  >0$. When $\frac{dP\left(  U_{c}\right)  }%
{dc}|_{c_{0}}>0$ and with the assumptions of Theorem
\ref{thm: stability 2D GGP} (ii), $U_{c_{0}}$ is linearly unstable, which
again implies that $P\left(  U_{c_{0}}\right)  >0$. Since otherwise $P\left(
U_{c_{0}}\right)  \leq0$, then $k_{0}^{\leq0}\left(  L_{c_{0}}\right)  \geq1$
and by Theorem \ref{theorem-counting}, $U_{c_{0}}$ is linearly stable, a contradiction.
\end{proof}

\begin{remark}
For 2D traveling wave solution $U_{c}$ satisfying (\ref{TW-GGP}), one can
prove the identity
\begin{equation}
cP\left(  U_{c}\right)  =2\int_{\mathbf{R}^{2}}V\left(  \left\vert
U_{c}\right\vert \right)  ^{2}dx, \label{virial-GP}%
\end{equation}
by using energy conservation and virial identity (see \cite{chiron-scheid} for
general $F$ and \cite{jones-et-stability} for GP). So for $F$ such that $V$ is
nonnegative (such as GP), we have $P\left(  U_{c}\right)  >0$ from
(\ref{virial-GP}). However, when $V$ also takes negative values (such as
cubic-quintic), then one can not conclude the sign of $P\left(  U_{c}\right)
$ from (\ref{virial-GP}). By using the index counting, above Corollary
\ref{cor-positive-momentum} shows that $P\left(  U_{c}\right)  >0$ is true for
any nonlinear term $F$ under the assumptions there.

Consider axial symmetric 3D traveling waves $U_{c}=\left(  x_{1},\left\vert
x^{\perp}\right\vert \right)  $ which are constrained energy-momentum
minimizers, as constructed in \cite{Maris-annal}. We can also prove that
$P\left(  U_{c}\right)  >0$ by the same arguments as in Corollary
\ref{cor-positive-momentum}. Actually, the argument for 3D is much simpler
than 2D and does not need the additional assumptions on $\ker L_{c}$. Let
\[
v_{1}=\partial_{c}U_{c},\ v_{j}=\partial_{c_{j}}U_{\vec{c}}|_{\left(
c,0,0\right)  },\ j=2,3,
\]
where $\vec{c}=\left(  c_{1},c_{2},c_{3}\right)  $ with $\left\vert \vec
{c}\right\vert =c\in\left(  0,\sqrt{2}\right)  $ and $U_{\vec{c}}$ is the
traveling wave with the velocity vector $\vec{c}$. Then we can compute in a
similar way that
\[
\left\langle L_{c}\cdot,\cdot\right\rangle |_{span\left\{  v_{1},v_{2}%
,v_{3}\right\}  }=\left(
\begin{array}
[c]{ccc}%
\frac{dP\left(  U_{c}\right)  }{dc} & 0 & 0\\
0 & \frac{P\left(  U_{c}\right)  }{c} & 0\\
0 & 0 & \frac{P\left(  U_{c}\right)  }{c}%
\end{array}
\right)  .
\]
Since
\[
n^{\leq0}\left(  L_{c}|_{span\left\{  v_{1},v_{2},v_{3}\right\}  }\right)
\leq n^{-}\left(  L_{c}\right)  \leq1
\]
by the index counting formula (\ref{counting-formula}), so regardless of the
sign of $\frac{dP\left(  U_{c}\right)  }{dc}$, we must have $P\left(
U_{c}\right)  >0$. The 3D analogue (see \cite{Maris-annal}) of the identity
(\ref{virial-GP}) is
\[
cP\left(  U_{c}\right)  =\int_{\mathbf{R}^{3}}\left\vert \frac{\partial U_{c}%
}{\partial x_{1}}\right\vert ^{2}dx+\int_{\mathbf{R}^{3}}V\left(  \left\vert
U_{c}\right\vert \right)  ^{2}dx,
\]
which is again not enough to conclude $P\left(  U_{c}\right)  >0$ when $V$
takes negative values.
\end{remark}

\section{Appendix}

\label{S:appendix}

In this appendix, we give some elementary properties of
\eqref{eqn-hamiltonian}, which are mostly based on theoretical functional
analysis arguments. They include some basic decomposition of the phase space,
the well-posedness of \eqref{eqn-hamiltonian}, and the standard
complexification procedure.

We start with some elementary properties of $L$. First we prove that $n^{-}(L)
= \dim X_{-}$ in assumption (\textbf{H2}) is the maximal dimension of
subspaces where $\langle L\cdot, \cdot\rangle<0$.

\begin{lemma}
\label{L:Morse-Index} If $N\subset X$ is a subspace such that $\langle Lu,
u\rangle<0$ for all $u \in N\backslash\{0\}$, then $\dim N \le n^{-}(L)$.
\end{lemma}

\begin{proof}
Let $X_{\pm}$ be given in (\textbf{H2}) and $P_{+,0,-}$ be the projections
associated to the decomposition $X=X_{+}\oplus\ker L\oplus X_{-}$. For any
$u\in X$, $P_{-}u=0$ would imply $u\in\ker L\oplus X_{+}$ and thus $\langle
Lu,u\rangle\geq0,\ $so $u\notin N$. Therefore, $P_{-}:N\rightarrow X_{-}$ is
injective and in turn it implies $\dim N\leq\dim X_{-}$.
\end{proof}

In order to proceed we have to introduce some notations. Given a closed
subspace $Y\subset X$, let $i_{Y}:Y\rightarrow X$ be the embedding and then
$i_{Y}^{\ast}:X^{\ast}\rightarrow Y^{\ast}$. Define
\begin{equation}%
\begin{split}
&  L_{Y}=i_{Y}^{\ast}Li_{Y}:Y\rightarrow Y^{\ast},\\
&  Y^{\perp_{L}}=\ker(i_{Y}^{\ast}L)=\{u\in X\mid\langle Lu,i_{Y}%
v\rangle=\langle Lu,v\rangle=0,\;\forall v\in Y\},
\end{split}
\label{E:L_y1}%
\end{equation}
which satisfy
\begin{equation}
L_{Y}^{\ast}=L_{Y}\;\text{ and }\langle L_{Y}u,v\rangle=\langle Lu,v\rangle
,\;\forall u,v\in Y. \label{E:L_Y2}%
\end{equation}
The following is a simple technical lemma.

\begin{lemma}
\label{L:non-degeneracy} Assume (\textbf{H1-3}). Let $Y\subset X$ be a closed subspace.

\begin{enumerate}
\item Suppose the quadratic form $\langle L\cdot, \cdot\rangle$ is
non-degenerate (in the sense of \eqref{E:non-degeneracy-def}) on $Y$, then $X=
Y \oplus Y^{\perp_{L}}$.

\item Assume $\dim\ker L < \infty$ and $\ker L_{Y}=\{0\}$, then $\langle
L\cdot, \cdot\rangle$ is non-degenerate on $Y$.

\item If $X= \ker L \oplus Y$ then $\langle L\cdot, \cdot\rangle$ is
non-degenerate on $Y$.
\end{enumerate}
\end{lemma}

\begin{proof}
We first notice that $L_{Y}$ being an isomorphism implies $Y\cap Y^{\perp_{L}}
= \{0\}$. For any $u \in X$, let
\[
u_{1} = L_{Y}^{-1} i_{Y}^{*} L u \in Y \; \Longrightarrow\; \langle Lu_{1} - L
u, v\rangle= 0, \; \forall v\in Y,
\]
and thus $u_{2} = u-u_{1} \in Y^{\perp_{L}}$ which implies $X = Y \oplus
Y^{\perp_{L}}$.

In order prove the second statement, from the standard argument, it suffices
to show that
\begin{equation}
\inf_{u\in Y\backslash\{0\}}\sup_{v\in Y\backslash\{0\}}\frac{|\langle
Lu,v\rangle|}{\Vert u\Vert\Vert v\Vert}>0.\label{E:non-degeneracy}%
\end{equation}
According to Remark \ref{R:assumptions} and the assumption of the lemma, there
exist closed subspaces $X_{\leq0}$ and $X_{+}$ such that the decomposition
$X=X_{\leq0}\oplus X_{+}$ is orthogonal with respect to both $(\cdot,\cdot)$
and $\langle L\cdot,\cdot\rangle$, $\dim X_{\leq0}<\infty$, $\langle
Lu,u\rangle\leq0$ for all $u\in X_{\leq0}$, and for some $\delta>0$, $\langle
Lu,u\rangle\geq\delta\Vert u\Vert^{2}$ for all $u\in X_{+}$. This splitting is
associated to the orthogonal projections $\mathcal{P}_{\leq0,+}:X\rightarrow
X_{\leq0,+}$. Let $Y_{+}=Y\cap X_{+}$ and
\[
Y_{1}=\{u\in Y\mid\langle Lu,v\rangle=0,\ \forall v\in Y_{+}\}.
\]
Clearly, $Y_{+}$ and $Y_{1}$ are both closed subspaces of $Y$. Much as in the
first statement, using the uniform positive definiteness of $\langle
Lu,u\rangle$ on $Y_{+}$, we have $Y=Y_{+}\oplus Y_{1}$ via
\[
u=u_{+}+(u-u_{+}),\text{ where }u_{+}=L_{Y_{+}}^{-1}i_{Y_{+}}^{\ast}Lu\in
Y_{+},\quad\forall u\in Y.
\]

For any $u_{1}\in Y_{1}\backslash\{0\}$, let $x_{\leq0,+}=\mathcal{P}%
_{X_{\leq0,+}}u_{1}$ and we have $u_{1}=x_{\leq0}+x_{+}$. Since $\mathcal{P}%
_{X_{\leq0}}u_{1}=0$ would imply $u_{1}\in X_{+}\cap Y=Y_{+}$ contradictory to
$Y=Y_{+}\oplus Y_{1}$, we obtain that the linear mapping $\mathcal{P}%
_{X_{\leq0}}|_{Y_{1}}$ is one-to-one. Therefore, $\dim Y_{1}<\infty$. From the
definition of $Y_{1}$, if $u_{1}\in Y_{1}\backslash\{0\}$ satisfies that
$\langle Lu_{1},v\rangle=0$ for all $v\in Y_{1}$, we would have $L_{Y}u_{1}=0$
which contradicts the assumption $\ker L_{Y}=\{0\}$. Therefore, $L_{Y}%
|_{Y_{1}}$ defines an isomorphism from $Y_{1}$ to $Y_{1}^{\ast}$ as $\dim
Y_{1}<\infty$ and thus there exists $\delta^{\prime}>0$ such that for any
$u_{1}\in Y_{1}\backslash\{0\}$, there exists $v\in Y_{1}$ such that $\langle
L_{Y}u_{1},v\rangle\geq\delta^{\prime}\Vert u_{1}\Vert\Vert v\Vert$.

Consider any $u=u_{1}+u_{+}\in Y$. If $\Vert u_{1}\Vert\geq\Vert u_{+}\Vert$,
there exists $v\in Y_{1}$ such that
\[
\langle Lu,v\rangle=\langle Lu_{1},v\rangle\geq\delta^{\prime}\Vert u_{1}%
\Vert\Vert v\Vert\geq\frac{\delta^{\prime}}{2}\Vert u\Vert\Vert v\Vert.
\]
If $\Vert u_{+}\Vert\geq\Vert u_{1}\Vert$, then let $v=u_{+}$ and we have
\[
\langle Lu,v\rangle=\langle Lu_{+},u_{+}\rangle\geq\delta\Vert u_{+}\Vert
^{2}\geq\frac{\delta}{2}\Vert u\Vert\Vert v\Vert.
\]
Therefore, \eqref{E:non-degeneracy} is obtained and the second statement is proved.

Finally we prove the last statement. We first show the non-degeneracy of
$\langle L\cdot,\cdot\rangle$ on $X_{+}\oplus X_{-}$ though a standard
procedure. The bounded symmetric quadratic form $\langle L\cdot,\cdot\rangle$
on $X_{+}\oplus X_{-}$ induces bounded linear operators
\[
L_{\alpha,\beta}=i_{X_{\alpha}}^{\ast}Li_{X_{\beta}}:X_{\beta}\rightarrow
X_{\alpha}^{\ast},\quad\alpha,\beta\in\{+,-\}.
\]
Since $L_{++}$ and $-L_{--}$ are both symmetric and bounded below, thus
isomorphic, and $L_{+-}=L_{-+}^{\ast}$, so the same are true for
\[
L_{++}-L_{+-}L_{--}^{-1}L_{-+}\;\text{ and }\;-(L_{--}-L_{-+}L_{++}^{-1}%
L_{+-}).
\]
It is easy to verify that
\[
L^{-1}=(L_{++}-L_{+-}L_{--}^{-1}L_{-+})^{-1}i_{X_{+}}^{\ast}+(L_{--}%
-L_{-+}L_{++}^{-1}L_{+-})^{-1}i_{X_{-}}^{\ast}%
\]
is a bounded operator from $(X_{+}\oplus X_{-})^{\ast}$ to $X_{+}\oplus X_{-}%
$. In general, if $X=\ker L\oplus Y$, there exists an isomorphism
$T:X_{-}\oplus X_{+}\rightarrow\ker L$ such that $Y=\graph (T)$. The
non-degeneracy of $\langle L\cdot,\cdot\rangle$ on $Y$ follows immediately
from its non-degeneracy on $X_{-}\oplus X_{+}$. The proof of the lemma is complete.
\end{proof}

\begin{remark}
\label{R:non-degeneracy} The first statement in the lemma holds actually for
any closed subspace $Y\subset X$ as long as $\langle L\cdot,\cdot\rangle$ is
non-degenerate on $Y$. The finite dimensionality assumption on $\ker L\ $is
essential for the second statement in the above lemma. A counter example is
\[
X=l^{2}\oplus l^{2},\;L=I\oplus0,\;Y=\big\{(\{x_{n}\},\{y_{n}\})\in X\mid
x_{n}=\frac{1}{n}y_{n}\big\},
\]
for which $\dim\ker L=\infty,\ n^{-}\left(  L\right)  =0,\ \ker L|_{Y}%
=\left\{  0\right\}  $, but $\langle L\cdot,\cdot\rangle$ is not
non-degenerate on $Y\ $in the sense of (\ref{E:non-degeneracy-def}).
\end{remark}

The next lemma will allow us to decompose equation \eqref{eqn-hamiltonian}.

\begin{lemma}
\label{L:decomJ} Suppose $X_{1,2} \subset X$ are closed subspaces satisfying
$X= X_{1} \oplus X_{2}$. Let $P_{1,2}: X\to X_{1,2}$ be the associated
projections, which imply $P_{1,2}^{*}: X_{1,2}^{*} \to X^{*}$, and
\[
J_{jk} = P_{j} JP_{k}^{*}: D(J_{jk}) \to X_{j}, \; D(J_{jk}) = (P_{k}%
^{*})^{-1} \big( D(J) \cap P_{k}^{*} X_{k}^{*}\big), \; j,k=1,2.
\]

\begin{enumerate}
\item If $\ker i_{X_{2}}^{*} \subset D(J)$, then $J_{11}$ and $J_{21}$ are
bounded operators defined on $X_{1}^{*}$, $J_{11}^{*} =- J_{11}$, $J_{22} =
-J_{22}^{*}$, and $J_{12}^{*}=-J_{21}$, and $J_{12}$ can be extended to the
bounded operator $-J_{21}^{*}= J_{12}^{**}$ defined on $X_{2}^{*}$.

\item If $\langle Lu_{1},u_{2}\rangle=0$, for all $u_{j}\in X_{j}$, $j=1,2$,
then $LX_{j}\subset\ker i_{X_{3-j}}^{\ast}$, $L_{X_{1,2}}$ satisfy
(\textbf{H2}) on $X_{1,2}$, $n^{-}(L)=n^{-}(L_{X_{1}})+n^{-}(L_{X_{2}})$, and
$\ker L=\ker L_{X_{1}}\oplus\ker L_{X_{2}}$.

\item Assume $\langle Lu_{1}, u_{2}\rangle=0$, for all $u_{j} \in X_{j}$,
$j=1,2$, and $\ker i_{X_{2}}^{*} \subset D(J)$, then the combinations $(X_{j},
L_{X_{j}}, J_{jj})$, $j=1,2$, satisfy (\textbf{H1-3}).
\end{enumerate}
\end{lemma}

\begin{proof}
For $j=1,2$, define $\tilde{X}_{j}^{\ast}$ as
\begin{equation}
\tilde{X}_{j}^{\ast}=P_{j}^{\ast}X_{j}^{\ast}=\ker i_{X_{3-j}}^{\ast}=\{f\in
X^{\ast}\mid\langle f,u\rangle=0,\,\forall u\in X_{3-j}\}\subset X^{\ast
}.\label{E:tildeX*}%
\end{equation}
Clearly, it holds
\begin{equation}
i_{X_{1}}P_{1}+i_{X_{2}}P_{2}=I_{X},\quad P_{1}^{\ast}i_{X_{1}}^{\ast}%
+P_{2}^{\ast}i_{X_{2}}^{\ast}=I_{X^{\ast}},\quad X^{\ast}=\tilde{X}_{1}^{\ast
}\oplus\tilde{X}_{2}^{\ast}.\label{E:splitX*}%
\end{equation}

Assume $\tilde{X}_{1}^{\ast}=P_{1}^{\ast}X_{1}^{\ast}\subset D(J)$. The Closed
Graph Theorem implies that the closed operator $JP_{1}^{\ast}:X_{1}^{\ast
}\rightarrow X$ is actually bounded, and thus $J_{11}$ and $J_{21}$ are
bounded as well. The property $J_{11}^{\ast}=-J_{11}$ is obvious from
$J^{\ast}=-J$ and the boundedness of $J_{11}$. We also obtain from this
assumption and \eqref{E:splitX*} that $D(J)\cap\tilde{X}_{2}^{\ast}$ is dense
in $\tilde{X}_{2}^{\ast}$ and thus $J_{12}$ and $J_{22}$ are densely defined,
as $P_{j}^{\ast}:X_{j}^{\ast}\rightarrow\tilde{X}_{j}^{\ast}$ is an
isomorphism. It remains to prove $J_{12}^{\ast}=-J_{21}$ and $J_{22}^{\ast
}=-J_{22}$.

Suppose $u=J_{12}^{\ast}g$, or equivalently, $g\in X_{1}^{\ast}$ and $u\in
X_{2}$ satisfy, $\forall f\in D(J_{12})\subset X_{2}^{\ast}$,
\begin{equation}
\langle P_{2}^{\ast}f,i_{X_{2}}u-i_{X_{1}}P_{1}JP_{1}^{\ast}g\rangle=\langle
f,u\rangle=\langle g,J_{12}f\rangle=\langle P_{1}^{\ast}g,JP_{2}^{\ast
}f\rangle,\label{E:dual1}%
\end{equation}
where we used $P_{j}i_{X_{j}}=id$ and $P_{3-j}i_{X_{j}}=0$ on $X_{j}$. For any
$h\in X_{1}^{\ast}$, we have
\[
\langle P_{1}^{\ast}h,i_{x_{2}}u-i_{X_{1}}P_{1}JP_{1}^{\ast}g\rangle=\langle
P_{1}^{\ast}g,JP_{1}^{\ast}h\rangle.
\]
Therefore, \eqref{E:splitX*} and \eqref{E:dual1} imply $u=J_{12}^{\ast}g$ is
equivalent to
\[%
\begin{split}
&  \langle\gamma,i_{X_{2}}u-i_{X_{1}}P_{1}JP_{1}^{\ast}g\rangle=\langle
P_{1}^{\ast}g,J\gamma\rangle,\quad\forall\gamma\in D(J)\\
\Longleftrightarrow &  i_{X_{2}}u-i_{X_{1}}P_{1}JP_{1}^{\ast}g=J^{\ast}%
P_{1}^{\ast}g=-JP_{1}^{\ast}g\\
\Longleftrightarrow &  u=-P_{2}JP_{1}^{\ast}g=-J_{21}g
\end{split}
\]
Therefore, $J_{12}^{\ast}=-J_{21}$.

Similarly, using the assumption $\tilde{X}_{1}^{\ast}\subset D(J)$, one can
prove $u=J_{22}^{\ast}g\in X_{2}$, $g\in X_{2}^{\ast}$, if and only if
\[
i_{X_{2}}u+i_{X_{1}}J_{21}^{\ast}g=J^{\ast}P_{2}^{\ast}g\Longleftrightarrow
u=-P_{2}JP_{2}^{\ast}g=-J_{22}g.
\]
Therefore, we obtain $J_{22}^{\ast}=-J_{22}$.

Assume $\langle Lu_{1},u_{2}\rangle=0$, for all $u_{j}\in X_{j}$, $j=1,2$. As
a direct consequence, we have $LX_{j}\subset\tilde{X}_{j}^{\ast}$, which,
along with \eqref{E:splitX*}, immediately implies
\[
L=P_{1}^{\ast}L_{X_{1}}P_{1}+P_{2}^{\ast}L_{X_{2}}P_{2},\quad P_{j}^{\ast
}L_{X_{j}}P_{j}(X)\subset\tilde{X}_{j},
\]
which in turn yield
\[
\ker L=\ker L_{X_{1}}\oplus\ker L_{X_{2}},\quad\ker L_{X_{j}}=X_{j}\cap\ker
L,\;j=1,2.
\]
Let
\[
Y_{1,2}=\{u\in X_{1,2}\mid(u,v)=0,\,\forall v\in\ker L_{X_{1,2}}%
\},\;Y=Y_{1}\oplus Y_{2},
\]
which implies%
\[
X=Y\oplus\ker L=Y_{1}\oplus Y_{2}\oplus\ker L,
\]
and
\[
\langle Ly_{j},y_{1}^{\prime}+y_{2}^{\prime}+u\rangle=\langle Ly_{j}%
,y_{j}^{\prime}\rangle=\langle L_{Y_{j}}y_{j},y_{j}^{\prime}\rangle,
\]
for any $y_{j},y_{j}^{\prime}\in Y_{j}$, $j=1,2$, and $u\in\ker L$. Let
$P_{Y_{1,2,0}}$ be the projections associated to this decomposition, then we
have
\[
L(Y_{j})\subset\tilde{Y}_{j}^{\ast}\triangleq P_{Y_{j}}^{\ast}Y_{j}^{\ast
}=\ker i_{\ker L\oplus Y_{3-j}}^{\ast},\ \ \ j=1,2.
\]
Assumption (\textbf{H2}) implies that $L|_{Y}:Y\rightarrow R(L)$ is an
isomorphism to the closed subspace $R(L)\subset X^{\ast}$. Therefore,
$L(Y_{1,2})\subset\tilde{Y}_{1,2}^{\ast}$ are closed subspaces and
$L|_{Y_{1,2}}:Y_{1,2}\rightarrow L(Y_{1,2})$ are isomorphisms. It implies that
$L_{Y_{1,2}}$ are isomorphisms from $Y_{1,2}$ to closed subspaces $L_{Y_{1,2}%
}(Y_{1,2})\subset Y_{1,2}^{\ast}$. Due to their boundedness and symmetry, we
obtain that $L_{Y_{1,2}}Y_{1,2}$ is equal to the orthogonal complement of
$\ker L_{Y_{1,2}}^{\ast}=\ker L_{Y_{1,2}}=\{0\}$. So $L_{Y_{1,2}}%
:Y_{1,2}\rightarrow Y_{1,2}^{\ast}$ are isomorphisms, which induce bounded
non-degenerate symmetric quadratic forms on $Y_{1,2}$. From the standard
theory on symmetric quadratic forms, $Y_{j}$, $j=1,2$, can be split into
$Y_{j}=Y_{j+}\oplus Y_{j-}$, where closed subspaces $Y_{j\pm}$ are orthogonal
with respect to both $(\cdot,\cdot)$ and $\langle L\cdot,\cdot\rangle$.
Moreover, there exists $\delta>0$ such that
\[
\pm\langle L_{X_{j}}u,u\rangle=\pm\langle Lu,u\rangle\geq\delta\Vert
u\Vert^{2},\;\forall u\in Y_{j\pm}.
\]
This proves that $X_{j}$ satisfies (\textbf{H2}) with
\[
X_{j}=Y_{J-}\oplus\ker L_{X_{j}}\oplus Y_{j+},\ \ \ j=1,2.
\]
Finally, since $X=X_{1}\oplus X_{2}$, there exists $C>0$ such that,
\[
\Vert u_{1}\Vert^{2}+\Vert u_{2}\Vert^{2}\leq C\Vert u_{1}+u_{2}\Vert
^{2},\ \ \ \forall\ u_{1,2}\in X_{1,2}.
\]
Therefore, the splitting
\[
X=(Y_{1-}\oplus Y_{2-})\oplus\ker L\oplus(Y_{1+}\oplus Y_{2+})
\]
satisfies the properties in (\textbf{H2}), which implies $n^{-}(L)=n^{-}%
(L_{X_{1}})+n^{-}(L_{X_{2}})$.

Finally, assume $\langle Lu_{1},u_{2}\rangle=0$, for all $u_{j}\in X_{j}$,
$j=1,2$, and $P_{1}^{\ast}X_{1}^{\ast}\subset D(J)$. To complete the proof of
the lemma, we only need to show that (\textbf{H3}) is satisfied by
$(X_{j},L_{X_{j}},J_{jj})$, $j=1,2$. This is obvious for $j=1$, as $J_{11}$ is
a bounded operator, and thus we only need to work on $j=2$. Let $X_{\pm
}\subset X$ be the closed subspaces assumed in (\textbf{H2-3}) and
$Z=X_{-}\oplus X_{+}$. Since $X=\ker L\oplus Z=\ker L\oplus Y$, $Z$ can be
represented as the graph of a bounded linear operator from $Y$ to $\ker L$. As
$\ker L=\ker L_{X_{1}}\oplus\ker L_{X_{2}}$ and $Y=Y_{1}\oplus Y_{2}$, there
exist bounded operators $S_{jk}:Y_{k}\rightarrow\ker L_{X_{j}}$ such that
\[
Z=\{y_{1}+y_{2}+\Sigma_{j,k=1}^{2}S_{jk}y_{k}\mid y_{1,2}\in Y_{1,2}\}.
\]
We will first show
\begin{equation}
W\triangleq\{f\in X_{2}^{\ast}\mid\langle f,u\rangle=0,\,u\in Z_{2}\}\subset
D(J_{22}),\label{E:tem1}%
\end{equation}
where
\[
Z_{2}=\{y_{2}+S_{22}y_{2}\mid y_{2}\in Y_{2}\}\subset X_{2}.
\]
Trivially extend $S_{jk}$ to be an operator from $X_{k}$ to $\ker L_{X_{j}%
}\subset X_{j}$ via
\[
S_{jk}(y_{k}+v_{k})=S_{jk}y_{k},\ \ \forall\ y_{k}\in Y_{k},\ v_{k}\in\ker
L_{X_{k}}.
\]
It leads to $S_{jk}S_{kl}=0$, $\forall j,k,l=1,2$. Given any $f\in W\subset
X_{2}^{\ast}$, one may compute, for any
\[
u=y_{1}+y_{2}+\Sigma_{j,k=1}^{2}S_{jk}y_{k}\in Z,
\]
using the definition of $W$, and the property of the extensions of $S_{jk}$,
\[%
\begin{split}
\langle P_{2}^{\ast}f-P_{1}^{\ast}S_{21}^{\ast}f,u\rangle= &  \langle
f,y_{2}+S_{21}y_{1}+S_{22}y_{2}\rangle-\langle S_{21}^{\ast}f,y_{1}%
+S_{11}y_{1}+S_{12}y_{2}\rangle\\
= &  \langle f,S_{21}y_{1}\rangle-\langle f,S_{21}y_{1}+S_{21}S_{11}%
y_{1}+S_{21}S_{12}y_{2}\rangle=0.
\end{split}
\]
Therefore, (\textbf{H3}) implies $P_{2}^{\ast}f-P_{1}^{\ast}S_{21}^{\ast}f\in
D(J)$. Since we assume $P_{1}^{\ast}X_{1}^{\ast}\subset D(J)$, we obtain
$P_{2}^{\ast}f\in D(J)$ and thus $f\in D(J_{22})$ which proves \eqref{E:tem1}.

Since $y_{2}\rightarrow y_{2}+S_{22}y_{2}$ is an isomorphism from $Y_{2}$ to
$Z_{2}$,
\[
\langle L(y_{2}+S_{22}y_{2}),y_{2}^{\prime}+S_{22}y_{2}^{\prime}%
\rangle=\langle Ly_{2},y_{2}^{\prime}\rangle,
\]
and $L_{Y_{2}}$ is isomorphic, we have $\langle L\cdot,\cdot\rangle$ is
non-degenerate on $Z_{2}$ and $L_{Z_{2}}$ is also an isomorphism. Therefore,
there exist closed subspaces $X_{2\pm}\subset Z_{2}$ and $\delta>0$ such that
$Z_{2}=X_{2-}\oplus X_{2+}$, $\dim X_{2-}=n^{-}(L_{X_{2}})$, and $\pm\langle
L_{X_{2}}u,u\rangle\geq\delta\Vert u\Vert^{2}$, for any $u\in X_{2\pm}$. It
along with \eqref{E:tem1} and $X_{2}=Z_{2}\oplus\ker L_{X_{2}}$ completes the
proof of the lemma.
\end{proof}

\begin{remark}
Under assumptions $\langle Lu_{1},u_{2}\rangle=0$, for all $u_{j}\in X_{j}$,
$j=1,2$, and $P_{1}^{\ast}X_{1}^{\ast}\subset D(J)$, $(X_{j},L_{X_{j}}%
,J_{jj})$, $j=1,2$, satisfies the same hypothesis (\textbf{H1-H3}) as
$(X,L,J)$ and $n^{-}(L)=n^{-}(L_{X_{1}})+n^{-}(L_{X_{2}})$. Moreover, it is
easily verified based on these assumptions that $J_{jj}L_{X_{j}}%
=P_{j}JL|_{X_{j}}$. Therefore, this lemma would often be applied to reduce the
problem to subspaces when $JL(X_{1})\subset X_{1}$, which implies $JL$ has
certain upper triangular structure.
\end{remark}

\begin{corollary}
\label{C:decomJ} $LJ: D(J) \to X^{*}$ is a closed operator and consequently
$(JL)^{*} = -LJ$.
\end{corollary}

\begin{proof}
Let $X_{\pm}$ and $\ker L$ satisfy the requirements in (\textbf{H2-3}) and let
$X_{1}=\ker L$ and $X_{2}=X_{-}\oplus X_{+}$. Clearly, we have, $L_{X_{1}}=0$,
$\langle Lu_{1},u_{2}\rangle=0$, for all $u_{j}\in X_{j}$, $j=1,2$, and
$P_{1}^{\ast}X_{1}^{\ast}\subset D(J)$ due to (\textbf{H3}). Using
\[
i_{X_{1}}P_{1}+i_{X_{2}}P_{2}=I_{X},\quad Li_{X_{1}}=0,\quad i_{X_{1}}^{\ast
}L=0,
\]
$LJ$ can be rewritten in this decomposition
\[
LJ\gamma=P_{2}^{\ast}L_{X_{2}}J_{21}i_{X_{1}}^{\ast}\gamma+P_{2}^{\ast
}L_{X_{2}}J_{22}i_{X_{2}}^{\ast}\gamma,\quad\forall\gamma\in X^{\ast},
\]
which is equivalent to using the blockwise decomposition of $J$ and $L$. Since
$J_{21}$ is continuous, $P_{2}^{\ast}L_{X_{2}}J_{21}i_{X_{1}}^{\ast}$ is
continuous too. Moreover, the facts that $L_{X_{2}}:X_{2}\rightarrow
X_{2}^{\ast}$ is an isomorphism, $P_{2}^{\ast}$ has a continuous left inverse
$i_{X_{2}}^{\ast}$ as $P_{2}i_{X_{2}}=I_{X_{2}}$, along with the closedness of
$J_{22}$ imply that $P_{2}^{\ast}L_{X_{2}}J_{22}$ and thus $P_{2}^{\ast
}L_{X_{2}}J_{22}i_{X_{2}}^{\ast}$ is a closed operator. Therefore, $LJ$ is closed.

Since $(LJ)^{*}= JL$ is densely defined and thus $(LJ)^{**} = - (JL)^{*}$ is
well defined. The closeness of $LJ$ implies $LJ=(LJ)^{**} = - (JL)^{*}$.
\end{proof}

\begin{remark}
\label{R:inf-dim-nage-1} We would like to point out that, in the proof Lemma
\ref{L:decomJ} and Corollary \ref{C:decomJ}, we do not use the assumption that
$n^{-}(L)<\infty$. Therefore, they actually hold even if $n^{-}(L)=\infty$
except that $n^{-}(L_{X_{1,2}})$ might be $\infty$.
\end{remark}

The following is a simple, but useful, technical lemma.

\begin{lemma}
\label{L:decom1} There exist closed subspaces $X_{\pm}\subset X$ satisfying
the properties in (\textbf{H2-3}) and in addition,

\begin{enumerate}
\item $X=X_{0}\oplus X_{-}\oplus X_{+}$ is a $L$-orthogonal splitting with
associated projections $P_{0,\pm}$, where $X_{0}=\ker L$;

\item $L_{X_{\pm}} : X_{\pm}\to X_{\pm}^{*}$ are isomorphic; and

\item $\tilde X_{0, -}^{*} \subset D(J)$ and $D(J) \cap\tilde X_{+}^{*}$ is
dense in $\tilde X_{+}^{*}$, where $\tilde X_{\pm, 0}^{*} \triangleq P_{\pm,
0}^{*} X_{\pm, 0}^{*}$ (see \eqref{E:L_y1} and \eqref{E:tildeX*}).
\end{enumerate}
\end{lemma}

\begin{proof}
Let $Y_{\pm}\subset X$ be closed subspaces satisfying hypothesis
(\textbf{H2-3}). Let $Y=Y_{-}\oplus Y_{+}$, $P:X\rightarrow Y$ be the
projection associated to the decomposition $X=X_{0}\oplus Y$, $\tilde{X}%
_{0}^{\ast}=(I-P)^{\ast}X_{0}^{\ast}$, and $\tilde{Y}^{\ast}=P^{\ast}Y^{\ast}%
$, which are closed subspaces. According to (\textbf{H3}), we have $\tilde
{X}_{0}^{\ast}\subset D(J)$. Consequently, $\tilde{Y}^{\ast}\cap D(J)$ is
dense in $\tilde{Y}^{\ast}$ as $X^{\ast}=\tilde{X}_{0}^{\ast}\oplus\tilde
{Y}^{\ast}$. Our assumptions imply $L_{Y}:Y\rightarrow\tilde{Y}^{\ast}$ is an
isomorphism, which induces a bounded symmetric quadratic form on $Y$ with
Morse index equal to $n^{-}(L)$. Therefore, there exists a closed subspace
$X_{-}\subset Y$ such that $\dim X_{-}=n^{-}(L)$, $L(X_{-})\subset D(J)$, and
$\langle Lu,u\rangle\leq-\delta\Vert u\Vert^{2}$, for all $u\in X_{-}$. Let
\[
X_{+}=\{u\in Y\mid\langle Lu,v\rangle=0,\,\forall v\in X_{-}\}.
\]
Since $L$ is uniformly negative on $X_{-}$, Lemma \ref{L:non-degeneracy}
implies the $L$-orthogonal splitting $Y=X_{-}\oplus X_{+}$ and thus the
$L$-orthogonal decomposition $X=X_{0}\oplus X_{-}\oplus X_{+}$ as well. The
rest of the proof follows easily from the facts that $L_{Y}$ is isomorphic,
$X^{\ast}=\tilde{X}_{0}^{\ast}\oplus\tilde{X}_{-}^{\ast}\oplus\tilde{X}%
_{+}^{\ast}$, $\dim X_{-}=n^{-}(L)$, $\tilde{X}_{0}^{\ast}\subset D(J)$, and
$\tilde{X}_{-}^{\ast}=L(X_{-})\subset D(J)$.
\end{proof}

\begin{remark}
\label{R:orthogonal-d-2} Under assumption \eqref{E:orthogonal-d}, it is
possible to choose $X_{\pm}$ such that $X_{+}\oplus X_{-}=(\ker L)^{\perp}$
satisfies all properties in Lemma \ref{L:decom1}, where $(\ker L)^{\perp}$ is
defined in \eqref{E:orthogonal-c}. In fact, let $Y=(\ker L)^{\perp}$, then
\eqref{E:orthogonal-d} implies that the splitting $X=\ker L\oplus Y$ satisfies
all assumptions in Lemma \ref{L:decomJ}. The rest of the construction of
$X_{\pm}\subset Y=(\ker L)^{\perp}$ follows in exactly the same procedure as
in the proof of Lemma \ref{L:decom1}.
\end{remark}

In order to establish the well-posedness of the linear equation in the next,
we start with the following lemma.

\begin{lemma}
\label{L:decomJL} There exists an equivalent inner product $(\cdot, \cdot
)_{L}$ on $X$, a linear operator $A: D(JL) \to X$ which is anti-self-adjoint
with respect to $(\cdot, \cdot)_{L}$, and a bound linear operator $B: X \to X$
such that $JL= A+ B$.
\end{lemma}

\begin{proof}
Let $X=X_{-}\oplus X_{0}\oplus X_{+}$ be a decomposition as given in Lemma
\ref{L:decom1} with $X_{0}=\ker L$. Let
\[
L_{\pm}=\pm P_{\pm}^{\ast}i_{X_{\pm}}^{\ast}Li_{X_{\pm}}P_{\pm}:X\rightarrow
X^{\ast},
\]
which satisfy
\[
L_{\pm}^{\ast}=L_{\pm},\quad L=L_{+}-L_{-},\quad\langle L_{\pm}u,v\rangle
=\pm\langle L_{X_{\pm}}u,v\rangle=\pm\langle Lu,v\rangle,\,\forall u,v\in
X_{\pm}.
\]
Let $R:X\rightarrow X^{\ast}$ be the isomorphism corresponding to
$(\cdot,\cdot)$ through the Riesz Representation Theorem and
\[
L_{0}=P_{0}^{\ast}i_{X_{0}}^{\ast}Ri_{X_{0}}P_{X_{0}}:X\rightarrow X^{\ast
}\ \longleftrightarrow\ \langle L_{0}u,v\rangle=(P_{0}u,P_{0}v).
\]
From Lemma \ref{L:decom1} and assumptions (\textbf{H2-3}), it is easy to
verify that
\[
(u,v)_{L}\triangleq\langle(L_{+}+L_{-}+L_{0})u,v\rangle=\langle L_{+}%
u,v\rangle+\langle L_{-}u,v\rangle+(P_{0}u,P_{0}v)
\]
is uniformly positive and defines an equivalent inner product on $X$. Let
\[
A=J(L_{+}+L_{-}+L_{0})=JL+2JL_{-}+JL_{0}\triangleq JL-B.
\]
Since $P_{0,-}^{\ast}X_{0,-}^{\ast}\subset D(J)$, the Closed Graph Theorem
implies that $B$ is bounded. If $\dim\ker L<\infty$, $B$ is obviously of
finite rank. The proof of the lemma is complete.
\end{proof}

A direct consequence of this lemma is the well-posedness of equation
\eqref{eqn-hamiltonian} which follows from the standard perturbation theory of semigroups.

\begin{proposition}
\label{P:well-posedness} $JL$ generates a $C^{0}$ group $e^{tJL}$ of bounded
linear operators on $X$.
\end{proposition}

\noindent\textbf{Complexification.} For considerations where complex
eigenvalues are involved, we have to work with the standard complexification
of $X$ and the associated operators. Let
\[
\tilde{X}=\{x=x_{1}+ix_{2}\mid x_{1,2}\in X\}\;\text{ with }\;\overline
{x_{1}+ix_{2}}=x_{1}-ix_{2}%
\]
equipped with the complexified inner product
\[
(x_{1}+ix_{2},x_{1}^{\prime}+ix_{2}^{\prime})=(x_{1},x_{1}^{\prime}%
)+(x_{2},x_{2}^{\prime})+i\big((x_{2},x_{1}^{\prime})-(x_{1},x_{2}^{\prime
})\big).
\]
Instead of complexifying $L$ as a linear operator directly, it is much more
convenient for us to complexify its corresponding real symmetric quadratic
form $\langle Lu,v\rangle$ into a complex Hermitian symmetric form
\begin{align}
&  \mathcal{B}(x_{1}^{\prime}+ix_{2}^{\prime},x_{1}+ix_{2})=\langle\tilde
{L}(x_{1}+ix_{2}),(x_{1}^{\prime}+ix_{2}^{\prime})\rangle\nonumber\\
=  &  \langle Lx_{1},x_{1}^{\prime}\rangle+\langle Lx_{2},x_{2}^{\prime
}\rangle+i\big(\langle Lx_{1},x_{2}^{\prime}\rangle-\langle Lx_{2}%
,x_{1}^{\prime}\rangle\big), \label{E:complexify}%
\end{align}
for any $x_{1,2},\ x_{1,2}^{\prime}\in X$. Accordingly $L$ is complexified to
a (anti-linear) mapping $\tilde{L}$ from $\tilde{X}$ to $\tilde{X}^{\ast}$
satisfying
\begin{equation}
\tilde{L}(cx+c^{\prime}x^{\prime})=\bar{c}\tilde{L}x+\bar{c}^{\prime}\tilde
{L}x^{\prime}. \label{E:anti-linear}%
\end{equation}
A similar complexification can also be carried out for $J$ corresponding to a
Hermitian symmetric form on $\tilde{X}^{\ast}$ and a (anti-linear) mapping
from $\tilde{X}^{\ast}\rightarrow\tilde{X}$.

The composition $\tilde{J}\circ\tilde{L}$ of (anti-linear) mappings $\tilde
{J}$ and $\tilde{L}$ is a closed complex linear operator from $D(\tilde
{J}\tilde{L})\subset\tilde{X}$ to $\tilde{X}$. The fact that $\widetilde{JL}$
is anti-symmetric with respect to the Hermitian symmetric form $\langle
\tilde{L}u,v\rangle$, that is,
\begin{equation}
\langle\tilde{L}(\tilde{J}\tilde{L}u),v\rangle=-\langle\tilde{L}u,\tilde
{J}\tilde{L}v\rangle, \label{E:AntiHermitian}%
\end{equation}
will be used frequently. According to Corollary \ref{C:decomJ}, the dual
operator of $\tilde{J}\tilde{L}$ is given by
\begin{equation}
(\tilde{J}\tilde{L})^{\ast}=-\tilde{L}\tilde{J}. \label{E:dual-JL}%
\end{equation}
It is easy to verify that $\tilde{L}$, $\tilde{J}$, $\tilde{J}\tilde{L}$ and
$\tilde{L}\tilde{J}$ are real in the sense
\begin{equation}
\overline{\langle\tilde{L}x,x^{\prime}\rangle}=\langle\tilde{L}\bar
{x},\overline{x^{\prime}}\rangle,\;\overline{\langle f,\tilde{J}g\rangle
}=\langle\bar{f},\tilde{J}\bar{g}\rangle,\;\overline{\tilde{J}\tilde{L}%
x}=\tilde{J}\tilde{L}\bar{x},\;\overline{\tilde{L}\tilde{J}x}=\tilde{L}%
\tilde{J}\bar{x}. \label{E:real}%
\end{equation}
This implies that the spectrum of $\tilde{J}\tilde{L}$ and $\tilde{L}\tilde
{J}$ are symmetric about the real axis in the complex plane.

\begin{remark}
\label{R:real} In fact, on the complexified Hilbert space $\tilde X$ (or on
$\tilde X^{*}$), a linear operator or a Hermitian form is the complexification
of a (real) operator or a symmetric quadratic form on $X$ (or on $\tilde
X^{*}$) if and only if \eqref{E:real} holds.
\end{remark}

In the rest of the paper, with slight abuse of notations, we will write $X,
JL, \langle Lu, v\rangle$ also for their complexifications unless confusion
might occur.

\begin{remark}
\label{R:complexification} The linear group of bounded operators $e^{tJL}$
obtained in Proposition \ref{P:well-posedness} is also complexified
accordingly when needed.
\end{remark}

\begin{remark}
\label{R:splitting} Exactly the same statements in Lemma
\ref{L:non-degeneracy}, \ref{lemma-L-form-preserve}, \ref{L:InvariantSubS}
hold in the complexified framework.
\end{remark}

\begin{center}
{\Large Acknowledgement}
\end{center}

Zhiwu Lin is supported in part by NSF grants DMS-1411803 and DMS-1715201.
Chongchun Zeng is supported in part by a NSF grant DMS-1362507.

\bigskip

\bigskip

\bigskip

\end{document}